\documentclass[a4paper,11pt]{article}

\usepackage{amsmath,amsthm}
\usepackage{aligned-overset}
\usepackage{newtxtext}
\usepackage{newtxmath}
\usepackage{cases}

\usepackage{adjustbox}
\usepackage{booktabs}
\usepackage{enumitem}

\usepackage{tabularx}
\usepackage{multirow}

\usepackage{graphicx}
\usepackage{subcaption}
\usepackage[percent]{overpic}
\usepackage{epstopdf}

\usepackage[hmargin={26mm,26mm}, vmargin={30mm,35mm}]{geometry}

\usepackage[style=numeric-comp,
  maxbibnames=99,
  bibencoding=utf8,
  date=year,
giveninits=true]{biblatex}
\usepackage{xcolor}
\usepackage{hyperref}

\usepackage{nicefrac}
\usepackage{comment}

\usepackage{authblk}

\newcommand{\domain}{\Omega}
\newcommand{\domainBoundary}{\partial \domain}
\newcommand{\dimDomain}{2}

\newcommand{\scalarSol}{u}
\newcommand{\gradScalarSol}{\grad \scalarSol}
\newcommand{\loadTerm}{f}

\newcommand{\vectSol}{\boldsymbol{\tau}}

\newcommand{\grad}{\boldsymbol{\nabla}}
\newcommand{\divergence}{\grad \cdot}
\newcommand{\rot}{\operatorname{rot}}
\newcommand{\projL}[1]{\pi_{0,h}^{#1}}
\newcommand{\locProjL}[1]{\pi_{0,\element}^{#1}}
\newcommand{\vectProjL}[1]{\boldsymbol{\pi}_{0,h}^{#1}}
\newcommand{\locVectProjL}[1]{\boldsymbol{\pi}_{0,\element}^{#1}}
\newcommand{\vectProjF}[1]{\boldsymbol{\pi}_{\mathrm{F},h}^{#1}}
\newcommand{\locDofVect}{\boldsymbol{\mathrm{dof}}_{\element}}
\newcommand{\locInternalDofVect}{\boldsymbol{\mathrm{dof}}_{\mathring{\element}}}
\newcommand{\locBoundaryDofVect}{\boldsymbol{\mathrm{dof}}_{\partial{\element}}}
\newcommand{\restr}[2]{#1|_{#2}}
\newcommand{\fluxFunction}{\boldsymbol{\sigma}}

\newcommand{\sobIndex}{p}
\newcommand{\sobIndexConj}{{\sobIndex}'}
\newcommand{\sobIndexConjMin}{\underline{\sobIndexConj}}
\newcommand{\sobIndexConjMax}{\overline{\sobIndexConj}}

\newcommand{\polyDegree}{k}

\newcommand{\euNorm}[1]{| #1 |}
\newcommand{\norm}[2]{\|#2\|_{#1}}
\newcommand{\seminorm}[2]{|#2|_{#1}}
\newcommand{\tnorm}[2]{\lvert\!\lvert\!\lvert #2 \rvert\!\rvert\!\rvert_{#1}}

\newcommand{\discScalarNorm}[1]{\norm{\discScalarSpace}{#1}}

\newcommand{\discFullVectNorm}[1]{\tnorm{\discVectSpace}{#1}}
\newcommand{\locDiscFullVectNorm}[1]{\tnorm{\discVectSpace (\element)}{#1}}
\newcommand{\discDivFreeNorm}[1]{\tnorm{a_{h}}{#1}}
\newcommand{\locDiscDivFreeNorm}[1]{\tnorm{a_{h},\element}{#1}}

\newcommand{\Natural}{\mathbb{N}}
\newcommand{\Real}{\mathbb{R}}
\newcommand{\sobSpace}[3]{W^{#1,#2}(#3)}
\newcommand{\HSpace}[2]{H^{#1} (#2)}
\newcommand{\lebSpace}[2]{L^{#1} (#2)}
\newcommand{\vectLebSpace}[2]{\boldsymbol{L}^{#1} (#2)}

\newcommand{\vectSobSpace}[3]{\boldsymbol{W}^{#1,#2}(#3)}
\newcommand{\Hdiv}[1]{\boldsymbol{H}(\operatorname{div};#1)}
\newcommand{\Hrot}[1]{\boldsymbol{H}(\operatorname{rot};#1)}
\newcommand{\sobDivSpace}[2]{\boldsymbol{W}^{#1}(\mathrm{div},#2)}
\newcommand{\scalarSolSpace}{Q}
\newcommand{\vectSolSpace}{\boldsymbol{V}}
\newcommand{\discVectSpace}{\boldsymbol{V}_h^{\polyDegree}}
\newcommand{\discDivFreeSpace}{\boldsymbol{Z}_h^{\polyDegree}}
\newcommand{\discScalarSpace}{Q_h^{\polyDegree}}

\newcommand{\polySpace}[1]{\mathcal{P}^{#1}}
\newcommand{\vectPolySpace}[1]{\boldsymbol{\mathcal{P}}^{#1}}

\newcommand{\vectTestFun}{\boldsymbol{v}}
\newcommand{\vectFunOne}{\boldsymbol{w}}
\newcommand{\scalarTestFun}{q}
\newcommand{\discVectTestFun}{\boldsymbol{v}_h}
\newcommand{\discVectFunOne}{\boldsymbol{w}_h}
\newcommand{\discVectFunTwo}{\boldsymbol{z}_h}
\newcommand{\discVectTestFunTilda}{\tilde{\boldsymbol{v}}_h}
\newcommand{\discVectFunOneTilda}{\tilde{\boldsymbol{w}}_h}
\newcommand{\discVectFunTwoTilda}{\tilde{\boldsymbol{z}}_h}
\newcommand{\discScalarTestFun}{q_h}
\newcommand{\discScalarSol}{{\scalarSol}_h}
\newcommand{\discVectSol}{{\vectSol}_h}

\newcommand{\scalarPolyFun}[1]{p^{#1}}

\newcommand{\element}{E}
\newcommand{\edge}{e}
\newcommand{\edges}{\mathcal{E}_h}
\newcommand{\normalDomain}{\boldsymbol{n}_{\domainBoundary}}
\newcommand{\normalEdge}{\boldsymbol{n}_{\edge}}
\newcommand{\normalElement}{\boldsymbol{n}_{\partial \element}}
\newcommand{\edgeElement}{\mathcal{E}_{\element}}

\newcommand{\email}[1]{\href{mailto:#1}{#1}}

\usepackage[normalem]{ulem}
\newcounter{corr}
\definecolor{violet}{rgb}{0.580,0.,0.827}
\newcommand{\corr}[3]{\typeout{Warning: a correction remains in page \thepage}
  \stepcounter{corr}
  {\color{blue}\ifmmode\text{\,\sout{\ensuremath{#1}}\,}\else\sout{#1}\fi}
  {\color{red}#2}
  {\color{violet} #3}
}

\newtheorem{theorem}{Theorem}

\newtheorem{lemma}[theorem]{Lemma}
\newtheorem{corollary}[theorem]{Corollary}
\theoremstyle{remark}
\newtheorem{remark}[theorem]{Remark}
\theoremstyle{definition}
\newtheorem{assumption}[theorem]{Assumption}

\begin{document}
\title{A Mixed Virtual Element Method for the p-Laplace equation}
\author[1]{Kirubell B. Haile}
\affil[1]{Dipartimento di Matematica e Applicazioni, Universit\`{a} degli Studi di Milano Bicocca, Piazza dell’Ateneo Nuovo 1, 20126
Milano, Italy, \email{k.haile@campus.unimib.it}}
\author[2]{Giuseppe Vacca}
\affil[2]{Dipartimento di Matematica, Universit\`{a} degli Studi di Bari, Via Edoardo Orabona 4, 70125 Bari, Italy, \email{giuseppe.vacca@uniba.it}}
\maketitle
\begin{abstract}
  We introduce and analyze a mixed Virtual Element Method for the $p$-Laplace equation in a non-Hilbertian setting, covering the full range $p \in (1, \infty)$. 
  The discrete framework combines standard mixed Virtual Element spaces with a novel non-linear stabilization term designed to mimic the power-law structure of the continuous operator. 
  We establish discrete inf-sup stability under non-Hilbertian norms and rigorously prove the continuity and coercivity of the discrete form.
  This guarantees the well-posedness of the problem and allows us to derive a priori error estimates for the primal variable and the flux.
  A set of numerical tests supports the theoretical derivations.
\end{abstract}
%

\section{Introduction}\label{sec:introduction}
The $p$-Laplace equation arises in the modeling of a wide variety of advanced physical and engineering processes. Unlike the classical linear Laplacian, the non-linear nature of the $p$-Laplacian, governed by the parameter $\sobIndex \in (1, \infty)$, allows for the accurate modeling of phenomena where the material properties exhibit a power-law behavior.
Notable applications include the fluid mechanics of non-Newtonian media \cite{ATKINSON.JONES:1974}, such as ice sheet dynamics in glaciology \cite{GLOWINSKI.RAPPAZ:2003}. In this context, the pure $p$-Laplacian operator acts as the fundamental mathematical building block for more complex rheological models, including the $p$-Stokes and $p$-Navier-Stokes equations; see, e.g.,~\cite{BARRETT.LIU:1994, HIRN:2012, KALTENBACH.RUZICKA:23,BEIRAODAVEIGA.DIPIETRO.ETAL:2025a}. Other relevant applications encompass non-linear filtration and diffusion processes in porous media \cite{PHILIP:1961,VAZQUEZ:2006}, as well as solid mechanics problems, including non-linear elasticity and the macroscopic growth of sandpiles \cite{ANDREU.MAZON.ETAL:2009}. 
Recently, it has also found applications in image processing, data clustering, and game theory on graphs \cite{ELMOATAZ.DESQUESNES.ETAL:2017}.

The numerical approximation of the $p$-Laplace equation and the derivation of corresponding error bounds have a well-established history. 
Early works focusing on lowest-order conforming finite element methods can be traced back to \cite{GLOWINSKI.MARROCO:1975, BARRETT.LIU:1993}, while high-order formulations within the non-conforming framework were later explored using interior penalty discontinuous Galerkin methods \cite{BURMAN.ERN:2008}, a framework subsequently extended to more complex $p$-power-law transport phenomena \cite{BEIRAODAVEIGA.DIPIETRO.ETAL:2024}. However, in many physical contexts, the flux constitutes a key variable of interest; to achieve its independent approximation, several mixed formulations have been developed.
In \cite{FARHLOUL.MANOUZI:2000}, the first mixed finite element scheme for the $p$-Laplacian was analyzed using lowest-order Raviart–Thomas elements (see, e.g.,~\cite{BOFFI.BREZZI.ETAL:2013}). This approach was subsequently extended to other advanced frameworks, including Hybridizable Discontinuous Galerkin methods \cite{CREUSE.FARHLOUL.ETAL:2007, COCKBURN.SHEN:2016} and Local Discontinuous Galerkin formulations \cite{DIENING.KRONER.ETAL:2013}.
Parallel to these advancements, polytopal methods have gained significant traction due to their ability to naturally handle complex geometries, hanging nodes, and general polygonal meshes.
Early developments for non-linear operators included Discrete Duality Finite Volume schemes \cite{ANDREIANOV.BOYER.ETAL:2007} and Mimetic Finite Difference methods \cite{ANTONIETTI.BIGONI.ETAL:2015}. 
Later, advanced arbitrary-order frameworks were established; among these, Hybrid High-Order methods have been extended to the general class of Leray--Lions equations on polygonal grids \cite{DIPIETRO.DRONIOU:2017}.

Within this framework, the Virtual Element Method (VEM) has targeted a wide variety of problems since its introduction in \cite{BEIRAODAVEIGA.BREZZI.ETAL:2013} and subsequent extension to mixed formulations in \cite{BREZZI.FALK.ETAL:2014}.
Non-linear problems have been addressed from both primal and dual-mixed perspectives. 
Primal formulations include applications to quasilinear diffusion \cite{CANGIANI.CHATZIPANTELIDIS.ETAL:2020} and Cahn–Hilliard equations \cite{ANTONIETTI.BEIRAODAVEIGA.ETAL:2016}. 
On the dual-mixed side, recent works have focused on non-linear fluid mechanics models, such as Newtonian and quasi-Newtonian Stokes and Brinkman flows \cite{CACERES.GATICA:2017, CACERES.GATICA.ETAL:2018, CACERES.GATICA.ETAL:2017}.
To the best of our knowledge, a mixed virtual element formulation for the classical $p$-Laplace equation is still missing from the literature. This paper aims to introduce and analyze a mixed Virtual Element Method for this problem within a non-Hilbertian setting.

Relying on the mixed Virtual Element spaces introduced in \cite{BREZZI.FALK.ETAL:2014}, we design a specific discrete framework for the $p$-Laplace equation.
In particular, we introduce a non-linear stabilization term, inspired by \cite{ANTONIETTI.BEIRAODAVEIGA.ETAL:2024}, that mimics the non-linear behavior of the volume term.
From a theoretical perspective, validating this approach requires several key steps: we establish discrete inf-sup stability under non-Hilbertian norms and prove the continuity and coercivity of the new stabilization form.
Consequently, the well-posedness of the discrete problem is guaranteed, allowing us to derive a priori error estimates for both the primal variable and the flux, covering the entire range $\sobIndex>1$. Specifically, our convergence rates match the theoretical bounds established in \cite{FARHLOUL.MANOUZI:2000} for $\sobIndex>2$, while a comparison with conforming and non-conforming methods for $1<\sobIndex \leq 2$ is discussed in Section~\ref{sec:global.summary}.

The rest of the paper is organized as follows.
In Section~\ref{sec:setting}, we establish the notation, recall key preliminary definitions and results, and review the weak mixed formulation of the continuous problem.
We introduce the mixed VEM scheme in Section~\ref{sec:scheme} and analyze its well-posedness in Section~\ref{sec:stability}.
In Section~\ref{sec:a-priori}, we derive a priori error estimates for both the $1 < \sobIndex \leq 2$ and $\sobIndex > 2$ regimes.
Finally, in Section~\ref{sec:numerical.tests}, we present several numerical experiments to validate the theoretical results.

\section{Setting and continuous problem} \label{sec:setting}
%
\subsection{Notation and functional spaces} \label{sec:functional.spaces}
Let $Y$ be a generic bounded domain in $\Real^d$ with $d \in \{1, \dimDomain\}$.
For an integer $m \geq 0$ and a real number $1 \leq r \leq \infty$, we denote by $\sobSpace{m}{r}{Y}$ the usual Sobolev space on $Y$ equipped with the standard norm
$\norm{\sobSpace{m}{r}{Y}}{\cdot}$ and seminorm $\seminorm{\sobSpace{m}{r}{Y}}{\cdot}$.
For $m=0$, we obtain the Lebesgue space $\lebSpace{r}{Y} \coloneqq \sobSpace{0}{r}{Y}$ while, for $r=2$, we have the Hilbert space $\HSpace{m}{Y} \coloneqq \sobSpace{m}{2}{Y}$.
Spaces of vector-valued functions, as well as their elements, are denoted in boldface font. For example, $\vectTestFun \in \vectSobSpace{m}{r}{Y}$ denotes a vector-valued function.
We denote by $\sobDivSpace{r}{Y}$ the space of functions in $\vectLebSpace{r}{Y}$ with divergence in $\lebSpace{r}{Y}$, equipped with the norm $\norm{\sobDivSpace{r}{Y}}{\vectTestFun}^r
\coloneqq \norm{\vectLebSpace{r}{Y}}{\vectTestFun}^r + \norm{\lebSpace{r}{Y}}{\divergence \vectTestFun}^r$.
For $r=2$, we obtain $\Hdiv{Y} \coloneqq \sobDivSpace{2}{Y}$. Similarly, $\Hrot{Y}$ is the space of functions $\vectTestFun \in \vectLebSpace{2}{Y}$ with scalar rotor $\rot \vectTestFun
\coloneqq \partial_1 \vectTestFun_2 - \partial_2 \vectTestFun_1$ in $\lebSpace{2}{Y}$.
%
\subsection{Notation for inequalities} \label{sec:notation.inequalities}
To improve readability in the statements of theorems and lemmas, we write ``$A \lesssim B$ with a hidden constant depending only on $X, Y, \dots$'' for $A \leq C B$, where $C$ is a positive
constant depending only on $X, Y, \dots$.
Similarly, $A \gtrsim B$ stands for $B \lesssim A$, and $A \simeq B$ means that both $A \lesssim B$ and $B \lesssim A$ hold.

When necessary, to carefully track dependencies without naming every constant, we use generic constants $C(\delta, \varepsilon, \dots)$ depending only on the parameters in the parentheses.
The exact value of such constants may change at each occurrence. 
Similarly, given an arbitrary $\varepsilon > 0$, we may implicitly redefine it to absorb positive fixed constants without changing the notation.
%
\subsection{The quasilinear Laplace equation} \label{sec:continuous.problem}
Let $\domain \subset \Real^{\dimDomain}$ denote an open, bounded, and connected polygonal domain with Lipschitz boundary $\domainBoundary$.
For a fixed real parameter $ \sobIndex \in (1,\infty)$ and a given real-valued function $\loadTerm: \domain \to \Real$, we consider the quasilinear Laplace equation: Find $\scalarSol: \domain \to \Real $ such that
\begin{subequations}\label{eq:continuous.problem.primal}
  \begin{alignat}{2}
    - \divergence \fluxFunction_{\dimDomain,\sobIndex} (\gradScalarSol) &= \loadTerm
    \qquad
    && \text{in } \domain, \label{eq:continuous.problem.primal.momentum} \\
    \scalarSol &= 0
    \qquad
    && \text{on } \domainBoundary, \label{eq:continuous.problem.primal.bc}
  \end{alignat}
\end{subequations}
where $\fluxFunction_{\dimDomain,\sobIndex}$ denotes the diffusive flux function described below.
%
\subsection{Diffusive flux function} \label{sec:diffusive.flux.function}
For any real number $r > 1$ and any integer $n \in \Natural$, we set
\begin{equation}\label{eq:sobolev.indexes}
  r'\coloneqq \frac{r}{r-1} \in (1,\infty), \quad \underline{r} \coloneqq \min\{r,2\} \in (1,2], \quad \text{and} \quad \overline{r} \coloneqq \max\{r,2\} \in [2,\infty).
\end{equation}
Notice that $r'$ is the H{\"o}lder-conjugate exponent of $r$ and satisfies $1/r+1/r'=1$.
Furthermore, we define the function
\begin{equation}\label{eq:def.sigma}
  \fluxFunction_{n,r}: \Real^n \ni \boldsymbol{A} \longmapsto \euNorm{\boldsymbol{A}}^{r-2} \boldsymbol{A} \in \Real^n,
\end{equation}
where $\euNorm{\cdot}$ denotes the Euclidean norm on $\Real^n$.
The following properties of the flux function $\fluxFunction_{n,r}$ are well-known (see, e.g., \cite[Section~2]{DIENING.ETTWEIN:2008}).
\begin{lemma}\label{lem:diffusive.flux.properties.original}
  For any real number $r \in (1,\infty)$, any integer $n \in \Natural$, and any $\boldsymbol{A}, \boldsymbol{B} \in \Real^n$, it holds
  \begin{subequations}\label{eq:diffusive.flux.properties.original}
    \begin{alignat}{2}
      \euNorm{\fluxFunction_{n,r}(\boldsymbol{A}) - \fluxFunction_{n,r}(\boldsymbol{B})} &\lesssim (\euNorm{\boldsymbol{A}} + \euNorm{\boldsymbol{B}})^{r-2}
      \euNorm{\boldsymbol{A}-\boldsymbol{B}}, \label{eq:diffusive.flux.properties.original.hc}\\
      \left( \fluxFunction_{n,r}(\boldsymbol{A}) - \fluxFunction_{n,r}(\boldsymbol{B}) \right) \cdot (\boldsymbol{A}-\boldsymbol{B})&\gtrsim (\euNorm{\boldsymbol{A}} +
      \euNorm{\boldsymbol{B}})^{r-2} \euNorm{\boldsymbol{A}-\boldsymbol{B}}^{2}, \label{eq:diffusive.flux.properties.original.mono}
    \end{alignat}
  \end{subequations}
  with a hidden constant depending only on $r$ in both the inequalities.
\end{lemma}
\begin{remark}\label{rem:diffusive.flux.properties}
  By using the notation introduced in \eqref{eq:sobolev.indexes}, the following estimates can be derived from \eqref{eq:diffusive.flux.properties.original} and will be useful in the
  subsequent analysis.
  \begin{subequations}\label{eq:diffusive.flux.properties}
    \begin{alignat}{2}
      \euNorm{\fluxFunction_{n,r}(\boldsymbol{A}) - \fluxFunction_{n,r}(\boldsymbol{B})} &\lesssim (\euNorm{\boldsymbol{A}}^r + \euNorm{\boldsymbol{B}}^r)^{\frac{r-\underline{r}}{r}}
      \euNorm{\boldsymbol{A}-\boldsymbol{B}}^{\underline{r}-1}, \label{eq:diffusive.flux.properties.hc} \\
      \left( \fluxFunction_{n,r}(\boldsymbol{A}) - \fluxFunction_{n,r}(\boldsymbol{B}) \right) \cdot (\boldsymbol{A}-\boldsymbol{B}) &\gtrsim (\euNorm{\boldsymbol{A}}^r +
      \euNorm{\boldsymbol{B}}^r)^{\frac{\underline{r}-2}{r}} \euNorm{\boldsymbol{A}-\boldsymbol{B}}^{\overline{r}}, \label{eq:diffusive.flux.properties.mono}
    \end{alignat}
  \end{subequations}
  with a hidden constant depending only on $r$ in both the inequalities.
\end{remark}
%
\subsection{Mixed weak formulation} \label{sec:mixed.formulation}
From the definition of the diffusive flux function in \eqref{eq:def.sigma}, it is easy to check that the inverse of $\fluxFunction_{n,r}$ is given by $\fluxFunction_{n,r'}$.
We introduce the auxiliary variable $\vectSol \coloneqq \fluxFunction_{\dimDomain,\sobIndex} (\gradScalarSol)$ and, setting $\fluxFunction \coloneqq
\fluxFunction_{\dimDomain,\sobIndexConj}$ to simplify the notation, the mixed form of Problem~\eqref{eq:continuous.problem.primal} reads:

\medskip
\noindent Find $\scalarSol: \domain \to \Real $ and $\vectSol: \domain \to \Real^{\dimDomain}$ such that
\begin{subequations}\label{eq:continuous.problem.dual}
  \begin{alignat}{2}
    \fluxFunction(\vectSol) - \gradScalarSol &= \boldsymbol{0} \qquad && \text{in } \domain, \label{eq:continuous.problem.dual.constitutive} \\
    - \divergence \vectSol &= \loadTerm \qquad && \text{in } \domain, \label{eq:continuous.problem.dual.momentum} \\
    \scalarSol &= 0 \qquad && \text{on } \domainBoundary. \label{eq:continuous.problem.dual.bc}
  \end{alignat}
\end{subequations}
We introduce the auxiliary variable $\vectSol \coloneqq \fluxFunction_{\dimDomain,\sobIndex} (\gradScalarSol)$ and, we set $\fluxFunction \coloneqq
\fluxFunction_{\dimDomain,\sobIndexConj}$ to simplify the notation.
To state the mixed weak formulation of the problem, we introduce the spaces
\begin{equation}\label{eq:continuous.spaces}
  \vectSolSpace \coloneqq \sobDivSpace{\sobIndexConj}{\domain} \quad \text{and} \quad \scalarSolSpace \coloneqq \lebSpace{\sobIndex}{\domain},
\end{equation}
and we define, for any $\vectFunOne, \vectTestFun \in \vectSolSpace$ and any $\scalarTestFun \in \scalarSolSpace$, the forms
\begin{equation}\label{eq:def.a}
  a(\boldsymbol{w},\vectTestFun) \coloneqq \int_{\domain} \fluxFunction(\boldsymbol{w}) \cdot \vectTestFun = \int_{\domain} \euNorm{\boldsymbol{w}}^{\sobIndexConj-2} \boldsymbol{w} \cdot
  \vectTestFun \quad \text{and} \quad b(\vectTestFun,\scalarTestFun) \coloneqq \int_{\domain} (\divergence \vectTestFun) \scalarTestFun. 
\end{equation}
The mixed weak formulation of Problem~\eqref{eq:continuous.problem.primal} reads: 
Given $\loadTerm \in \lebSpace{\sobIndexConj}{\domain}$, find $\vectSol \in \vectSolSpace$ and $\scalarSol \in \scalarSolSpace$ such that
\begin{subequations}\label{eq:continuous.problem.dual.weak}
  \begin{alignat}{2}
    a(\vectSol,\vectTestFun) + b(\vectTestFun, \scalarSol) &= 0
    \qquad && \text{for all } \vectTestFun \in \vectSolSpace, \label{eq:continuous.problem.dual.weak.1} \\
    -b(\vectSol,\scalarTestFun) &= \int_{\domain} \loadTerm \scalarTestFun
    \qquad && \text{for all } \scalarTestFun \in \scalarSolSpace. \label{eq:continuous.problem.dual.weak.2}
  \end{alignat}
\end{subequations}
Existence and uniqueness of a solution to Problem~\eqref{eq:continuous.problem.dual.weak} are established in \cite[Theorem 2.1]{FARHLOUL.MANOUZI:2000}:
\begin{theorem}
  Problem~\eqref{eq:continuous.problem.dual.weak} has a unique solution $(\vectSol, \scalarSol) \in \vectSolSpace \times \scalarSolSpace$ which satisfies
  \begin{equation}\label{eq:cont.sols.stability}
    \norm{\vectSolSpace}{\vectSol} \lesssim \norm{\lebSpace{\sobIndexConj}{\domain}}{\loadTerm}
    \quad \text{and} \quad
    \norm{\scalarSolSpace}{\scalarSol} \lesssim \norm{\lebSpace{\sobIndexConj}{\domain}}{\loadTerm}^{\sobIndexConj-1},
  \end{equation}
  with a hidden constant depending only on $\sobIndex$ in both the inequalities.
\end{theorem}
%
\subsection{Mesh and geometric assumptions}\label{sec:mesh}
Let $\{ \domain_h \}_h$ be a family of decompositions of the domain $\domain \subset \Real^{\dimDomain}$ into nonoverlapping polygonal elements.
We denote the set of edges by $\edges$.
For each element $\element \in \domain_h$, let $h_{\element}$ be its diameter, $\normalElement$ its unit outer normal to the boundary $\partial \element$, $\euNorm{\element}$ its area, and
$\edgeElement$ its set of edges.
For any $\edge \in \edges$, we denote its diameter by $h_{\edge}$.
and denote by $\normalEdge$ one of the its two unit normal vectors, which is fixed once and
for all, with the convection that $\normalEdge$ points outward $\domain$ if $\edge \subset \partial \domain$.
We set $h \coloneqq \max_{\element \in \domain_h} h_{\element}$.
\begin{assumption}\label{ass:mesh.assumptions}
  There exists a uniform constant $\rho > 0$ such that, for any element $\element \in \{ \domain_h \}_h$,
  \begin{enumerate}
    \item $\element$ is star-shaped with respect to a disk $\boldsymbol{B}_{\element}$ with radius $ \geq \rho h_{\element}$;
    \item any edge $\edge \in \edgeElement$ satisfies $h_{\edge} \geq \rho h_{\element}$.
  \end{enumerate}
\end{assumption}
\noindent
We refer to $\rho$ as the mesh regularity parameter.
\begin{remark}\label{lem.element.triangulation}
  As a consequence of Assumption~\ref{ass:mesh.assumptions}, any element $\element \in \domain_h$ admits a conforming simplicial mesh $\mathcal{T}_h(\element)$ in the sense of \cite[Chapter 2]{CIARLET:2002}. Such a decomposition is obtained by
  connecting each edge with the center of the disk $\boldsymbol{B}_{\element}$ in Assumption~\ref{ass:mesh.assumptions} (see \cite[Remark 2.1]{BEIRAODAVEIGA.MASCOTTO.ETAL:2022} for further details).
\end{remark}
Finally, for any integer $m \geq 0$ and any $r \in [1, \infty]$, we introduce the broken Sobolev spaces defined over the mesh $\domain_h$ as
\begin{equation*}
  \sobSpace{m}{r}{\domain_h} \coloneqq \left\{ v \in \lebSpace{r}{\domain} \colon v_{|\element} \in \sobSpace{m}{r}{\element} \quad \forall \element \in \domain_h \right\}.
\end{equation*}
These spaces are naturally equipped with the broken norm and broken seminorm, defined respectively as
\begin{equation*}
  \norm{\sobSpace{m}{r}{\domain_h}}{v}^r \coloneqq \sum_{\element \in \domain_h} \norm{\sobSpace{m}{r}{\element}}{v}^r, \qquad
  \seminorm{\sobSpace{m}{r}{\domain_h}}{v}^r \coloneqq \sum_{\element \in \domain_h} \seminorm{\sobSpace{m}{q}{\element}}{v}^r,
\end{equation*}
with the usual modifications for $r = \infty$.
Similarly, we denote by $\vectSobSpace{m}{r}{\domain_h}$ the corresponding vector-valued broken Sobolev space.
%
\subsection{Polynomial projections and properties}\label{sec:poly.proj}
Let $Y \in \domain_h \cup \edges$. For an integer $\ell \geq 0$, we denote the space of polynomials of total degree up to $\ell$ over $Y$ by $\polySpace{\ell}(Y)$.
We define the broken polynomial space over the mesh $\domain_h$ as
\begin{equation*}
  \polySpace{\ell} (\domain_h) \coloneqq \{ \scalarPolyFun{\ell} \in \lebSpace{1}{\domain}: \restr{\scalarPolyFun{\ell}}{\element} \in \polySpace{\ell}(\element) \text{ for all } \element
  \in \domain_h \}.
\end{equation*}
We also introduce the set of scaled scalar monomial basis on $\edge \in \edges$ with midpoint $x_{\edge}$, and on $\element \in \domain_h$ with barycenter $\boldsymbol{x}_{\element}$, as
\begin{equation*}
  \mathcal{M}_h^{\ell}(\edge) \coloneqq \left\{ \left(\frac{x-x_{\edge}}{h_{\edge}}\right)^{\alpha} : \alpha = {0,\ldots,\ell} \right\} \quad \text{and} \quad 
  \mathcal{M}_h^{\ell}(\element) \coloneqq \left\{ \left(\frac{\boldsymbol{x}-\boldsymbol{x}_{\element}}{h_{\element}}\right)^{\boldsymbol{\alpha}} : \euNorm{\boldsymbol{\alpha}} = {0,\ldots,\ell} \right\},
\end{equation*}
where, for a non-negative multi-index $\boldsymbol{\alpha}=(\alpha_1,\alpha_2)$, we set $\euNorm{\boldsymbol{\alpha}} \coloneqq \alpha_1 + \alpha_2$ and $ \boldsymbol{x}^{\boldsymbol{\alpha}} \coloneqq x_1^{\alpha_1} x_2^{\alpha_2}$.
The local $L^2$-orthogonal projector $\locProjL{\ell} \colon \lebSpace{1}{\element} \to \polySpace{\ell}(\element)$ is defined, for any $v \in \lebSpace{1}{\element}$, by
\begin{equation} \label{eq:def.loc.projL}
  \int_{\element} (v - \locProjL{\ell} v) \scalarPolyFun{\ell} = 0 \quad \text{for all } \scalarPolyFun{\ell} \in \polySpace{\ell}(\element).
\end{equation}
The global $L^2$-orthogonal projector $\projL{\ell} \colon \lebSpace{1}{\domain} \to \polySpace{\ell}(\domain_h)$ is defined by setting, for all $ v \in \lebSpace{1}{\domain}$,
$\restr{(\projL{\ell} v)}{\element} \coloneqq \locProjL{\ell} \restr{v}{\element}$ for all $\element \in \domain_h$.
Local and global $L^2$-orthogonal projectors for vector-valued functions, denoted by $\locVectProjL{\ell}$ and $\vectProjL{\ell}$ respectively, are defined component-wise.
We collect here some well-known results concerning polynomial approximation and inverse inequalities on general polygonal domains (see, e.g., \cite[Section~1.3]{DIPIETRO.DRONIOU:2020}) that
will be used in the following analysis.
\begin{lemma}[$L^r$-boundedness and approximation property of $\locVectProjL{\ell}$]\label{lem:projL.approx}
  Let $\element \in \domain_h$. For any integers $\ell \geq 0$, $n\in \{0,1,\ldots,\ell+1\}$ and $m \in \{0,1,\ldots, n\}$, and any real numbers $r\in [1,\infty]$, it holds
  \begin{subequations}\label{eq:vectProjL.prop}
    \begin{alignat}{2}
      \norm{\vectLebSpace{r}{\element}}{\locVectProjL{\ell}\vectTestFun} &\lesssim \norm{\vectLebSpace{r}{\element}}{\vectTestFun}
      \qquad &&\text{for all} \quad \vectTestFun \in \vectLebSpace{r}{\element}, \label{eq:vectProjL.stab.element} \\
      \seminorm{\vectSobSpace{m}{r}{\element}}{\vectTestFun - \locVectProjL{\ell}\vectTestFun} &\lesssim h_{\element}^{n-m} \seminorm{\vectSobSpace{n}{r}{\element}}{\vectTestFun}
      \qquad &&\text{for all} \quad \vectTestFun \in \vectSobSpace{n}{r}{\element}, \label{eq:vectProjL.approx.element}
    \end{alignat}
  \end{subequations}
  with the hidden constants in \eqref{eq:vectProjL.stab.element} and \eqref{eq:vectProjL.approx.element} depending only on $\{ \ell, r,\rho \}$ and $\{\ell, r, m, n, \rho \}$, respectively.
\end{lemma}
\begin{lemma}[Polynomial Lebesgue embeddings]\label{lem:poly.sob.embedding}
  Let $Y \in \domain_h \cup \edges$. For any integer $\ell \geq 0$, any real numbers $r,s\in [1,\infty]$, and any $\scalarPolyFun{\ell} \in \polySpace{\ell}(Y)$, it holds
  \begin{equation}\label{eq:poly.sob.embedding}
    \norm{\lebSpace{r}{Y}}{\scalarPolyFun{\ell}} \simeq \euNorm{Y}^{\frac{1}{r} - \frac{1}{s}} \norm{\lebSpace{s}{Y}}{\scalarPolyFun{\ell}},
  \end{equation}
  with a hidden constant depending only on $\ell, r, s$ and $\rho$.
\end{lemma}
\begin{lemma}[Polynomial inverse inequality]\label{lem:grad.poly.inv.est}
  Let $Y \subset \Real^{\dimDomain}$ be a bounded polygonal domain. For any integer $\ell \geq 0$, any real number $r\in [1,\infty]$, and any $\scalarPolyFun{\ell} \in \polySpace{\ell}(Y)$, it holds
  \begin{equation}\label{eq:grad.poly.inv.est}
    \norm{\lebSpace{r}{Y}}{\grad \scalarPolyFun{\ell}} \lesssim h_{Y}^{-1} \norm{\lebSpace{r}{Y}}{ \scalarPolyFun{\ell}},
  \end{equation}
  with a hidden constant depending only on $\ell, r$ and the shape regularity of $Y$.
\end{lemma}
\begin{remark}
  The scalar version of Lemma \ref{lem:projL.approx} and the vectorial version of Lemma \ref{lem:poly.sob.embedding} hold true under the same assumptions on the parameters.
\end{remark}
%
\section{A mixed virtual element scheme} \label{sec:scheme}
Throughout the remainder of this paper, we fix the polynomial degree $\polyDegree \geq 0$.
\subsection{Discrete spaces} \label{sec:discrete.spaces}
We define the local discrete spaces for each element $\element \in \domain_h$ as (see \cite{BREZZI.FALK.ETAL:2014})
\begin{subequations}\label{eq:def.local.discrete.spaces}
  \begin{equation}\label{eq:def.local.discrete.vect.space}
    \discVectSpace (\element) \coloneqq \left\{ \discVectTestFun \in \Hdiv{\element} \cap \Hrot{\element}:
      \begin{alignedat}{2}
        &(i) && \discVectTestFun \cdot \normalEdge \in \polySpace{\polyDegree}(\edge) \quad \text{for each edge } \edge \in \edgeElement,\\
        &(ii) && \divergence \discVectTestFun \in \polySpace{\polyDegree}(\element), \\
        &(iii) && \rot \discVectTestFun \in \polySpace{\polyDegree-1}(\element)
      \end{alignedat}
    \right\},
  \end{equation}
  and
  \begin{equation}\label{eq:def.local.discrete.scalar.space}
    \discScalarSpace (\element) \coloneqq \polySpace{\polyDegree}(\element).
  \end{equation}
\end{subequations}
It is straightforward to check that $\vectPolySpace{\polyDegree}(\element) \subset \discVectSpace (\element)$. 
The local degrees of freedom for the space $\discVectSpace(\element)$ are chosen as
\begin{subequations}\label{eq:DoFs}
  \begin{alignat}{2}
    &\frac{1}{h_{\edge}}\int_{\edge} (\discVectTestFun \cdot \normalEdge) \scalarPolyFun{\polyDegree} \quad && \text{for each edge } \edge \in \edgeElement \text{ and all } \scalarPolyFun{\polyDegree} \in
    \mathcal{M}_h^{\polyDegree}(\edge), \label{eq:DoF.normal.on.edge} \\
    &\frac{h_{\element}}{\euNorm{\element}}\int_{\element} \discVectTestFun \cdot \grad \scalarPolyFun{\polyDegree} \quad && \text{ for all } \scalarPolyFun{\polyDegree} \in
    \mathcal{M}_h^{\polyDegree}(\element) \setminus \mathcal{M}_h^{0}(\element), \label{eq:DoF.grad.on.element} \\
    &\frac{1}{\euNorm{\element}}\int_{\element} \discVectTestFun \cdot\left(\frac{\boldsymbol{x} - \boldsymbol{x}_{\element}}{h_{\element}}\right)^{\perp} \scalarPolyFun{\polyDegree-1} \quad && \text{ for all }
    \scalarPolyFun{\polyDegree-1} \in \mathcal{M}_h^{\polyDegree-1}(\element), \label{eq:DoF.perp.grad.on.element}
  \end{alignat}
\end{subequations}
where, for any $\boldsymbol{a}=(a_1,a_2) \in \Real^2$, we use the notation $(\boldsymbol{a})^{\perp} \coloneqq (-a_2,a_1) \in \Real^2$.
We denote by $\locDofVect(\discVectTestFun) \in \Real^{N_{\element}}$ the vector of degrees of freedom \eqref{eq:DoFs} of $\discVectTestFun$, with $N_{\element} \coloneqq \dim(\discVectSpace(\element))$.
\begin{remark}
  A crucial property of the degrees of freedom \eqref{eq:DoFs} is that they allow for the explicit computation of the local $L^2$-orthogonal projector. In particular, for any
  $\discVectTestFun \in \discVectSpace(\element)$, the polynomial $\locVectProjL{\polyDegree} \discVectTestFun$ is uniquely determined by the values of $\locDofVect(\discVectTestFun)$; we
  refer to \cite[Section 3]{BEIRAODAVEIGA.BREZZI.ETAL:2015} for further details.
\end{remark}
The global discrete spaces are defined as
\begin{equation} \label{eq:def.discrete.spaces}
  \discVectSpace \coloneqq \{ \discVectTestFun \in \Hdiv{\domain}: \restr{\discVectTestFun}{\element} \in \discVectSpace (\element) \text{ for all } \element \in \domain_h\} \quad \text{and} \quad
  \discScalarSpace \coloneqq \polySpace{\polyDegree} (\domain_h),
\end{equation}
with global degrees of freedom given by
\begin{subequations}\label{eq:DoFs.global}
  \begin{alignat}{2}
    &\frac{1}{h_{\edge}}\int_{\edge} (\discVectTestFun \cdot \normalEdge) \scalarPolyFun{\polyDegree} \quad && \text{for each edge } \edge \in \edges \text{ and all } \scalarPolyFun{\polyDegree} \in
    \mathcal{M}_h^{\polyDegree}(\edge), \label{eq:DoF.normal.on.edge.global} \\
    &\frac{h_{\element}}{\euNorm{\element}}\int_{\element} \discVectTestFun \cdot \grad \scalarPolyFun{\polyDegree} \quad && \text{for each element } \element \in \domain_h \text{ and all } \scalarPolyFun{\polyDegree} \in
    \mathcal{M}_h^{\polyDegree}(\element) \setminus \mathcal{M}_h^{0}(\element), \label{eq:DoF.grad.on.element.global} \\
    &\frac{1}{\euNorm{\element}}\int_{\element} \discVectTestFun \cdot\left(\frac{\boldsymbol{x} - \boldsymbol{x}_{\element}}{h_{\element}}\right)^{\perp} \scalarPolyFun{\polyDegree-1} \quad && \text{for each element } \element \in \domain_h \text{ and all }
    \scalarPolyFun{\polyDegree-1} \in \mathcal{M}_h^{\polyDegree-1}(\element), \label{eq:DoF.perp.grad.on.element.global}
  \end{alignat}
\end{subequations}
%
\subsection{The Fortin interpolation operator and commuting properties}
We define the Fortin operator $\vectProjF{\polyDegree}: \vectSobSpace{1}{r}{\domain} \to \discVectSpace$ (with $r > 1$) as the DoF interpolation operator. For any $\vectTestFun \in
\vectSobSpace{1}{r}{\domain}$, the function $\vectProjF{\polyDegree} \vectTestFun$ is the unique element in $\discVectSpace$ such that
\begin{subequations}\label{eq:FortinProj}
  \begin{alignat}{2}
    \int_{\edge} ((\vectTestFun - \vectProjF{\polyDegree} \vectTestFun) \cdot \normalEdge) \scalarPolyFun{\polyDegree} &= 0
    \qquad
    && \text{for each edge } \edge \in \edges \text{ and all } \scalarPolyFun{\polyDegree} \in \mathcal{M}_h^{\polyDegree}(\edge); \label{eq:FortinProj.normal.on.edge} \\
    \int_{\element} (\vectTestFun - \vectProjF{\polyDegree} \vectTestFun) \cdot \grad \scalarPolyFun{\polyDegree} &= 0
    \qquad
    &&\text{for each element } \element \in \domain_h \text{ and all } \scalarPolyFun{\polyDegree} \in \mathcal{M}_h^{\polyDegree}(\element) \setminus \mathcal{M}_h^{0}(\element); \label{eq:FortinProj.grad.on.element} \\
    \int_{\element} (\vectTestFun - \vectProjF{\polyDegree} \vectTestFun) \cdot \left(\frac{\boldsymbol{x} - \boldsymbol{x}_{\element}}{h_{\element}}\right)^{\perp} \scalarPolyFun{\polyDegree-1} &= 0
    \qquad
    &&\text{for each element } \element \in \domain_h \text{ and all } \scalarPolyFun{\polyDegree-1} \in \mathcal{M}_h^{\polyDegree-1}(\element). \label{eq:FortinProj.perp.grad.on.element}
  \end{alignat}
\end{subequations}
Note that the requirement $r > 1$ ensures that the normal traces on the edges in \eqref{eq:FortinProj.normal.on.edge} are well-defined.
Using the definition above, it is straightforward to check that the operator satisfies the commuting property
\begin{equation}\label{eq:projF.div.commute}
  \divergence \vectProjF{\polyDegree} \vectTestFun = \projL{\polyDegree} \divergence \vectTestFun \text{ for all } \vectTestFun \in \vectSobSpace{1}{r}{\domain}.
\end{equation}
The stability and approximation properties of the operator $\vectProjF{\polyDegree}$ in the $L^r$-framework will be discussed in Section~\ref{sec:DoF.interpolation.prop}. Having established the properties of the discrete functional spaces and the interpolation operator, we proceed to define the discrete bilinear and non-linear forms that compose the mixed VEM scheme.
\subsection{Discrete operators and problem}\label{sub:disc-operators}
For all $\element \in \domain_h$, and all $\discVectFunOne,\discVectFunTwo \in \discVectSpace (\element)$, we define
\begin{equation} \label{eq:def.ah.local}
  a_{h,\element} (\discVectFunOne,\discVectFunTwo) \coloneqq a_{\element} (\locVectProjL{\polyDegree}\discVectFunOne,\locVectProjL{\polyDegree}\discVectFunTwo) +
  s_{\element}(\discVectFunOne-\locVectProjL{\polyDegree}\discVectFunOne,\discVectFunTwo-\locVectProjL{\polyDegree}\discVectFunTwo)
\end{equation}
where the local operator $a_{\element}$ is given by $a_{\element} (\discVectFunOne,\discVectFunTwo) \coloneqq \int_{\element} \fluxFunction (\discVectFunOne) \cdot \discVectFunTwo$, while the stabilization term $s_{\element}$ is designed to match the non-linear structure of the problem (see, e.g., \cite{ANTONIETTI.BEIRAODAVEIGA.ETAL:2024,ANTONIETTI.BEIRAODAVEIGA.ETAL:2026} for the case of non-Newtonian
Stokes flow) as
\begin{equation}\label{eq:def.sE}
  s_{\element} (\discVectFunOne,\discVectFunTwo)
  \coloneqq h_{\element}^{2} \fluxFunction_{N_{\element}} (\locDofVect (\discVectFunOne)) \cdot \locDofVect (\discVectFunTwo)
  = h_{\element}^{2} \euNorm{\locDofVect (\discVectFunOne)}^{\sobIndexConj-2} \locDofVect (\discVectFunOne) \cdot \locDofVect (\discVectFunTwo),
\end{equation}
with $\fluxFunction_{N_{\element}} \coloneqq \fluxFunction_{N_{\element},\sobIndexConj}$ representing the discrete flux function acting on the vectors of local degrees of freedom.
The global discrete operator is defined naturally, for all $\discVectFunOne,\discVectFunTwo \in \discVectSpace$, as
\begin{equation}\label{eq:def.ah.bh}
  a_h (\discVectFunOne,\discVectFunTwo) \coloneqq \sum_{\element \in \domain_h} a_{h,\element}(\discVectFunOne,\discVectFunTwo).
\end{equation}
Note that, owing to the choice of the discrete spaces, the divergence of any discrete vector field is explicitly a polynomial. Consequently, the continuous form $b(\cdot,\cdot)$ is used directly in the discrete setting without requiring any approximation or stabilization, as it can be evaluated exactly using the degrees of freedom.

The discrete problem reads: Find $\discVectSol \in \discVectSpace$ and $\discScalarSol \in \discScalarSpace$ such that
\begin{subequations}\label{eq:discrete.problem}
  \begin{alignat}{2}
    a_h(\discVectSol,\discVectTestFun) + b(\discVectTestFun, \discScalarSol) &= 0
    \qquad
    &&\text{for all } \discVectTestFun \in \discVectSpace, \label{eq:discrete.problem.line1} \\
    -b(\discVectSol,\discScalarTestFun) &= \int_{\domain} \loadTerm \discScalarTestFun
    \qquad
    &&\text{for all } \discScalarTestFun \in \discScalarSpace. \label{eq:discrete.problem.line2}
  \end{alignat}
\end{subequations}
%

\section{Well-posedness of the discrete scheme} \label{sec:stability}
\subsection{Abstract framework} \label{sec:framework}
To analyze the well-posedness of the discrete problem, we cast it into an abstract framework.
We introduce the non-linear operator $\mathcal{A}_h \colon \discVectSpace \to (\discVectSpace)'$ and the linear operator $\mathcal{B}_h \colon \discVectSpace \to (\discScalarSpace)'$, along
with its dual $\mathcal{B}_h' \colon \discScalarSpace \to (\discVectSpace)'$, defined by the following duality pairings
\begin{equation}\label{eq:def.abstract.operators}
  \langle \mathcal{A}_h \discVectFunOne, \discVectTestFun \rangle \coloneqq a_h(\discVectFunOne,\discVectTestFun), \quad
  \langle \mathcal{B}_h \discVectFunOne, \discScalarTestFun \rangle \coloneqq b(\discVectFunOne,\discScalarTestFun),
  \text{ and} \quad
  \langle \mathcal{B}_h' \discScalarTestFun, \discVectTestFun \rangle \coloneqq b(\discVectTestFun,\discScalarTestFun).
\end{equation}
Furthermore, we define the functional $\mathcal{F}_h \in (\discScalarSpace)'$ as $\langle \mathcal{F}_h, \discScalarTestFun \rangle \coloneqq -\int_{\domain} \loadTerm \discScalarTestFun$.
We can then introduce the discrete kernel $\discDivFreeSpace$ of the operator $\mathcal{B}_h$ and the corresponding affine manifold $\boldsymbol{\mathcal{K}} (\mathcal{F}_h)$ as
\begin{equation*}
  \discDivFreeSpace \coloneqq \{ \discVectTestFun \in \discVectSpace: \mathcal{B}_h \discVectTestFun = 0 \}, \text{ and }
  \boldsymbol{\mathcal{K}} (\mathcal{F}_h) \coloneqq \{ \discVectTestFun \in \discVectSpace: \mathcal{B}_h \discVectTestFun = \mathcal{F}_h \}.
\end{equation*}
Using this notation, Problem~\eqref{eq:discrete.problem} can be equivalently rewritten as follows: Find $\discVectSol \in \boldsymbol{\mathcal{K}} (\mathcal{F}_h)$ and $\discScalarSol \in \discScalarSpace$ such that
\begin{equation}\label{eq:discrete.problem.functional}
  \mathcal{A}_h \discVectSol + \mathcal{B}_h' \discScalarSol = 0 \quad \text{in } (\discVectSpace)'.
\end{equation}
To analyze the behavior of $\mathcal{A}_h$ on the discrete kernel, for any fixed $\tilde{\vectSol}_h \in \boldsymbol{\mathcal{K}} (\mathcal{F}_h)$ we introduce the operator
$\tilde{\mathcal{A}}_{h,\tilde{\vectSol}_h} \colon \discDivFreeSpace \to (\discDivFreeSpace)'$ defined by
\begin{equation}\label{eq:shifted.operator}
  \langle \tilde{\mathcal{A}}_{h,\tilde{\vectSol}_h} \discVectFunOne, \discVectTestFun \rangle \coloneqq \langle \mathcal{A}_h(\discVectFunOne + \tilde{\vectSol}_h), \discVectTestFun
  \rangle \qquad \text{for all } \discVectFunOne, \discVectTestFun \in \discDivFreeSpace.
\end{equation}
According to the classical theory for non-linear saddle-point problems (we refer to, e.g., \cite[Proposition~2.3]{SCHEURER:1977}), the existence of a solution to
\eqref{eq:discrete.problem.functional} is guaranteed provided that the following three conditions hold
\begin{enumerate}[label=(\roman*), ref=(\roman*)]
  \item \textit{Discrete inf-sup condition:} there exists a constant $\tilde{\beta} > 0$, independent of $h$, such that
    \begin{equation}\label{eq:discrete.infsup.generic}
      \inf_{\discScalarTestFun \in \discScalarSpace} \sup_{\discVectTestFun \in \discVectSpace} \frac{ \langle \mathcal{B}_h \discVectTestFun ,\discScalarTestFun  \rangle}{\discFullVectNorm{\discVectTestFun}
      \discScalarNorm{\discScalarTestFun}} \geq \tilde{\beta};
    \end{equation}
  \item \textit{Continuity on the kernel:} for any given $\tilde{\vectSol}_h \in \boldsymbol{\mathcal{K}} (\mathcal{F}_h)$, the operator $\tilde{\mathcal{A}}_{h,\tilde{\vectSol}_h}$ is
    continuous, i.e., for all $\discVectFunOne, \discVectFunTwo \in \discDivFreeSpace$
    \begin{equation}\label{eq:Ah.tilde.continuity}
      \norm{(\discDivFreeSpace)'}{\tilde{\mathcal{A}}_{h,\tilde{\vectSol}_h} \discVectFunOne - \tilde{\mathcal{A}}_{h,\tilde{\vectSol}_h} \discVectFunTwo} \to 0 \quad \text{as }
      \norm{\discDivFreeSpace}{\discVectFunOne - \discVectFunTwo} \to 0;
    \end{equation}
  \item \textit{Coercivity on the kernel:} for any given $\tilde{\vectSol}_h \in \boldsymbol{\mathcal{K}} (\mathcal{F}_h)$, the operator $\tilde{\mathcal{A}}_{h,\tilde{\vectSol}_h}$ is
    coercive, meaning that for all $\discVectFunOne \in \discDivFreeSpace$
    \begin{equation}\label{eq:Ah.tilde.coercivity}
      \frac{\langle \tilde{\mathcal{A}}_{h,\tilde{\vectSol}_h} \discVectFunOne, \discVectFunOne \rangle}{\norm{\discDivFreeSpace}{\discVectFunOne}} \longrightarrow +\infty \quad \text{as }
      \norm{\discDivFreeSpace}{\discVectFunOne} \to +\infty,
    \end{equation}
\end{enumerate}
with the norms $\discFullVectNorm{\cdot}$, $\discScalarNorm{\cdot}$, and $\norm{\discDivFreeSpace}{\cdot}$ defined later.
\begin{remark}
  Observing that $\divergence \discVectSpace = \discScalarSpace$, it follows that, given $\discVectTestFun \in \discVectSpace$, the condition $\langle \mathcal{B}_h \discVectTestFun ,\discScalarTestFun  \rangle = 0$ for
  all $\discScalarTestFun \in \discScalarSpace$ is equivalent to $\divergence \discVectTestFun = 0$, so that we can write
  \begin{equation}\label{eq:def.Zhk}
    \discDivFreeSpace = \{ \discVectTestFun \in \discVectSpace: \divergence \discVectTestFun = 0 \}.
  \end{equation}
\end{remark}
%
\subsection{Properties of the discrete norm}
For any $\element \in \domain_h$ and any $\discVectFunOne \in \discVectSpace(\element)$, we define the local discrete norm as
\begin{equation}\label{eq:def.local.discrete.norm.NEW}
  \locDiscFullVectNorm{\discVectFunOne}^{\sobIndexConj}
  \coloneqq
  \locDiscDivFreeNorm{\discVectFunOne}^{\sobIndexConj} + h_{\element}^{\sobIndexConj} \norm{\lebSpace{\sobIndexConj}{\element}}{\divergence \discVectFunOne}^{\sobIndexConj},
\end{equation}
where 
$\locDiscDivFreeNorm{\discVectFunOne}^{\sobIndexConj} \coloneqq
  a_{h,\element}(\discVectFunOne,\discVectFunOne) =
  \norm{\vectLebSpace{\sobIndexConj}{\element}}{\locVectProjL{\polyDegree} \discVectFunOne}^{\sobIndexConj}
  + s_{\element}(\discVectFunOne - \locVectProjL{\polyDegree} \discVectFunOne, \discVectFunOne - \locVectProjL{\polyDegree} \discVectFunOne)$.

The global discrete norm is given, for any $\discVectFunOne \in \discVectSpace$, by
\begin{equation}\label{eq:def.discrete.norm.NEW}
  \begin{aligned}
    \discFullVectNorm{\discVectFunOne}^{\sobIndexConj}
    &\coloneqq
    \sum_{\element \in \domain_h} \locDiscFullVectNorm{\discVectFunOne}^{\sobIndexConj}
    =
    \discDivFreeNorm{\discVectFunOne}^{\sobIndexConj} + \sum_{\element \in \domain_h} h_{\element}^{\sobIndexConj} \norm{\lebSpace{\sobIndexConj}{\element}}{\divergence \discVectFunOne}^{\sobIndexConj},
  \end{aligned}
\end{equation}
where
\begin{equation}\label{eq:def.ah.norm}
  \discDivFreeNorm{\discVectFunOne}^{\sobIndexConj} \coloneqq \sum_{\element \in \domain_h} \locDiscDivFreeNorm{\discVectFunOne}^{\sobIndexConj}.
\end{equation}
It is straightforward to check that
\begin{equation}\label{eq:disc.full.norm.eq.ah.norm}
  \discFullVectNorm{\discVectFunOne} = \discDivFreeNorm{\discVectFunOne} \quad \text{for all } \discVectFunOne \in \discDivFreeSpace.
\end{equation}
Finally, we endow $\discScalarSpace$ with the norm $\discScalarNorm{\cdot} \coloneqq \norm{\lebSpace{\sobIndex}{\domain}}{\cdot}$.
\begin{remark}\label{rem:boundary.dof.eu.norm.NEW}
  Given an element $\element \in \domain_h$, denote by $\locInternalDofVect$ the vector of internal degrees of freedom defined in \eqref{eq:DoF.grad.on.element} and
  \eqref{eq:DoF.perp.grad.on.element}, and by $\locBoundaryDofVect$ the vector of boundary degrees of freedom defined in \eqref{eq:DoF.normal.on.edge}.
  Given $\discVectFunOneTilda \in \discVectSpace(\element)$ such that $\discVectFunOneTilda$ is $L^2$-orthogonal to $\vectPolySpace{\polyDegree}(\element)$, noticing that both $\grad \scalarPolyFun{\polyDegree}$ and $\left(\frac{\boldsymbol{x} - \boldsymbol{x}_{\element}}{h_{\element}}\right)^{\perp} \scalarPolyFun{\polyDegree-1}$ are polynomials up to degree $\polyDegree$, we deduce that $\locInternalDofVect
  (\discVectFunOneTilda) = \boldsymbol{0}$.
  In particular, since $\euNorm{\locDofVect (\discVectFunOneTilda)} = \euNorm{\locBoundaryDofVect (\discVectFunOneTilda)}$,
  \begin{equation}\label{eq:sE.just.boundary.DoF}
    s_{\element}(\discVectFunOneTilda,\discVectFunOneTilda) = h_{\element}^{2} \euNorm{\locBoundaryDofVect (\discVectFunOneTilda)}^{\sobIndexConj}.
  \end{equation}
\end{remark}
We begin by proving an equivalence between the Euclidean norm of the boundary degrees of freedom and the $L^r$-norm of the normal trace.
\begin{lemma}\label{lem:boundaryDof.equivalence}
  Let $\element \in \domain_h$. For any $\discVectFunOne \in \discVectSpace(\element)$ and any real number $r>1$, it holds that
  \begin{equation}\label{eq:boundaryDof.equivalence}
    \euNorm{\locBoundaryDofVect \discVectFunOne }^{r} \simeq \sum_{\edge \in \edgeElement} h_{\edge}^{-1} \norm{\lebSpace{r}{\edge}}{ \discVectFunOne \cdot \normalEdge}^r,
  \end{equation}
  with the hidden constants depending only on $\polyDegree, r$ and $\rho$.
\end{lemma}
\begin{proof}
  Fix $\element \in \domain_h$, $\discVectFunOne \in \discVectSpace(\element)$ and $r>1$. We first observe that
  \begin{equation}\label{eq:boundaryDof.equivalence.initial}
    \euNorm{\locBoundaryDofVect \discVectFunOne }^{r}
    \simeq 
    \sum_{\edge \in \edgeElement} \frac{1}{h_{\edge}^r}\sum_{\scalarPolyFun{\polyDegree} \in \mathcal{M}_h^{\polyDegree}(\edge)} \left| \int_{\edge} (\discVectFunOne \cdot \normalEdge) \scalarPolyFun{\polyDegree} \right|^r,
  \end{equation}
  with the hidden constants depending only on $\polyDegree, r$ and $\rho$.
  For a fixed edge $\edge \in \edgeElement$, using a standard scaling argument, it can be shown that, for all $\scalarPolyFun{\star} \in \polySpace{\polyDegree}(\edge)$,
  \begin{equation}\label{eq:edge-r-norm.poly}
    \sum_{\scalarPolyFun{\polyDegree} \in \mathcal{M}_h^{\polyDegree}(\edge)} \euNorm{\int_{\edge} \scalarPolyFun{\star} \scalarPolyFun{\polyDegree} }^r \simeq h_{\edge}^{r-1}\norm{\lebSpace{r}{\edge}}{\scalarPolyFun{\star}}^r,
  \end{equation}
  with the hidden constants depending only on $\polyDegree, r$ and $\rho$.
  Recalling that $\discVectFunOne \cdot \normalEdge \in \polySpace{\polyDegree}(\edge)$, we substitute $\scalarPolyFun{\star}=\discVectFunOne \cdot \normalEdge$ into
  \eqref{eq:edge-r-norm.poly} and insert the result into \eqref{eq:boundaryDof.equivalence.initial} to get the desired equivalence.
\end{proof}
\begin{remark}
  Combining \eqref{eq:sE.just.boundary.DoF} and \eqref{eq:boundaryDof.equivalence} with $r = \sobIndexConj$ we obtain that, given $\element \in \domain_h$, for any $\discVectFunOneTilda \in
  \discVectSpace(\element)$ $L^2$-orthogonal to $\vectPolySpace{\polyDegree}(\element)$,
  \begin{equation}\label{eq:sE.simeq.r-norm.normal.component}
    s_{\element} (\discVectFunOneTilda,\discVectFunOneTilda) \simeq h_{\element} \norm{\lebSpace{\sobIndexConj}{\partial \element}}{\discVectFunOneTilda\cdot \normalElement}^{\sobIndexConj}.
  \end{equation}
  In particular, recalling the definition of $a_{h,\element}$ in \eqref{eq:def.ah.local} we have, for any $\discVectFunOne \in \discVectSpace(\element)$,
  \begin{equation}\label{eq:ahE.simeq.disc.div.free.norm}
    \locDiscDivFreeNorm{\discVectFunOne}^{\sobIndexConj} \simeq \norm{\vectLebSpace{\sobIndexConj}{\element}}{\vectProjL{\polyDegree} \discVectFunOne}^{\sobIndexConj} +
    h_{\element} \norm{\lebSpace{\sobIndexConj}{\partial \element}}{(\discVectFunOne-\locVectProjL{\polyDegree} \discVectFunOne) \cdot \normalElement}^{\sobIndexConj},
  \end{equation}
  with the hidden constants in \eqref{eq:sE.simeq.r-norm.normal.component} and \eqref{eq:ahE.simeq.disc.div.free.norm} depending only on $\polyDegree$, $\sobIndex$, and $\rho$.
\end{remark}
\begin{remark}\label{rem:conforming}
  If all the mesh elements $\element$ are convex, standard regularity results on Lipschitz domains ensure that the inclusion $\discVectSpace \subset \vectSolSpace$ holds for all
  $\sobIndexConj \in (1, \infty)$. Without the convexity assumption, we have the inclusion $\discVectSpace \subset \vectSolSpace$ if and only if $\sobIndexConj \in (1,2]$. This is the
  reason why in the following results we restrict ourselves to the $L^r$ setting with $r \in (1,2]$.
\end{remark}
%
\subsubsection{Preliminary results}
The following Lemma provides a discrete trace inequality for virtual functions in the $L^r$-setting. Its proof is reported in Appendix~\ref{app:proof.r-norm.normal.component} for the sake of completeness.
\begin{lemma}\label{lem:r-norm.normal.component}
  Let $\element \in \domain_h$. For any $\discVectFunOne \in \discVectSpace(\element)$ and any real number $1<r \leq 2$ it holds that
  \begin{equation}\label{eq:r-norm.normal.component}
    \norm{\lebSpace{r}{\partial \element}}{\discVectFunOne \cdot \normalElement}
    \lesssim
    h_{\element}^{-\frac{1}{r}} \norm{\vectLebSpace{r}{\element}}{\discVectFunOne} + h_{\element}^{1-\frac{1}{r}} \norm{\lebSpace{r}{\element}}{\divergence \discVectFunOne},
  \end{equation}
  with a hidden constant depending only on $\polyDegree$, $r$, and $\rho$.
\end{lemma}
Next, we establish an inverse-type inequality bounding the $L^r$-norm of the divergence of a discrete function by the $L^r$-norm of the function itself. The $L^2$ case was already proven in
\cite{BEIRAODAVEIGA.MASCOTTO.ETAL:2022}, and the proof presented here relies on similar arguments.
\begin{lemma}\label{lem:r-norm.div.virtual.func}
  Let $\element \in \domain_h$. For any $\discVectFunOne \in \discVectSpace(\element)$ and any real number $1<r \leq 2$, it holds that
  \begin{equation}\label{eq:r-norm.div.virtual.func}
    \norm{\lebSpace{r}{\element}}{\divergence \discVectFunOne} \lesssim h_{\element}^{-1} \norm{\vectLebSpace{r}{\element}}{\discVectFunOne},
  \end{equation}
  where the hidden constant depends on $r$, $\polyDegree$, and $\rho$.
\end{lemma}
\begin{proof}
  Let $\element \in \domain_h$, $\discVectFunOne \in \discVectSpace(\element)$ and $1<r \leq 2$.
  We first apply \eqref{eq:poly.sob.embedding} to write
  \begin{equation}\label{eq:r-norm.div.virtual.func.poly.inv.est}
    \norm{\lebSpace{r}{\element}}{\divergence \discVectFunOne}^2 \leq C(r,\polyDegree,\rho) \euNorm{\element}^{\frac{2}{r} - 1} \norm{\lebSpace{2}{\element}}{\divergence \discVectFunOne}^2.
  \end{equation}
  Let $b_{\element}$ be the usual cubic piecewise bubble function associated with the shape-regular triangulation $\mathcal{T}_h(\element)$ of $\element$ (cf. Remark
  \ref{lem.element.triangulation}), scaled so that it has a unitary $L^{\infty}$ norm over $\element$. 
  Applying a triangle-by-triangle argument together with integration by parts and \eqref{eq:grad.poly.inv.est} leads to
  \begin{align*}
    \norm{\lebSpace{2}{\element}}{\divergence \discVectFunOne}^2
    &\leq
    C(\polyDegree,\rho) \int_{\element} (\divergence \discVectFunOne)(b_{\element}\divergence \discVectFunOne)
    = - C(\polyDegree,\rho)
    \int_{\element} \discVectFunOne \cdot \grad (b_{\element}\divergence \discVectFunOne) \\
    &\leq C(\polyDegree,\rho) \norm{\vectLebSpace{r}{\element}}{\discVectFunOne} \norm{\vectLebSpace{r'}{\element}}{ \grad (b_{\element} \divergence \discVectFunOne)} \\
    &\leq C(r,\polyDegree,\rho) \euNorm{\element}^{\frac{1}{r'} - \frac{1}{r}} h_{\element}^{-1} \norm{\vectLebSpace{r}{\element}}{\discVectFunOne}
    \norm{\vectLebSpace{r}{\element}}{\divergence \discVectFunOne}.
  \end{align*}
  Combining the above inequality with \eqref{eq:r-norm.div.virtual.func.poly.inv.est}, observing that $\frac{2}{r} - 1 + \frac{1}{r'} - \frac{1}{r} = 0$, and simplifying
  $\norm{\lebSpace{r}{\element}}{\divergence \discVectFunOne}$ on both sides, we obtain the desired result.
\end{proof}
\begin{remark}
  From \eqref{eq:sE.simeq.r-norm.normal.component} and by taking $r=\sobIndexConj$ in \eqref{eq:r-norm.normal.component} and \eqref{eq:r-norm.div.virtual.func}, we obtain that, for any
  $\element \in \domain_h$ and any $\discVectFunOneTilda \in \discVectSpace(\element)$ $L^2$-orthogonal to $\vectPolySpace{\polyDegree}(\element)$
  \begin{equation}\label{eq:sE.less.cont.norm}
    s_{\element} (\discVectFunOneTilda,\discVectFunOneTilda) \lesssim \norm{\vectLebSpace{\sobIndexConj}{\element}}{\discVectFunOneTilda}^{\sobIndexConj}.
  \end{equation}
  In particular, recalling the definition of $a_{h,\element}$ in \eqref{eq:def.ah.local} we have, for any $\discVectFunOne \in \discVectSpace(\element)$,
  \begin{equation}\label{eq:ahE.less.cont.norm}
    a_{h,\element} (\discVectFunOne,\discVectFunOne) \lesssim
    \norm{\vectLebSpace{\sobIndexConj}{\element}}{\discVectFunOne}^{\sobIndexConj},
  \end{equation}
  with the hidden constants in \eqref{eq:sE.less.cont.norm} and \eqref{eq:ahE.less.cont.norm} depending only on $\polyDegree$, $\sobIndex$, and $\rho$.
\end{remark}
The following result extends the $L^2$-stability of discrete vector fields (see \cite[Lemma 3.1]{BEIRAODAVEIGA.MASCOTTO.ETAL:2022}) to the $L^r$ setting by means of duality and polynomial inverse inequalities. 
%
\begin{lemma}\label{lem:r-norm.virtual.fun}
  Let $\element \in \domain_h$. For any $\discVectFunOne \in \discVectSpace(\element)$ and any real number $1<r \leq 2$ it holds that
  \begin{equation}\label{eq:r-norm.virtual.fun}
    \norm{\lebSpace{r}{\element}}{\discVectFunOne}
    \lesssim
    h_{\element} \norm{\lebSpace{r}{\element}}{\divergence \discVectFunOne} +
    h_{\element}^{\frac{1}{r}} \norm{\lebSpace{r}{\partial \element}}{ \discVectFunOne \cdot \normalElement} +
    \sup_{\scalarPolyFun{\polyDegree-1} \in \polySpace{\polyDegree-1}(\element)} \frac{\int_{\element} \discVectFunOne \cdot (\boldsymbol{x}^{\perp}
    \scalarPolyFun{\polyDegree-1})}{\norm{\lebSpace{r'}{\element}}{\boldsymbol{x}^{\perp} \scalarPolyFun{\polyDegree-1}}},
  \end{equation}
  with a hidden constant depending only on $\polyDegree$, $r$, and $\rho$.
\end{lemma}
\begin{proof}
  Fix $\element \in \domain_h$, $\discVectFunOne \in \discVectSpace(\element)$ and $1<r \leq 2$.
  Applying a $(\frac{2}{r},\frac{2}{2-r})$-H{\"o}lder inequality, we get $\norm{\lebSpace{r}{\element}}{\discVectFunOne} \leq \euNorm{\element}^{\frac{1}{r}-\frac{1}{2}} \norm{\lebSpace{2}{\element}}{\discVectFunOne}$. Furthermore, from a simple adaptation of \cite[Lemma 3.1]{BEIRAODAVEIGA.MASCOTTO.ETAL:2022} to our local discrete virtual space $\discVectSpace (\element)$, it holds that
  \begin{equation*}
    \norm{\lebSpace{2}{\element}}{\discVectFunOne}
    \leq C(\polyDegree,\rho)
    \left[
      h_{\element} \norm{\lebSpace{2}{\element}}{\divergence \discVectFunOne} +
      h_{\element}^{\frac{1}{2}} \norm{\lebSpace{2}{\partial \element}}{ \discVectFunOne \cdot \normalElement} +
      \sup_{\Theta \in \polySpace{\polyDegree-1}(\element)} \frac{\int_{\element} \discVectFunOne \cdot (\boldsymbol{x}^{\perp}
      \Theta)}{\norm{\lebSpace{2}{\element}}{\boldsymbol{x}^{\perp} \Theta}}
    \right].
  \end{equation*}
  The conclusion follows by applying $\eqref{eq:poly.sob.embedding}$.
\end{proof}
%
\subsubsection{Equivalence between discrete and continuous norms}
Now we are in position to state and prove the following key result for the subsequent analysis.
\begin{lemma}
  If $ 1 < \sobIndexConj \leq 2$, then, for any $\element \in \domain_h$ and any $\discVectFunOne \in \discVectSpace (\element)$, it holds
  \begin{equation}\label{eq:disc.norm.simeq.cont.norm}
    \locDiscFullVectNorm{\discVectFunOne} \simeq \norm{\vectLebSpace{\sobIndexConj}{\element}}{\discVectFunOne},
  \end{equation}
  with the hidden constants depending only on $\polyDegree$, $\sobIndex$, and $\rho$.
\end{lemma}
\begin{proof}
  Recalling the defintion of $ \locDiscFullVectNorm{\cdot} $ in \eqref{eq:def.local.discrete.norm.NEW}, the fact that $\locDiscFullVectNorm{\discVectFunOne} \lesssim
  \norm{\vectLebSpace{\sobIndexConj}{\element}}{\discVectFunOne}$ is a straightforward combination of \eqref{eq:ahE.less.cont.norm} and
  \eqref{eq:r-norm.div.virtual.func} with $r=\sobIndexConj$.

  Let us prove the converse inequality.
  Fix $\element \in \domain_h$ and $\discVectFunOne \in \discVectSpace (\element)$.
  We add and subtract $\locVectProjL{\polyDegree}\discVectFunOne$ to write
  \begin{equation}\label{eq:cont.norm.less.disc.norm.initial}
    \norm{\vectLebSpace{\sobIndexConj}{\element}}{\discVectFunOne}
    \leq
    \locDiscFullVectNorm{\discVectFunOne}
    + \norm{\vectLebSpace{\sobIndexConj}{\element}}{\discVectFunOne - \locVectProjL{\polyDegree}\discVectFunOne}
  \end{equation}
  For the second term in the right-hand side, we apply \eqref{eq:r-norm.virtual.fun}, and recalling the definition of $\locVectProjL{\polyDegree}\discVectFunOne$ in
  \eqref{eq:def.loc.projL}, which ensures that the term in the $\sup$ in \eqref{eq:r-norm.virtual.fun} vanishes, we obtain
  \begin{equation}\label{eq:cont.norm.less.disc.norm.term.2}
    \begin{aligned}
      \norm{\vectLebSpace{\sobIndexConj}{\element}}{\discVectFunOne - \locVectProjL{\polyDegree}\discVectFunOne}
      &\leq
      C(\polyDegree,\sobIndex,\rho)
      \left[
        h_{\element} \norm{\lebSpace{\sobIndexConj}{\element}}{\divergence( \discVectFunOne- \locVectProjL{\polyDegree}\discVectFunOne)} +
      h_{\element}^{\frac{1}{\sobIndexConj}} \norm{\lebSpace{\sobIndexConj}{\partial \element}}{ (\discVectFunOne - \locVectProjL{\polyDegree}\discVectFunOne) \cdot \normalElement} \right] \\
      &\leq
      C(\polyDegree,\sobIndex,\rho)
      \locDiscFullVectNorm{\discVectFunOne},
    \end{aligned}
  \end{equation}
  where the last inequality follows from the triangle inequality on $\norm{\lebSpace{\sobIndexConj}{\element}}{\divergence( \discVectFunOne- \locVectProjL{\polyDegree}\discVectFunOne)}$ together with \eqref{eq:r-norm.div.virtual.func}, and from the equivalence in \eqref{eq:ahE.simeq.disc.div.free.norm}.
  Plugging \eqref{eq:cont.norm.less.disc.norm.term.2} into \eqref{eq:cont.norm.less.disc.norm.initial} leads to the desired result.
\end{proof}
%
\subsection{Properties of the DoF interpolation operator}\label{sec:DoF.interpolation.prop}
We now state and prove the local approximation results for the DoF interpolation operator $\vectProjF{\polyDegree}$.
\begin{lemma}
  Let $\element \in \domain_h$. For all integers $s \in \{1,\ldots,\polyDegree +1\}$ and $j \in \{0,1,\ldots,\polyDegree +1 \}$, it holds,
  for any $\vectTestFun \in \vectSobSpace{s}{r}{\element}$ with $\divergence \vectTestFun \in \sobSpace{j}{r}{\element}$, that
  \begin{subequations}\label{eq:Fortin.approx}
    \begin{alignat}{2}
      \norm{\vectLebSpace{r}{\element}}{\vectTestFun - \vectProjF{\polyDegree} \vectTestFun } &\lesssim h_{\element}^s \seminorm{\vectSobSpace{s}{r}{\element}}{\vectTestFun} &\quad &\text{for all
      } 1< r \leq 2, \label{eq:Fortin.approx.r-norm} \\
      \norm{\lebSpace{r}{\element}}{\divergence (\vectTestFun - \vectProjF{\polyDegree} \vectTestFun)} &\lesssim h_{\element}^j \seminorm{\sobSpace{j}{r}{\element}}{\divergence
      \vectTestFun} &\quad &\text{for all } 1\leq r \leq \infty \label{eq:Fortin.approx.div.r-norm},
    \end{alignat}
  \end{subequations}
  with the hidden constants in \eqref{eq:Fortin.approx.r-norm} and \eqref{eq:Fortin.approx.div.r-norm} depending only on $\{ \polyDegree, r, s, \rho \}$ and $\{ \polyDegree, r, j, \rho\}$,
  respectively.
\end{lemma}
\begin{proof}
  Fix $\element \in \domain_h$, $s \in \{1,\ldots,\polyDegree +1\}$, $j \in \{0,1,\ldots,\polyDegree +1 \}$ and $\vectTestFun$ as in the statement of the Lemma.
  For $1\leq r \leq \infty$, inequality \eqref{eq:Fortin.approx.div.r-norm} follows from \eqref{eq:projF.div.commute} and the scalar version of \eqref{eq:vectProjL.approx.element}.
  For inequality \eqref{eq:Fortin.approx.r-norm}, assuming $1 < r \leq 2$, splitting the approximation error by adding and subtracting $\locVectProjL{\polyDegree} \vectTestFun$ and applying
  \eqref{eq:vectProjL.approx.element}, leads to
  \begin{equation}\label{eq:Fortin.approx.initial}
    \begin{aligned}
      \norm{\vectLebSpace{r}{\element}}{\vectTestFun - \vectProjF{\polyDegree} \vectTestFun }
      &\leq C(\polyDegree,r,s,\rho) h_{\element}^s\seminorm{\vectSobSpace{s}{r}{\element}}{\vectTestFun} + \norm{\vectLebSpace{r}{\element}}{\locVectProjL{\polyDegree} \vectTestFun -
      \vectProjF{\polyDegree} \vectTestFun}.
    \end{aligned}
  \end{equation}
  We define define $\boldsymbol{\xi}_h \coloneqq \locVectProjL{\polyDegree} \vectTestFun - \vectProjF{\polyDegree} \vectTestFun$.
  We apply \eqref{eq:r-norm.virtual.fun} and, after noticing that $\boldsymbol{\xi}_h$ is $L^2$-orthogonal to
  $\polySpace{\polyDegree}(\element)$ to make the term in the $\sup$ vanish, we obtain
  \begin{equation}\label{eq:Fortin.approx.xi.initial}
    \norm{\vectLebSpace{r}{\element}}{\boldsymbol{\xi}_h}
    \leq C(r,\polyDegree,\rho)
    \left[
      h_{\element} \norm{\lebSpace{r}{\element}}{\divergence \boldsymbol{\xi}_h} +
      h_{\element}^{\frac{1}{r}} \norm{\lebSpace{r}{\partial \element}}{ \boldsymbol{\xi}_h \cdot \normalElement}
    \right].
  \end{equation}
  We analyze each of the two terms in the square brackets.

  \medskip
  \noindent
  \textbf{First term.}
  By adding and subtracting $\divergence \vectTestFun$ inside the norm we get
  \begin{equation}\label{eq:Fortin.approx.xi.term1}
    \begin{aligned}
      h_{\element} \norm{\lebSpace{r}{\element}}{\divergence \boldsymbol{\xi}_h}
      & \leq h_{\element} \left(C(r) \seminorm{\vectSobSpace{1}{r}{\element}}{\vectTestFun - \locVectProjL{\polyDegree} \vectTestFun} + \norm{\lebSpace{r}{\element}}{\divergence (\vectTestFun -
      \vectProjF{\polyDegree} \vectTestFun)} \right) \\
      &\leq C(\polyDegree,r,s,\rho) h_{\element}^{s} \left( \seminorm{\vectSobSpace{s}{r}{\element}}{\vectTestFun} + \seminorm{\sobSpace{s-1}{r}{\element}}{\divergence \vectTestFun} \right) \\
      &\leq C(\polyDegree,r,s,\rho) h_{\element}^{s} \seminorm{\vectSobSpace{s}{r}{\element}}{\vectTestFun},
    \end{aligned}
  \end{equation}
  where the second inequality follows by applying \eqref{eq:vectProjL.approx.element} to the first term, and \eqref{eq:Fortin.approx.div.r-norm} to the second term.

  \medskip
  \noindent
  \textbf{Second term.}
  We fix $\edge \in \edgeElement$. Since $\boldsymbol{\xi}_h \cdot \normalEdge \in \polySpace{\polyDegree}(\edge)$ and $\vectProjF{\polyDegree} \vectTestFun$ satisfies
  \eqref{eq:FortinProj.normal.on.edge}, we have that
  \begin{equation}\label{eq:Fortin.approx.xi.term2}
    \begin{aligned}
      h_{\element}^{\frac{1}{r}} \norm{\lebSpace{r}{\edge}}{ \boldsymbol{\xi}_h \cdot \normalEdge}
      &\leq h_{\element}^{\frac{1}{r}} \sup_{\scalarPolyFun{\polyDegree} \in \polySpace{\polyDegree}(\edge)} \frac{\int_{\edge} (\locVectProjL{\polyDegree} \vectTestFun - \vectTestFun) \cdot \normalEdge
      \scalarPolyFun{\polyDegree} }{\norm{\lebSpace{r'}{\edge}}{\scalarPolyFun{\polyDegree}}} \leq
      h_{\element}^{\frac{1}{r}} \norm{\vectLebSpace{r}{\edge}}{\locVectProjL{\polyDegree} \vectTestFun - \vectTestFun} \\
      &\leq
      C(r,\rho)
      \left[ \norm{\vectLebSpace{r}{\element}}{\locVectProjL{\polyDegree} \vectTestFun - \vectTestFun} + h_{\element}
        \seminorm{\vectSobSpace{1}{r}{\element}}{\locVectProjL{\polyDegree} \vectTestFun - \vectTestFun}
      \right] \\
      &\leq
      C(\polyDegree,r,s,\rho) h_{\element}^s \seminorm{\vectSobSpace{s}{r}{\element}}{\vectTestFun}
    \end{aligned}
  \end{equation}
  where the second inequality follows from a $(r,\infty,r')$-H{\"o}lder inequality. For the third inequality we apply a continuous local trace inequality (see, e.g., \cite[Lemma
  1.31]{DIPIETRO.DRONIOU:2020}), while the last inequality follows by applying \eqref{eq:vectProjL.approx.element}.

  \medskip
  \noindent
  \textbf{Conclusion.}
  Plugging \eqref{eq:Fortin.approx.xi.term1} and \eqref{eq:Fortin.approx.xi.term2} into \eqref{eq:Fortin.approx.xi.initial}, and combining the resulting
  estimate with \eqref{eq:Fortin.approx.initial}, leads to the desired result.
\end{proof}
In the next result we establish the local approximation property of the operator $\vectProjL{\polyDegree} \circ \vectProjF{\polyDegree}$.
\begin{lemma}
  Let $\element \in \domain_h$ and $r > 1$. For all integers $s \in \{1,\ldots,\polyDegree +1\}$ and all $\vectTestFun \in \vectSobSpace{s}{r}{\element}$, it holds
  \begin{equation}\label{eq:projL.Fortin.approx}
    \norm{\vectLebSpace{r}{\element}}{\vectTestFun - \locVectProjL{\polyDegree} \vectProjF{\polyDegree} \vectTestFun } \lesssim h_{\element}^s \seminorm{\vectSobSpace{s}{r}{\element}}{\vectTestFun}
  \end{equation}
  with a hidden constant depending only on $\polyDegree, r, s$, and $\rho$.
\end{lemma}
\begin{proof}
  Fix $\element \in \domain_h$, $s \in \{1,\ldots,\polyDegree +1\}$, and $\vectTestFun \in \vectSobSpace{s}{r}{\element}$. We add and subtract $\locVectProjL{\polyDegree}
  \vectTestFun$, and apply \eqref{eq:vectProjL.approx.element} to write
  \begin{align*}
    \norm{\vectLebSpace{r}{\element}}{\vectTestFun - \locVectProjL{\polyDegree} \vectProjF{\polyDegree} \vectTestFun}
    &\leq
    C(\polyDegree,r,s,\rho) h_{\element}^{s}\seminorm{\vectSobSpace{s}{r}{\element}}{\vectTestFun}
    +
    \norm{\vectLebSpace{r}{\element}}{\locVectProjL{\polyDegree}(\vectTestFun - \vectProjF{\polyDegree} \vectTestFun)}.
  \end{align*}
  If $1<r\leq 2$, we apply \eqref{eq:vectProjL.stab.element} and \eqref{eq:Fortin.approx.r-norm} to write
  \begin{align*}
    \norm{\vectLebSpace{r}{\element}}{\locVectProjL{\polyDegree}(\vectTestFun - \vectProjF{\polyDegree} \vectTestFun)} \leq
    C(\polyDegree,r,s,\rho) h_{\element}^{s}\seminorm{\vectSobSpace{s}{r}{\element}}{\vectTestFun}.
  \end{align*}
  If $r>2$, we first apply \eqref{eq:poly.sob.embedding} and then \eqref{eq:vectProjL.stab.element} and \eqref{eq:Fortin.approx.r-norm} together with a $(\frac{r}{2},\frac{r}{r-2})$-H{\"o}lder inequality, thus obtaining
  \begin{align*}
    \norm{\vectLebSpace{r}{\element}}{\locVectProjL{\polyDegree}(\vectTestFun - \vectProjF{\polyDegree} \vectTestFun)} 
    \leq 
    C(\polyDegree,r,s,\rho) \euNorm{\element}^{\frac{1}{r} - \frac{1}{2}} h_{\element}^{s}\seminorm{\vectSobSpace{s}{2}{\element}}{\vectTestFun}
    \leq 
    C(\polyDegree,r,s,\rho) h_{\element}^{s}\seminorm{\vectSobSpace{s}{r}{\element}}{\vectTestFun}
  \end{align*}
\end{proof}
%

\subsection{Main stability result}

\subsubsection{Discrete inf-sup condition}
In this section we prove property \eqref{eq:discrete.infsup.generic}.
\begin{lemma}[Discrete inf-sup condition] \label{lem:discrete.infsup}
  For all $\discScalarTestFun \in \discScalarSpace$, it holds that
  \begin{equation}\label{eq:discrete.infsup}
    \sup_{\discVectTestFun \in \discVectSpace} \frac{b(\discVectTestFun,\discScalarTestFun)}{\discFullVectNorm{\discVectTestFun}}
    \gtrsim
    \discScalarNorm{\discScalarTestFun},
  \end{equation}
  with a hidden constant depending only on $\polyDegree$, $\sobIndex$, and $\rho$.
\end{lemma}
\begin{proof}
  Let $\discScalarTestFun \in \discScalarSpace \subset \scalarSolSpace$. We define $G \coloneqq \euNorm{\discScalarTestFun}^{\sobIndex-2}\discScalarTestFun$. Since
  $\euNorm{G}^{\sobIndexConj} = \euNorm{\discScalarTestFun}^{\sobIndex} \in \lebSpace{1}{\domain}$, it follows that
  \begin{equation}\label{eq:identity.norm.G.qh}
    G \in \lebSpace{\sobIndexConj}{\domain} \text{ with }
    \norm{\lebSpace{\sobIndexConj}{\domain}}{G}^{\sobIndexConj} = \norm{\lebSpace{\sobIndex}{\domain}}{\discScalarTestFun}^{\sobIndex}.
  \end{equation}
  From \cite[Theorem~1]{BOGOVSKII:1979}, we know that there exists $\vectTestFun \in \vectSobSpace{1}{\sobIndexConj}{\domain}$ such that
  \begin{equation}\label{eq:div.surj.stable}
    \divergence \vectTestFun = G  \text{ in } \domain,
    \quad \text{and} \quad
    \norm{\vectSobSpace{1}{\sobIndexConj}{\domain}}{\vectTestFun} \leq C(\sobIndex) \norm{\lebSpace{\sobIndexConj}{\domain}}{G}.
  \end{equation}
  Using the commuting property \eqref{eq:projF.div.commute}, the fact that $\discScalarTestFun \in \discScalarSpace$ together with the definition of $\projL{\polyDegree}$, and then
  \eqref{eq:div.surj.stable} together with the definition of $G$, we obtain
  \begin{align}\label{eq:Fortin.cond.1}
    b(\vectProjF{\polyDegree} \vectTestFun,\discScalarTestFun)
    =
    \int_{\Omega} (\projL{\polyDegree}(\divergence \vectTestFun))\discScalarTestFun
    =
    \int_{\Omega} G \discScalarTestFun
    =
    \discScalarNorm{\discScalarTestFun}^{\sobIndex}.
  \end{align}
  From now on, we distinguish the case $ 1 < \sobIndexConj \leq 2$ and $\sobIndexConj > 2$.

  \medskip
  \noindent
  \textbf{Case $1 < \sobIndexConj \leq 2$.}
  For any $\element \in \domain_h$, we first apply \eqref{eq:disc.norm.simeq.cont.norm}, then we add and subtract
  $\vectTestFun$ followed by the triangle inequality and \eqref{eq:Fortin.approx.r-norm}, to write
  \begin{align*}
    \locDiscFullVectNorm{\vectProjF{\polyDegree} \vectTestFun}
    &\leq C(\polyDegree, \sobIndex, \rho)
    \left[
      h_{\element} \seminorm{\vectSobSpace{1}{\sobIndexConj}{\element}}{\vectTestFun} + \norm{\vectLebSpace{\sobIndexConj}{\element}}{\vectTestFun}
    \right] \leq
    C(\polyDegree, \sobIndex, \rho)
    \norm{\vectSobSpace{1}{\sobIndexConj}{\element}}{\vectTestFun}.
  \end{align*}

  \medskip
  \noindent
  \textbf{Case $\sobIndexConj > 2$.}
  For any $\element \in \domain_h$, we first apply \eqref{eq:poly.sob.embedding}, and then, applying stability bound \eqref{eq:vectProjL.stab.element}, we obtain
  \begin{equation}\label{eq:disc.inf.sup.g2.initial}
    \begin{aligned}
      \locDiscFullVectNorm{\vectProjF{\polyDegree} \vectTestFun}^{\sobIndexConj}
      &\leq C(\polyDegree, \sobIndex, \rho)
      \euNorm{\element}^{1-\frac{\sobIndexConj}{2}}
      \big[
        \norm{\vectLebSpace{2}{\element}}{\vectProjF{\polyDegree}\vectTestFun}^{\sobIndexConj}
        +
        \left(
          h_{\element}^{\frac{1}{2}} \norm{\lebSpace{2}{\partial \element}}{(\vectProjF{\polyDegree}\vectTestFun - \locVectProjL{\polyDegree} \vectProjF{\polyDegree}\vectTestFun)\cdot
        \normalElement}\right)^{\sobIndexConj} \\
        &\quad
        +
        \left(
        h_{\element} \norm{\lebSpace{2}{\element}}{\divergence\vectProjF{\polyDegree}\vectTestFun}\right)^{\sobIndexConj}
      \big].
    \end{aligned}
  \end{equation}
  For the first term in the square bracket, we add and subtract $\vectTestFun$, and then we apply \eqref{eq:Fortin.approx.r-norm} to write
  \begin{equation}\label{eq:disc.inf.sup.g2.term1}
    \begin{aligned}
      \norm{\vectLebSpace{2}{\element}}{\vectProjF{\polyDegree}\vectTestFun} \leq
      \norm{\vectLebSpace{2}{\element}}{\vectTestFun} + \norm{\vectLebSpace{2}{\element}}{\vectTestFun-\vectProjF{\polyDegree}\vectTestFun}
      \leq C(\polyDegree, \rho)
      \norm{\vectSobSpace{1}{2}{\element}}{\vectTestFun},
    \end{aligned}
  \end{equation}
  where in the last inequality we have used the fact that $h_{\element} \leq C$.
  For the second term, combining \eqref{eq:r-norm.normal.component} and \eqref{eq:r-norm.div.virtual.func}, and then applying a triangle inequality followed by \eqref{eq:vectProjL.stab.element}, we obtain
  \begin{equation}\label{eq:disc.inf.sup.g2.term2}
    h_{\element}^{\frac{1}{2}} \norm{\lebSpace{2}{\partial \element}}{(\vectProjF{\polyDegree}\vectTestFun - \locVectProjL{\polyDegree} \vectProjF{\polyDegree}\vectTestFun)\cdot \normalElement}
    \leq C(\polyDegree,\rho)
    \norm{\vectLebSpace{2}{\element}}{\vectProjF{\polyDegree}\vectTestFun}
    \overset{\eqref{eq:disc.inf.sup.g2.term1}}\leq C(\polyDegree,\rho)
    \norm{\vectSobSpace{1}{2}{\element}}{\vectTestFun}.
  \end{equation}
  For the third term, we apply \eqref{eq:r-norm.div.virtual.func} together with \eqref{eq:disc.inf.sup.g2.term1} to write
  \begin{equation}\label{eq:disc.inf.sup.g2.term3}
    h_{\element} \norm{\lebSpace{2}{\element}}{\divergence\vectProjF{\polyDegree}\vectTestFun}
    \leq C(\polyDegree,\rho)
    \norm{\vectSobSpace{1}{2}{\element}}{\vectTestFun}.
  \end{equation}
  Plugging \eqref{eq:disc.inf.sup.g2.term1}, \eqref{eq:disc.inf.sup.g2.term2}, and \eqref{eq:disc.inf.sup.g2.term3} into \eqref{eq:disc.inf.sup.g2.initial} and then applying a
  $(\frac{\sobIndexConj}{2},\frac{\sobIndexConj}{\sobIndexConj-2})$-H{\"o}lder inequality, we get
  \begin{equation*}
    \locDiscFullVectNorm{\vectProjF{\polyDegree} \vectTestFun}^{\sobIndexConj}
    \leq C(\polyDegree, \sobIndex, \rho)
    \euNorm{\element}^{1-\frac{\sobIndexConj}{2}}
    \norm{\vectSobSpace{1}{2}{\element}}{\vectTestFun}^{\sobIndexConj}
    \leq C(\polyDegree, \sobIndex, \rho)
    \norm{\vectSobSpace{1}{\sobIndexConj}{\element}}{\vectTestFun}^{\sobIndexConj}.
  \end{equation*}

  \medskip
  \noindent
  \textbf{Conclusion.} In both cases, i.e. for $1<\sobIndexConj \leq 2$ and $\sobIndexConj > 2$, we have that
  \begin{equation*}
    \locDiscFullVectNorm{\vectProjF{\polyDegree} \vectTestFun}
    \leq C(\polyDegree, \sobIndex, \rho)
    \norm{\vectSobSpace{1}{\sobIndexConj}{\element}}{\vectTestFun}.
  \end{equation*}
  Summing the above inequality over $\element \in \domain_h$, using the stability bound in \eqref{eq:div.surj.stable}, and observing that $\frac{\sobIndex}{\sobIndexConj} = \sobIndex - 1$, we get
  \begin{equation} \label{eq:Fortin.cond.2}
    \discFullVectNorm{\vectProjF{\polyDegree} \vectTestFun}
    \leq
    C(\polyDegree, \sobIndex, \rho)
    \norm{\lebSpace{\sobIndexConj}{\domain}}{G}
    \overset{\eqref{eq:identity.norm.G.qh}}\leq
    C(\polyDegree, \sobIndexConj, \rho)
    \discScalarNorm{\discScalarTestFun}^{\sobIndex-1}
  \end{equation}
  The conclusion follows by combining \eqref{eq:Fortin.cond.1} and \eqref{eq:Fortin.cond.2}.
\end{proof}
%
\subsubsection{Continuity}
In this section, we prove property \eqref{eq:Ah.tilde.continuity}.
\begin{lemma}[$a_{\element}$ and $s_{\element}$ H{\"o}lder continuity]\label{lem:aE.Holder.continuity.disc.norms}
  Let $\element \in \domain_h$. Then, for any $\discVectTestFunTilda, \discVectFunOneTilda, \discVectFunTwoTilda \in \discVectSpace (\element)$, setting $\tilde{\boldsymbol{e}}_h \coloneqq \discVectFunOneTilda - \discVectFunTwoTilda$, it holds that
  \begin{align}
    \euNorm{ a_{\element}(\discVectFunOneTilda,\discVectTestFunTilda) - a_{\element}(\discVectFunTwoTilda,\discVectTestFunTilda)}
    &\lesssim
    \left( \norm{\vectLebSpace{\sobIndexConj}{\element}}{ \discVectFunOneTilda}^{\sobIndexConj} + \norm{\vectLebSpace{\sobIndexConj}{\element}}{ \discVectFunTwoTilda}^{\sobIndexConj}
    \right)^{\frac{\sobIndexConj - \sobIndexConjMin}{\sobIndexConj}}
    \norm{\vectLebSpace{\sobIndexConj}{\element}}{ \tilde{\boldsymbol{e}}_h}^{\sobIndexConjMin - 1}
    \norm{\vectLebSpace{\sobIndexConj}{\element}}{ \discVectTestFunTilda}, \label{eq:aE.Holder.continuity} \\
    \euNorm{ s_{\element}(\discVectFunOneTilda,\discVectTestFunTilda) - s_{\element}(\discVectFunTwoTilda,\discVectTestFunTilda)}
      &\lesssim
      \left[ s_{\element}(\discVectFunOneTilda,\discVectFunOneTilda) + s_{\element}(\discVectFunTwoTilda,\discVectFunTwoTilda) \right]^{\frac{\sobIndexConj - \sobIndexConjMin}{\sobIndexConj}} 
      \left[ s_{\element}(\tilde{\boldsymbol{e}}_h,\tilde{\boldsymbol{e}}_h) \right]^{\frac{\sobIndexConjMin - 1}{\sobIndexConj}}
      \left[ s_{\element}(\discVectTestFunTilda,\discVectTestFunTilda) \right]^{\frac{1}{\sobIndexConj}}, \label{eq:sE.Holder.continuity}
  \end{align}
  with a hidden constant depending only on $\sobIndex$.
\end{lemma}
\begin{proof}
  The proof of \eqref{eq:aE.Holder.continuity} is a direct consequence of \eqref{eq:diffusive.flux.properties.hc} with
  $(\boldsymbol{A}, \boldsymbol{B}, r, n)=(\discVectFunOneTilda, \discVectFunTwoTilda, \sobIndexConj,2)$ and a $\left(\frac{\sobIndexConj}{\sobIndexConj - \sobIndexConjMin}, \frac{\sobIndexConj}{\sobIndexConjMin - 1}, \sobIndexConj \right)$-H{\"o}lder inequality.
  While applying \eqref{eq:diffusive.flux.properties.hc} with
  $(\boldsymbol{A}, \boldsymbol{B}, r, n)=(\locDofVect \discVectFunOneTilda, \locDofVect \discVectFunTwoTilda, \sobIndexConj, N_{{\element}})$ and the identity $\frac{\sobIndexConj - \sobIndexConjMin}{\sobIndexConj} + \frac{\sobIndexConjMin - 1}{\sobIndexConj} + \frac{1}{\sobIndexConj} = 1$ lead to \eqref{eq:sE.Holder.continuity}.
\end{proof}
From the previous Lemma we immediately obtain the following result.
\begin{lemma}[$a_{h,\element}$ H{\"o}lder continuity]\label{lem:ahE.Holder.continuity}
  Let $\element \in \domain_h$. For any $\discVectTestFun, \discVectFunOne, \discVectFunTwo \in \discVectSpace (\element)$ it holds that
  \begin{equation}\label{eq:ahE.Holder.continuity}
    \euNorm{ a_{h,\element}(\discVectFunOne,\discVectTestFun) - a_{h,\element}(\discVectFunTwo,\discVectTestFun)}
    \lesssim
    \left( \locDiscFullVectNorm{ \discVectFunOne}^{\sobIndexConj} + \locDiscFullVectNorm{\discVectFunTwo}^{\sobIndexConj} \right)^{\frac{\sobIndexConj - \sobIndexConjMin}{\sobIndexConj}}
    \locDiscFullVectNorm{ \discVectFunOne - \discVectFunTwo}^{\sobIndexConjMin - 1}
    \locDiscFullVectNorm{ \discVectTestFun},
  \end{equation}
  with a hidden constant depending only on $\sobIndex$.
\end{lemma}
\begin{proof}
  Recalling the definition of $a_{h,\element}$ in \eqref{eq:def.ah.local}, the proof is a straightforward application of \eqref{eq:aE.Holder.continuity} with
  $(\discVectTestFunTilda, \discVectFunOneTilda, \discVectFunTwoTilda)=(\locVectProjL{\polyDegree}\discVectTestFun, \locVectProjL{\polyDegree}\discVectFunOne,
  \locVectProjL{\polyDegree}\discVectFunTwo) \in [\discVectSpace (\element)]^3$, and \eqref{eq:sE.Holder.continuity} with
  $(\discVectTestFunTilda, \discVectFunOneTilda, \discVectFunTwoTilda)=(\discVectTestFun - \locVectProjL{\polyDegree}\discVectTestFun, \discVectFunOne -
  \locVectProjL{\polyDegree}\discVectFunOne, \discVectFunTwo - \locVectProjL{\polyDegree}\discVectFunTwo) \in [\discVectSpace (\element)]^3$, together with a discrete $(\frac{\sobIndexConj}{\sobIndexConj-\sobIndexConjMin}, \frac{\sobIndexConj}{\sobIndexConjMin-1}, \sobIndexConj)$-H{\"o}lder inequality.
\end{proof}
\begin{lemma}[$a_{h}$ H{\"o}lder continuity]\label{lem:ah.Holder.continuity}
  For any $\discVectTestFun, \discVectFunOne, \discVectFunTwo \in \discVectSpace$ it holds that
  \begin{equation}\label{eq:ah.Holder.continuity}
    \euNorm{ a_{h}(\discVectFunOne,\discVectTestFun) - a_{h}(\discVectFunTwo,\discVectTestFun)}
    \lesssim
    \left( \discFullVectNorm{ \discVectFunOne}^{\sobIndexConj} + \discFullVectNorm{\discVectFunTwo}^{\sobIndexConj} \right)^{\frac{\sobIndexConj - \sobIndexConjMin}{\sobIndexConj}}
    \discFullVectNorm{ \discVectFunOne - \discVectFunTwo}^{\sobIndexConjMin - 1}
    \discFullVectNorm{ \discVectTestFun},
  \end{equation}
  with a hidden constant depending only on $\sobIndex$.
\end{lemma}
\begin{proof}
  Recalling the definition of $a_{h}$ in \eqref{eq:def.ah.bh}, applying \eqref{eq:ahE.Holder.continuity}, and a discrete $\left(\frac{\sobIndexConj}{\sobIndexConj - \sobIndexConjMin},
  \frac{\sobIndexConj}{\sobIndexConjMin - 1},\sobIndexConj \right)$-H{\"o}lder inequality yields the desired result.
\end{proof}
\begin{lemma}[Continuity of $\mathcal{A}_h$]\label{lem:Ah.Holder.cont}
  For any $\discVectFunOne,\discVectFunTwo, \discVectTestFun \in \discVectSpace$, it holds that
  \begin{equation}\label{eq:Ah.Holder.cont}
    \euNorm{ \langle \mathcal{A}_h\discVectFunOne - \mathcal{A}_h \discVectFunTwo, \discVectTestFun \rangle}
    \lesssim
    \left( \discFullVectNorm{ \discVectFunOne}^{\sobIndexConj} + \discFullVectNorm{\discVectFunTwo}^{\sobIndexConj} \right)^{\frac{\sobIndexConj - \sobIndexConjMin}{\sobIndexConj}}
    \discFullVectNorm{ \discVectFunOne - \discVectFunTwo}^{\sobIndexConjMin - 1}
    \discFullVectNorm{ \discVectTestFun},
  \end{equation}
  with a hidden constant depending only on $\sobIndex$.
\end{lemma}
\begin{proof}
  The proof is a straightforward application of the definition of $\mathcal{A}_h$ in \eqref{eq:def.abstract.operators} and Lemma~\ref{lem:ah.Holder.continuity}.
\end{proof}
\begin{lemma}[Continuity of $\tilde{\mathcal{A}}_{h,\tilde{\boldsymbol{\tau}}_h}$]\label{lem:continuity.Atildeh}
  For any $\tilde{\boldsymbol{\tau}}_h \in \discVectSpace$, and any $\discVectFunOne,\discVectFunTwo \in \discDivFreeSpace$ such that $\norm{\discDivFreeSpace}{\discVectFunOne -
  \discVectFunTwo} \to 0$, it holds that
  \[
    \norm{(\discDivFreeSpace)'}{\tilde{\mathcal{A}}_{h,\tilde{\boldsymbol{\tau}}_h} \discVectFunOne -\tilde{\mathcal{A}}_{h,\tilde{\boldsymbol{\tau}}_h}\discVectFunTwo} \to 0
  \]
\end{lemma}
\begin{proof}
  Let $\tilde{\boldsymbol{\tau}}_h \in \discVectSpace$ and $\discVectFunOne,\discVectFunTwo,\discVectTestFun \in \discDivFreeSpace$.
  Applying \eqref{eq:Ah.Holder.cont} with $(\discVectFunOne,\discVectFunTwo,\discVectTestFun) = (\discVectFunOne + \tilde{\boldsymbol{\tau}}_h, \discVectFunTwo +
  \tilde{\boldsymbol{\tau}}_h, \discVectTestFun) \in [\discVectSpace]^2 \times \discDivFreeSpace \subset [\discVectSpace]^3 $ and recalling the definition of
  $\tilde{\mathcal{A}}_{h,\tilde{\boldsymbol{\tau}}_h}$ in \eqref{eq:shifted.operator}, we obtain
  \begin{equation}\label{eq:continuity.Atildeh.initial}
    \euNorm{\langle \tilde{\mathcal{A}}_{h,\tilde{\boldsymbol{\tau}}_h} \discVectFunOne -\tilde{\mathcal{A}}_{h,\tilde{\boldsymbol{\tau}}_h}\discVectFunTwo, \discVectTestFun \rangle}
    \leq C(\sobIndex)
    \left( \discFullVectNorm{ \discVectFunOne + \tilde{\boldsymbol{\tau}}_h}^{\sobIndexConj} + \discFullVectNorm{\discVectFunTwo + \tilde{\boldsymbol{\tau}}_h}^{\sobIndexConj}
    \right)^{\frac{\sobIndexConj - \sobIndexConjMin}{\sobIndexConj}}
    \discFullVectNorm{ \discVectFunOne - \discVectFunTwo}^{\sobIndexConjMin - 1}
    \discFullVectNorm{ \discVectTestFun}.
  \end{equation}
  We endow $\discDivFreeSpace \subset \discVectSpace$ with the discrete norm $\discDivFreeNorm{\cdot}$ (cf. \eqref{eq:def.ah.norm}). Recalling that $\discDivFreeNorm{\cdot} =
  \discFullVectNorm{\cdot}$ on $\discDivFreeSpace$, applying \eqref{eq:continuity.Atildeh.initial}, and observing that $\sobIndexConjMin - 1 > 0$, we deduce the desired result from
  \begin{align*}
    \norm{(\discDivFreeSpace)'}{\tilde{\mathcal{A}}_{h,\tilde{\boldsymbol{\tau}}_h}\discVectFunOne - \tilde{\mathcal{A}}_{h,\tilde{\boldsymbol{\tau}}_h}\discVectFunTwo }
    \leq C(\sobIndex)
    \left( \discFullVectNorm{ \discVectFunOne + \tilde{\boldsymbol{\tau}}_h}^{\sobIndexConj} + \discFullVectNorm{\discVectFunTwo+ \tilde{\boldsymbol{\tau}}_h}^{\sobIndexConj}
    \right)^{\frac{\sobIndexConj - \sobIndexConjMin}{\sobIndexConj}}
    \discFullVectNorm{ \discVectFunOne - \discVectFunTwo}^{\sobIndexConjMin - 1}.
  \end{align*}
\end{proof}
%
\subsubsection{Coercivity}
In this section, we prove property \eqref{eq:Ah.tilde.coercivity}.
\begin{lemma}[$a_{\element}$ monotonicity]\label{lem:aE.monotonicity}
  For any $\discVectFunOneTilda, \discVectFunTwoTilda \in \discVectSpace$ it holds that
  \begin{align}\label{eq:aE.monotonicity}
    \sum_{\element \in \domain_h}
    a_{\element}(\discVectFunOneTilda,\discVectFunOneTilda-\discVectFunTwoTilda) - a_{\element}(\discVectFunTwoTilda,\discVectFunOneTilda-\discVectFunTwoTilda)
    &\gtrsim
    \left( \norm{\vectLebSpace{\sobIndexConj}{\domain}}{\discVectFunOneTilda}^{\sobIndexConj} + \norm{\vectLebSpace{\sobIndexConj}{\domain}}{\discVectFunTwoTilda}^{\sobIndexConj}
    \right)^{\frac{\sobIndexConjMin-2}{\sobIndexConj}}
    \norm{\vectLebSpace{\sobIndexConj}{\domain}}{ \discVectFunOneTilda - \discVectFunTwoTilda}^{\sobIndexConjMax},
  \end{align}
  with a hidden constant depending only on $\sobIndex$.
\end{lemma}
\begin{proof}
  We divide the proof into two cases.

  \medskip
  \noindent
  \textbf{Case $\sobIndexConj > 2$.}
  Applying \eqref{eq:diffusive.flux.properties.mono} with $(\boldsymbol{A}, \boldsymbol{B},r, n) = (\discVectFunOneTilda, \discVectFunTwoTilda, r, 2)$ on each $\element \in \domain_h$ and then summing over $\element \in \domain_h$ we obtain (notice that in this case $\sobIndexConjMax= \sobIndexConj$ and $\sobIndexConjMin = 2$)
  \begin{equation}\label{eq:aE.monotonicity.g2}
    \sum_{\element \in \domain_h}a_{\element}(\discVectFunOneTilda,\discVectFunOneTilda-\discVectFunTwoTilda) - a_{\element}(\discVectFunTwoTilda,\discVectFunOneTilda-\discVectFunTwoTilda)
    \geq C(\sobIndex)
    \norm{\vectLebSpace{\sobIndexConj}{\domain}}{\discVectFunOneTilda-\discVectFunTwoTilda}^{\sobIndexConj}.
  \end{equation}

  \medskip
  \noindent
  \textbf{Case $1 < \sobIndexConj \leq 2$.}
  Let $\element \in \domain_h$. We rewrite property \eqref{eq:diffusive.flux.properties.mono} with $(\boldsymbol{A}, \boldsymbol{B}, r,n)=(\discVectFunOneTilda,\discVectFunTwoTilda,\sobIndexConj,2)$ as (notice that in this case $\sobIndexConjMax= 2$ and $\sobIndexConjMin = \sobIndexConj$)
  \begin{equation*}
    \euNorm{\discVectFunOneTilda-\discVectFunTwoTilda}^{\sobIndexConj}
    \leq C(\sobIndex)
    \left[ \left(\fluxFunction (\discVectFunOneTilda) - \fluxFunction (\discVectFunTwoTilda) \right) \cdot (\discVectFunOneTilda-\discVectFunTwoTilda) \right]^{\frac{\sobIndexConj}{2}}
    \left( \euNorm{\discVectFunOneTilda}^{\sobIndexConj} + \euNorm{\discVectFunTwoTilda}^{\sobIndexConj} \right)^{\frac{2 -\sobIndexConj}{2}}.
  \end{equation*}
  We integrate the above inequality over $\element \in \domain_h$, use a $\left(\frac{2}{\sobIndexConj},\frac{2}{2 -\sobIndexConj}\right)$-H{\"o}lder inequality on the term in the right-hand, and then we sum over $\element \in \domain_h$ applying a discrete $\left(\frac{2}{\sobIndexConj},\frac{2}{2 -\sobIndexConj}\right)$-H{\"o}lder inequality obtaining
  \begin{equation*}
    \norm{\vectLebSpace{\sobIndexConj}{\domain}}{\discVectFunOneTilda-\discVectFunTwoTilda}^{\sobIndexConj}
    \leq C(\sobIndex)
    \left[ \sum_{\element \in \domain_h}a_{\element}(\discVectFunOneTilda,\discVectFunOneTilda-\discVectFunTwoTilda) -
    a_{\element}(\discVectFunTwoTilda,\discVectFunOneTilda-\discVectFunTwoTilda)\right]^{\frac{\sobIndexConj}{2}}
    \left( \norm{\vectLebSpace{\sobIndexConj}{\domain}}{\discVectFunOneTilda}^{\sobIndexConj} + \norm{\vectLebSpace{\sobIndexConj}{\domain}}{\discVectFunTwoTilda}^{\sobIndexConj}
    \right)^{\frac{2 -\sobIndexConj}{2}}.
  \end{equation*}
  We then conclude
  \begin{equation}\label{eq:aE.monotonicity.s2}
    \sum_{\element \in \domain_h}a_{\element}(\discVectFunOneTilda,\discVectFunOneTilda-\discVectFunTwoTilda) - a_{\element}(\discVectFunTwoTilda,\discVectFunOneTilda-\discVectFunTwoTilda)
    \geq C(\sobIndex)
    \left( \norm{\vectLebSpace{\sobIndexConj}{\domain}}{\discVectFunOneTilda}^{\sobIndexConj} + \norm{\vectLebSpace{\sobIndexConj}{\domain}}{\discVectFunTwoTilda}^{\sobIndexConj}
    \right)^{\frac{\sobIndexConj-2}{\sobIndexConj}}
    \norm{\vectLebSpace{\sobIndexConj}{\domain}}{\discVectFunOneTilda-\discVectFunTwoTilda}^{2}.
  \end{equation}

  \medskip
  \noindent
  \textbf{Conclusion.} Combining \eqref{eq:aE.monotonicity.s2} and \eqref{eq:aE.monotonicity.g2} yields the desired result.
\end{proof}
\begin{lemma}[$s_{\element}$ monotonicity]\label{lem:sE.monotonicity}
  For any $\discVectFunOneTilda, \discVectFunTwoTilda \in \discVectSpace$, setting $\tilde{\boldsymbol{e}}_h \coloneqq\discVectFunOneTilda - \discVectFunTwoTilda $, it holds that
  \begin{equation}\label{eq:sE.monotonicity}
    \begin{aligned}
      \sum_{\element \in \domain_h}
      \left[ s_{\element}(\discVectFunOneTilda,\tilde{\boldsymbol{e}}_h) - s_{\element}(\discVectFunTwoTilda,\tilde{\boldsymbol{e}}_h) \right]
      &\gtrsim
      \left[\sum_{\element \in \domain_h}
      \left[ s_{\element} (\discVectFunOneTilda,\discVectFunOneTilda) + s_{\element} (\discVectFunTwoTilda,\discVectFunTwoTilda)
      \right] \right]^{\frac{\sobIndexConjMin-2}{\sobIndexConj}} 
      \left[
      \sum_{\element \in \domain_h} 
      s_{\element} (\tilde{\boldsymbol{e}}_h, \tilde{\boldsymbol{e}}_h) \right]^{\frac{\sobIndexConjMax}{\sobIndexConj}},
    \end{aligned}
  \end{equation}
  with a hidden constant depending only on $\sobIndex$.
\end{lemma}
\begin{proof}
  We divide the proof into two cases.

  \medskip
  \noindent
  \textbf{Case $\sobIndexConj > 2$.} 
  Applying property \eqref{eq:diffusive.flux.properties.mono} with $(\boldsymbol{A}, \boldsymbol{B}, r, n)= (\discVectFunOneTilda,\discVectFunTwoTilda,\sobIndexConj,N_{\element})$ on each $\element \in \domain_h$ and then summing over $\element \in \domain_h$, we obtain
  \begin{equation}\label{eq:sE.monotonicity.g2}
    \sum_{\element \in \domain_h} \left[ s_{\element}(\discVectFunOneTilda,\tilde{\boldsymbol{e}}_h) - s_{\element}(\discVectFunTwoTilda,\tilde{\boldsymbol{e}}_h) \right]
    \geq C(\sobIndex)
    \sum_{\element \in \domain_h}
    s_{\element} (\tilde{\boldsymbol{e}}_h, \tilde{\boldsymbol{e}}_h)
  \end{equation}

  \medskip
  \noindent
  \textbf{Case $1 < \sobIndexConj \leq 2$.}
  Let $\element \in \domain_h$. We rewrite property \eqref{eq:diffusive.flux.properties.mono} with $(\boldsymbol{A}, \boldsymbol{B}, r, n)= (\locDofVect \discVectFunOneTilda,\locDofVect
  \discVectFunTwoTilda,\sobIndexConj,N_{\element})$ as
  \begin{equation*}
    \begin{aligned}
      \euNorm{\locDofVect(\tilde{\boldsymbol{e}}_h)}^{\sobIndexConj}
      &\leq C(\sobIndex)
      \left[ \left(\fluxFunction (\locDofVect \discVectFunOneTilda) - \fluxFunction (\locDofVect \discVectFunTwoTilda) \right) \cdot (\locDofVect(\tilde{\boldsymbol{e}}_h))
      \right]^{\frac{\sobIndexConj}{2}} 
      \left( \euNorm{\locDofVect \discVectFunOneTilda}^{\sobIndexConj} + \euNorm{\locDofVect \discVectFunTwoTilda}^{\sobIndexConj} \right)^{\frac{2 -\sobIndexConj}{2}}.
    \end{aligned}
  \end{equation*}
  We multiply both sides of the above inequality by $h_{\element}^{2}$ and, after observing that $\frac{\sobIndexConj}{2} + \frac{2 -\sobIndexConj}{2} = 1$, we sum over $\element \in \domain_h$, use a discrete $\left(\frac{2}{\sobIndexConj}, \frac{2}{2 -\sobIndexConj}\right)$-H{\"o}lder inequality and, after basic algebraic manipulations we obtain
  \begin{equation}\label{eq:sE.monotonicity.s2}
    \begin{aligned}
      \sum_{\element \in \domain_h}
      \left[s_{\element}(\discVectFunOneTilda,\tilde{\boldsymbol{e}}_h) - s_{\element}(\discVectFunTwoTilda,\tilde{\boldsymbol{e}}_h) \right]
      &\geq C(\sobIndex)
      \left[\sum_{\element \in \domain_h} \left[ s_{\element} (\discVectFunOneTilda,\discVectFunOneTilda) + s_{\element} (\discVectFunTwoTilda,\discVectFunTwoTilda) \right] \right]^{\frac{\sobIndexConj-2}{\sobIndexConj}} \left[ \sum_{\element \in \domain_h}
    s_{\element} (\tilde{\boldsymbol{e}}_h,\tilde{\boldsymbol{e}}_h) \right]^{\frac{2}{\sobIndexConj}}.
    \end{aligned}
  \end{equation}

  \medskip
  \noindent
  \textbf{Conclusion.}
  The desired result follows by combining \eqref{eq:sE.monotonicity.g2} and \eqref{eq:sE.monotonicity.s2}.
\end{proof}
\begin{lemma}[$a_{h}$ monotonicity]\label{lem:ah.monotonicity}
  For any $ \discVectFunOne, \discVectFunTwo \in \discVectSpace$ such that $\discVectFunOne - \discVectFunTwo \in \discDivFreeSpace $, it holds that
  \begin{align}\label{eq:ah.monotonicity}
    a_{h}(\discVectFunOne,\discVectFunOne-\discVectFunTwo) - a_{h}(\discVectFunTwo,\discVectFunOne-\discVectFunTwo)
    &\gtrsim
    \left( \discFullVectNorm{\discVectFunOne}^{\sobIndexConj} + \discFullVectNorm{\discVectFunTwo}^{\sobIndexConj} \right)^{\frac{\sobIndexConjMin-2}{\sobIndexConj}}
    \discFullVectNorm{\discVectFunOne - \discVectFunTwo }^{\sobIndexConjMax},
  \end{align}
  with a hidden constant depending only on $\sobIndex$.
\end{lemma}
\begin{proof}
  Let $\boldsymbol{\xi}_h \coloneqq \discVectFunOne - \discVectFunTwo$.
  Applying \eqref{eq:aE.monotonicity} with $(\discVectFunOneTilda,\discVectFunTwoTilda
  )=(\vectProjL{\polyDegree}\discVectFunOne,\vectProjL{\polyDegree}\discVectFunTwo) $ and \eqref{eq:sE.monotonicity} with $(\discVectFunOneTilda,\discVectFunTwoTilda )=(\discVectFunOne -
  \vectProjL{\polyDegree}\discVectFunOne,\discVectFunTwo - \vectProjL{\polyDegree}\discVectFunTwo) $, and recalling that $\frac{\sobIndexConjMin-2}{\sobIndexConj} < 0$ and the equality in \eqref{eq:disc.full.norm.eq.ah.norm},  we get
  \begin{align*}
    a_{h}(\discVectFunOne,\boldsymbol{\xi}_h) - a_{h}(\discVectFunTwo,\boldsymbol{\xi}_h)
    &\geq
    C(\sobIndex) \left( \discFullVectNorm{\discVectFunOne}^{\sobIndexConj} + \discFullVectNorm{\discVectFunTwo}^{\sobIndexConj} \right)^{\frac{\sobIndexConjMin-2}{\sobIndexConj}} \discFullVectNorm{\boldsymbol{\xi}_h}^{\sobIndexConjMax}.
  \end{align*}
\end{proof}
\begin{lemma}[$\tilde{\mathcal{A}}_{h,\tilde{\boldsymbol{\tau}}_h}$ coercivity]\label{lem:Ahtilde.coercivity}
  For any $\tilde{\boldsymbol{\tau}}_h \in \boldsymbol{\mathcal{K}} (\mathcal{F}_h) $ and any $\discVectFunOne \in \discDivFreeSpace$ such that $\discDivFreeNorm{\discVectFunOne } \to \infty$, it holds that
  \[
    \frac{\langle\tilde{\mathcal{A}}_{h,\tilde{\boldsymbol{\tau}}_h} \discVectFunOne, \discVectFunOne \rangle }{\discDivFreeNorm{\discVectFunOne }} \to \infty.
  \]
\end{lemma}
\begin{proof}
  Applying \eqref{eq:ah.monotonicity} with $(\discVectFunOne,\discVectFunTwo)=(\discVectFunOne + \tilde{\boldsymbol{\tau}}_h,\tilde{\boldsymbol{\tau}}_h) \in [\discVectSpace]^2$ and taking $(\discVectFunOne,\discVectFunTwo,\discVectTestFun)=(\tilde{\boldsymbol{\tau}}_h,\boldsymbol{0},\discVectFunOne) \in [\discVectSpace]^3$ in
  \eqref{eq:ah.Holder.continuity} leads to
  \begin{align*}
    a_{h}(\discVectFunOne + \tilde{\boldsymbol{\tau}}_h,\discVectFunOne)
    &\geq C(\sobIndex) \left[
      \left( \discFullVectNorm{\discVectFunOne + \tilde{\boldsymbol{\tau}}_h}^{\sobIndexConj} + \discFullVectNorm{\tilde{\boldsymbol{\tau}}_h}^{\sobIndexConj}
      \right)^{\frac{\sobIndexConjMin-2}{\sobIndexConj}}
      \discFullVectNorm{\discVectFunOne}^{\sobIndexConjMax}
      - \discFullVectNorm{ \tilde{\boldsymbol{\tau}}_h}^{\sobIndexConj-1}
    \discFullVectNorm{ \discVectFunOne}\right].
  \end{align*}
  Using the definition of $\tilde{\mathcal{A}}_{h,\tilde{\boldsymbol{\tau}}_h} $, the above estimate and the fact that $\discDivFreeNorm{ \discVectFunOne}=\discFullVectNorm{ \discVectFunOne}$, we obtain
  \begin{align*}
    \frac{\langle\tilde{\mathcal{A}}_{h,\tilde{\boldsymbol{\tau}}_h} \discVectFunOne, \discVectFunOne\rangle}{\discDivFreeNorm{\discVectFunOne}}
    &\geq C(\sobIndex) \left[
      \left( \discFullVectNorm{\discVectFunOne + \tilde{\boldsymbol{\tau}}_h}^{\sobIndexConj} + \discFullVectNorm{\tilde{\boldsymbol{\tau}}_h}^{\sobIndexConj}
      \right)^{\frac{\sobIndexConjMin-2}{\sobIndexConj}}
      \discDivFreeNorm{\discVectFunOne}^{\sobIndexConjMax-1}
    - \discFullVectNorm{ \tilde{\boldsymbol{\tau}}_h}^{\sobIndexConj-1}\right].
  \end{align*}
  The conclusion follows by taking the limit $\discDivFreeNorm{\discVectFunOne} \to \infty$
  and observing that $\sobIndexConjMin +
  \sobIndexConjMax-3 = \sobIndexConj-1 >0$.
\end{proof}
%
\subsubsection{Well-posedness of the discrete problem}
We are now in a position to prove the main result of this section.
\begin{theorem}
  Problem~\eqref{eq:discrete.problem} has a unique solution $(\discVectSol,\discScalarSol) \in \discVectSpace \times \scalarSolSpace$ which satisfies
  \begin{equation}\label{eq:stab.bound.disc.sols}
    \discFullVectNorm{\discVectSol} \lesssim \norm{\lebSpace{\sobIndexConj}{\domain}}{\loadTerm} \quad \text{and} \quad \discScalarNorm{\discScalarSol} \lesssim
    \norm{\lebSpace{\sobIndexConj}{\domain}}{\loadTerm}^{\sobIndexConj-1},
  \end{equation}
  with the hidden constants depending only on $\polyDegree$, $\sobIndex$, and $\rho$.
\end{theorem}
\begin{proof}
  We prove existence, uniqueness and the stability bound \eqref{eq:stab.bound.disc.sols} separately.

  \medskip
  \noindent
  \textbf{Existence.}
  The existence of a solution to Problem~\eqref{eq:discrete.problem} is guaranteed in light of the discussion in Section \ref{sec:framework} and the results in Lemmas \ref{lem:discrete.infsup},
  \ref{lem:continuity.Atildeh}, and \ref{lem:Ahtilde.coercivity}.

  \medskip
  \noindent
  \textbf{Stability bound.}
  Let $(\discVectSol,\discScalarSol)$ be a solution to the discrete problem. Then from the discrete inf-sup condition we have that
  \begin{equation}\label{eq:stab.bound.disc.sols.uh.initial}
    \begin{aligned}
      C(\polyDegree, \sobIndex, \rho) \discScalarNorm{\discScalarSol}
      &\leq
      \sup_{\discVectTestFun \in \discVectSpace} \frac{b(\discVectTestFun,\discScalarSol)}{\discFullVectNorm{\discVectTestFun}}
      \overset{\eqref{eq:discrete.problem.line1}}= \sup_{\discVectTestFun \in \discVectSpace} \frac{ - a_h(\discVectSol,\discVectTestFun)}{\discFullVectNorm{\discVectTestFun}}
      \leq \discFullVectNorm{\discVectSol}^{\sobIndexConj-1},
    \end{aligned}
  \end{equation}
  where in the last inequality we have applied \eqref{eq:ah.Holder.continuity} with $\discVectFunTwo = \boldsymbol{0}$.
  Furthermore, from the definition of the discrete norm in \eqref{eq:def.discrete.norm.NEW} and \eqref{eq:discrete.problem}, we have that
  \begin{equation}\label{eq:stab.bound.disc.sols.tauh.initial.initial}
    \begin{aligned}
      \discFullVectNorm{\discVectSol}^{\sobIndexConj}
      &= \int_{\domain} \loadTerm \discScalarSol + \sum_{\element \in \domain_h} h_{\element}^{\sobIndexConj}
      \norm{\lebSpace{\sobIndexConj}{\element}}{\loadTerm}^{\sobIndexConj} 
      \overset{\eqref{eq:stab.bound.disc.sols.uh.initial}}\leq C(\polyDegree, \sobIndex, \rho) \left(
        \norm{\lebSpace{\sobIndexConj}{\domain}}{\loadTerm} \discFullVectNorm{\discVectSol}^{\sobIndexConj-1} + h^{\sobIndexConj}
      \norm{\lebSpace{\sobIndexConj}{\domain}}{\loadTerm}^{\sobIndexConj} \right) \\
      &\leq \varepsilon \discFullVectNorm{\discVectSol}^{\sobIndexConj} + C(\polyDegree, \sobIndex, \rho,\varepsilon) \norm{\lebSpace{\sobIndexConj}{\domain}}{\loadTerm}^{\sobIndexConj},
    \end{aligned}
  \end{equation}
  where the last inequality follows by applying a $(\sobIndexConj,\sobIndex)$-Young inequality along with $h\leq C$.
  Choosing $\varepsilon = \frac{1}{2}$, after basic algebraic manipulations we get the first inequality in \eqref{eq:stab.bound.disc.sols}, which combined with
  \eqref{eq:stab.bound.disc.sols.uh.initial} yields the second inequality in \eqref{eq:stab.bound.disc.sols}.

  \medskip
  \noindent
  \textbf{Uniqueness.}
  Let $(\boldsymbol{\tau}_{h,1}, u_{h,1})$ and $(\boldsymbol{\tau}_{h,2},u_{h,2})$ be two solutions to Problem~\eqref{eq:discrete.problem}.
  Let $\boldsymbol{\xi}_h \coloneqq \boldsymbol{\tau}_{h,1} - \boldsymbol{\tau}_{h,2}$. It is straightforward to check that $\boldsymbol{\xi}_h \in \discDivFreeSpace$. We then have
  \begin{equation*}
    0\overset{\eqref{eq:discrete.problem.line1}}=a_h(\boldsymbol{\tau}_{h,1},\boldsymbol{\xi}_h) - a_h(\boldsymbol{\tau}_{h,2},\boldsymbol{\xi}_h)  
    \overset{\eqref{eq:stab.bound.disc.sols} \, \eqref{eq:ah.monotonicity}}
    \geq C(\polyDegree,\sobIndex,\rho,\loadTerm)
    \discFullVectNorm{\boldsymbol{\xi}_h}^{\overline{\sobIndexConj}}.
  \end{equation*}
  Therefore, $\boldsymbol{\tau}_{h,1} = \boldsymbol{\tau}_{h,2}$.
  From \eqref{eq:discrete.problem.line1}, since $\boldsymbol{\tau}_{h,1} = \boldsymbol{\tau}_{h,2}$, it follows that $b(\discVectTestFun,u_{h,1} - u_{h,2}) = a_h(\boldsymbol{\tau}_{h,1},\discVectTestFun) - a_h(\boldsymbol{\tau}_{h,2},\discVectTestFun) = 0$ for
  all $\discVectTestFun \in \discVectSpace$.
  We then have, using \eqref{eq:discrete.infsup} and $\boldsymbol{\tau}_{h,1} = \boldsymbol{\tau}_{h,2}$,
  \begin{equation*}
    \discScalarNorm{u_{h,1} - u_{h,2}} \leq C(\polyDegree,\sobIndex,\rho) \sup_{\discVectTestFun \in \discVectSpace} \frac{b(\discVectTestFun,u_{h,1} - u_{h,2})}{\discFullVectNorm{\discVectTestFun}} = 0.
  \end{equation*}
  Therefore, $u_{h,1} = u_{h,2}$.
\end{proof}
%
\section{A priori error estimates}\label{sec:a-priori}
This section is devoted to the derivation of a priori error estimates for both the flux $\vectSol$ and the scalar variable $\scalarSol$. Due to the non-linear nature of the constitutive law, the analysis is naturally divided into two main regimes depending on the Sobolev exponent $\sobIndexConj$.
As a preliminary step, in the next Lemma we collect two algebraic results that will play a fundamental role in the a priori error estimates in both the $1 < \sobIndexConj \leq 2$ and $\sobIndexConj > 2$ cases. The first one is an application of the generalized Young inequality (see \cite[Lemma 2.3]{BARRETT.LIU:1994a}) while the second is a straightforward vector equivalence.
\begin{lemma}[Young-type inequality and vector equivalence]
  For any real numbers $x,y,z \geq 0$, $r>1$, and $\varepsilon > 0$, there exists $c(\varepsilon) > 0$ such that
  \begin{equation}\label{eq:young.type.inequality}
    x (z+y)^{r-2} y \leq \varepsilon (z+x)^{r-2} x^2 + c(\varepsilon) (z+y)^{r-2} y^2.
  \end{equation}
  For any vectors $\boldsymbol{A}$, $\boldsymbol{B} \in \Real^n$, with $n \in \Natural$, it holds
  \begin{equation}\label{eq:simple.eu.norm.equivalence}
    2^{-1} \left(\euNorm{\boldsymbol{A}} + \euNorm{\boldsymbol{A} - \boldsymbol{B}}\right) \leq \euNorm{\boldsymbol{A}} + \euNorm{\boldsymbol{B}} \leq 2 \left(\euNorm{\boldsymbol{A}}+
    \euNorm{\boldsymbol{A} - \boldsymbol{B}} \right).
  \end{equation}
\end{lemma}
For the sake of readability, in what follows, for any sufficiently regular vector field $\vectTestFun$, we employ the following notation:
\begin{equation*}
  \vectTestFun^{\perp} \coloneqq \vectTestFun - \vectProjL{\polyDegree} \vectTestFun \quad \text{ and } \quad \vectTestFun_{I} \coloneqq \vectProjF{\polyDegree} \vectTestFun 
\end{equation*}
%
\subsection{A priori error estimates for $\vectSol$: $1<\sobIndexConj \leq 2$}\label{sec:a-priori.vect.1}
\subsubsection{Preliminary results}
%
\begin{lemma}[Discrete stability of the projected solution]
  Let $\vectSol \in \vectSolSpace$ be the unique solution to Problem~\eqref{eq:continuous.problem.dual.weak}. Assume $1 < \sobIndexConj \leq 2$ and assume furthermore that $\vectSol \in \vectSobSpace{1}{\sobIndexConj}{\domain_h}$. Then, it holds
  \begin{equation}\label{eq:projF.vectSol.stab}
    \discFullVectNorm{\vectSol_{I}}
    \lesssim  \norm{\lebSpace{\sobIndexConj}{\domain}}{\loadTerm} + \seminorm{\vectSobSpace{1}{\sobIndexConj}{\domain_h}}{\vectSol}
  \end{equation}
  with a hidden constant depending only on $\polyDegree$, $\sobIndex$, and $\rho$.
\end{lemma}
\begin{proof}
  Summing over $\element \in \domain_h$ the equivalence in \eqref{eq:disc.norm.simeq.cont.norm}, and then adding and subtracting
  $\vectSol$, we get
  \begin{align*}
    \discFullVectNorm{\vectSol_{I}}
    &\leq C(\polyDegree, \sobIndex, \rho) \left(\norm{\lebSpace{\sobIndexConj}{\domain}}{\vectSol - \vectSol_{I}} +
    \norm{\lebSpace{\sobIndexConj}{\domain}}{\vectSol}\right) \\
    &\leq C(\polyDegree, \sobIndex, \rho)
    \left(
      h \seminorm{\vectSobSpace{1}{\sobIndexConj}{\domain_h}}{\vectSol}
    + \norm{\lebSpace{\sobIndexConj}{\domain}}{\loadTerm} \right)\\
    &\leq C(\polyDegree, \sobIndex, \rho) \left( \seminorm{\vectSobSpace{1}{\sobIndexConj}{\domain_h}}{\vectSol} +\norm{\lebSpace{\sobIndexConj}{\domain}}{\loadTerm}\right),
  \end{align*}
  where the second inequality follows by summing over $\element \in \domain_h$ inequality \eqref{eq:Fortin.approx.r-norm}, by
  $h_{\element} \leq h$ for all $\element \in \domain_h$, and by the continuous stability bound \eqref{eq:cont.sols.stability}, while the last inequality holds since  $h \leq C$.
\end{proof}
\begin{lemma}
  Let $(\vectSol,\scalarSol) \in \vectSolSpace \times \scalarSolSpace$ and $(\discVectSol,\discScalarSol) \in \discVectSpace \times \discScalarSpace$ be the unique solutions to Problem~\eqref{eq:continuous.problem.dual.weak} and Problem~\eqref{eq:discrete.problem}, respectively. Assume $1<\sobIndexConj \leq 2$ and assume furthermore that $\vectSol \in \vectSobSpace{1}{\sobIndexConj}{\domain_h}$. Then, setting $\boldsymbol{\xi}_h \coloneqq \discVectSol-\vectSol_{I}$,
  \begin{equation}\label{eq:error.estimate.s2.mono}
    a_h(\discVectSol,\boldsymbol{\xi}_h) - a_h(\vectSol_{I},\boldsymbol{\xi}_h)
    \gtrsim
    \left(\norm{\lebSpace{\sobIndexConj}{\domain}}{\loadTerm}+ \seminorm{\vectSobSpace{1}{\sobIndexConj}{\domain_h}}{\vectSol} \right)^{\sobIndexConj-2} \discFullVectNorm{\boldsymbol{\xi}_h}^2,
  \end{equation}
  with a hidden constant depending only on $\polyDegree$, $\sobIndex$, and $\rho$.
\end{lemma}
\begin{proof}
  The proof is a straightforward application of \eqref{eq:ah.monotonicity} and the stability
  estimates \eqref{eq:stab.bound.disc.sols} and \eqref{eq:projF.vectSol.stab}.
\end{proof}
%
\subsubsection{Discrete error estimate}
%
%
\begin{lemma}[Error equation for $\vectSol$: $1 < \sobIndexConj \leq 2$]
  Let $(\vectSol,\scalarSol) \in \vectSolSpace \times \scalarSolSpace$ and $(\discVectSol,\discScalarSol) \in \discVectSpace \times \discScalarSpace$ be the unique solutions to Problem~\eqref{eq:continuous.problem.dual.weak} and Problem~\eqref{eq:discrete.problem}, respectively. Assume $1<\sobIndexConj \leq 2$. Then, for any $\discVectFunTwo \in \discDivFreeSpace$, the following error equation holds
  \begin{equation}\label{eq:error.equation.s2}
    a_h(\discVectSol,\discVectFunTwo) - a_h(\vectSol_{I},\discVectFunTwo) = a(\vectSol,\discVectFunTwo) - a_h(\vectSol_{I},\discVectFunTwo).
  \end{equation}
\end{lemma}
\begin{proof}
  Taking $\discVectTestFun = \discVectFunTwo$ in \eqref{eq:discrete.problem.line1}, $\vectTestFun = \discVectFunTwo$ in \eqref{eq:continuous.problem.dual.weak.1} and using the fact that
  $\divergence \discVectFunTwo = 0$, we obtain $a_h(\discVectSol,\discVectFunTwo) = 0 = a(\vectSol,\discVectFunTwo)$. Subtracting $a_h(\vectSol_{I},\discVectFunTwo)$ from
  both sides leads to \eqref{eq:error.equation.s2}.
\end{proof}
\begin{lemma}[Consistency error for $\vectSol$: $1 < \sobIndexConj \leq 2$]
  Let $(\vectSol,\scalarSol) \in \vectSolSpace \times \scalarSolSpace$ and $(\discVectSol,\discScalarSol) \in \discVectSpace \times \discScalarSpace$ be the unique solutions to Problem~\eqref{eq:continuous.problem.dual.weak} and Problem~\eqref{eq:discrete.problem}, respectively.
  Assume that $1<\sobIndexConj \leq 2$, and suppose furthermore that the exact solution satisfies the following regularities:
  \begin{itemize}
    \item[] $\vectSol \in \vectSobSpace{\polyDegree_{\vectSol}}{\sobIndexConj}{\domain_h}$ for $\polyDegree_{\vectSol} \in \{1,2,\ldots,\polyDegree +1 \}$;
    \item[] $\fluxFunction(\vectSol) \in \vectSobSpace{\polyDegree_{\fluxFunction}}{\sobIndex}{\domain_h}$ for $\polyDegree_{\fluxFunction} \in \{0,1,\ldots,\polyDegree +1 \}$.
  \end{itemize}
  Then, setting $\boldsymbol{\xi}_h \coloneqq \discVectSol - \vectSol_{I}$, for any $\varepsilon > 0$, it holds
  \begin{equation}\label{eq:consistency.error.s2}
    \begin{aligned}
      a(\vectSol,\boldsymbol{\xi}_h) - a_h(\vectSol_{I},\boldsymbol{\xi}_h)
      &\leq
      \varepsilon \discFullVectNorm{\boldsymbol{\xi}_h}^2 +
      \frac{1}{2} \left[ a_h(\discVectSol,\boldsymbol{\xi}_h) - a_h(\vectSol_{I},\boldsymbol{\xi}_h) \right] \\
      &\quad 
      +\frac{C(\polyDegree,\sobIndex,\polyDegree_{\fluxFunction},\rho)}{\varepsilon} \mathcal{R}_h^{(1)} 
      + C(\polyDegree,\sobIndex,\polyDegree_{\vectSol},\rho) \mathcal{R}_h^{(2)} ,
    \end{aligned}
  \end{equation}
  where 
  \begin{equation}\label{eq:def.consistency.error.s2.conv.terms}
    \begin{aligned}
      \mathcal{R}_h^{(1)} &\coloneqq \left[\sum_{\element \in \domain_h} h_{\element}^{\sobIndex \polyDegree_{\fluxFunction}} \seminorm{\vectSobSpace{\polyDegree_{\fluxFunction}}{\sobIndex}{\element}}{\fluxFunction
      (\vectSol) }^{\sobIndex}\right]^{\frac{2}{\sobIndex}}, \\
      \mathcal{R}_h^{(2)} &\coloneqq \sum_{\element \in \domain_h}
      h_{\element}^{\sobIndexConj \polyDegree_{\vectSol} }\seminorm{\vectSobSpace{\polyDegree_{\vectSol}}{\sobIndexConj}{\element}}{\vectSol}^{\sobIndexConj}.
    \end{aligned}
  \end{equation}
\end{lemma}
\begin{proof}
  Let $\boldsymbol{\xi}_h \coloneqq \discVectSol - \vectSol_{I}$.
  From the definitions of $a$ in \eqref{eq:def.a} and $a_h$ in \eqref{eq:def.ah.bh} we have that
  \begin{equation}\label{eq:error.estimate.s2.rhs.initial}
    \begin{aligned}
      a(\vectSol,\boldsymbol{\xi}_h) - a_h(\vectSol_{I},\boldsymbol{\xi}_h)
      &=
      \sum_{\element \in \domain_h} \int_{\element} \left(
        \fluxFunction (\vectSol) \cdot \boldsymbol{\xi}_h - \fluxFunction (\locVectProjL{\polyDegree} \vectSol_{I})\cdot \locVectProjL{\polyDegree}\boldsymbol{\xi}_h \right) 
      + \sum_{\element \in \domain_h}
        - s_{\element} (\vectSol_{I}^{\perp}, \boldsymbol{\xi}_h^{\perp}) \\
      &\eqqcolon
      \sum_{\element \in \domain_h} \mathfrak{T}_{1,\element}
      + \sum_{\element \in \domain_h} \mathfrak{T}_{2,\element}.
    \end{aligned}
  \end{equation}
  We analyze each of the two terms separately.

  \medskip
  \noindent
  \textbf{First term.}
  We first add and subtract $\locVectProjL{\polyDegree} \fluxFunction (\vectSol) \cdot \boldsymbol{\xi}_h$ inside the integral for $\mathfrak{T}_{1,\element}$, and exploiting the
  orthogonality property of $\locVectProjL{\polyDegree}$, we get
  \begin{equation}\label{eq:error.estimate.s2.rhs.T1E.initial}
    \begin{aligned}
      \mathfrak{T}_{1,\element}
      &=\int_{\element}
      \left[\fluxFunction (\vectSol) - \locVectProjL{\polyDegree} \fluxFunction (\vectSol) \right] \cdot \boldsymbol{\xi}_h
      + \int_{\element}
      \left[ \fluxFunction (\vectSol) - \fluxFunction (\locVectProjL{\polyDegree} \vectSol_{I}) \right] \cdot \locVectProjL{\polyDegree} \boldsymbol{\xi}_h 
      \eqqcolon \mathfrak{T}_{1,\element}^{(1)} + \mathfrak{T}_{1,\element}^{(2)}
    \end{aligned}
  \end{equation}
  For the first term $\mathfrak{T}_{1,\element}^{(1)}$ we apply a $(\sobIndex,\sobIndexConj)$-H{\"o}lder inequality together with \eqref{eq:vectProjL.approx.element} and \eqref{eq:disc.norm.simeq.cont.norm} to write
  \begin{equation*}
    \begin{aligned}
      \mathfrak{T}_{1,\element}^{(1)}
      \leq \norm{\vectLebSpace{\sobIndex}{\element}}{\fluxFunction (\vectSol) - \locVectProjL{\polyDegree} \fluxFunction (\vectSol)}
      \norm{\vectLebSpace{\sobIndexConj}{\element}}{\boldsymbol{\xi}_h } 
      \leq
      C(\polyDegree,\sobIndex,\polyDegree_{\fluxFunction},\rho) h_{\element}^{\polyDegree_{\fluxFunction}}
      \seminorm{\vectSobSpace{\polyDegree_{\fluxFunction}}{\sobIndex}{\element}}{\fluxFunction (\vectSol) } \locDiscFullVectNorm{\boldsymbol{\xi}_h}.
    \end{aligned}
  \end{equation*}
  Summing the above inequality over $\element \in \domain_h$, applying a discrete $(\sobIndex,\sobIndexConj)$-H{\"o}lder inequality and a generalized Young inequality yields, for any $\varepsilon > 0$,
  \begin{equation}\label{eq:error.estimate.s2.rhs.T1E.1}
    \begin{aligned}
      \sum_{\element \in \domain_h}\mathfrak{T}_{1,\element}^{(1)}
      &\leq \frac{C(\polyDegree,\sobIndex,\polyDegree_{\fluxFunction},\rho)}{\varepsilon}
      \left[\sum_{\element \in \domain_h} h_{\element}^{\sobIndex \polyDegree_{\fluxFunction}} \seminorm{\vectSobSpace{\polyDegree_{\fluxFunction}}{\sobIndex}{\element}}{\fluxFunction
      (\vectSol) }^{\sobIndex}\right]^{\frac{2}{\sobIndex}} +
      \varepsilon \discFullVectNorm{\boldsymbol{\xi}_h}^2.
    \end{aligned}
  \end{equation}
  For the second term $\mathfrak{T}_{1,\element}^{(2)}$, we apply \eqref{eq:diffusive.flux.properties.original.hc} and \eqref{eq:young.type.inequality}, together with \eqref{eq:simple.eu.norm.equivalence}, to write, for any $\varepsilon > 0$,
  \begin{equation*}
    \begin{aligned}\mathfrak{T}_{1,\element}^{(2)}
      &\leq C(\sobIndex)
      \int_{\element}
      \left( \euNorm{\locVectProjL{\polyDegree} \vectSol_{I}} + \euNorm{\vectSol - \locVectProjL{\polyDegree} \vectSol_{I}} \right)^{\sobIndexConj - 2}
      \euNorm{\vectSol - \locVectProjL{\polyDegree} \vectSol_{I}}
      \euNorm{\locVectProjL{\polyDegree} \boldsymbol{\xi}_h} \\ 
      &\leq
      \varepsilon
      \int_{\element}
      \left( \euNorm{\locVectProjL{\polyDegree} \vectSol_{I}} + \euNorm{\locVectProjL{\polyDegree} \discVectSol} \right)^{\sobIndexConj - 2}
      \euNorm{\locVectProjL{\polyDegree} \boldsymbol{\xi}_h}^2 
      + C(\sobIndex,\varepsilon)
      \norm{\vectLebSpace{\sobIndexConj}{\element}}{\vectSol - \locVectProjL{\polyDegree} \vectSol_{I}}^{\sobIndexConj}
    \end{aligned}
  \end{equation*}
  Applying \eqref{eq:diffusive.flux.properties.original.mono} and \eqref{eq:projL.Fortin.approx} to the term multiplied by $\varepsilon$ and by $C(\sobIndex,\varepsilon)$, respectively, and then summing over $\element \in \domain_h$ after choosing $\varepsilon = 1/2$, we get
  \begin{equation}\label{eq:error.estimate.s2.rhs.T1E.2}
    \begin{aligned}
      \sum_{\element \in \domain_h}
      \mathfrak{T}_{1,\element}^{(2)}
      &\leq
      \frac{1}{2} \sum_{\element \in \domain_h}
      \left[
        a_{\element} (\locVectProjL{\polyDegree} \discVectSol,\locVectProjL{\polyDegree} \boldsymbol{\xi}_h) - a_{\element} (\locVectProjL{\polyDegree}
      \vectSol_{I},\locVectProjL{\polyDegree} \boldsymbol{\xi}_h) \right]
      \\
      &\quad 
      + C(\polyDegree,\sobIndex,\polyDegree_{\vectSol},\rho)
      \sum_{\element \in \domain_h}
      h_{\element}^{\sobIndexConj \polyDegree_{\vectSol}} \seminorm{\vectSobSpace{\polyDegree_{\vectSol}}{\sobIndexConj}{\element}}{\vectSol}^{\sobIndexConj}
    \end{aligned}
  \end{equation}
  Plugging \eqref{eq:error.estimate.s2.rhs.T1E.1} and \eqref{eq:error.estimate.s2.rhs.T1E.2} into \eqref{eq:error.estimate.s2.rhs.T1E.initial} we get
  \begin{equation}\label{eq:error.estimate.s2.rhs.T1E}
    \begin{aligned}
      \sum_{\element \in \domain_h}
      \mathfrak{T}_{1,\element}
      &\leq
      \varepsilon \discFullVectNorm{\boldsymbol{\xi}_h}^2 +
      \frac{1}{2} \sum_{\element \in \domain_h}
      \left[
        a_{\element} (\locVectProjL{\polyDegree} \discVectSol,\locVectProjL{\polyDegree} \boldsymbol{\xi}_h) - a_{\element} (\locVectProjL{\polyDegree}
      \vectSol_{I},\locVectProjL{\polyDegree} \boldsymbol{\xi}_h) \right]\\
      &\quad +\frac{C(\polyDegree,\sobIndex,\polyDegree_{\fluxFunction},\rho)}{\varepsilon} \mathcal{R}_h^{(1)} 
      + C(\polyDegree,\sobIndex,\polyDegree_{\vectSol},\rho) \mathcal{R}_h^{(2)}.
    \end{aligned}
  \end{equation}

  \medskip
  \noindent
  \textbf{Second term.} 
  Recalling the definition of $s_{\element}$ in \eqref{eq:def.sE}, we apply \eqref{eq:diffusive.flux.properties.original.hc} and \eqref{eq:young.type.inequality} 
  , together with \eqref{eq:simple.eu.norm.equivalence} to obtain
  \begin{equation}\label{eq:aux.sE.ineq.for.g2.case}
    \begin{aligned}
      \mathfrak{T}_{2,\element}
      &\leq C(\sobIndex) h_{\element}^{2}
      \left( \euNorm{\locBoundaryDofVect (\vectSol_{I}^{\perp})} \right)^{\sobIndexConj-2}
      \euNorm{\locBoundaryDofVect (\vectSol_{I}^{\perp})}
      \euNorm{ \locBoundaryDofVect (\boldsymbol{\xi}_h^{\perp})} \\
      &\leq
      \varepsilon h_{\element}^{2}
      \left( \euNorm{\locBoundaryDofVect (\vectSol_{I}^{\perp})} + \euNorm{ \locBoundaryDofVect (\discVectSol^{\perp})} \right)^{\sobIndexConj-2}
      \euNorm{ \locBoundaryDofVect (\boldsymbol{\xi}_h^{\perp})}^2 + C(\sobIndex,\varepsilon) s_{\element} (\vectSol_{I}^{\perp},\vectSol_{I}^{\perp}).
    \end{aligned}
  \end{equation}
  Applying \eqref{eq:diffusive.flux.properties.original.mono} to the term multiplied by $\varepsilon$ and bound \eqref{eq:sE.less.cont.norm} to the term multiplied by $C(\sobIndex,\varepsilon)$, we get
  \begin{equation*}
    \begin{aligned}
      \mathfrak{T}_{2,\element}
      &\leq
      \varepsilon
      \left[s_{\element}(\discVectSol^{\perp},\boldsymbol{\xi}_h^{\perp}) -
      s_{\element}(\vectSol_{I}^{\perp},\boldsymbol{\xi}_h^{\perp})\right] + C(\polyDegree,\sobIndex,\rho,\varepsilon)
      \norm{\vectLebSpace{\sobIndexConj}{\element}}{\vectSol_{I}^{\perp}}^{\sobIndexConj}.
    \end{aligned}
  \end{equation*}
  Triangle inequality together with approximation results \eqref{eq:Fortin.approx.r-norm} and \eqref{eq:projL.Fortin.approx} yields
  \begin{equation*}
    \norm{\vectLebSpace{\sobIndexConj}{\element}}{\vectSol_{I}^{\perp}}^{\sobIndexConj} \leq C(\polyDegree,\sobIndex,\polyDegree_{\vectSol},\rho)
    h_{\element}^{\sobIndexConj \polyDegree_{\vectSol} }\seminorm{\vectSobSpace{\polyDegree_{\vectSol}}{\sobIndexConj}{\element}}{\vectSol}^{\sobIndexConj}.
  \end{equation*}
  Therefore, choosing $\varepsilon = 1/2$ and summing over $\element \in \domain_h$, we get
  \begin{equation}\label{eq:error.estimate.s2.rhs.T2E}
    \begin{aligned}
      \sum_{\element \in \domain_h}
      \mathfrak{T}_{2,\element}
      &\leq
      \frac{1}{2}
      \sum_{\element \in \domain_h}
      \left[s_{\element}(\discVectSol^{\perp},\boldsymbol{\xi}_h^{\perp}) -
      s_{\element}(\vectSol_{I}^{\perp},\boldsymbol{\xi}_h^{\perp})\right] + C(\polyDegree,\sobIndex,\polyDegree_{\vectSol},\rho) \mathcal{R}_h^{(2)}.
    \end{aligned}
  \end{equation}

  \medskip
  \noindent
  \textbf{Conclusion.}
  Plugging \eqref{eq:error.estimate.s2.rhs.T1E} and \eqref{eq:error.estimate.s2.rhs.T2E} into \eqref{eq:error.estimate.s2.rhs.initial}, and recalling the definition of $a_h$ in \eqref{eq:def.ah.bh} we obtain the desired result.
\end{proof}
\begin{theorem}[Discretization error for $\vectSol$: $1 < \sobIndexConj \leq 2$]\label{thm:discrete.error.s2}
  Let $(\vectSol,\scalarSol) \in \vectSolSpace \times \scalarSolSpace$ and $(\discVectSol,\discScalarSol) \in \discVectSpace \times \discScalarSpace$ be the unique solutions to Problem~\eqref{eq:continuous.problem.dual.weak} and Problem~\eqref{eq:discrete.problem}, respectively. Assume $1<\sobIndexConj \leq 2$, and assume furthermore
  \begin{itemize}
    \item[] $\vectSol \in \vectSobSpace{\polyDegree_{\vectSol}}{\sobIndexConj}{\domain_h}$ for $\polyDegree_{\vectSol} \in \{1,2,\ldots,\polyDegree +1 \}$;
    \item[] $\fluxFunction(\vectSol) \in \vectSobSpace{\polyDegree_{\fluxFunction}}{\sobIndex}{\domain_h}$ for $\polyDegree_{\fluxFunction} \in \{0,1,\ldots,\polyDegree +1 \}$.
  \end{itemize}
  Setting $\boldsymbol{\xi}_h \coloneqq \discVectSol - \vectSol_{I}$ and recalling the definition of $\mathcal{R}_h^{(1)}$ and $\mathcal{R}_h^{(2)}$ in \eqref{eq:def.consistency.error.s2.conv.terms}, it holds
  \begin{equation}\label{eq:discrete.error.s2}
    \discFullVectNorm{\boldsymbol{\xi}_h}^2
    \lesssim \mathcal{R}_h^{(1)} + \mathcal{R}_h^{(2)}
  \end{equation}
  with a hidden constant depending only on $\seminorm{\vectSobSpace{1}{\sobIndexConj}{\domain_h}}{\vectSol}$, $\norm{\lebSpace{\sobIndexConj}{\domain}}{\loadTerm}$, $\polyDegree$, $\sobIndex$, $\polyDegree_{\vectSol}$, $\polyDegree_{\fluxFunction}$, and $\rho$.
\end{theorem}
\begin{proof}
  Applying \eqref{eq:error.equation.s2} together with \eqref{eq:consistency.error.s2}, we obtain, for any $\varepsilon >0$,
  \begin{equation*}
    \begin{aligned}
      a_h(\discVectSol,\boldsymbol{\xi}_h) - a_h(\vectSol_{I},\boldsymbol{\xi}_h)
      &\leq
      \varepsilon \discFullVectNorm{\boldsymbol{\xi}_h}^2 +
      \frac{1}{2} \left[ a_h(\discVectSol,\boldsymbol{\xi}_h) - a_h(\vectSol_{I},\boldsymbol{\xi}_h) \right] \\
      &\quad +\frac{C(\polyDegree,\sobIndex,\polyDegree_{\fluxFunction},\rho)}{\varepsilon} \mathcal{R}_h^{(1)} 
      + C(\polyDegree,\sobIndex,\polyDegree_{\vectSol},\rho) \mathcal{R}_h^{(2)}.
    \end{aligned}
  \end{equation*}
  Rearranging the terms and applying \eqref{eq:error.estimate.s2.mono} we have that
  \begin{equation*}
    \begin{aligned}
      C(\polyDegree,\sobIndex,\rho)
      \left( \norm{\lebSpace{\sobIndexConj}{\domain}}{\loadTerm} + \seminorm{\vectSobSpace{1}{\sobIndexConj}{\domain_h}}{\vectSol} \right)^{\sobIndexConj-2} \discFullVectNorm{\boldsymbol{\xi}_h}^2
      &\leq
      \varepsilon \discFullVectNorm{\boldsymbol{\xi}_h}^2
      +\frac{C(\polyDegree,\sobIndex,\polyDegree_{\fluxFunction},\rho)}{\varepsilon} \mathcal{R}_h^{(1)} \\
      &\quad + C(\polyDegree,\sobIndex,\polyDegree_{\vectSol},\rho) \mathcal{R}_h^{(2)}.
    \end{aligned}
  \end{equation*}
  Choosing $\varepsilon \coloneqq \frac{1}{2} C_1 \left( \norm{\lebSpace{\sobIndexConj}{\domain}}{\loadTerm} + \seminorm{\vectSobSpace{1}{\sobIndexConj}{\domain_h}}{\vectSol} \right)^{\sobIndexConj-2}$, with $C_1$ denoting the constant on the left-hand side of the above inequality, after basic algebraic manipulations, we obtain the desired result.
\end{proof}
%
\subsubsection{Interpolation error estimate}
%
\begin{lemma}[Interpolation error for $\vectSol$: $1 < \sobIndexConj \leq 2$]\label{lem:interpolation.error.norm.s2}
  Let $(\vectSol,\scalarSol) \in \vectSolSpace \times \scalarSolSpace$ be the unique solution to Problem~\eqref{eq:continuous.problem.dual.weak}.
  Assume $1 < \sobIndexConj \leq 2$ and that $\vectSol \in \vectSobSpace{\polyDegree_{\vectSol}}{\sobIndexConj}{\domain_h}$ with $\polyDegree_{\vectSol} \in \{ 1, 2, \ldots, \polyDegree+1 \}$. Then, it holds
  \begin{equation}\label{eq:interpolation.error.s2}
    \discFullVectNorm{\vectSol - \vectSol_{I}}^{\sobIndexConj} \lesssim \sum_{\element \in \domain_h} h_{\element}^{\sobIndexConj \polyDegree_{\vectSol}}
    \seminorm{\vectSobSpace{\polyDegree_{\vectSol}}{\sobIndexConj}{\domain_h}}{\vectSol}^{\sobIndexConj},
  \end{equation}
  with a hidden constant depending only on $\polyDegree$, $\sobIndex$, $\polyDegree_{\vectSol}$, and $\rho$.
\end{lemma}
\begin{proof}
  First, we observe that the stabilization form $s_{\element}(\cdot, \cdot)$ is well-defined for functions in $\vectSobSpace{1}{\sobIndexConj}{\element}$. Indeed, from the definition of
  $s_{\element}$ in \eqref{eq:def.sE} and the definition of the degrees of freedom in \eqref{eq:DoFs}, everything is well-defined since functions in
  $\vectSobSpace{1}{\sobIndexConj}{\element}$ have normal traces in $\lebSpace{1}{\partial \element}$.
  Then, from the definition of the discrete norm, we have
  \begin{equation*}
    \begin{aligned}
      \discFullVectNorm{\vectSol - \vectSol_{I}}^{\sobIndexConj}
      &= \norm{\vectLebSpace{\sobIndexConj}{\domain}}{\vectProjL{\polyDegree}(\vectSol - \vectSol_{I})}^{\sobIndexConj} + \sum_{\element \in \domain_h} s_{\element} \left( (\vectSol - \vectSol_{I})^{\perp},
      (\vectSol - \vectSol_{I})^{\perp} \right)\\
      &\quad + \sum_{\element \in \domain_h}h_{\element}^{\sobIndexConj}\norm{\lebSpace{\sobIndexConj}{\element}}{\divergence (\vectSol - \vectSol_{I})}^{\sobIndexConj}.
    \end{aligned}
  \end{equation*}
  For the first term, using the boundedness of $\vectProjL{\polyDegree}$ and \eqref{eq:Fortin.approx.r-norm} we get
  \begin{equation*}
    \begin{aligned}
      \norm{\vectLebSpace{\sobIndexConj}{\domain}}{\vectProjL{\polyDegree}(\vectSol - \vectSol_{I})}^{\sobIndexConj}
      \leq C(\polyDegree,\sobIndexConj,\polyDegree_{\vectSol},\rho)
      \sum_{\element \in \domain_h}
      h_{\element}^{\sobIndexConj \polyDegree_{\vectSol}} \seminorm{\vectSobSpace{\polyDegree_{\vectSol}}{\sobIndexConj}{\element}}{\vectSol}^{\sobIndexConj}.
    \end{aligned}
  \end{equation*}
  For the second term, observing first that $\locBoundaryDofVect(\vectSol - \vectSol_{I}) = 0$, and then applying bound \eqref{eq:sE.less.cont.norm} together with the stability of $\vectProjL{\polyDegree}$ and the approximation property \eqref{eq:Fortin.approx.r-norm}, we get
  \begin{equation*}
    \begin{aligned}
      \sum_{\element \in \domain_h} s_{\element} \left( (\vectSol - \vectSol_{I})^{\perp}, (\vectSol -
      \vectSol_{I})^{\perp} \right) &= \sum_{\element \in \domain_h} s_{\element} \left(\vectProjL{\polyDegree}(\vectSol - \vectSol_{I}), \vectProjL{\polyDegree}(\vectSol -
      \vectSol_{I}) \right) \\
      &\leq C(\polyDegree,\sobIndexConj,\rho,\polyDegree_{\vectSol}) \sum_{\element \in \domain_h} h_{\element}^{\sobIndexConj
      \polyDegree_{\vectSol}} \seminorm{\vectSobSpace{\polyDegree_{\vectSol}}{\sobIndexConj}{\element}}{\vectSol}^{\sobIndexConj}.
    \end{aligned}
  \end{equation*}
  Finally, for the third term involving the divergence, bound \eqref{eq:Fortin.approx.div.r-norm} yields
    \begin{equation*}
      \sum_{\element \in \domain_h} h_{\element}^{\sobIndexConj}\norm{\lebSpace{\sobIndexConj}{\element}}{\divergence (\vectSol - \vectSol_{I})}^{\sobIndexConj}
      \leq C(\polyDegree,\sobIndexConj,\rho,\polyDegree_{\vectSol})
      \sum_{\element \in \domain_h} h_{\element}^{\sobIndexConj \polyDegree_{\vectSol}} \seminorm{\vectSobSpace{\polyDegree_{\vectSol}}{\sobIndexConj}{\element}}{\vectSol}^{\sobIndexConj}.
    \end{equation*}
    Combining the three bounds above yields the desired result.
\end{proof}
As a direct consequence of Theorem~\ref{thm:discrete.error.s2}, Lemma~\ref{lem:interpolation.error.norm.s2}, and the triangle inequality, we obtain the following result.
\begin{corollary}\label{cor:total.error}
  Under the same assumptions of Theorem~\ref{thm:discrete.error.s2}, it holds
  \begin{equation}\label{eq:total.error.s2}
    \discFullVectNorm{\vectSol - \discVectSol}
    \lesssim \left(\mathcal{R}_h^{(1)}\right)^{\frac{1}{2}} + \left(\mathcal{R}_h^{(2)}\right)^{\frac{1}{2}} + \left(\mathcal{R}_h^{(2)}\right)^{\frac{1}{\sobIndexConj}},
  \end{equation}
  with a hidden constant depending only on $\seminorm{\vectSobSpace{1}{\sobIndexConj}{\domain_h}}{\vectSol}$, $\norm{\lebSpace{\sobIndexConj}{\domain}}{\loadTerm}$, $\polyDegree$, $\sobIndexConj$, $\polyDegree_{\vectSol}$, $\polyDegree_{\fluxFunction}$, and $\rho$.

  In particular, assuming full regularity, i.e., $\polyDegree_{\vectSol} = \polyDegree + 1$ and $\polyDegree_{\fluxFunction} = \polyDegree + 1$, estimate \eqref{eq:total.error.s2} yields the explicit asymptotic convergence rate
  \begin{equation}\label{eq:asymptotic.rate.s2}
    \discFullVectNorm{\vectSol - \discVectSol} \lesssim h^{\frac{\sobIndexConj}{2}(\polyDegree+1)}.
  \end{equation}
\end{corollary}
%
\subsection{A priori error estimates for $\vectSol$: $\sobIndexConj > 2$}\label{sec:a-priori.vect.2}
We start this section by observing that, in the case $\sobIndexConj > 2$, a straightforward application of \eqref{eq:ah.monotonicity} leads to
\begin{equation}\label{eq:error.estimate.g2.mono}
  a_h(\discVectSol,\discVectSol-\vectSol_{I}) - a_h(\vectSol_{I},\discVectSol-\vectSol_{I})
  \geq
  C(\sobIndex)
  \discFullVectNorm{\discVectSol-\vectSol_{I}}^{\sobIndexConj}.
\end{equation}
%
\subsubsection{Discrete error estimate}
%
\begin{lemma}[Error equation for $\vectSol$: $\sobIndexConj > 2$]
  Let $(\vectSol,\scalarSol) \in \vectSolSpace \times \scalarSolSpace$ and $(\discVectSol,\discScalarSol) \in \discVectSpace \times \discScalarSpace$ be the unique solutions to Problem
  \ref{eq:continuous.problem.dual.weak} and Problem \ref{eq:discrete.problem}, respectively. Assume $\sobIndexConj > 2$ and $\fluxFunction (\vectSol) \in
  \vectSobSpace{1}{\sobIndex}{\domain_h}$. Then, for any $\discVectFunTwo \in \discVectSpace$ such that $\divergence \discVectFunTwo = 0$, the following error equation holds
  \begin{equation}\label{eq:error.equation.g2}
    a_h(\discVectSol,\discVectFunTwo) - a_h(\vectSol_{I},\discVectFunTwo) = \int_{\domain} \fluxFunction(\vectSol) \cdot \discVectFunTwo -
    a_h(\vectSol_{I},\discVectFunTwo).
  \end{equation}
\end{lemma}
\begin{proof}
  Taking $\discVectTestFun = \discVectFunTwo$ in \eqref{eq:discrete.problem.line1}, and using the fact that $\divergence \discVectFunTwo = 0$, we obtain $a_h(\discVectSol,\discVectFunTwo) = 0$.
  Furthermore, since $\scalarSol$ is the solution of \eqref{eq:continuous.problem.dual.weak.1}, we have that $\grad \scalarSol = \fluxFunction (\vectSol) \in
  \vectSobSpace{1}{\sobIndex}{\domain_h} \subset \vectLebSpace{2}{\domain}$ and also $\scalarSol = 0$ on $\domainBoundary$.
  Therefore, recalling that $\discVectFunTwo \in \Hdiv{\domain}$ and applying the integration by parts formula, we obtain
  \begin{align*}
    0&=\int_{\domain} (\fluxFunction(\vectSol) - \grad \scalarSol) \cdot \discVectFunTwo \\
    &=
    \int_{\domain} \fluxFunction(\vectSol) \cdot \discVectFunTwo
    - \left(
      \int_{\domainBoundary} (\discVectFunTwo \cdot \normalDomain) \scalarSol - \int_{\domain} \scalarSol \divergence \discVectFunTwo
    \right)
    =\int_{\domain} \fluxFunction(\vectSol) \cdot \discVectFunTwo.
  \end{align*}
  We conclude then $a_h(\discVectSol,\discVectFunTwo) = \int_{\domain} \fluxFunction(\vectSol) \cdot \discVectFunTwo$. Subtracting from both sides
  $a_h(\vectSol_{I},\discVectFunTwo)$ leads to the desired result.
\end{proof}
\begin{lemma}[Consistency error for $\vectSol$: $\sobIndexConj > 2$]
  Let $(\vectSol,\scalarSol) \in \vectSolSpace \times \scalarSolSpace$ and $(\discVectSol,\discScalarSol) \in \discVectSpace \times \discScalarSpace$ be the unique solutions to Problem~\eqref{eq:continuous.problem.dual.weak} and Problem \eqref{eq:discrete.problem}, respectively. Assume $\sobIndexConj > 2$, and suppose furthermore that the exact solution satisfies the
  following regularities:
  \begin{itemize}
    \item[] $\vectSol \in \vectSobSpace{\polyDegree_{\vectSol}}{\sobIndexConj}{\domain_h}$ for $\polyDegree_{\vectSol} \in \{1,2,\ldots,\polyDegree +1 \}$;
    \item[] $\fluxFunction(\vectSol) \in \vectSobSpace{\polyDegree_{\fluxFunction}}{\sobIndex}{\domain_h}$ for $\polyDegree_{\fluxFunction} \in \{1,\ldots,\polyDegree +1 \}$.
  \end{itemize}
  Then, setting $\boldsymbol{\xi}_h\coloneqq\discVectSol - \vectSol_{I}$, for any $\varepsilon > 0$, it holds
  \begin{equation}\label{eq:consistency.error.g2}
    \begin{aligned}
      \int_{\domain} \fluxFunction(\vectSol) \cdot \boldsymbol{\xi}_h - a_h(\vectSol_{I},\boldsymbol{\xi}_h)
      &\leq
      \varepsilon \discFullVectNorm{\boldsymbol{\xi}_h}^{\sobIndexConj}+
      \frac{1}{2} \left[ a_h(\discVectSol,\boldsymbol{\xi}_h) - a_h(\vectSol_{I},\boldsymbol{\xi}_h) \right] \\
      &\quad + \frac{C(\polyDegree,\sobIndex,\polyDegree_{\fluxFunction},\rho)}{\varepsilon} \tilde{\mathcal{R}}_h^{(1)} 
      + C(\polyDegree,\sobIndex,\polyDegree_{\vectSol},\rho) \left( \tilde{\mathcal{R}}_h^{(2)} \right)^{\frac{\sobIndexConj}{2}} \\
      &\quad + C(\polyDegree,\sobIndex,\polyDegree_{\vectSol},\rho,\domain) \left( \norm{\lebSpace{\sobIndexConj}{\domain}}{\loadTerm} + \seminorm{\vectSobSpace{1}{\sobIndexConj}{\domain_h}}{\vectSol} \right)^{\sobIndexConj-2} \tilde{\mathcal{R}}_h^{(2)} ,
    \end{aligned}
  \end{equation}
  where 
  \begin{equation}\label{eq:def.consistency.error.g2.conv.terms}
    \begin{aligned}
      \tilde{\mathcal{R}}_h^{(1)} &\coloneqq \sum_{\element \in \domain_h} h_{\element}^{\sobIndex \polyDegree_{\fluxFunction}} \seminorm{\vectSobSpace{\polyDegree_{\fluxFunction}}{\sobIndex}{\element}}{\fluxFunction
      (\vectSol) }^{\sobIndex}, \\
      \tilde{\mathcal{R}}_h^{(2)} &\coloneqq \left[\sum_{\element \in \domain_h}
      h_{\element}^{\sobIndexConj \polyDegree_{\vectSol} }\seminorm{\vectSobSpace{\polyDegree_{\vectSol}}{\sobIndexConj}{\element}}{\vectSol}^{\sobIndexConj}\right]^{\frac{2}{\sobIndexConj}}.
    \end{aligned}
  \end{equation}
\end{lemma}
\begin{proof}
  Let $\boldsymbol{\xi}_h \coloneqq \discVectSol - \vectSol_{I}$.
  From the definition of $a_h$ in \eqref{eq:def.ah.bh} we have that
  \begin{equation}\label{eq:error.estimate.g2.rhs.initial}
    \begin{aligned}
      \int_{\domain} \fluxFunction(\vectSol) \cdot \boldsymbol{\xi}_h - a_h(\vectSol_{I},\boldsymbol{\xi}_h)
      &= \sum_{\element \in \domain_h} \int_{\element} \left(\fluxFunction(\vectSol) \cdot \boldsymbol{\xi}_h - \fluxFunction(\locVectProjL{\polyDegree}\vectSol_{I}) \cdot
      \locVectProjL{\polyDegree}\boldsymbol{\xi}_h \right) + \sum_{\element \in \domain_h}
        - s_{\element} (\vectSol_{I}^{\perp},\vectSol_{I}^{\perp}) \\
      &\eqqcolon
      \sum_{\element \in \domain_h} \mathfrak{T}_{1,\element}
      + \sum_{\element \in \domain_h} \mathfrak{T}_{2,\element}.
    \end{aligned}
  \end{equation}

  \medskip
  \noindent
  \textbf{First term.}
  Applying the same argument that leads to \eqref{eq:error.estimate.s2.rhs.T1E.initial} we obtain
  \begin{equation}\label{eq:error.estimate.g2.rhs.T1E.initial}
    \begin{aligned}
      \mathfrak{T}_{1,\element}
      &=\int_{\element}
      \left[\fluxFunction (\vectSol) - \locVectProjL{\polyDegree} \fluxFunction (\vectSol) \right] \cdot \boldsymbol{\xi}_h
      + \int_{\element}
      \left[ \fluxFunction (\vectSol) - \fluxFunction (\locVectProjL{\polyDegree} \vectSol_{I}) \right] \cdot \locVectProjL{\polyDegree} \boldsymbol{\xi}_h \eqqcolon \mathfrak{T}_{1,\element}^{(1)} + \mathfrak{T}_{1,\element}^{(2)}.
    \end{aligned}
  \end{equation}
  For the first term in \eqref{eq:error.estimate.g2.rhs.T1E.initial}, using the orthogonality property of $\locVectProjL{\polyDegree}$,  applying the Cauchy-Schwarz inequality along with the standard polynomial interpolation result (see, e.g.,
  \cite[Section 4.3]{BRENNER.SCOTT:2008})
  $\norm{\vectLebSpace{2}{\element}}{\fluxFunction (\vectSol) - \locVectProjL{\polyDegree} \fluxFunction (\vectSol)} \leq C(\polyDegree,\polyDegree_{\fluxFunction},\rho)
  h_{\element}^{\polyDegree_{\fluxFunction}-1} \seminorm{\vectSobSpace{\polyDegree_{\fluxFunction}}{1}{\element}}{\fluxFunction (\vectSol)}$, and following the same argument that leads to
  \eqref{eq:cont.norm.less.disc.norm.term.2} together with \eqref{eq:poly.sob.embedding}, we obtain
  \begin{equation}\label{eq:aux.for.u.g2}
    \begin{aligned}
      \mathfrak{T}_{1,\element}^{(1)}
      &\leq \norm{\vectLebSpace{2}{\element}}{\fluxFunction (\vectSol) - \locVectProjL{\polyDegree} \fluxFunction (\vectSol)} \norm{\vectLebSpace{2}{\element}}{\boldsymbol{\xi}_h^{\perp} } \\
      &\leq C(\polyDegree,\polyDegree_{\fluxFunction},\rho) h_{\element}^{\polyDegree_{\fluxFunction}-1} \seminorm{\vectSobSpace{\polyDegree_{\fluxFunction}}{1}{\element}}{\fluxFunction
      (\vectSol)} \left[ \norm{\vectLebSpace{2}{\element}}{\locVectProjL{\polyDegree} \boldsymbol{\xi}_h} + h_{\element}^{\frac{1}{2}}\norm{\lebSpace{2}{\partial
      \element}}{\boldsymbol{\xi}_h^{\perp} \cdot \normalElement} \right] \\
      &\leq
      C(\polyDegree,\polyDegree_{\fluxFunction},\rho) h_{\element}^{\polyDegree_{\fluxFunction}} \seminorm{\vectSobSpace{\polyDegree_{\fluxFunction}}{\sobIndex}{\element}}{\fluxFunction (\vectSol)}
      \locDiscFullVectNorm{\boldsymbol{\xi}_h}.
    \end{aligned}
  \end{equation}
  Therefore, summing over $\element \in \domain_h$, applying a discrete $(\sobIndex,\sobIndexConj)$-H{\"o}lder inequality and then a generalized $(\sobIndex,\sobIndexConj)$-Young's inequality, we arrive at
  \begin{equation}\label{eq:error.estimate.g2.rhs.T1E.1}
    \begin{aligned}
      \sum_{\element \in \domain_h}\mathfrak{T}_{1,\element}^{(1)}
      &\leq \frac{C(\polyDegree,\sobIndex,\polyDegree_{\fluxFunction},\rho)}{\varepsilon}
      \sum_{\element \in \domain_h} h_{\element}^{\sobIndex \polyDegree_{\fluxFunction}} \seminorm{\vectSobSpace{\polyDegree_{\fluxFunction}}{\sobIndex}{\element}}{\fluxFunction
      (\vectSol)}^{\sobIndex} +
      \varepsilon \discFullVectNorm{\boldsymbol{\xi}_h}^{\sobIndexConj}.
    \end{aligned}
  \end{equation}
  For the second term in \eqref{eq:error.estimate.g2.rhs.T1E.initial}, applying first \eqref{eq:diffusive.flux.properties.original.hc} together with \eqref{eq:young.type.inequality} and \eqref{eq:simple.eu.norm.equivalence}, and then applying \eqref{eq:diffusive.flux.properties.original.mono} we obtain for any $\varepsilon > 0$
  \begin{equation}\label{eq:error.estimate.g2.rhs.T1E.2.initial}
    \begin{aligned}
      \mathfrak{T}_{1,\element}^{(2)}
      &\leq\varepsilon
      \int_{\element}
      \left( \fluxFunction(\locVectProjL{\polyDegree} \discVectSol) - \fluxFunction(\locVectProjL{\polyDegree} \vectSol_{I}) \right) \cdot
      \locVectProjL{\polyDegree} \boldsymbol{\xi}_h \\
      &\quad
      + C(\sobIndex,\varepsilon) \int_{\element}
      \left( \euNorm{\locVectProjL{\polyDegree} \vectSol_{I}} + \euNorm{\vectSol - \locVectProjL{\polyDegree} \vectSol_{I}} \right)^{\sobIndexConj - 2}
      \euNorm{\vectSol - \locVectProjL{\polyDegree} \vectSol_{I}}^2.
    \end{aligned}
  \end{equation}
  Applying a $\left(\frac{\sobIndexConj}{\sobIndexConj - 2},\frac{\sobIndexConj}{2}\right)$-H{\"o}lder inequality to the second integral in \eqref{eq:error.estimate.g2.rhs.T1E.2.initial},
  summing over $\element
  \in \domain_h$, and then applying a discrete $\left(\frac{\sobIndexConj}{\sobIndexConj - 2},\frac{\sobIndexConj}{2}\right)$-H{\"o}lder inequality, we arrive at
  %
  \begin{equation}\label{eq:error.estimate.g2.rhs.sum.T1E.2.initial}
    \begin{aligned}
      \sum_{\element \in \domain_h} \mathfrak{T}_{1,\element}^{(2)}
      &\leq
      \varepsilon
      \sum_{\element \in \domain_h}
      \left[ a_{\element}(\locVectProjL{\polyDegree} \discVectSol, \locVectProjL{\polyDegree} \boldsymbol{\xi}_h ) -a_{\element}(\locVectProjL{\polyDegree} \vectSol_{I},
      \locVectProjL{\polyDegree} \boldsymbol{\xi}_h ) \right] \\
      &\quad
      +
      C(\sobIndex,\varepsilon)
      \left[
        \norm{\vectLebSpace{\sobIndexConj}{\domain}}{\locVectProjL{\polyDegree} \vectSol_{I}}^{\sobIndexConj} +
        \norm{\vectLebSpace{\sobIndexConj}{\domain}}{\vectSol}^{\sobIndexConj}
      \right]^{\frac{\sobIndexConj-2}{\sobIndexConj}}
      \left[
        \sum_{\element \in \domain_h}\norm{\vectLebSpace{\sobIndexConj}{\element}}{\vectSol - \locVectProjL{\polyDegree} \vectSol_{I}}^{\sobIndexConj}
      \right]^{\frac{2}{\sobIndexConj}}.
    \end{aligned}
  \end{equation}
  We first observe that, from the triangle inequality, \eqref{eq:projL.Fortin.approx} and \eqref{eq:cont.sols.stability}, it follows that
  \begin{align*}
    \left[ \norm{\vectLebSpace{\sobIndexConj}{\domain}}{\locVectProjL{\polyDegree} \vectSol_{I}}^{\sobIndexConj} +
    \norm{\vectLebSpace{\sobIndexConj}{\domain}}{\vectSol}^{\sobIndexConj} \right]^{\frac{\sobIndexConj-2}{\sobIndexConj}} 
    &\leq C(\polyDegree,\sobIndex,\rho)
    \left[ h^{\sobIndexConj} \seminorm{\vectSobSpace{1}{\sobIndexConj}{\domain}}{\vectSol}^{\sobIndexConj} +
    \norm{\vectLebSpace{\sobIndexConj}{\domain}}{\vectSol}^{\sobIndexConj} \right]^{\frac{\sobIndexConj-2}{\sobIndexConj}} \\
    &\leq C(\polyDegree,\sobIndex,\rho) \left(\seminorm{\vectSobSpace{1}{\sobIndexConj}{\domain}}{\vectSol} + \norm{\lebSpace{\sobIndexConj}{\domain}}{\loadTerm}\right)^{\sobIndexConj-2},
  \end{align*}
  Then plugging this last estimate into \eqref{eq:error.estimate.g2.rhs.sum.T1E.2.initial}, using again \eqref{eq:projL.Fortin.approx} and choosing $\varepsilon = \frac{1}{2}$, we obtain
  \begin{align*}
    \sum_{\element \in \domain_h} \mathfrak{T}_{1,\element}^{(2)}
    &\leq
    \frac{1}{2}
    \left[ a_{\element}(\locVectProjL{\polyDegree} \discVectSol, \locVectProjL{\polyDegree} \boldsymbol{\xi}_h ) -a_{\element}(\locVectProjL{\polyDegree} \vectSol_{I},
    \locVectProjL{\polyDegree} \boldsymbol{\xi}_h ) \right] \\
    &\quad
    +
    C(\polyDegree,\sobIndex,\polyDegree_{\vectSol},\rho) \left(\seminorm{\vectSobSpace{1}{\sobIndexConj}{\domain}}{\vectSol} + \norm{\lebSpace{\sobIndexConj}{\domain}}{\loadTerm}\right)^{\sobIndexConj-2}
    \left[
      \sum_{\element \in \domain_h}
      h_{\element}^{\sobIndexConj \polyDegree_{\vectSol}} \seminorm{\vectSobSpace{\polyDegree_{\vectSol}}{\sobIndexConj}{\element}}{\vectSol}^{\sobIndexConj}
    \right]^{\frac{2}{\sobIndexConj}}
  \end{align*}
  Combining this last estimate with \eqref{eq:error.estimate.g2.rhs.T1E.1} yields, for any $\varepsilon >0$,
  \begin{equation}\label{eq:error.estimate.g2.rhs.T1E}
    \begin{aligned}
      \sum_{\element \in \domain_h} \mathfrak{T}_{1,\element}
      &\leq
      \varepsilon \discFullVectNorm{\boldsymbol{\xi}_h}^{\sobIndexConj} +
      \frac{1}{2}
      \left[ a_{\element}(\locVectProjL{\polyDegree} \discVectSol, \locVectProjL{\polyDegree} \boldsymbol{\xi}_h ) -a_{\element}(\locVectProjL{\polyDegree} \vectSol_{I},
      \locVectProjL{\polyDegree} \boldsymbol{\xi}_h ) \right] \\
      &\quad
      +
      C(\polyDegree,\sobIndex,\polyDegree_{\vectSol},\rho) \left(\seminorm{\vectSobSpace{1}{\sobIndexConj}{\domain}}{\vectSol} + \norm{\lebSpace{\sobIndexConj}{\domain}}{\loadTerm}\right)^{\sobIndexConj-2}
      \left[
        \sum_{\element \in \domain_h}
        h_{\element}^{\sobIndexConj \polyDegree_{\vectSol}} \seminorm{\vectSobSpace{\polyDegree_{\vectSol}}{\sobIndexConj}{\element}}{\vectSol}^{\sobIndexConj}
      \right]^{\frac{2}{\sobIndexConj}} \\
      &\quad +\frac{C(\polyDegree,\sobIndex,\polyDegree_{\fluxFunction},\rho)}{\varepsilon}
      \sum_{\element \in \domain_h} h_{\element}^{\sobIndex \polyDegree_{\fluxFunction}} \seminorm{\vectSobSpace{\polyDegree_{\fluxFunction}}{\sobIndex}{\element}}{\fluxFunction
      (\vectSol)}^{\sobIndex}.
    \end{aligned}
  \end{equation}

  \medskip
  \noindent
  \textbf{Second term.}
  For any $\element \in \domain_h$, following the same argument that leads to \eqref{eq:aux.sE.ineq.for.g2.case} and applying \eqref{eq:diffusive.flux.properties.original.mono}, we deduce that, for any $\varepsilon > 0$,
  \begin{equation}\label{eq:error.estimate.g2.rhs.T2E.initial}
    \begin{aligned}
      \mathfrak{T}_{2,\element}
      &\leq
      \varepsilon
      \left[s_{\element}(\discVectSol^{\perp},\boldsymbol{\xi}_h^{\perp}) -
      s_{\element}(\vectSol_{I}^{\perp},\boldsymbol{\xi}_h^{\perp})\right] + C(\sobIndex,\varepsilon)
      h_{\element}^{2}
      \euNorm{\locBoundaryDofVect (\vectSol_{I}^{\perp})}^{\sobIndexConj}.
    \end{aligned}
  \end{equation}
  Using \eqref{eq:boundaryDof.equivalence}, we first observe that $h_{\element}^{2} \euNorm{\locBoundaryDofVect (\vectSol_{I}^{\perp})}^{\sobIndexConj} \leq C(\polyDegree,\sobIndex,\rho) h_{\element} \norm{\lebSpace{\sobIndexConj}{\partial \element}}{\vectSol_{I}^{\perp}\cdot \normalElement}^{\sobIndexConj}.$
  For any $\edge \in \edgeElement$, recalling that $\vectSol_{I}^{\perp}\cdot \normalEdge \in
  \polySpace{\polyDegree}(\edge)$ and applying \eqref{eq:poly.sob.embedding} followed by \eqref{eq:r-norm.normal.component} and \eqref{eq:r-norm.div.virtual.func}, we have
  \begin{align*}
    \norm{\lebSpace{\sobIndexConj}{\edge}}{(\vectSol_{I}^{\perp})\cdot \normalEdge}
    &\leq
    C(\polyDegree,\sobIndex,\rho)
    \euNorm{\edge}^{\frac{1}{\sobIndexConj} - \frac{1}{2}} h_{\element}^{-\frac{1}{2}}
    \norm{\lebSpace{2}{\element}}{\vectSol_{I}^{\perp}} \\
    &\leq
    C(\polyDegree,\sobIndex,\rho)
    \euNorm{\edge}^{\frac{1}{\sobIndexConj} - \frac{1}{2}} h_{\element}^{-\frac{1}{2}} \left(
      \norm{\lebSpace{2}{\element}}{\vectSol - \vectSol_{I}} + \norm{\lebSpace{2}{\element}}{\vectSol - \locVectProjL{\polyDegree} \vectSol_{I}}
    \right) \\
    &\leq
    C(\polyDegree,\sobIndex,\polyDegree_{\vectSol},\rho)
    h_{\element}^{-\frac{1}{\sobIndexConj}}
    h_{\element}^{\polyDegree_{\vectSol}} \seminorm{\vectSobSpace{\polyDegree_{\vectSol}}{\sobIndexConj}{\element}}{\vectSol},
  \end{align*}
  where the last inequality follows by applying the approximation properties \eqref{eq:Fortin.approx} and \eqref{eq:projL.Fortin.approx} together with a H{\"o}lder inequality.
  Therefore, $h_{\element}^{2} \euNorm{\locBoundaryDofVect (\vectSol_{I}^{\perp})}^{\sobIndexConj} C(\polyDegree,\sobIndex,\polyDegree_{\vectSol},\rho) h_{\element}^{\sobIndexConj \polyDegree_{\vectSol}} \seminorm{\vectSobSpace{\polyDegree_{\vectSol}}{\sobIndexConj}{\element}}{\vectSol}^{\sobIndexConj}$.
  Plugging this last estimate into \eqref{eq:error.estimate.g2.rhs.T2E.initial} and choosing $\varepsilon= \frac{1}{2}$, yields
  \begin{equation}\label{eq:error.estimate.g2.rhs.T2E}
    \begin{aligned}
      \sum_{\element \in \domain_h}
      \mathfrak{T}_{2,\element}
      &\leq
      \frac{1}{2}
      \sum_{\element \in \domain_h}
      \left[s_{\element}(\discVectSol^{\perp},\boldsymbol{\xi}_h^{\perp}) -
      s_{\element}(\vectSol_{I}^{\perp},\boldsymbol{\xi}_h^{\perp})\right] +
      C(\polyDegree,\sobIndex,\polyDegree_{\vectSol},\rho)
      \sum_{\element \in \domain_h}
      h_{\element}^{\sobIndexConj \polyDegree_{\vectSol}} \seminorm{\vectSobSpace{\polyDegree_{\vectSol}}{\sobIndexConj}{\element}}{\vectSol}^{\sobIndexConj}.
    \end{aligned}
  \end{equation}

  \medskip
  \noindent
  \textbf{Conclusion.}
  Plugging \eqref{eq:error.estimate.g2.rhs.T1E} and \eqref{eq:error.estimate.g2.rhs.T2E} into \eqref{eq:error.estimate.g2.rhs.initial} leads to the desired result.
\end{proof}
\begin{theorem}[Discretization error for $\vectSol$: $\sobIndexConj > 2$]\label{thm:discrete.error.g2}
  Let $(\vectSol,\scalarSol) \in \vectSolSpace \times \scalarSolSpace$ and $(\discVectSol,\discScalarSol) \in \discVectSpace \times \discScalarSpace$ be the unique solutions to Problem~\eqref{eq:continuous.problem.dual.weak} and Problem~\eqref{eq:discrete.problem}, respectively. Assume $\sobIndexConj > 2$, and assume furthermore
  \begin{itemize}
    \item[] $\vectSol \in \vectSobSpace{\polyDegree_{\vectSol}}{\sobIndexConj}{\domain_h}$ for $\polyDegree_{\vectSol} \in \{1,2,\ldots,\polyDegree +1 \}$;
    \item[] $\fluxFunction(\vectSol) \in \vectSobSpace{\polyDegree_{\fluxFunction}}{\sobIndex}{\domain_h}$ for $\polyDegree_{\fluxFunction} \in \{1,\ldots,\polyDegree +1 \}$.
  \end{itemize}
  Setting $\boldsymbol{\xi}_h \coloneqq \discVectSol - \vectSol_{I}$ and recalling the definition of $\tilde{\mathcal{R}}^{(1)}_h$ and $\tilde{\mathcal{R}}^{(2)}_h$ in \eqref{eq:def.consistency.error.g2.conv.terms}, it holds
  \begin{equation}\label{eq:discrete.error.g2}
    \discFullVectNorm{\boldsymbol{\xi}_h}^{\sobIndexConj}
    \lesssim \tilde{\mathcal{R}}^{(1)}_h + \tilde{\mathcal{R}}^{(2)}_h + \left(\tilde{\mathcal{R}}^{(2)}_h\right)^{\frac{\sobIndexConj}{2}},
  \end{equation}
  with a hidden constant depending only on $\seminorm{\vectSobSpace{1}{\sobIndexConj}{\domain_h}}{\vectSol}$, $\norm{\lebSpace{\sobIndexConj}{\domain}}{\loadTerm}$,  $\polyDegree$, $\sobIndex$, $\polyDegree_{\vectSol}$, $\polyDegree_{\fluxFunction}$, and $\rho$.
\end{theorem}
\begin{proof}
  Applying \eqref{eq:error.equation.g2} with $\discVectFunTwo = \boldsymbol{\xi}_h$, and \eqref{eq:consistency.error.g2}, we obtain that, for any $\varepsilon >0$, it holds
  \begin{equation}\label{eq:another.aux.for.u.g2}
    \begin{aligned}
      a_h(\discVectSol,\boldsymbol{\xi}_h) - a_h(\vectSol_{I},\boldsymbol{\xi}_h)
      &\leq
      \varepsilon \discFullVectNorm{\boldsymbol{\xi}_h}^{\sobIndexConj}+
      \frac{1}{2} \left[ a_h(\discVectSol,\boldsymbol{\xi}_h) - a_h(\vectSol_{I},\boldsymbol{\xi}_h) \right] \\
      &\quad 
      + \frac{C(\polyDegree,\sobIndex,\polyDegree_{\fluxFunction},\rho)}{\varepsilon} \tilde{\mathcal{R}}_h^{(1)} 
      + C(\polyDegree,\sobIndex,\polyDegree_{\vectSol},\rho) \left( \tilde{\mathcal{R}}_h^{(2)} \right)^{\frac{\sobIndexConj}{2}} \\
      &\quad + C(\polyDegree,\sobIndex,\polyDegree_{\vectSol},\rho) \left( \norm{\lebSpace{\sobIndexConj}{\domain}}{\loadTerm} + \seminorm{\vectSobSpace{1}{\sobIndexConj}{\domain_h}}{\vectSol} \right)^{\sobIndexConj-2} \tilde{\mathcal{R}}_h^{(2)} .
    \end{aligned}
  \end{equation}
  Rearranging the terms and applying \eqref{eq:error.estimate.g2.mono} we have that
  \begin{equation*}
    \begin{aligned}
      C(\sobIndex) \discFullVectNorm{\boldsymbol{\xi}_h}^{\sobIndexConj}
      &\leq
      \varepsilon \discFullVectNorm{\boldsymbol{\xi}_h}^{\sobIndexConj} 
      + \frac{C(\polyDegree,\sobIndex,\polyDegree_{\fluxFunction},\rho)}{\varepsilon} \tilde{\mathcal{R}}_h^{(1)} 
      + C(\polyDegree,\sobIndex,\polyDegree_{\vectSol},\rho) \left( \tilde{\mathcal{R}}_h^{(2)} \right)^{\frac{\sobIndexConj}{2}} \\ 
      &\quad + C(\polyDegree,\sobIndex,\polyDegree_{\vectSol},\rho) \left( \norm{\lebSpace{\sobIndexConj}{\domain}}{\loadTerm} + \seminorm{\vectSobSpace{1}{\sobIndexConj}{\domain_h}}{\vectSol} \right)^{\sobIndexConj-2} \tilde{\mathcal{R}}_h^{(2)} .
    \end{aligned}
  \end{equation*}
  Denoting by $C_1$ the constant on the left-hand side of the above inequality and choosing $\varepsilon \coloneqq \frac{C_1}{2}$, after basic algebraic manipulations, we obtain the desired result.
\end{proof}
%
\subsubsection{Interpolation error estimate}
%
\begin{lemma}[Interpolation error for $\vectSol$: $\sobIndexConj > 2$]\label{lem:interpolation.error.norm.g2}
  Let $(\vectSol,\scalarSol) \in \vectSolSpace \times \scalarSolSpace$ be the unique solution to Problem~\eqref{eq:continuous.problem.dual.weak}.
  Assume $\sobIndexConj > 2$ and that $\vectSol \in \vectSobSpace{\polyDegree_{\vectSol}}{\sobIndexConj}{\domain_h}$ with $\polyDegree_{\vectSol} \in \{ 1, 2, \ldots, \polyDegree+1 \}$. Then, it holds
  \begin{equation}\label{eq:interpolation.error.g2}
    \discFullVectNorm{\vectSol - \vectSol_{I}}^{\sobIndexConj} \lesssim \sum_{\element \in \domain_h} h_{\element}^{\sobIndexConj \polyDegree_{\vectSol}}
    \seminorm{\vectSobSpace{\polyDegree_{\vectSol}}{\sobIndexConj}{\domain_h}}{\vectSol}^{\sobIndexConj},
  \end{equation}
  with a hidden constant depending only on $\polyDegree$, $\sobIndex$, $\polyDegree_{\vectSol}$, and $\rho$.
\end{lemma}
\begin{proof}
  Recalling the remark at the beginning of the proof of Lemma~\ref{lem:interpolation.error.norm.s2}, we write
  \begin{equation*}
    \begin{aligned}
      \discFullVectNorm{\vectSol - \vectSol_{I}}^{\sobIndexConj}
      &= \norm{\vectLebSpace{\sobIndexConj}{\domain}}{\vectProjL{\polyDegree}(\vectSol - \vectSol_{I})}^{\sobIndexConj}  + \sum_{\element \in \domain_h} s_{\element} \left( (\vectSol - \vectSol_{I})^{\perp},
      (\vectSol - \vectSol_{I})^{\perp} \right)\\
      &\quad + \sum_{\element \in \domain_h} h_{\element}^{\sobIndexConj}\norm{\lebSpace{\sobIndexConj}{\element}}{\divergence (\vectSol - \vectSol_{I})}^{\sobIndexConj}.
    \end{aligned}
  \end{equation*}
  For the first term, applying \eqref{eq:poly.sob.embedding} and \eqref{eq:vectProjL.stab.element}, and then using the approximation properties \eqref{eq:Fortin.approx.r-norm} and \eqref{eq:projL.Fortin.approx} together with a H{\"o}lder inequalilty, we have
  \begin{align*}
    \norm{\vectLebSpace{\sobIndexConj}{\domain}}{\vectProjL{\polyDegree}(\vectSol - \vectSol_{I})}^{\sobIndexConj}
    &\leq C(\polyDegree,\sobIndex,\rho) \sum_{\element \in \domain_h} \euNorm{\element}^{1 - \frac{\sobIndexConj}{2}} \left(
      \norm{\vectLebSpace{2}{\element}}{\vectSol - \vectProjL{\polyDegree} \vectSol_{I}}^{\sobIndexConj}
      +
    \norm{\vectLebSpace{2}{\element}}{\vectSol - \vectSol_{I}}^{\sobIndexConj}\right) \\
    &\leq
    C(\polyDegree,\sobIndex,\rho,\polyDegree_{\vectSol}) \sum_{\element \in \domain_h}
    h_{\element}^{\sobIndexConj \polyDegree_{\vectSol}} \seminorm{\vectSobSpace{\polyDegree_{\vectSol}}{\sobIndexConj}{\element}}{\vectSol}^{\sobIndexConj}
  \end{align*}
  For the second term, we first recall that $\locBoundaryDofVect(\vectSol - \vectSol_{I}) = 0$. 
  Then, recalling the definition of $s_{\element}$ and applying bound \eqref{eq:boundaryDof.equivalence} and the polynomial embedding \eqref{eq:poly.sob.embedding} together with \eqref{eq:r-norm.normal.component} and \eqref{eq:r-norm.div.virtual.func}, we obtain
  \begin{equation*}
    \begin{aligned}
      \sum_{\element \in \domain_h} s_{\element} \left( (\vectSol - \vectSol_{I})^{\perp}, (\vectSol -
      \vectSol_{I})^{\perp} \right) &\leq C(\polyDegree,\sobIndex,\rho)
      \sum_{\element \in \domain_h} h_{\element} \norm{\lebSpace{\sobIndexConj}{\partial \element}}{(\vectProjL{\polyDegree}(\vectSol - \vectSol_{I}))\cdot
      \normalElement}^{\sobIndexConj} \\
      &\leq C(\polyDegree,\sobIndex,\rho)
      \sum_{\element \in \domain_h} h_{\element}^{2-\sobIndexConj} \norm{\vectLebSpace{2}{\element}}{\vectProjL{\polyDegree}(\vectSol - \vectSol_{I})}^{\sobIndexConj} \\
      &\leq C(\polyDegree,\sobIndex,\rho,\polyDegree_{\vectSol}) \sum_{\element \in \domain_h} h_{\element}^{\sobIndexConj \polyDegree_{\vectSol}}
      \seminorm{\vectSobSpace{\polyDegree_{\vectSol}}{\sobIndexConj}{\element}}{\vectSol}^{\sobIndexConj},
    \end{aligned}
  \end{equation*}
  where in the last inequality we have applied \eqref{eq:vectProjL.stab.element} and \eqref{eq:Fortin.approx.r-norm} together with a H{\"o}lder inequality.
  Finally, for the third term involving the divergence, bound \eqref{eq:Fortin.approx.div.r-norm} yields
  \begin{equation*}
    \sum_{\element \in \domain_h} h_{\element}^{\sobIndexConj}\norm{\lebSpace{\sobIndexConj}{\element}}{\divergence (\vectSol - \vectSol_{I})}^{\sobIndexConj}
    \leq C(\polyDegree,\sobIndex,\rho,\polyDegree_{\vectSol})
    \sum_{\element \in \domain_h} h_{\element}^{\sobIndexConj \polyDegree_{\vectSol}} \seminorm{\vectSobSpace{\polyDegree_{\vectSol}}{\sobIndexConj}{\element}}{\vectSol}^{\sobIndexConj}.
  \end{equation*}
  Combining the three bounds above yields the desired result.
\end{proof}
As a direct consequence of Theorem~\ref{thm:discrete.error.g2}, Lemma~\ref{lem:interpolation.error.norm.g2}, and the triangle inequality, we obtain the following result.
\begin{corollary}\label{cor:total.error.g2}
  Under the same assumptions of Theorem~\ref{thm:discrete.error.g2}, it holds
  \begin{equation}\label{eq:total.error.g2}
    \discFullVectNorm{\vectSol - \discVectSol}
    \lesssim \left(\tilde{\mathcal{R}}_h^{(1)}\right)^{\frac{1}{\sobIndexConj}} + \left(\tilde{\mathcal{R}}_h^{(2)}\right)^{\frac{1}{\sobIndexConj}} + \left(\tilde{\mathcal{R}}_h^{(2)}\right)^{\frac{1}{2}},
  \end{equation}
  with a hidden constant depending only on $\seminorm{\vectSobSpace{1}{\sobIndexConj}{\domain_h}}{\vectSol}$, $\norm{\lebSpace{\sobIndexConj}{\domain}}{\loadTerm}$, $\polyDegree$, $\sobIndexConj$, $\polyDegree_{\vectSol}$, $\polyDegree_{\fluxFunction}$, and $\rho$.

  In particular, assuming full regularity, i.e., $\polyDegree_{\vectSol} = \polyDegree + 1$ and $\polyDegree_{\fluxFunction} = \polyDegree + 1$, estimate \eqref{eq:total.error.g2} yields the explicit asymptotic convergence rate
  \begin{equation}\label{eq:asymptotic.rate.g2}
    \discFullVectNorm{\vectSol - \discVectSol} \lesssim h^{\frac{1}{\sobIndexConj-1}(\polyDegree+1)}.
  \end{equation}
\end{corollary}
%
\subsection{A priori error estimates for $\scalarSol$: $1 < \sobIndexConj \leq 2$}\label{sec:a-priori.scalar.1}
\subsubsection{Preliminary results}
The next two lemmas are auxiliary for the error analysis of the scalar solution $\scalarSol$. Their proofs are reported in Appendix~\ref{app:proof.sE.approx.for.u} and Appendix~\ref{app:proof.sigma.approx.for.u}.
\begin{lemma}\label{lem:sE.approx.for.u}
  Let $\vectSol \in \vectSolSpace$ and $\discVectSol \in \discVectSpace$ be the unique solutions to Problem~\eqref{eq:continuous.problem.dual.weak} and Problem~\eqref{eq:discrete.problem}, respectively. 
  Assume that $\vectSol \in \vectSobSpace{\polyDegree_{\vectSol}}{\sobIndexConj}{\domain_h}$ for $\polyDegree_{\vectSol} \in \{1,2,\ldots,\polyDegree +1 \}$.
  Then, for any $\element \in \domain_h$, it holds
  \begin{equation}\label{eq:sE.approx.for.u}
    \begin{aligned}
      s_{\element}(\discVectSol^{\perp}, \discVectSol^{\perp})
      &\leq
      C(\polyDegree,\sobIndex,\polyDegree_{\vectSol},\rho)
            h_{\element}^{\sobIndexConj \polyDegree_{\vectSol} }\seminorm{\vectSobSpace{\polyDegree_{\vectSol}}{\sobIndexConj}{\element}}{\vectSol}^{\sobIndexConj} + 
      C(\sobIndex)\left[ s_{\element}(\discVectSol^{\perp} , \boldsymbol{\xi}_h^{\perp} ) - s_{\element}( \vectSol_{I}^{\perp} , \boldsymbol{\xi}_h^{\perp} ) \right],
    \end{aligned}
  \end{equation}
  where $\boldsymbol{\xi}_h \coloneqq \discVectSol - \vectSol_{I}$.
\end{lemma}
\begin{lemma}\label{lem:sigma.approx.for.u}
  Let $\vectSol \in \vectSolSpace$ and $\discVectSol \in \discVectSpace$ be the unique solutions to Problem~\eqref{eq:continuous.problem.dual.weak} and Problem~\eqref{eq:discrete.problem}, respectively. 
  Assume $1<\sobIndexConj \leq 2$, and assume furthermore
    \begin{itemize}
      \item[] $\vectSol \in \vectSobSpace{\polyDegree_{\vectSol}}{\sobIndexConj}{\domain_h}$ for $\polyDegree_{\vectSol} \in \{1,2,\ldots,\polyDegree +1 \}$.
    \end{itemize}
  Then, for any $\element \in \domain_h$, it holds
  \begin{equation}\label{eq:sigma.approx.for.u}
    \begin{aligned}
      \norm{\vectLebSpace{\sobIndex}{\element}}{\fluxFunction(\vectSol) -\fluxFunction(\vectProjL{\polyDegree}\discVectSol)}^{\sobIndex}
      &\leq
      C(\polyDegree,\sobIndex,\polyDegree_{\vectSol},\rho) h_{\element}^{\sobIndexConj \polyDegree_{\vectSol}}
      \seminorm{\vectSobSpace{\polyDegree_{\vectSol}}{\sobIndexConj}{\element}}{\vectSol}^{\sobIndexConj} \\ 
      &\quad + C(\sobIndex) \left[ a_{\element}(\vectProjL{\polyDegree}
      \discVectSol,\vectProjL{\polyDegree}\boldsymbol{\xi}_h) - a_{\element}(\vectProjL{\polyDegree}\vectSol_{I},\vectProjL{\polyDegree}\boldsymbol{\xi}_h) \right],
    \end{aligned}
  \end{equation}
  where $\boldsymbol{\xi}_h \coloneqq \discVectSol - \vectSol_{I}$.
\end{lemma}
%
\subsubsection{Discrete error estimate}
\begin{theorem}[Discretization error for $\scalarSol$: $1 < \sobIndexConj \leq 2$]\label{thm:discrete.error.scalar.s2}
  Let $(\vectSol,\scalarSol) \in \vectSolSpace \times \scalarSolSpace$ and $(\discVectSol,\discScalarSol) \in \discVectSpace \times \discScalarSpace$ be the unique solutions to Problem~\eqref{eq:continuous.problem.dual.weak} and Problem~\eqref{eq:discrete.problem}, respectively. Assume $1<\sobIndexConj \leq 2$, and assume furthermore
  \begin{enumerate}
    \item[] $\vectSol \in \vectSobSpace{\polyDegree_{\vectSol}}{\sobIndexConj}{\domain_h}$ for $\polyDegree_{\vectSol} \in \{1,2,\ldots,\polyDegree +1 \}$;
    \item[] $\fluxFunction(\vectSol) \in \vectSobSpace{\polyDegree_{\fluxFunction}}{\sobIndex}{\domain_h}$ for $\polyDegree_{\fluxFunction} \in \{0,1,\ldots,\polyDegree +1 \}$.
  \end{enumerate}
  Then, recalling the definition of $\mathcal{R}_h^{(1)}$ and $\mathcal{R}_h^{(2)}$ in \eqref{eq:def.consistency.error.s2.conv.terms},
  \begin{equation}\label{eq:discrete.error.scalar.s2}
    \discScalarNorm{\discScalarSol - \projL{\polyDegree} \scalarSol} \lesssim 
     \left(\mathcal{R}_h^{(1)}\right)^{\frac{1}{\sobIndex}} + \left(\mathcal{R}_h^{(2)}\right)^{\frac{1}{\sobIndex}} + \left(\mathcal{R}_h^{(2)}\right)^{\frac{1}{2}},
  \end{equation}
  with a hidden constant depending only on $\seminorm{\vectSobSpace{1}{\sobIndexConj}{\domain_h}}{\vectSol}$, $\norm{\lebSpace{\sobIndexConj}{\domain}}{\loadTerm}$, $\polyDegree$, $\sobIndexConj$, $\polyDegree_{\vectSol}$, $\polyDegree_{\fluxFunction}$, and $\rho$.
\end{theorem}
\begin{proof}
  Let $\discVectFunOne \in \discVectSpace$. 
  Using the discrete and continuous equations \eqref{eq:discrete.problem.line1}-\eqref{eq:continuous.problem.dual.weak.1}, and the commuting property $b(\discVectFunOne, \projL{\polyDegree} \scalarSol) = b(\discVectFunOne, \scalarSol)$, we have $b(\discVectFunOne, \discScalarSol - \projL{\polyDegree} \scalarSol) = a(\vectSol, \discVectFunOne) - a_h(\discVectSol, \discVectFunOne)$.
  By adding and subtracting $\int_{\element} \vectProjL{\polyDegree}\fluxFunction(\vectSol) \cdot \discVectFunOne$ on each $\element \in \domain_h$, and exploiting the properties of $\vectProjL{\polyDegree}$, we decompose the error as
  \begin{equation}\label{eq:scalar.err.estimate.initial.s2}
    \begin{aligned}
      b(\discVectFunOne,\discScalarSol - \projL{\polyDegree} \scalarSol)&=
      \sum_{\element \in \domain_h}
      \int_{\element} (\fluxFunction(\vectSol) - \vectProjL{\polyDegree}\fluxFunction(\vectSol)) \cdot \discVectFunOne + \sum_{\element \in \domain_h} \int_{\element}(\fluxFunction(\vectSol)
      -\fluxFunction(\vectProjL{\polyDegree}\discVectSol)) \cdot \vectProjL{\polyDegree} \discVectFunOne \\
      &\quad +
      \sum_{\element \in \domain_h} -s_{\element} (\discVectSol^{\perp}, \discVectFunOne - \vectProjL{\polyDegree} \discVectFunOne) \\
      &\eqqcolon
      \sum_{\element \in \domain_h}\mathfrak{T}_{1,\element} + \sum_{\element \in \domain_h}\mathfrak{T}_{2,\element} + \sum_{\element \in \domain_h}\mathfrak{T}_{3,\element}
    \end{aligned}
  \end{equation}
  We bound each of the above terms separately.

  \medskip
  \noindent
  \textbf{First term.}
  For any $\element \in \domain_h$, we apply a $(\sobIndex,\sobIndexConj)$-H{\"o}lder inequality along with $\vectProjL{\polyDegree}$ approximation property and the equivalence in \eqref{eq:disc.norm.simeq.cont.norm}, to write
  \begin{align*}
    \mathfrak{T}_{1,\element}
    &\leq
    C(\polyDegree,\sobIndex,\polyDegree_{\fluxFunction},\rho)
    h_{\element}^{\polyDegree_{\fluxFunction}} \seminorm{\vectSobSpace{\polyDegree_{\fluxFunction}}{\sobIndex}{\element}}{\fluxFunction(\vectSol)}
    \locDiscFullVectNorm{\discVectFunOne}.
  \end{align*}
  Summing over $\element \in \domain_h$ and applying a discrete $(\sobIndex,\sobIndexConj)$-H{\"o}lder inequality, we conclude that
  \begin{equation}\label{eq:scalar.err.estimate.T1.s2}
    \sum_{\element \in \domain_h}
    \mathfrak{T}_{1,\element}
    \leq
    C(\polyDegree,\sobIndex,\polyDegree_{\fluxFunction},\rho)
    \left[
      \sum_{\element \in \domain_h}
      h_{\element}^{\sobIndex \polyDegree_{\fluxFunction}} \seminorm{\vectSobSpace{\polyDegree_{\fluxFunction}}{\sobIndex}{\element}}{\fluxFunction(\vectSol)}^\sobIndex
    \right]^{\frac{1}{\sobIndex}}
    \discFullVectNorm{\discVectFunOne}.
  \end{equation}

  \medskip
  \noindent
  \textbf{Second term.}
  For any $\element \in \domain_h$, we apply a $(\sobIndex,\sobIndexConj)$-H{\"o}lder inequality, and recalling the definition of the local discrete norm in \eqref{eq:def.local.discrete.norm.NEW}, we write
  \begin{align*}
    \mathfrak{T}_{2,\element}
    \leq
    \norm{\vectLebSpace{\sobIndex}{\element}}{\fluxFunction(\vectSol)
    -\fluxFunction(\vectProjL{\polyDegree}\discVectSol)} \locDiscFullVectNorm{\discVectFunOne}
  \end{align*}
  Summing over $\element \in \domain_h$ and applying \eqref{eq:sigma.approx.for.u}, we get
  \begin{equation}\label{eq:scalar.err.estimate.T2.s2}
    \begin{aligned}
      \sum_{\element \in \domain_h} \mathfrak{T}_{2,\element}
      &\leq
      \Bigg[ C(\polyDegree,\sobIndex,\polyDegree_{\vectSol},\rho)
      \left( 
      \sum_{\element \in \domain_h}
       h_{\element}^{\sobIndexConj \polyDegree_{\vectSol}}
      \seminorm{\vectSobSpace{\polyDegree_{\vectSol}}{\sobIndexConj}{\element}}{\vectSol}^{\sobIndexConj}
      \right)^{\frac{1}{\sobIndex}} \\
      &\quad + C(\sobIndex) \left( \sum_{\element \in \domain_h} a_{\element}(\vectProjL{\polyDegree}
      \discVectSol,\vectProjL{\polyDegree}\boldsymbol{\xi}_h) - a_{\element}(\vectProjL{\polyDegree}\vectSol_{I},\vectProjL{\polyDegree}\boldsymbol{\xi}_h) \right)^{\frac{1}{\sobIndex}}
      \Bigg] \discFullVectNorm{\discVectFunOne},
    \end{aligned}
  \end{equation}
  where $\boldsymbol{\xi}_h \coloneqq \discVectSol - \vectSol_{I}$.

  \medskip
  \noindent
  \textbf{Third term.}
  Applying \eqref{eq:sE.Holder.continuity} and recalling the definition of the local discrete norm in \eqref{eq:def.local.discrete.norm.NEW} yields, for any $\element \in \domain_h$, $\mathfrak{T}_{3,\element}
    \leq
    C(\sobIndex) \left[s_{\element}(\discVectSol^{\perp}, \discVectSol^{\perp})\right]^{\frac{\sobIndexConj-1}{\sobIndexConj}} \locDiscFullVectNorm{\discVectFunOne}$.
  Summing over $\element \in \domain_h$, applying \eqref{eq:sE.approx.for.u} and recalling that ${\frac{\sobIndexConj-1}{\sobIndexConj}}={\frac{1}{\sobIndex}}$, we get
  \begin{equation}\label{eq:scalar.err.estimate.T3.s2}
    \begin{aligned}
      \sum_{\element \in \domain_h} \mathfrak{T}_{3,\element}
      &\leq
      \discFullVectNorm{\discVectFunOne} \Bigg[
      C(\polyDegree,\sobIndex,\polyDegree_{\vectSol},\rho) 
      \left(\sum_{\element \in \domain_h} 
      h_{\element}^{\sobIndexConj \polyDegree_{\vectSol} }\seminorm{\vectSobSpace{\polyDegree_{\vectSol}}{\sobIndexConj}{\element}}{\vectSol}^{\sobIndexConj}\right)^{\frac{1}{\sobIndex}} \\
      &\quad+
      C(\sobIndex) \sum_{\element \in \domain_h} \left( s_{\element}(\discVectSol^{\perp} , \boldsymbol{\xi}_h^{\perp} ) - s_{\element}( \vectSol_{I}^{\perp} , \boldsymbol{\xi}_h^{\perp} ) \right)^{\frac{1}{\sobIndex}} \Bigg]
    \end{aligned}
  \end{equation}

  \medskip
  \noindent
  \textbf{Conclusion.}
  Plugging \eqref{eq:scalar.err.estimate.T1.s2}, \eqref{eq:scalar.err.estimate.T2.s2} and \eqref{eq:scalar.err.estimate.T3.s2} into \eqref{eq:scalar.err.estimate.initial.s2}, after recalling the definition of $a_h$ in \eqref{eq:def.ah.bh} and observing that from \eqref{eq:error.equation.s2}, \eqref{eq:consistency.error.s2} and \eqref{eq:discrete.error.s2} follows
  \begin{equation*}
    a_h(\discVectSol,\boldsymbol{\xi}_h) - a_h(\vectSol_{I},\boldsymbol{\xi}_h) \leq C(\polyDegree,\sobIndex,\polyDegree_{\vectSol},\polyDegree_{\fluxFunction},\seminorm{\vectSobSpace{1}{\sobIndexConj}{\domain_h}}{\vectSol},\norm{\lebSpace{\sobIndexConj}{\domain}}{\loadTerm},\rho) \left( \mathcal{R}_h^{(1)} + \mathcal{R}_h^{(2)} \right),
  \end{equation*}
  the discrete inf-sup inequality \eqref{eq:discrete.infsup} yields the desired result.
\end{proof}
As a direct consequence of Theorem~\ref{thm:discrete.error.scalar.s2}, the approximation properties of $\projL{\polyDegree}$, and the triangle inequality, we obtain the following result.
\begin{corollary}\label{cor:total.error.scalar}
  Under the same assumptions of Theorem~\ref{thm:discrete.error.scalar.s2} and assuming furthermore 
  \begin{itemize}
    \item[] $\scalarSol \in \sobSpace{\polyDegree_{\scalarSol}}{\sobIndex}{\domain_h}$ for $\polyDegree_{\scalarSol} \in \{0,1,\ldots,\polyDegree +1 \}$,
  \end{itemize}
  we have that
  \begin{equation}\label{eq:total.error.scalar.s2}
    \discScalarNorm{\scalarSol - \discScalarSol}
    \lesssim 
    \left(\mathcal{R}_h^{(1)}\right)^{\frac{1}{\sobIndex}} + \left(\mathcal{R}_h^{(2)}\right)^{\frac{1}{\sobIndex}} + \left(\mathcal{R}_h^{(2)}\right)^{\frac{1}{2}}+ \left(\sum_{\element \in \domain_h} h_{\element}^{\sobIndex \polyDegree_{\scalarSol}} \seminorm{\sobSpace{\polyDegree_{\scalarSol}}{\sobIndex}{\element}}{\scalarSol}\right)^{\frac{1}{\sobIndex}},
  \end{equation}
  with a hidden constant depending only on $\seminorm{\vectSobSpace{1}{\sobIndexConj}{\domain_h}}{\vectSol}$, $\norm{\lebSpace{\sobIndexConj}{\domain}}{\loadTerm}$, $\polyDegree$, $\sobIndexConj$, $\polyDegree_{\vectSol}$, $\polyDegree_{\fluxFunction}$, $\polyDegree_{\scalarSol}$, and $\rho$.

  In particular, assuming full regularity, i.e., $\polyDegree_{\vectSol} = \polyDegree + 1$, $\polyDegree_{\fluxFunction} = \polyDegree + 1$, and $\polyDegree_{\scalarSol} = \polyDegree+1$, estimate \eqref{eq:total.error.scalar.s2} yields the explicit asymptotic convergence rate
  \begin{equation}\label{eq:asymptotic.rate.scalar.s2}
    \discScalarNorm{\scalarSol - \discScalarSol} \lesssim h^{\frac{1}{\sobIndex-1}(\polyDegree+1)}.
  \end{equation}
\end{corollary}
%
\subsection{A priori error estimates for $\scalarSol$: $\sobIndexConj > 2$}\label{sec:a-priori.scalar.2}
\begin{theorem}[Discretization error for $\scalarSol$: $\sobIndexConj > 2$]\label{thm:discrete.error.scalar.g2}
  Let $(\vectSol,\scalarSol) \in \vectSolSpace \times \scalarSolSpace$ and $(\discVectSol,\discScalarSol) \in \discVectSpace \times \discScalarSpace$ be the unique solutions to Problem~\eqref{eq:continuous.problem.dual.weak} and Problem~\eqref{eq:discrete.problem}, respectively. Assume $\sobIndexConj > 2$, and assume furthermore
  \begin{itemize}
    \item[] $\vectSol \in \vectSobSpace{\polyDegree_{\vectSol}}{\sobIndexConj}{\domain_h}$ for $\polyDegree_{\vectSol} \in \{1,2,\ldots,\polyDegree +1 \}$;
    \item[] $\fluxFunction(\vectSol) \in \vectSobSpace{\polyDegree_{\fluxFunction}}{\sobIndex}{\domain_h}$ for $\polyDegree_{\fluxFunction} \in \{1,\ldots,\polyDegree +1 \}$.
  \end{itemize}
  Then, recalling the definition of $\tilde{\mathcal{R}}^{(1)}_h$ and $\tilde{\mathcal{R}}^{(2)}_h$  in \eqref{eq:def.consistency.error.g2.conv.terms},
  \begin{equation}\label{eq:discrete.error.scalar.g2}
    \discScalarNorm{\discScalarSol - \projL{\polyDegree} \scalarSol} 
    \lesssim 
    \left( \tilde{\mathcal{R}}_h^{(1)} \right)^{\frac{1}{\sobIndex}} + \left(\tilde{\mathcal{R}}^{(2)}_h\right)^{{\frac{1}{\sobIndex}}}
    + \left(\tilde{\mathcal{R}}_h^{(1)}\right)^{\frac{1}{\sobIndexConj}} + \left(\tilde{\mathcal{R}}_h^{(2)}\right)^{\frac{1}{\sobIndexConj}} + \left(\tilde{\mathcal{R}}_h^{(2)}\right)^{\frac{1}{2}}
     + \left(\tilde{\mathcal{R}}^{(2)}_h\right)^{\frac{1}{2(\sobIndex-1)}}
  \end{equation}
  with a hidden constant depending only on $\seminorm{\vectSobSpace{1}{\sobIndexConj}{\domain_h}}{\vectSol}$, $\norm{\lebSpace{\sobIndexConj}{\domain}}{\loadTerm}$, $\polyDegree$, $\sobIndexConj$, $\polyDegree_{\vectSol}$, $\polyDegree_{\fluxFunction}$, and $\rho$.
\end{theorem}
\begin{proof}
  Let $\discVectFunOne \in \discVectSpace$. 
  Using the discrete equation \eqref{eq:discrete.problem.line1} and the commuting property $b(\discVectFunOne, \projL{\polyDegree} \scalarSol) = b(\discVectFunOne, \scalarSol)$, we have $b(\discVectFunOne, \discScalarSol - \projL{\polyDegree} \scalarSol) = \int_{\domain} \fluxFunction(\vectSol) \cdot \discVectFunOne - a_h(\discVectSol, \discVectFunOne)$, where we have also applied integration by parts on $b(\discVectFunOne, \scalarSol)$ since $\fluxFunction(\vectSol) \in \vectSobSpace{1}{\sobIndex}{\domain_h} \subset \vectLebSpace{2}{\domain}$.
  By adding and subtracting $\int_{\element} \vectProjL{\polyDegree}\fluxFunction(\vectSol) \cdot \discVectFunOne$ on each $\element \in \domain_h$, and exploiting the properties of $\vectProjL{\polyDegree}$, we decompose the error as
  \begin{equation}\label{eq:scalar.err.estimate.initial.g2}
    \begin{aligned}
      b(\discVectFunOne,\discScalarSol - \projL{\polyDegree} \scalarSol)
      &=
      \sum_{\element \in \domain_h}
      \int_{\element} (\fluxFunction(\vectSol) - \vectProjL{\polyDegree}\fluxFunction(\vectSol)) \cdot \discVectFunOne
      + \sum_{\element \in \domain_h} \int_{\element}(\fluxFunction(\vectSol) -\fluxFunction(\vectProjL{\polyDegree}\discVectSol)) \cdot \vectProjL{\polyDegree} \discVectFunOne \\
      &\quad
      + \sum_{\element \in \domain_h} -s_{\element} (\discVectSol^{\perp}, \discVectFunOne - \vectProjL{\polyDegree} \discVectFunOne)  \\
      &\eqqcolon
      \sum_{\element \in \domain_h}\mathfrak{T}_{1,\element} + \sum_{\element \in \domain_h}\mathfrak{T}_{2,\element} + \sum_{\element \in \domain_h}\mathfrak{T}_{3,\element}.
    \end{aligned}
  \end{equation}
  We bound each of the above terms separately.

  \medskip
  \noindent
  \textbf{First term.}
  Observing that $\mathfrak{T}_{1,\element}$ is identical to the first term in \eqref{eq:error.estimate.g2.rhs.T1E.initial}, we can apply the same argument that leads to \eqref{eq:aux.for.u.g2} and write, for any $\element \in \domain_h$,
  \begin{align*}
    \mathfrak{T}_{1,\element}
    \leq
    C(\polyDegree,\sobIndex,\polyDegree_{\fluxFunction},\rho) h_{\element}^{\polyDegree_{\fluxFunction}}
    \seminorm{\vectSobSpace{\polyDegree_{\fluxFunction}}{\sobIndex}{\element}}{\fluxFunction (\vectSol)}
    \locDiscFullVectNorm{\discVectFunOne}
  \end{align*}
  Summing over $\element \in \domain_h$ and applying a discrete $(\sobIndex,\sobIndexConj)$-H{\"o}lder inequality, we conclude that
  \begin{equation}\label{eq:scalar.err.estimate.T1.g2}
    \sum_{\element \in \domain_h}
    \mathfrak{T}_{1,\element}
    \leq
    C(\polyDegree,\sobIndex,\polyDegree_{\fluxFunction},\rho) \left( \tilde{\mathcal{R}}_h^{(1)} \right)^{\frac{1}{\sobIndex}}
    \discFullVectNorm{\discVectFunOne}
  \end{equation}

  \medskip
  \noindent
  \textbf{Second term.}
  Applying \eqref{eq:diffusive.flux.properties.original.hc} together with a $(\frac{\sobIndexConj}{\sobIndexConj-2},\sobIndexConj,\sobIndexConj)$-H{\"o}lder inequality, we get for any $\element \in \domain_h$
  \begin{align*}
    \mathfrak{T}_{2,\element}
    &\leq C(\sobIndex)
    \left[\norm{\vectLebSpace{\sobIndexConj}{\element}}{\vectSol}^{\sobIndexConj} + \norm{\vectLebSpace{\sobIndexConj}{\element}}{\vectProjL{\polyDegree}\discVectSol}^{\sobIndexConj}
    \right]^{\frac{\sobIndexConj-2}{\sobIndexConj}}
    \norm{\vectLebSpace{\sobIndexConj}{\element}}{\vectSol - \vectProjL{\polyDegree}\discVectSol}
    \norm{\vectLebSpace{\sobIndexConj}{\element}}{\vectProjL{\polyDegree} \discVectFunOne} \\
    &\leq C(\sobIndex)
    \left[\norm{\vectLebSpace{\sobIndexConj}{\element}}{\vectSol}^{\sobIndexConj} + \locDiscFullVectNorm{\discVectSol}^{\sobIndexConj} \right]^{\frac{\sobIndexConj-2}{\sobIndexConj}}
    \left[ \norm{\vectLebSpace{\sobIndexConj}{\element}}{\vectSol - \vectProjL{\polyDegree}\vectSol} + \norm{\vectLebSpace{\sobIndexConj}{\element}}{\vectProjL{\polyDegree}(\vectSol -
    \discVectSol)} \right]
    \norm{\vectLebSpace{\sobIndexConj}{\element}}{\vectProjL{\polyDegree} \discVectFunOne} \\
    &\leq C(\polyDegree,\sobIndex,\polyDegree_{\vectSol},\rho)
    \left[\norm{\vectLebSpace{\sobIndexConj}{\element}}{\vectSol}^{\sobIndexConj} + \locDiscFullVectNorm{\discVectSol}^{\sobIndexConj} \right]^{\frac{\sobIndexConj-2}{\sobIndexConj}}
    \left[
      h_{\element}^{\polyDegree_{\vectSol}} \seminorm{\vectSobSpace{\polyDegree_{\vectSol}}{\sobIndexConj}{\element}}{\vectSol}
      +
      \euNorm{\element}^{\frac{1}{\sobIndexConj} - \frac{1}{2}}
    \norm{\vectLebSpace{2}{\element}}{\vectSol - \discVectSol} \right]
    \locDiscFullVectNorm{\discVectFunOne},
  \end{align*}
  where in the last inequality we have used \eqref{eq:poly.sob.embedding} together with $\vectProjL{\polyDegree}$ stability and approximation property.
  We preliminarly observe that, setting $\boldsymbol{\xi}_h \coloneqq \discVectSol -\vectSol_{I}$, and applying first \eqref{eq:Fortin.approx.r-norm} and then \eqref{eq:r-norm.virtual.fun} together with \eqref{eq:poly.sob.embedding}, we have
  \begin{align*}
    \norm{\vectLebSpace{2}{\element}}{\vectSol - \discVectSol}
    &\leq
    C(\polyDegree,\polyDegree_{\vectSol},\rho) h_{\element}^{\polyDegree_{\vectSol}} \seminorm{\vectSobSpace{\polyDegree_{\vectSol}}{2}{\element}}{\vectSol} +
    \norm{\vectLebSpace{2}{\element}}{\boldsymbol{\xi}_h} \\
    &\leq
    C(\polyDegree,\polyDegree_{\vectSol},\rho) \euNorm{\element}^{\frac{1}{2} - \frac{1}{\sobIndexConj}} \left[ h_{\element}^{\polyDegree_{\vectSol}}
      \seminorm{\vectSobSpace{\polyDegree_{\vectSol}}{\sobIndexConj}{\element}}{\vectSol} + h_{\element}^{\frac{1}{\sobIndexConj}} \norm{\lebSpace{\sobIndexConj}{\partial
    \element}}{\boldsymbol{\xi}_h \cdot \normalElement} \right] \\
    &\leq
    C(\polyDegree,\polyDegree_{\vectSol},\rho) \euNorm{\element}^{\frac{1}{2} - \frac{1}{\sobIndexConj}} \left[ h_{\element}^{\polyDegree_{\vectSol}}
    \seminorm{\vectSobSpace{\polyDegree_{\vectSol}}{\sobIndexConj}{\element}}{\vectSol} + \locDiscFullVectNorm{\boldsymbol{\xi}_h} \right].
  \end{align*}
  Therefore,
  \begin{align*}
    \mathfrak{T}_{2,\element}
    &\leq C(\polyDegree,\sobIndex,\polyDegree_{\vectSol},\rho)
    \left[\norm{\vectLebSpace{\sobIndexConj}{\element}}{\vectSol}^{\sobIndexConj} + \locDiscFullVectNorm{\discVectSol}^{\sobIndexConj} \right]^{\frac{\sobIndexConj-2}{\sobIndexConj}}
    \left[
      h_{\element}^{\polyDegree_{\vectSol}} \seminorm{\vectSobSpace{\polyDegree_{\vectSol}}{\sobIndexConj}{\element}}{\vectSol}
      +
      \locDiscFullVectNorm{\boldsymbol{\xi}_h} \right]
    \locDiscFullVectNorm{\discVectFunOne}.
  \end{align*}
  Summing the above inequality over $\element \in \domain_h$, applying a discrete $(\frac{\sobIndexConj}{\sobIndexConj-2},\sobIndexConj,\sobIndexConj)$-H{\"o}lder inequality, and recalling the bounds \eqref{eq:cont.sols.stability}, \eqref{eq:stab.bound.disc.sols}, and \eqref{eq:discrete.error.g2} we arrive at
  \begin{equation}\label{eq:scalar.err.estimate.T2.g2}
    \begin{aligned}
      \sum_{\element \in \domain_h}
      \mathfrak{T}_{2,\element}
      &\leq
      C(\polyDegree,\sobIndex,\polyDegree_{\vectSol},\polyDegree_{\fluxFunction},\seminorm{\vectSobSpace{1}{\sobIndexConj}{\domain_h}}{\vectSol},\norm{\lebSpace{\sobIndexConj}{\domain}}{\loadTerm},\rho) \\ 
      &\quad \times
      \left(
      \left(\tilde{\mathcal{R}}_h^{(1)}\right)^{\frac{1}{\sobIndexConj}} + \left(\tilde{\mathcal{R}}_h^{(2)}\right)^{\frac{1}{\sobIndexConj}} + \left(\tilde{\mathcal{R}}_h^{(2)}\right)^{\frac{1}{2}} \right)   \discFullVectNorm{\discVectFunOne}
    \end{aligned}
  \end{equation}

  \medskip
  \noindent
  \textbf{Third term.}
  From \eqref{eq:sE.Holder.continuity} and \eqref{eq:sE.approx.for.u} it follows, for any $\element \in \domain_h$,
  \begin{equation}\label{eq:scalar.err.estimate.T3.g2.init}
    \begin{aligned}
      \mathfrak{T}_{3,\element}
      &\leq
      C(\sobIndex) \left[s_{\element}(\discVectSol^{\perp},\discVectSol^{\perp}) \right]^{\frac{\sobIndexConj-1}{\sobIndexConj}}
      \locDiscFullVectNorm{\discVectFunOne} \\ 
      &\leq C(\polyDegree,\sobIndex,\polyDegree_{\vectSol},\rho) \locDiscFullVectNorm{\discVectFunOne} 
      \Big[ h_{\element}^{\sobIndexConj \polyDegree_{\vectSol}} \seminorm{\vectSobSpace{\polyDegree_{\vectSol}}{\sobIndexConj}{\element}}{\vectSol}^{\sobIndexConj} \\ 
      &\quad 
      + \left[
        s_{\element}(\discVectSol^{\perp},\boldsymbol{\xi}_h^{\perp}) - s_{\element}(\vectProjF {\polyDegree} \vectSol -
        \vectProjL{\polyDegree} \vectProjF {\polyDegree} \vectSol,\boldsymbol{\xi}_h^{\perp}) \right]\Big]^{\frac{\sobIndexConj-1}{\sobIndexConj}}.
    \end{aligned}
  \end{equation}
  We preliminarly observe that from \eqref{eq:aE.monotonicity} and combining \eqref{eq:error.equation.g2}, \eqref{eq:consistency.error.g2} and \eqref{eq:discrete.error.g2} we have
  \begin{align*}
    &\sum_{\element \in \domain_h} \left[s_{\element}(\discVectSol^{\perp},\boldsymbol{\xi}_h^{\perp}) - s_{\element}((\vectProjF {\polyDegree} \vectSol)^{\perp},\boldsymbol{\xi}_h^{\perp})\right] \\
    &\quad \leq
    a_{h}(\discVectSol,\boldsymbol{\xi}_h) - a_{h}(\vectProjF {\polyDegree} \vectSol,\boldsymbol{\xi}_h) 
    \\ 
    &\quad 
    \leq
    C(\polyDegree,\sobIndex, \polyDegree_{\vectSol}, \polyDegree_{\fluxFunction},\seminorm{\vectSobSpace{1}{\sobIndexConj}{\domain_h}}{\vectSol},\norm{\lebSpace{\sobIndexConj}{\domain}}{\loadTerm},\rho)
    \left( \tilde{\mathcal{R}}^{(1)}_h + \tilde{\mathcal{R}}^{(2)}_h + \left(\tilde{\mathcal{R}}^{(2)}_h\right)^{\frac{\sobIndexConj}{2}} \right).
  \end{align*}
  Therefore, summing \eqref{eq:scalar.err.estimate.T3.g2.init} over $\element \in \domain_h$, applying a discrete $(\sobIndexConj,\sobIndex)$-H{\"o}lder inequality, and the above bound, we conclude
  \begin{equation}\label{eq:scalar.err.estimate.T3.g2}
    \begin{aligned}
      \sum_{\element \in \domain_h}
      \mathfrak{T}_{3,\element}
      &\leq
      C(\polyDegree,\sobIndex, \polyDegree_{\vectSol}, \polyDegree_{\fluxFunction},\seminorm{\vectSobSpace{1}{\sobIndexConj}{\domain_h}}{\vectSol},\norm{\lebSpace{\sobIndexConj}{\domain}}{\loadTerm},\rho) \\ 
      &\quad \times
      \left( \left(\tilde{\mathcal{R}}^{(1)}_h\right)^{\frac{1}{\sobIndex}} + \left(\tilde{\mathcal{R}}^{(2)}_h\right)^{\frac{1}{\sobIndex}} + \left(\tilde{\mathcal{R}}^{(2)}_h\right)^{\frac{1}{2(\sobIndex-1)}} \right)
      \discFullVectNorm{\discVectFunOne}.
    \end{aligned}
  \end{equation}

  \medskip
  \noindent
  \textbf{Conclusion.}
  Plugging \eqref{eq:scalar.err.estimate.T1.g2}, \eqref{eq:scalar.err.estimate.T2.g2} and \eqref{eq:scalar.err.estimate.T3.g2} into \eqref{eq:scalar.err.estimate.initial.g2} and combining the result with the discrete inf-sup inequality \eqref{eq:discrete.infsup} yields the desired result.
\end{proof}
As a direct consequence of Theorem~\ref{thm:discrete.error.scalar.g2}, the approximation properties of $\projL{\polyDegree}$, and the triangle inequality, we obtain the following result.
\begin{corollary}\label{cor:total.error.scalar.g2}
  Under the same assumptions of Theorem~\ref{thm:discrete.error.scalar.g2} and assuming furthermore 
  \begin{itemize}
    \item[] $\scalarSol \in \sobSpace{\polyDegree_{\scalarSol}}{\sobIndex}{\domain_h}$ for $\polyDegree_{\scalarSol} \in \{0,1,\ldots,\polyDegree +1 \}$,
  \end{itemize}
  we have that
  \begin{equation}\label{eq:total.error.scalar.g2}
    \begin{aligned}
      \discScalarNorm{\scalarSol - \discScalarSol}
      &\lesssim 
      \left( \tilde{\mathcal{R}}_h^{(1)} \right)^{\frac{1}{\sobIndex}} + \left(\tilde{\mathcal{R}}^{(2)}_h\right)^{{\frac{1}{\sobIndex}}}
      + \left(\tilde{\mathcal{R}}_h^{(1)}\right)^{\frac{1}{\sobIndexConj}} + \left(\tilde{\mathcal{R}}_h^{(2)}\right)^{\frac{1}{\sobIndexConj}} + \left(\tilde{\mathcal{R}}_h^{(2)}\right)^{\frac{1}{2}}
       + \left(\tilde{\mathcal{R}}^{(2)}_h\right)^{\frac{1}{2(\sobIndex-1)}} \\
       &\quad + \left(\sum_{\element \in \domain_h} h_{\element}^{\sobIndex \polyDegree_{\scalarSol}} \seminorm{\sobSpace{\polyDegree_{\scalarSol}}{\sobIndex}{\element}}{\scalarSol} \right)^{\frac{1}{\sobIndex}},
     \end{aligned}
  \end{equation}
  with a hidden constant depending only on $\seminorm{\vectSobSpace{1}{\sobIndexConj}{\domain_h}}{\vectSol}$, $\norm{\lebSpace{\sobIndexConj}{\domain}}{\loadTerm}$, $\polyDegree$, $\sobIndexConj$, $\polyDegree_{\vectSol}$, $\polyDegree_{\fluxFunction}$, $\polyDegree_{\scalarSol}$, and $\rho$.

  In particular, assuming full regularity, i.e., $\polyDegree_{\vectSol} = \polyDegree + 1$, $\polyDegree_{\fluxFunction} = \polyDegree + 1$, and $\polyDegree_{\scalarSol} = \polyDegree+1$, estimate \eqref{eq:total.error.scalar.g2} yields the explicit asymptotic convergence rate
  \begin{equation}\label{eq:asymptotic.rate.scalar.g2}
    \discScalarNorm{\scalarSol - \discScalarSol} \lesssim h^{(\sobIndex-1)(\polyDegree+1)}.
  \end{equation}
\end{corollary}
%
\subsection{Summary of convergence rates} \label{sec:global.summary}
Now that the error analysis for both variables and regimes is complete, we collect and summarize the asymptotic convergence rates established in Corollaries~\ref{cor:total.error}, \ref{cor:total.error.g2}, \ref{cor:total.error.scalar}, and~\ref{cor:total.error.scalar.g2} in Table~\ref{tab:convergence.rates}.
\begin{table}[h]
  \centering
  \bgroup
  \def\arraystretch{1.4}
  \begin{tabular}{c|cc}
    \hline
    & $\discFullVectNorm{\vectSol - \discVectSol}$ & $\discScalarNorm{\scalarSol - \discScalarSol}$ \\
    \hline
    $1 < \sobIndexConj \leq 2$ & $h^{\frac{\sobIndexConj}{2} (\polyDegree +1)}$  & $h^{\frac{1}{\sobIndex-1} (\polyDegree +1)}$ \\
    $\sobIndexConj > 2$        & $h^{\frac{1}{\sobIndexConj-1} (\polyDegree +1)}$ & $h^{(\sobIndex-1) (\polyDegree +1)}$ \\
    \hline
  \end{tabular}
  \egroup
  \caption{Dominant error components assuming high regularity for both $\vectSol$ and $\scalarSol$.}
  \label{tab:convergence.rates}
\end{table}
As show in Table~\ref{tab:convergence.rates}, for the lowest-order case ($\polyDegree = 0$), the derived convergence rates match the theoretical bounds established for classical conforming mixed finite element methods \cite{FARHLOUL.MANOUZI:2000} in the regime $1<\sobIndexConj \leq 2$. Conversely, for $\sobIndexConj > 2$, our analysis yields a rate of $\mathcal{O}(h^{\frac{1}{\sobIndexConj-1}})$ for the flux (instead of $\mathcal{O}(h^{\frac{2}{\sobIndexConj}})$) and $\mathcal{O}(h^{\sobIndex-1})$ for the scalar solution (instead of $\mathcal{O}(h^{1})$). 

While these bounds are lower than those achieved in \cite{FARHLOUL.MANOUZI:2000}, the discrepancy stems from the fact that their analysis fundamentally relies on the conforming nature of the discrete functional spaces. In our polytopal framework, as pointed out in Remark~\ref{rem:conforming}, the presence of general non-convex elements causes the discrete space to lose conformity within the range $1<p<2$. Nevertheless, when compared to other non-conforming approaches, such as the Hybrid High-Order methods analyzed in \cite{DIPIETRO.DRONIOU:2017}, our scheme recovers the exact same convergence rates for the scalar variable.
\section{Numerical tests}\label{sec:numerical.tests}
In this section, we investigate the numerical performance of the proposed scheme \eqref{eq:discrete.problem}. 
We first describe the  Ka\v{c}anov-based linearization strategy adopted to solve the resulting non-linear discrete problem.
In Test \ref{sub:test1} we consider the $p$-Laplace equation with a manufactured smooth solution in order to numerically validate the a priori error estimates derived in Section~\ref{sec:a-priori}. The study covers both ranges $1<p'\leq 2$ and $p'>2$, and is carried out on families of meshes consisting of convex and non-convex polygonal elements. This allows us to assess the robustness and accuracy of the method with respect to both the non-linear exponent and the mesh geometry.
In Test \ref{sub:test2}, we further assess the performance of the proposed scheme on a benchmark problem  from torsional creep.

\subsection{Relaxed Ka\v{c}anov method}
\label{sub:kacanov}
We first describe the linearization strategy based on a fixed-point Ka\v{c}anov  iteration \cite{kacanov} adopted to solve the non-linear $p$-Laplace equation. 
To this end, we introduce a linearized version of the diffusive flux
function \eqref{eq:def.sigma} and of the associated operator
\eqref{eq:def.a}. More precisely, for any $p > 1$ and for any fixed $\boldsymbol{\chi} \in \vectSolSpace$ we set
\[
\begin{aligned}
\fluxFunction_{p'}(\boldsymbol{\chi}; \boldsymbol{w})  &\coloneqq \euNorm{\boldsymbol{\chi}}^{p'-2} \boldsymbol{w} 
&\quad &\text{for all $\boldsymbol{w} \in \vectSolSpace$,} 
\\
a_{p'}(\boldsymbol{\chi}; \boldsymbol{w},\vectTestFun) &\coloneqq \int_{\domain}\fluxFunction_{p'}(\boldsymbol{\chi}; \boldsymbol{w}) \cdot \vectTestFun = \int_{\domain} \euNorm{\boldsymbol{\chi}}^{\sobIndexConj-2} \boldsymbol{w} \cdot
  \vectTestFun, 
&\quad &\text{for all $\boldsymbol{w}$, $\vectTestFun \in \vectSolSpace$.}
\end{aligned}
\]
Similarly, $a_{h, p'}(\cdot; \cdot,  \cdot)$ will denote the VEM counterpart of $a_{p'}(\cdot; \cdot,  \cdot)$ obtained generalizing in the linear setting the construction in Subsection \ref{sub:disc-operators}.

Consider the discrete Problem \eqref{eq:discrete.problem} associated with the $p$-Laplace equation with exponent $p > 1$.
Then, in the light of the notation introduced above, Problem \eqref{eq:discrete.problem} can be formulated as follows:

\medskip
\noindent
Find $\discVectSol \in \discVectSpace$ and $\discScalarSol \in \discScalarSpace$ such that
\begin{subequations}\label{eq:discrete.problemK}
  \begin{alignat}{2}
    a_{h, p'}(\discVectSol; \discVectSol,\discVectTestFun) + b(\discVectTestFun, \discScalarSol) &= 0
    \qquad
    &&\text{for all } \discVectTestFun \in \discVectSpace, \label{eq:discrete.problemK1} \\
    -b(\discVectSol,\discScalarTestFun) &= \int_{\domain} \loadTerm \discScalarTestFun
    \qquad
    &&\text{for all } \discScalarTestFun \in \discScalarSpace. \label{eq:discrete.problemK2}
  \end{alignat}
\end{subequations}
To solve the non-linear system \eqref{eq:discrete.problemK} we adopt the  \textit{relaxed Ka\v{c}anov method} described below.

\vspace{0.2cm}

\noindent
Let $(\discVectSol^{\rm D}, \discScalarSol^{\rm D})$ be the solution of the discrete Darcy problem obtained from \eqref{eq:discrete.problemK} by setting $p=p'=2$.

\noindent
Starting from $(\discVectSol^0, \discScalarSol^0) =(\discVectSol^{\rm D}, \discScalarSol^{\rm D})$, until convergence repeat the following steps:
 
\texttt{STEP 1.} For $n \geq 0$ solve
\[
  \begin{aligned}
     a_{h, p'}(\discVectSol^n; \widetilde{\discVectSol}^{n+1},\discVectTestFun) + b(\discVectTestFun, \widetilde{\discScalarSol}^{n+1}) &= 0
    \qquad
    &&\text{for all } \discVectTestFun \in \discVectSpace,
    \\
     -b(\widetilde{\discVectSol}^{n+1},\discScalarTestFun) &= \int_{\domain} \loadTerm \discScalarTestFun
    \qquad
    &&\text{for all } \discScalarTestFun \in \discScalarSpace.
  \end{aligned}
\]

\texttt{STEP 2.} Given a \textit{relaxation parameter} $\omega \in (0, 1]$, set
\[
\discVectSol^{n+1} = \omega \widetilde{\discVectSol}^{n+1} + (1 - \omega) \discVectSol^{n} \,,
\qquad
\discScalarSol^{n+1} = \widetilde{\discScalarSol}^{n+1} \,.
\]

\noindent
Given a tolerance $\texttt{tol} >0$, the iteration is terminated when the relative variation between two consecutive iterates satisfies
\[
\frac{
\vert \boldsymbol{\mathrm{dof}_{\discVectSpace}}(\discVectSol^{n+1} - \discVectSol^{n}) \vert + 
\vert \boldsymbol{\mathrm{dof}_{\discScalarSpace}}(\discScalarSol^{n+1} - \discScalarSol^{n}) \vert
}
{\vert \boldsymbol{\mathrm{dof}_{\discVectSpace}}(\discVectSol^{n+1}) \vert
+ 
\vert \boldsymbol{\mathrm{dof}_{\discScalarSpace}}(\discScalarSol^{n+1}) \vert} \leq \texttt{tol} \,, 
\]
where $\boldsymbol{\mathrm{dof}_{\discVectSpace}}(\cdot)$  and $\boldsymbol{\mathrm{dof}_{\discScalarSpace}}(\cdot)$ denote the vector containing the global degrees of freedom in $\discVectSpace$ and $\discScalarSpace$ respectively.

In the numerical experiments reported below, we set the relaxation parameter to $\omega=\texttt{0.25}$ and the stopping tolerance to $\texttt{tol}=\texttt{1e-6}$. Moreover, we consider the lowest-order approximation, corresponding to the polynomial degree $k=0$.

\subsection{Test 1. Convergence analysis}
\label{sub:test1}

The aim of this test is to validate the theoretical results established in Section \ref{sec:a-priori}.
To this end, we consider Problem \ref{eq:continuous.problem.dual}
posed on the domain $\Omega=(0,1)^2$ for  different values of the Sobolev exponent $p$. 
The load term $f$ (depending on $p$) and the non homogeneous boundary conditions on $\partial \Omega$ are chosen according to the exact solution 
\[
\scalarSol(x_1, x_2) =  e^{x_1 + x_2} + \sin(2\pi x_1) \, \sin(2\pi x_2)\,,
\qquad 
\vectSol(x_1, x_2) = \vectSol(p; x_1, x_2) = 
\fluxFunction_{\dimDomain,\sobIndex} (\gradScalarSol) \,.
\]
We remark that an additional term appears in the right-hand side of Equation~\eqref{eq:continuous.problem.dual.weak.1} 
as a consequence of the integration by parts formula and the prescribed non homogeneous boundary conditions.
In order to verify the a priori error estimates of Section \ref{sec:a-priori} in both ranges $1 < p' \leq 2$ and $p' > 2$, 
numerical tests are performed with the following values of $p'$:
\begin{equation}
\label{eq:test1-pp}
\begin{aligned}
1 <p' \leq 2&:
& \qquad 
p'&=
\texttt{1.10} \,,  
\texttt{1.25} \,,
\texttt{1.50} \,,
\texttt{1.75} \,,
\texttt{2.00} \,,
\\
p' > 2&:
& \qquad 
p'&=
\texttt{2.25} \,,
\texttt{2.50} \,,
\texttt{3.00} \,,
\texttt{4.00} \,,
\texttt{6.00} \,.
\end{aligned}
\end{equation}
To compute the VEM error between the exact solution 
$(\vectSol,\scalarSol) \in \vectSolSpace \times \scalarSolSpace$ and the VEM solution $(\discVectSol,\discScalarSol) \in \discVectSpace \times \discScalarSpace$, recalling definitions \eqref{eq:def.discrete.norm.NEW} and \eqref{eq:def.ah.norm}, for any $(\boldsymbol{\eta}_h, v_h) \in \discVectSpace \times \discScalarSpace$ we consider the computable error quantities 
\begin{equation}
\label{eq:err_quant}
\begin{gathered}
\begin{aligned}
\texttt{err}(\boldsymbol{\eta}_h, \discFullVectNorm{\cdot}) &:=
\frac{\discFullVectNorm{\vectSol - \boldsymbol{\eta}_h} }
{\discFullVectNorm{\vectSol} }\,,
& \quad
\texttt{err}(\boldsymbol{\eta}_h, \discDivFreeNorm{\cdot}) &:=
\frac{\discDivFreeNorm{\vectSol - \boldsymbol{\eta}_h} }
{\discDivFreeNorm{\vectSol} }\,,
\\
\texttt{err}(\boldsymbol{\eta}_h, \boldsymbol{L}^{p'}) &:=
\frac{\|\vectSol - \vectProjL{\polyDegree}\boldsymbol{\eta}_h\|_{\boldsymbol{L}^{p'}(\Omega)}}
{\|\vectSol\|_{\boldsymbol{L}^{p'}(\Omega)}}  \,,
& \quad
\texttt{err}(v_h, L^{p}) &:= \frac{\|\scalarSol - v_h\|_{L^{p}(\Omega)}}
{\|\scalarSol\|_{L^{p}(\Omega)}} \,, 
\\
\end{aligned}
\\
\texttt{err}(\fluxFunction_{p'}(\boldsymbol{\eta}_h), \boldsymbol{L}^{p}) :=
\frac{\|\fluxFunction_{p'}(\vectSol) - \fluxFunction_{p'}(\vectProjL{\polyDegree} \boldsymbol{\eta}_h)\|_{\boldsymbol{L}^{p}(\Omega)}}
{\|\fluxFunction_{p'}(\vectSol)\|_{\boldsymbol{L}^{p}(\Omega)}}
\,.
\end{gathered}
\end{equation}
Furthermore, given a sequence of $N+1$ meshes with mesh diameters $h_0 > \dots > h_N$, and denoting by $\texttt{err}(h_n)$ any of the error quantities listed in \eqref{eq:err_quant} computed on the mesh with diameter $h_n$, we define average experimental order of convergence \texttt{AEOC}  as  
\begin{equation}
\label{eq:err_aeoc}
\texttt{AEOC}= \frac{1}{N} \sum_{n=1}^N 
\frac{\log(\texttt{err}(h_{n-1})/\texttt{err}(h_{n}))}{\log(h_{n-1}/h_{n})} \,.
\end{equation}
The domain $\Omega$  is partitioned with the following sequences of polygonal meshes:
\texttt{QUADRILATERAL} distorted meshes,
\texttt{RANDOM} Voronoi meshes, and 
\texttt{NON CONVEX} polygonal meshes (see Fig.~\ref{fig:meshes}).
For the generation of the Voronoi meshes we used the code \texttt{Polymesher} \cite{polymesher}.
For each mesh family, we consider a mesh sequence with diameters
$\texttt{h} = \texttt{1/4}, \, \texttt{1/8}, \, \texttt{1/16}, \, \texttt{1/32}, \, \texttt{1/64}$.
\begin{figure}
\centering
\begin{overpic}[scale=0.23]{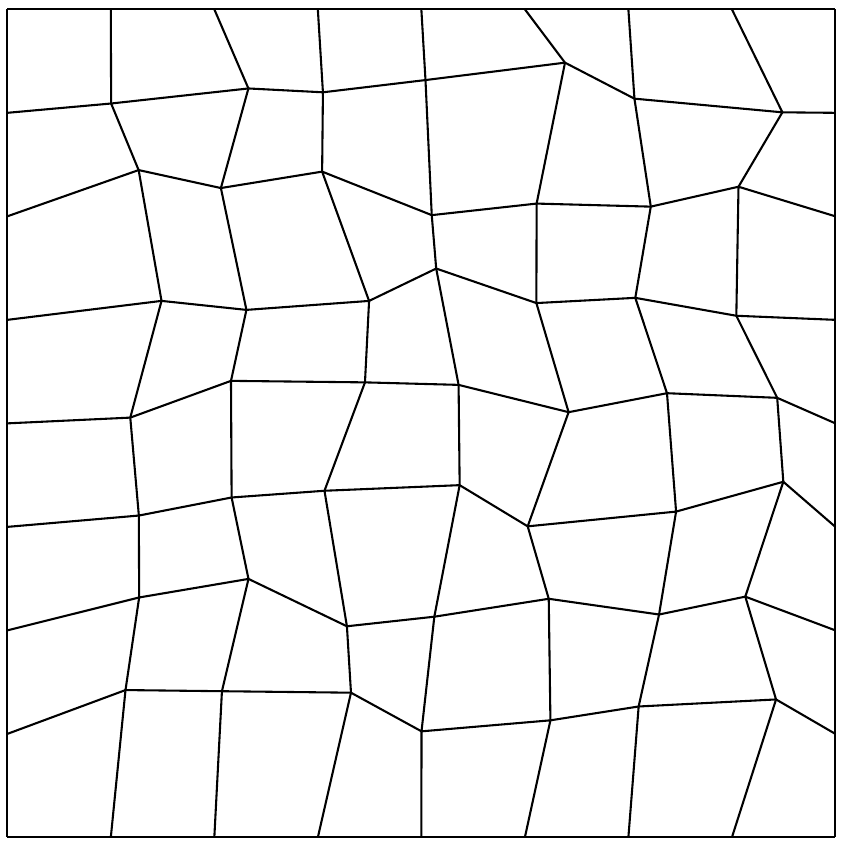}
\put(10,-15){{{\texttt{QUADRILATERAL}}}}
\end{overpic}
\qquad
\begin{overpic}[scale=0.23]{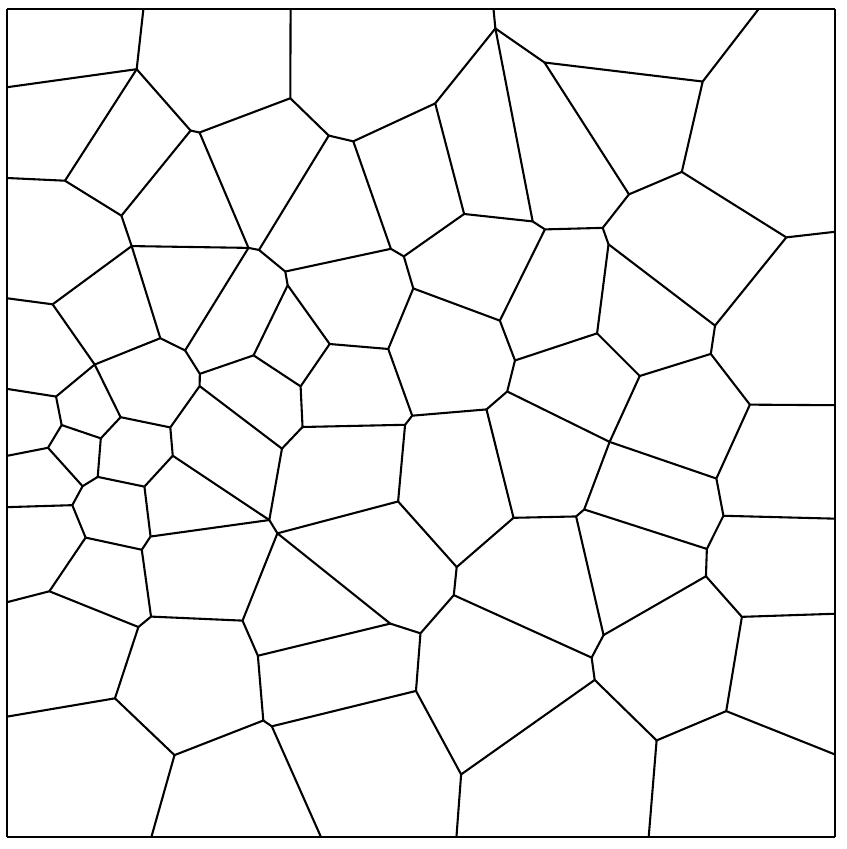}
\put(30,-15){{{\texttt{RANDOM}}}}
\end{overpic}
\qquad
\begin{overpic}[scale=0.20]{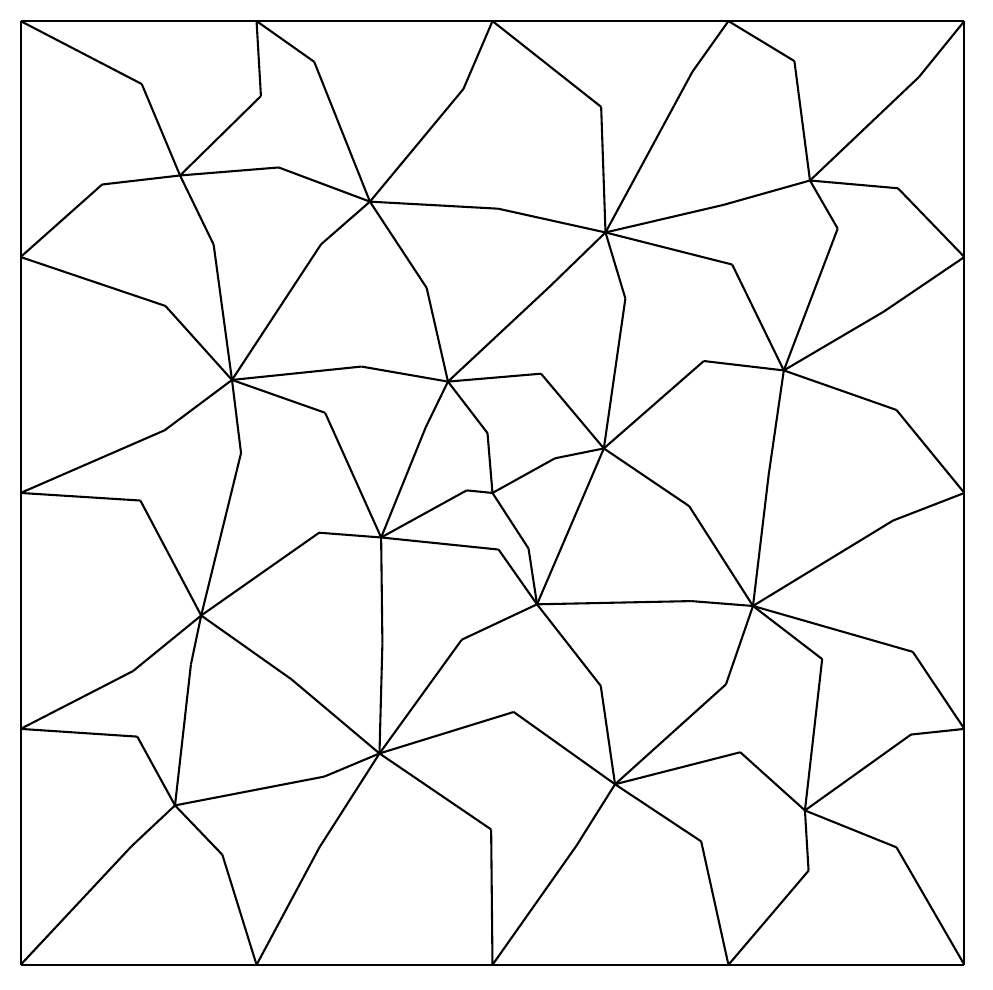}
\put(20,-15){{{\texttt{NON CONVEX}}}}
\end{overpic}
\vspace{0.5cm}
\caption{Example of the adopted polygonal meshes.}
\label{fig:meshes}
\end{figure}
We emphasize that the \texttt{QUADRILATERAL} meshes  and the \texttt{RANDOM} meshes consist exclusively of convex elements. Therefore, according to Remark~\ref{rem:conforming}, the inclusion
$\discVectSpace \subset \vectSolSpace$ holds for all
  $\sobIndexConj \in (1, \infty)$. 
In contrast, the \texttt{NON CONVEX} meshes family satisfy the inclusion $\discVectSpace \subset \vectSolSpace$  if $\sobIndexConj \in (1,2]$.

In Fig.~\ref{fig:test1-Q}, Fig.~\ref{fig:test1-V} and Fig.~\ref{fig:test1-W}, we plot the computed error quantities in \eqref{eq:err_quant} 
for the sequences of aforementioned mesh families (\texttt{QUADRILATERAL}, \texttt{RANDOM}, \texttt{NON CONVEX}, respectively)
and for the values of $p'$ reported in \eqref{eq:test1-pp}.
Whenever the asymptotic behavior of the errors is not readily apparent from the plots, the corresponding average experimental orders of convergence \texttt{AEOC} are also reported.

%
\begin{figure}[htbp]
    \centering
    {\texttt{QUADRILATERAL MESHES}}
\\
    \begin{subfigure}{0.45\textwidth}
        \centering
        \includegraphics[width=\linewidth]{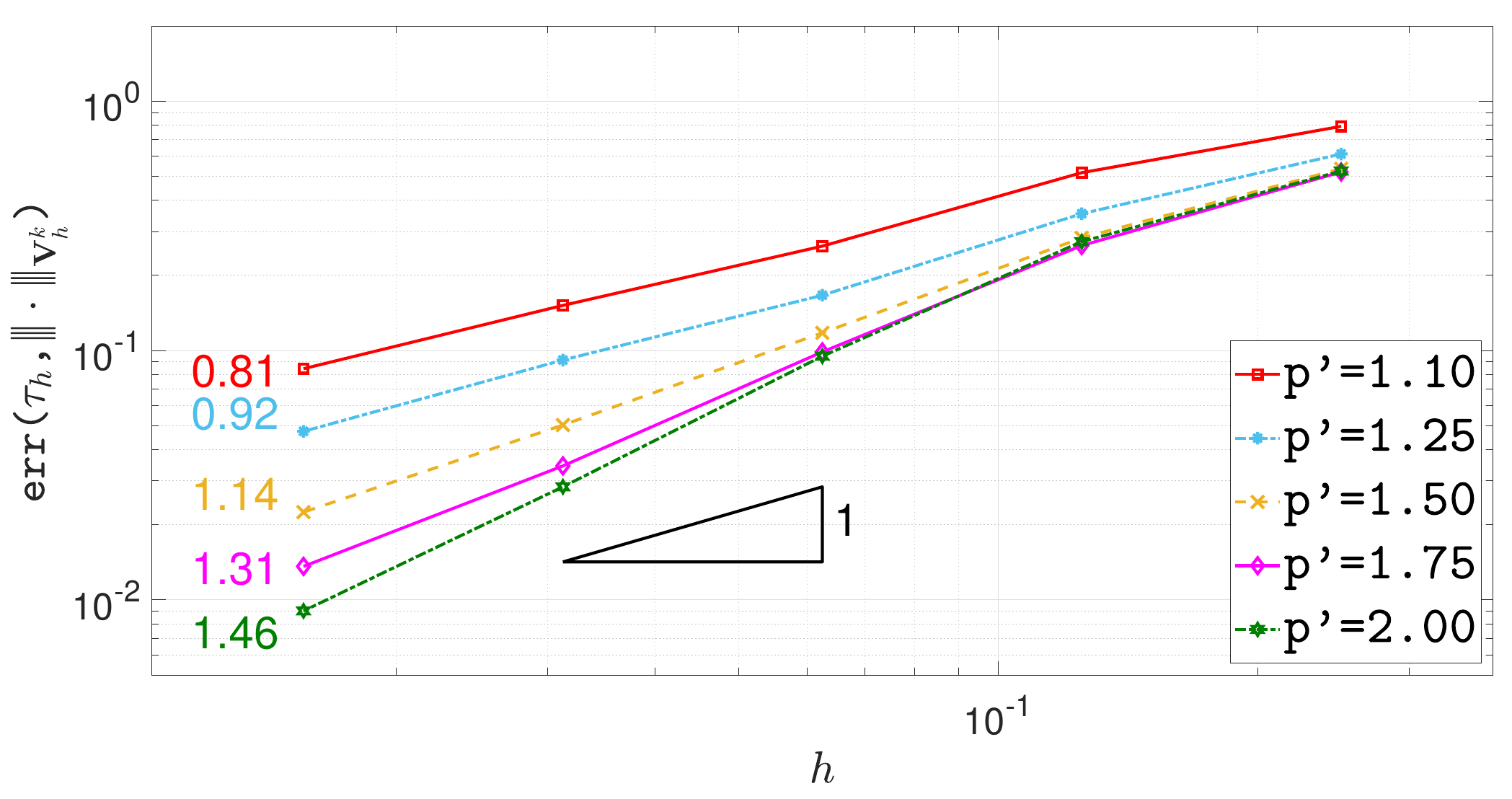}
    \end{subfigure}\qquad
    \begin{subfigure}{0.45\textwidth}
        \centering
        \includegraphics[width=\linewidth]{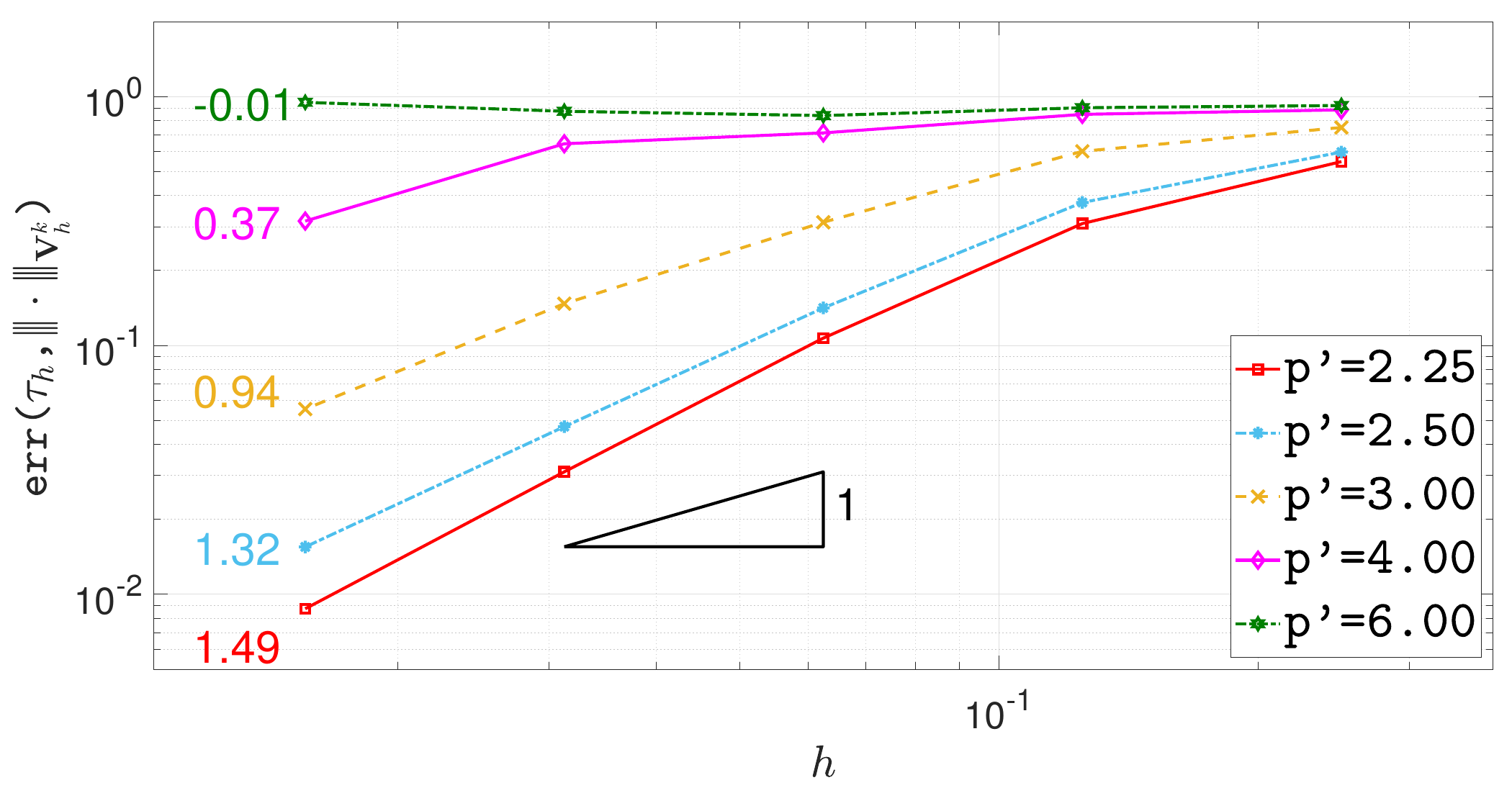}
    \end{subfigure}

    \vspace{0.25cm}
    
    \begin{subfigure}{0.45\textwidth}
        \centering
        \includegraphics[width=\linewidth]{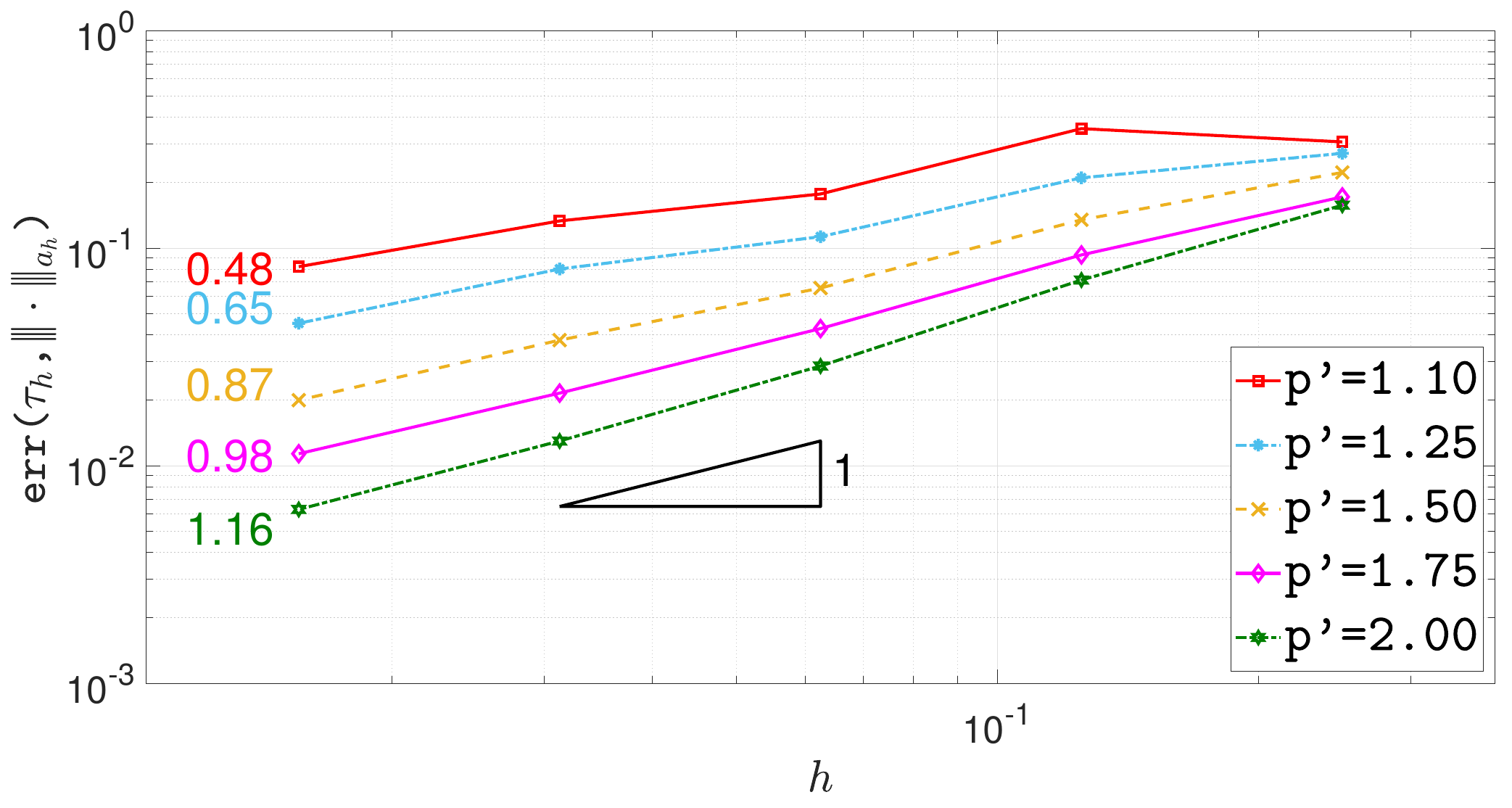}
    \end{subfigure}\qquad
    \begin{subfigure}{0.45\textwidth}
        \centering
        \includegraphics[width=\linewidth]{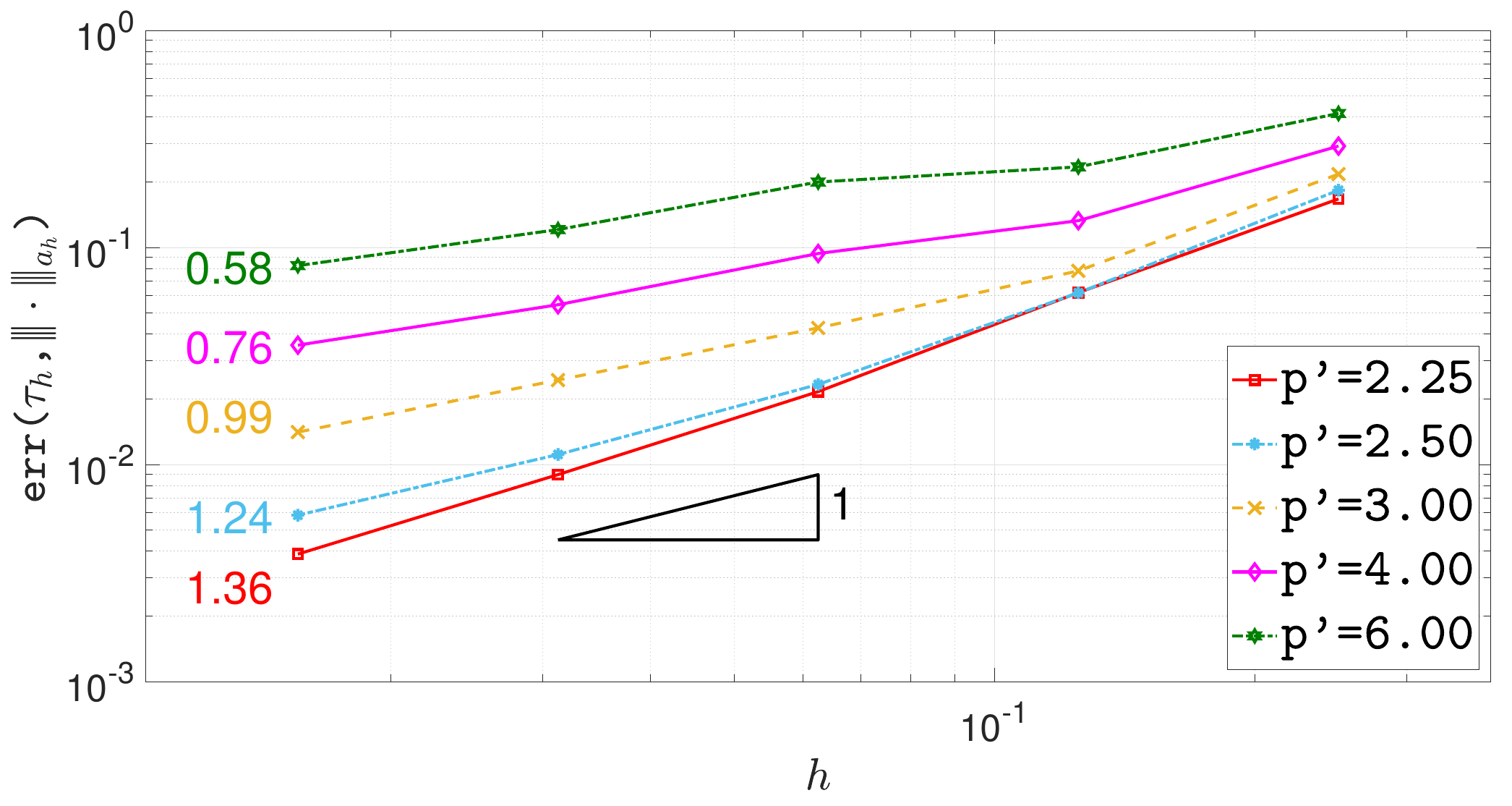}
    \end{subfigure}
    
    \vspace{0.25cm}
    
    \begin{subfigure}{0.45\textwidth}
        \centering
        \includegraphics[width=\linewidth]{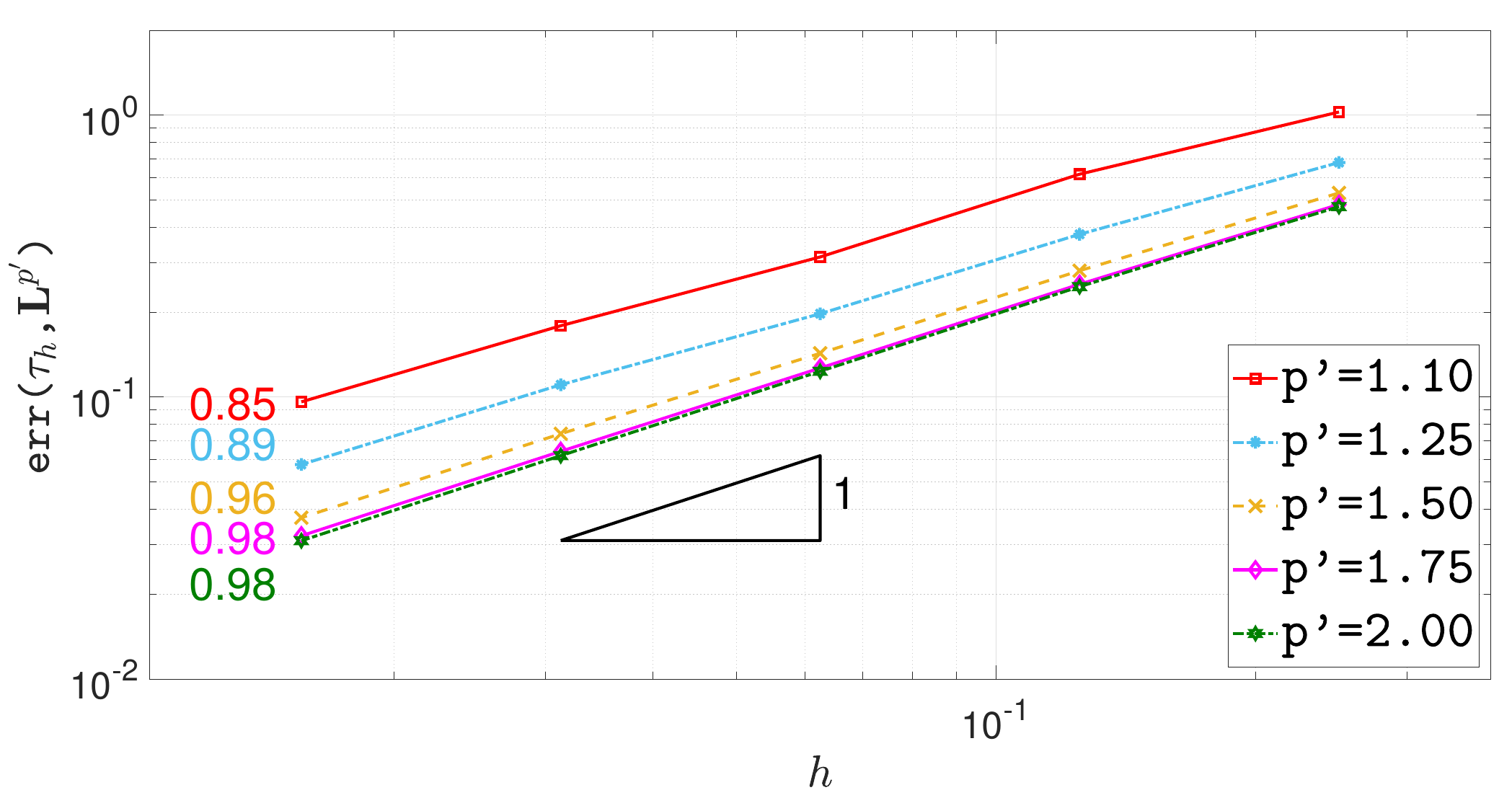}
    \end{subfigure}\qquad
    \begin{subfigure}{0.45\textwidth}
        \centering
        \includegraphics[width=\linewidth]{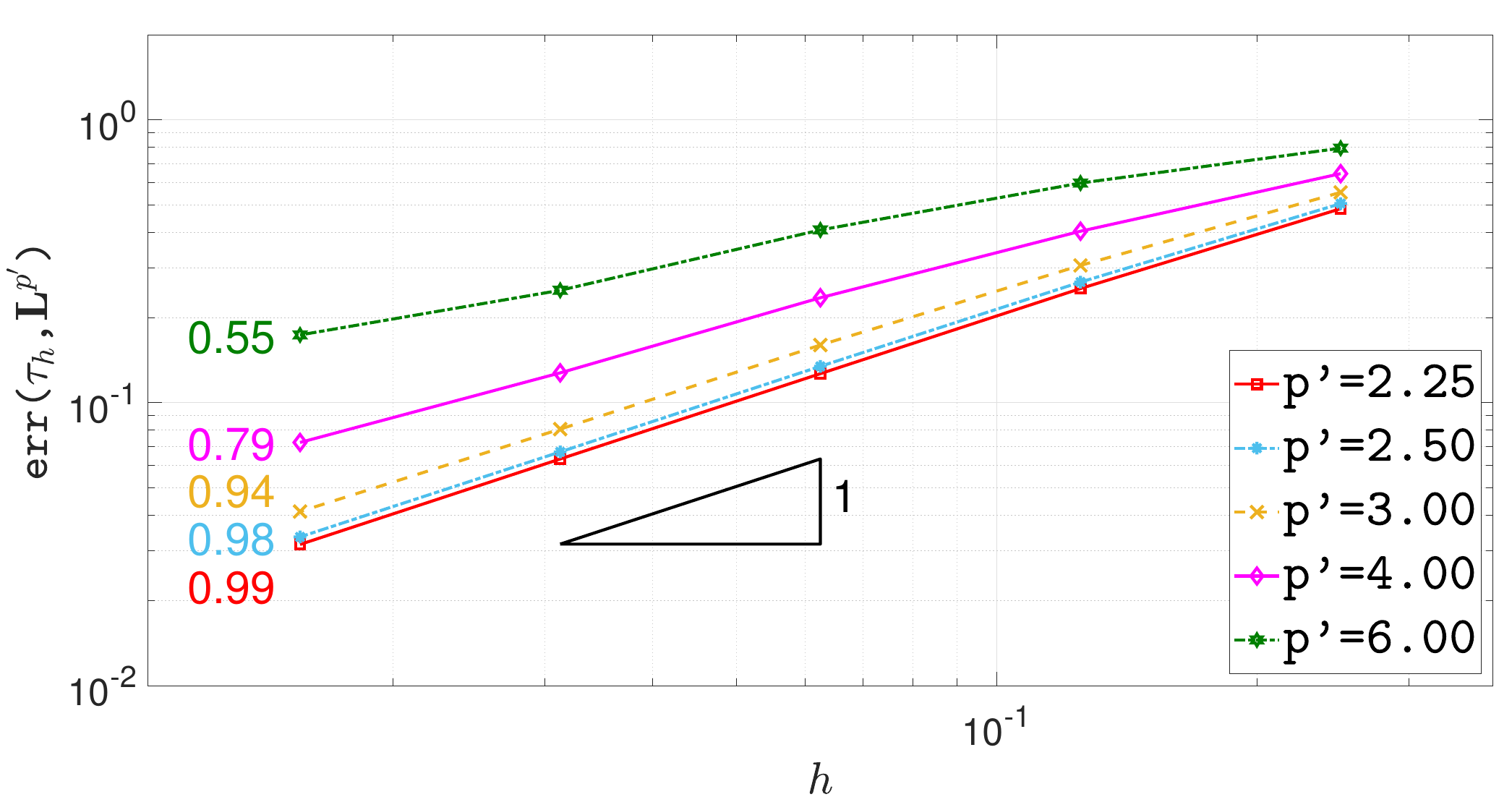}
    \end{subfigure}

    \vspace{0.25cm}    
    
    \begin{subfigure}{0.45\textwidth}
        \centering
        \includegraphics[width=\linewidth]{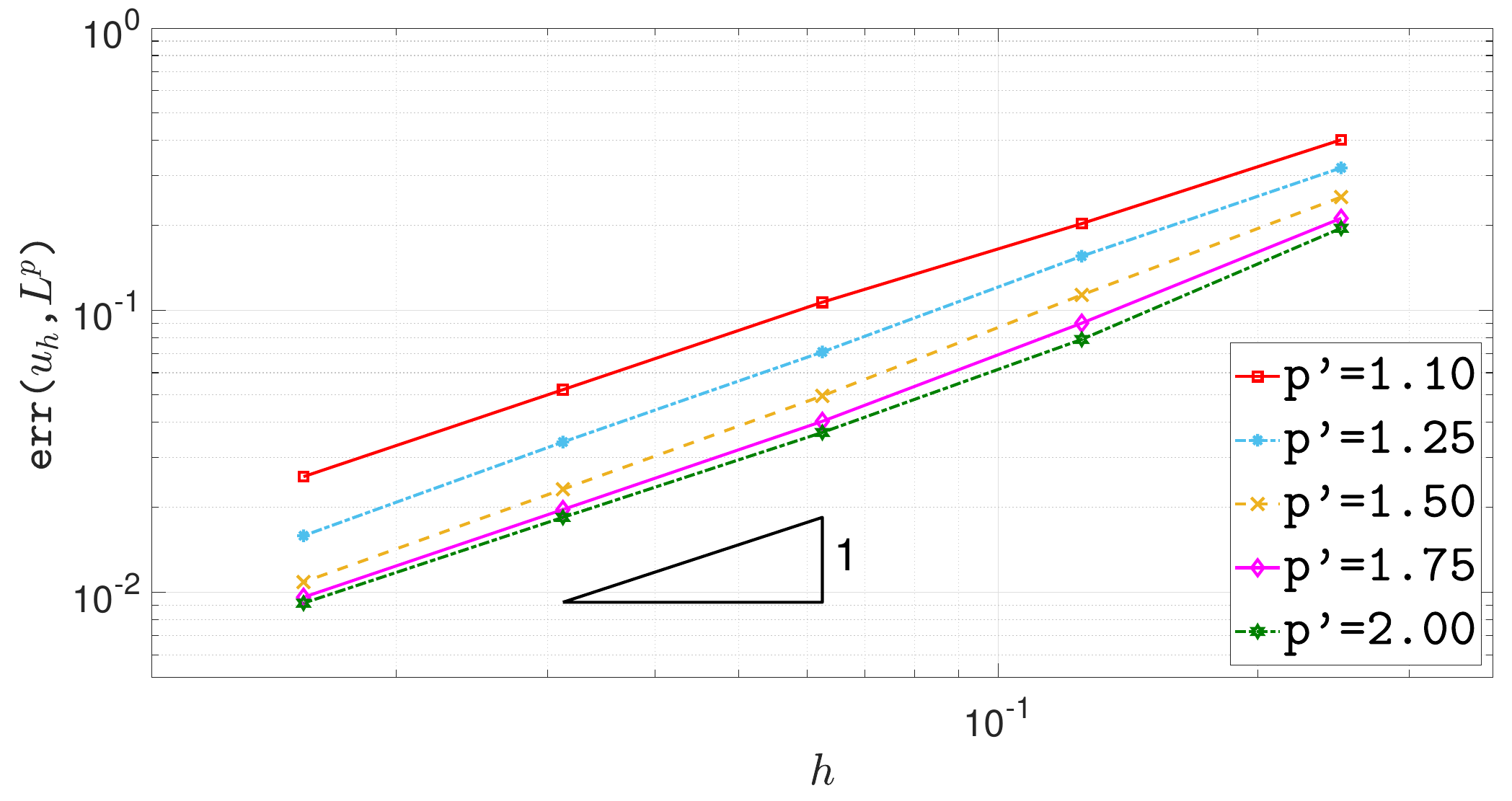}
        
    \end{subfigure}\qquad
    \begin{subfigure}{0.45\textwidth}
        \centering
        \includegraphics[width=\linewidth]{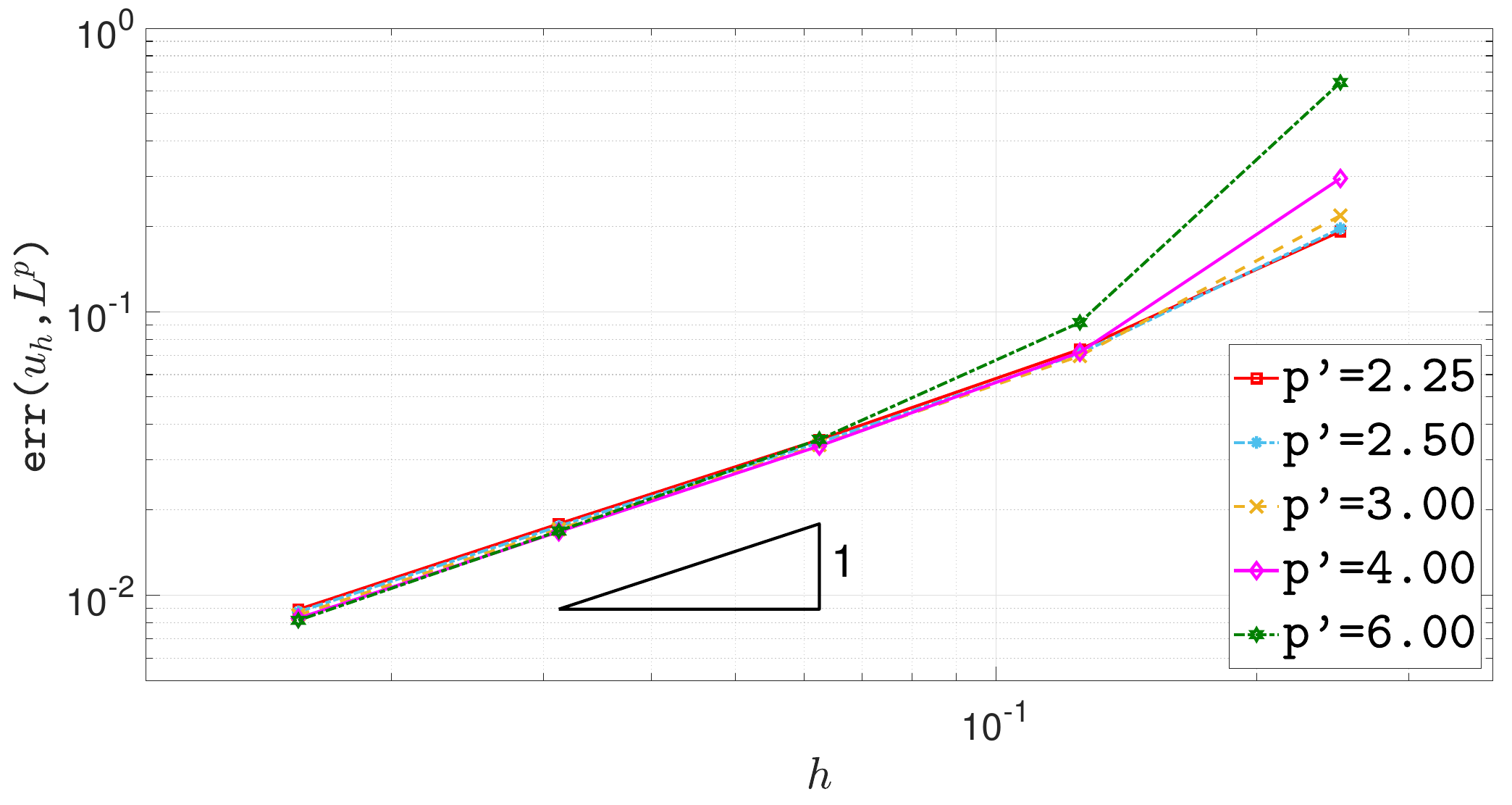}
    \end{subfigure}

    \vspace{0.25cm}

    \begin{subfigure}{0.45\textwidth}
        \centering
        \includegraphics[width=\linewidth]{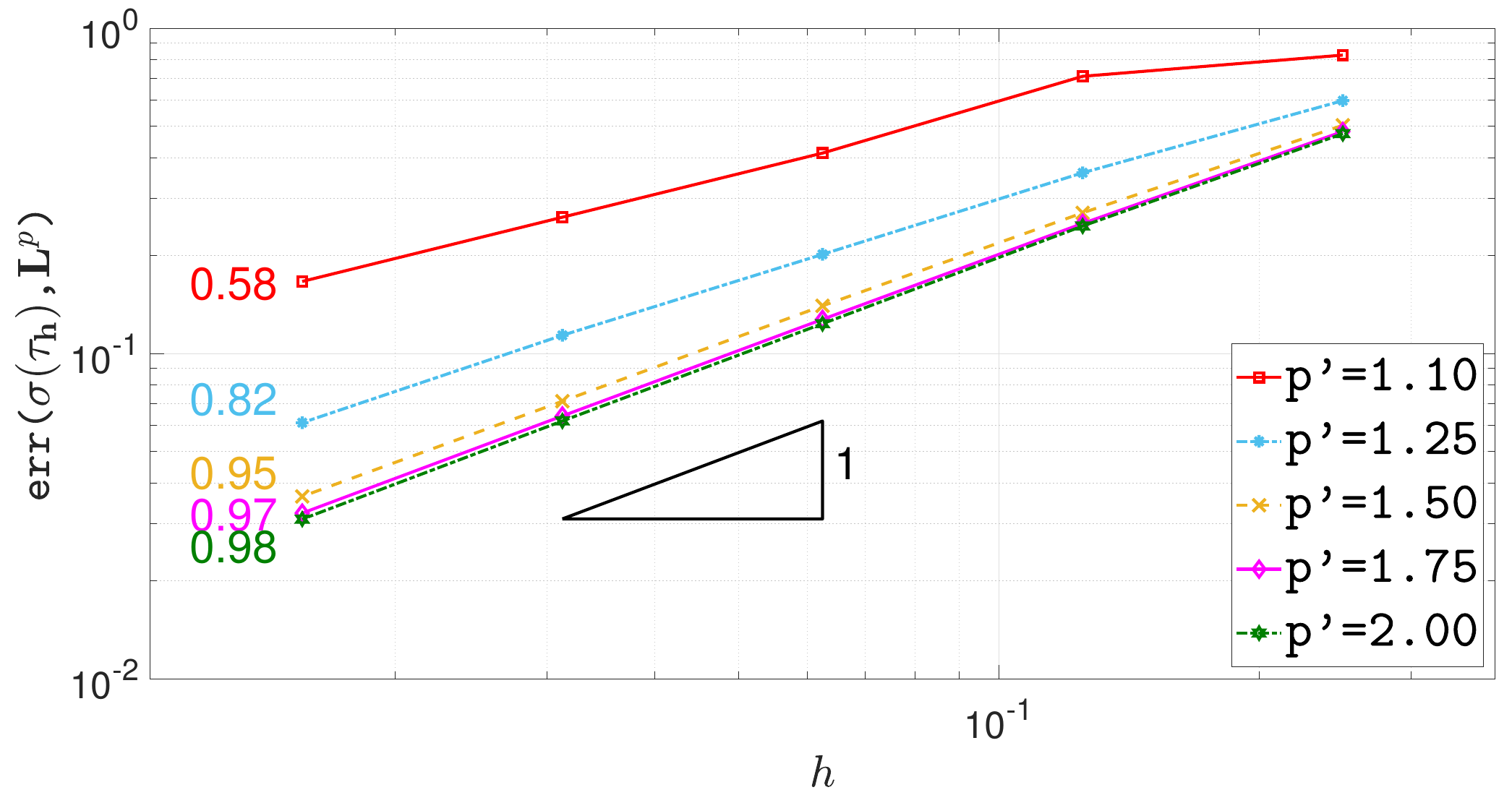}
    \end{subfigure}\qquad
    \begin{subfigure}{0.45\textwidth}
        \centering
        \includegraphics[width=\linewidth]{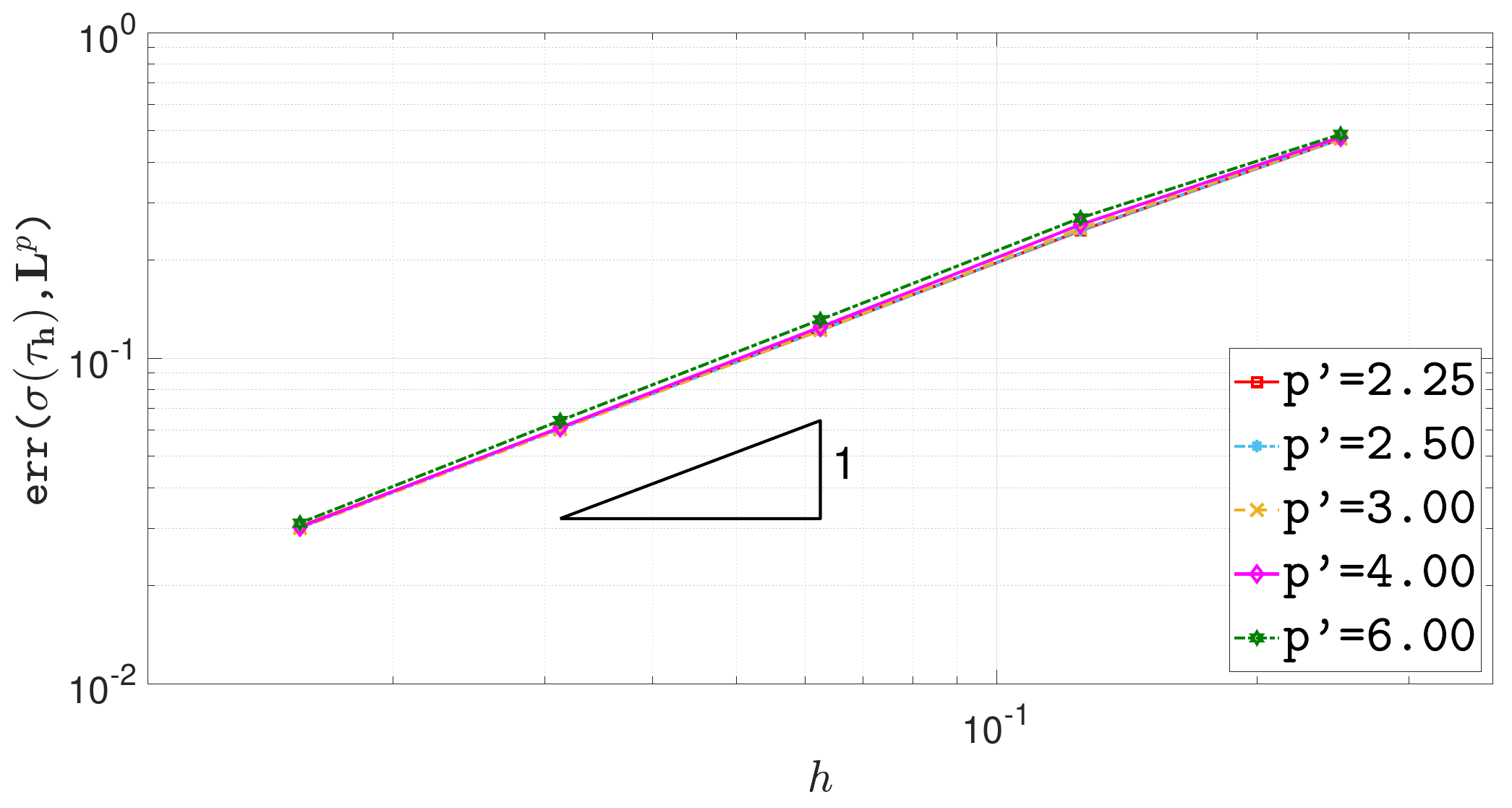}
    \end{subfigure}

   \caption{Test 1. Computed errors defined as in \eqref{eq:err_quant} as a function of the mesh size (loglog scale), for the mesh family \texttt{QUADRILATERAL}. Left panel: $1 <  \sobIndexConj \leq 2$, right panel: $\sobIndexConj > 2$.} 
    \label{fig:test1-Q}
\end{figure}
\begin{figure}[htbp]
    \centering
    {\texttt{RANDOM MESHES}}
\\
    \begin{subfigure}{0.45\textwidth}
        \centering
        \includegraphics[width=\linewidth]{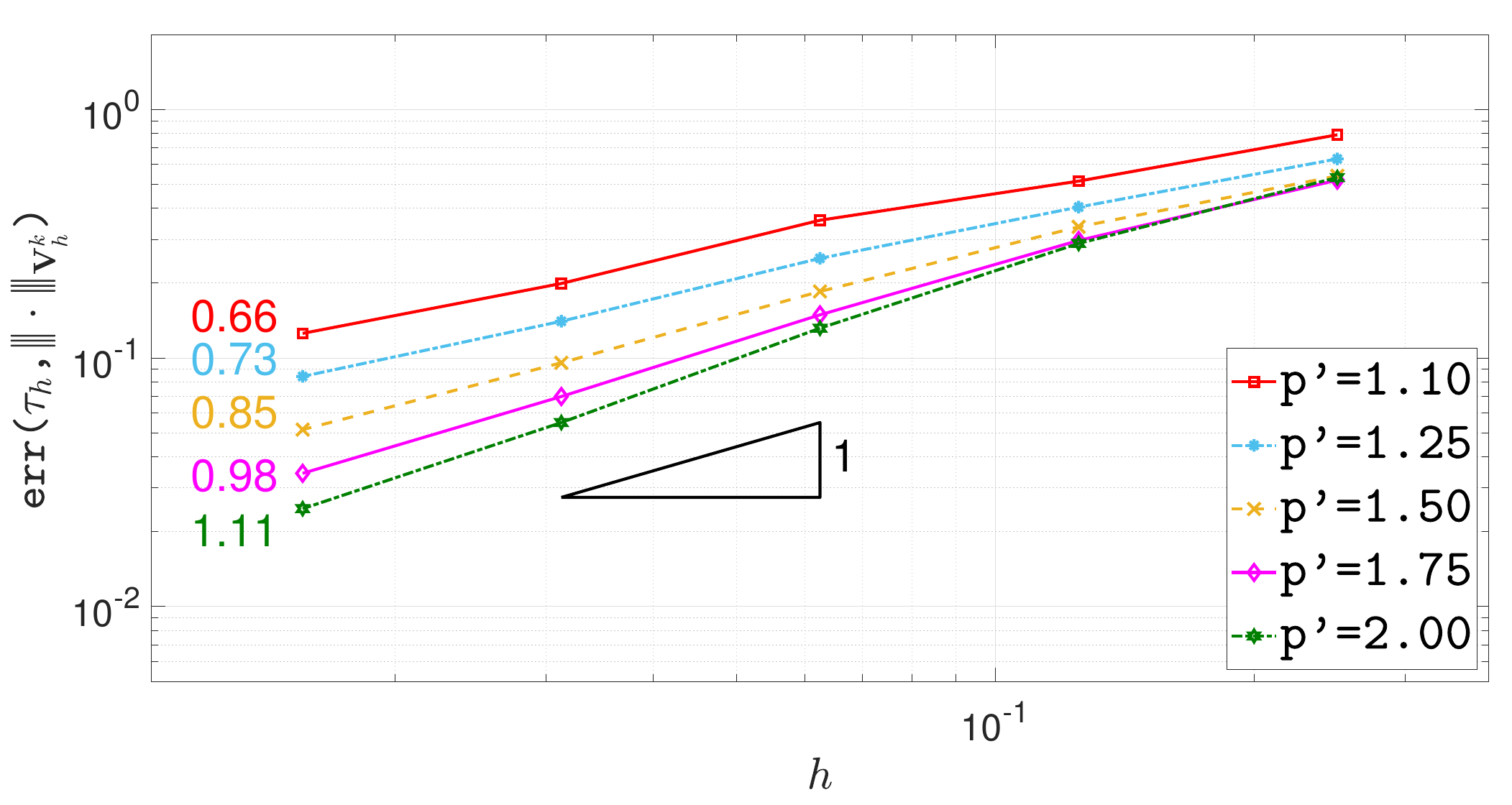}
    \end{subfigure}\qquad
    \begin{subfigure}{0.45\textwidth}
        \centering
        \includegraphics[width=\linewidth]{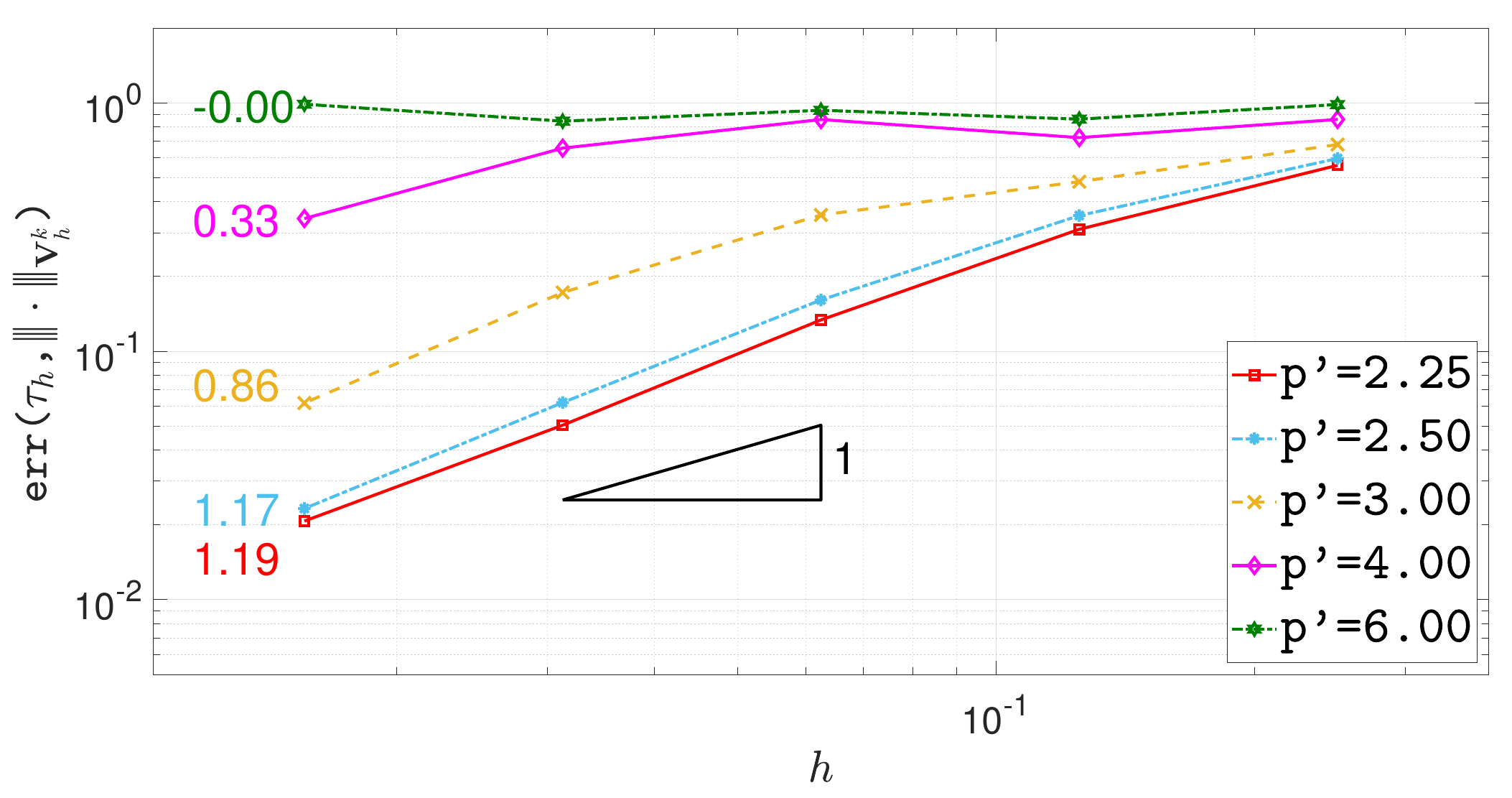}
    \end{subfigure}

    \vspace{0.25cm}
    
    \begin{subfigure}{0.45\textwidth}
        \centering
        \includegraphics[width=\linewidth]{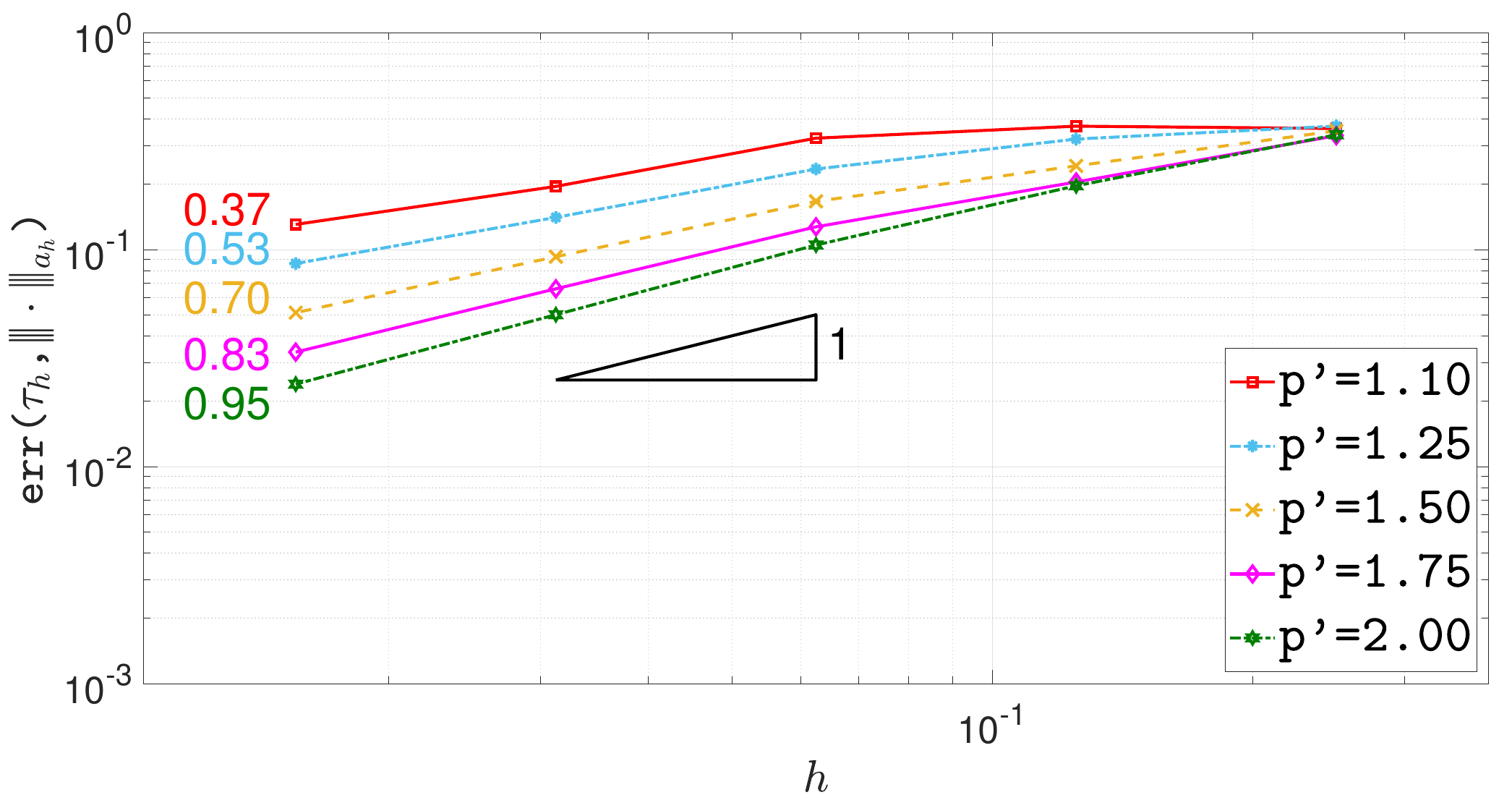}
    \end{subfigure}\qquad
    \begin{subfigure}{0.45\textwidth}
        \centering
        \includegraphics[width=\linewidth]{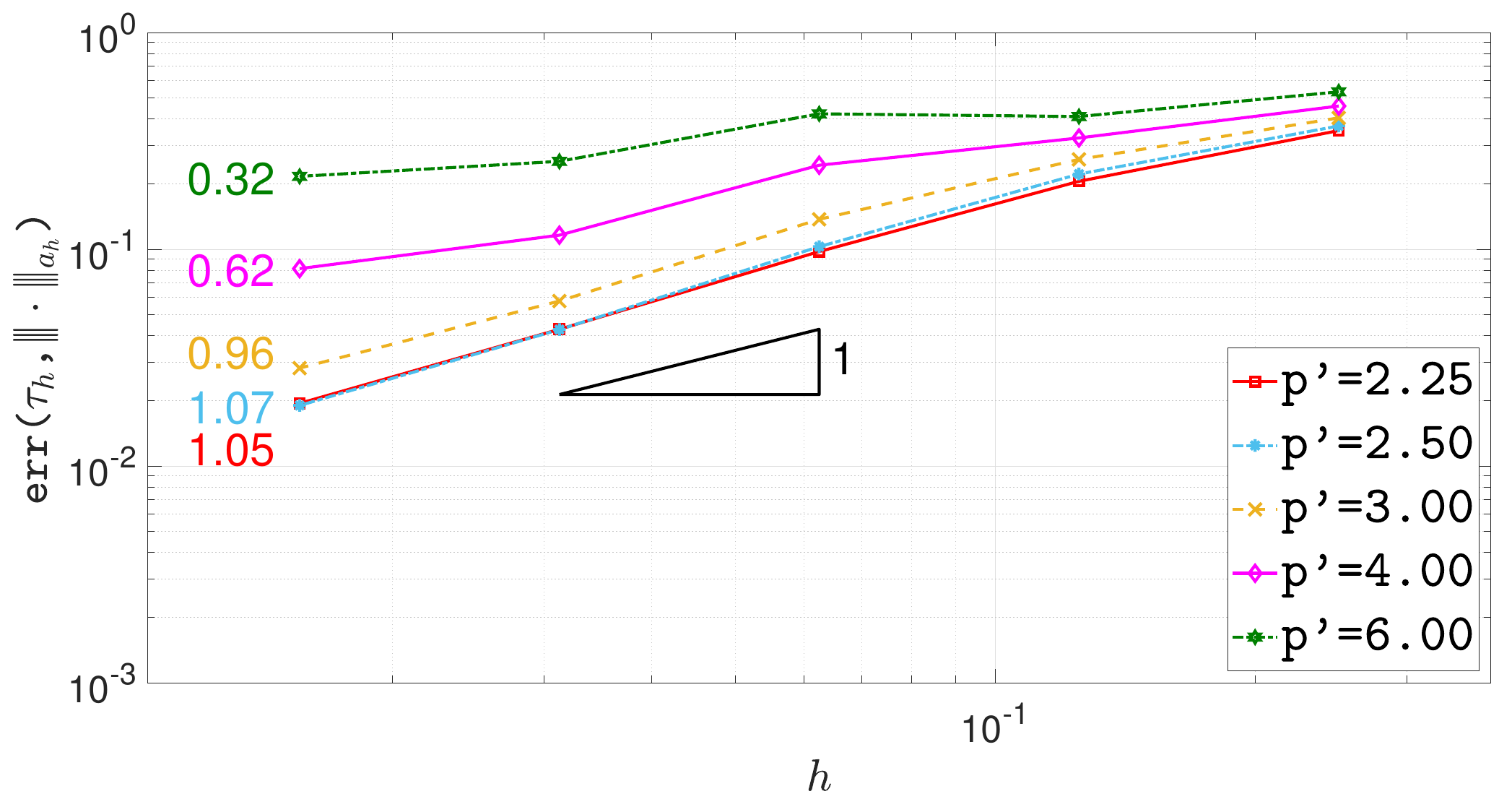}
    \end{subfigure}
    
    \begin{subfigure}{0.45\textwidth}
        \centering
        \includegraphics[width=\linewidth]{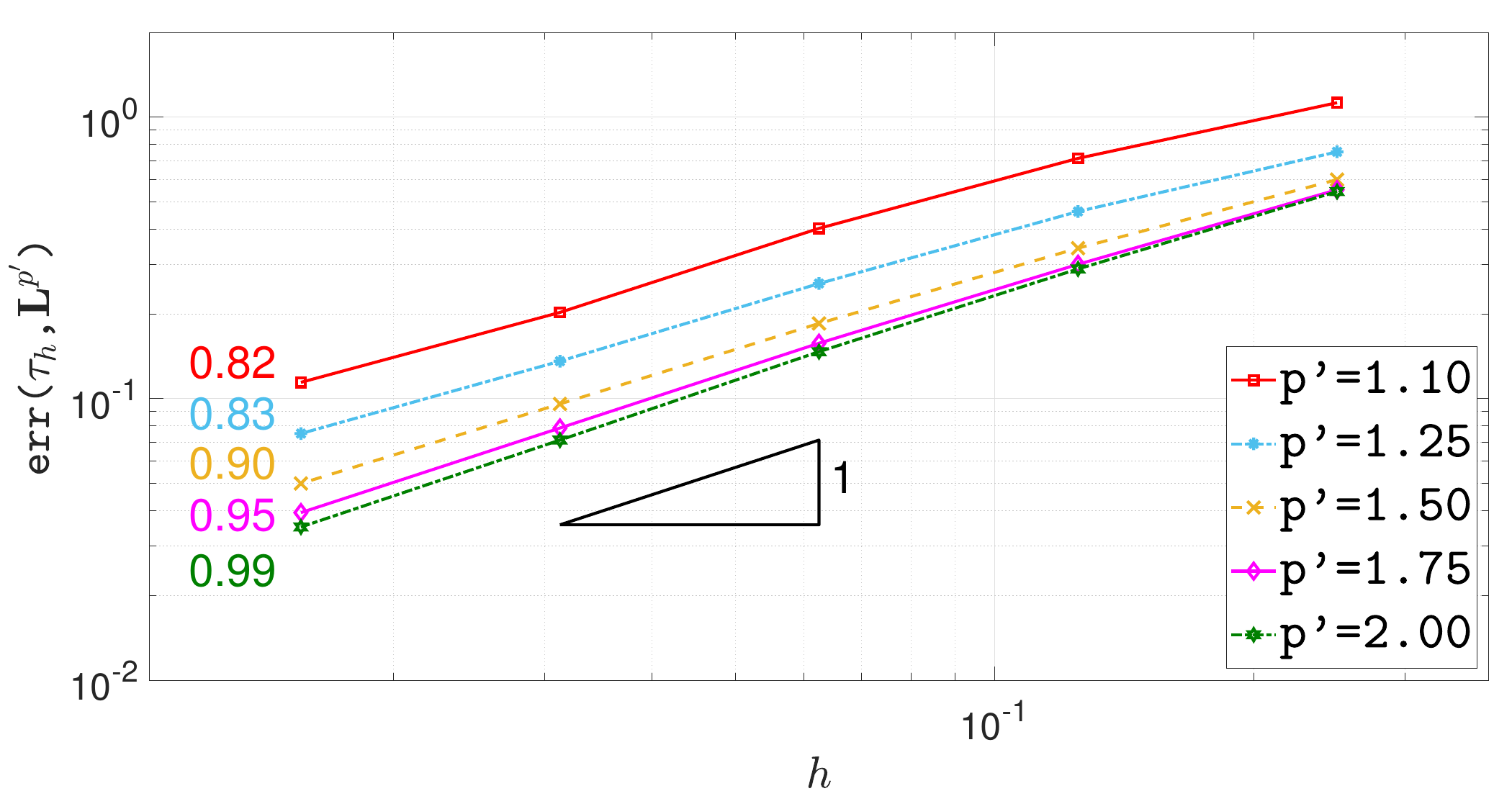}
    \end{subfigure}\qquad
    \begin{subfigure}{0.45\textwidth}
        \centering
        \includegraphics[width=\linewidth]{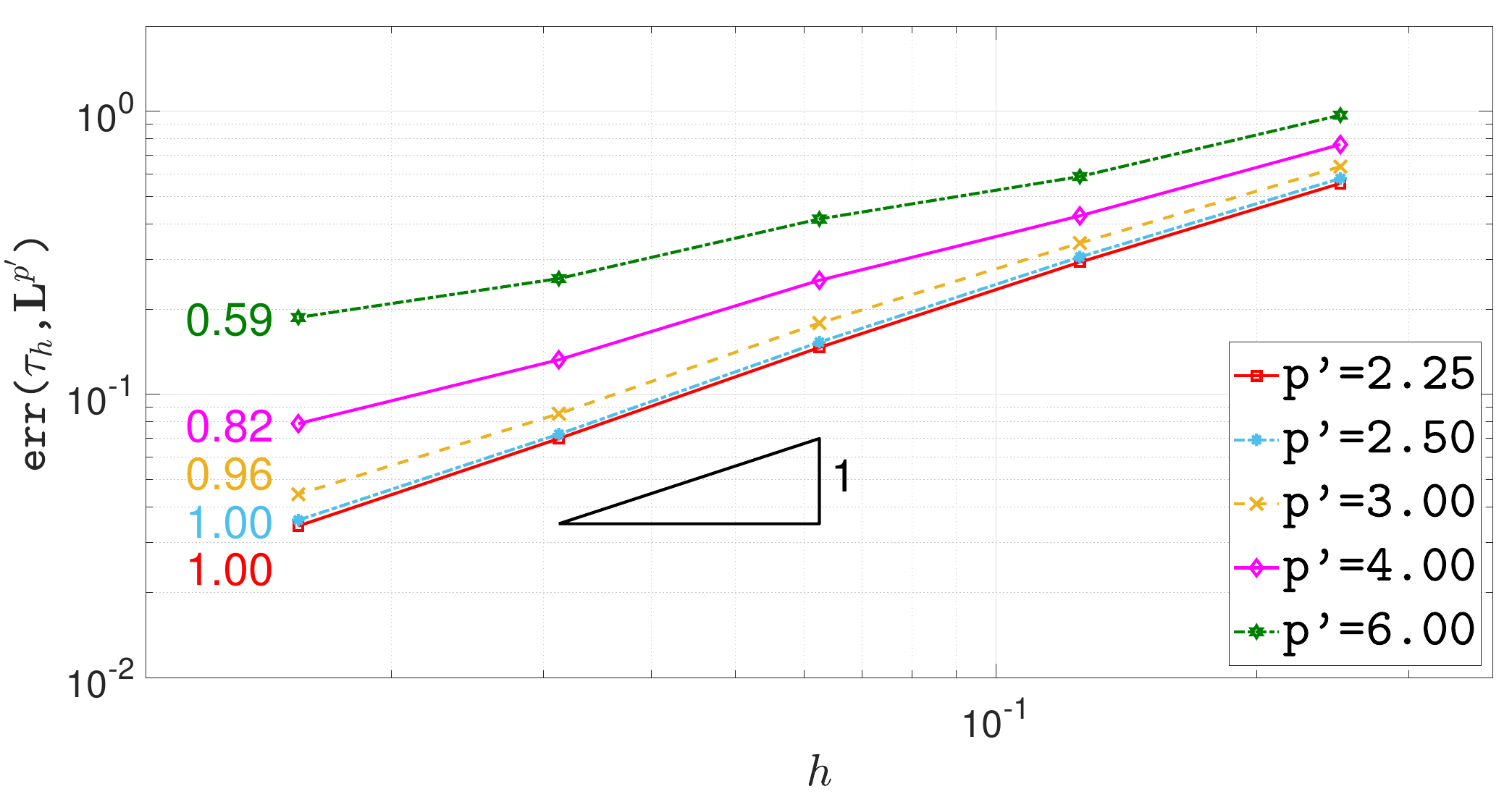}
    \end{subfigure}

    \vspace{0.25cm}

    \vspace{0.25cm}
    
    \begin{subfigure}{0.45\textwidth}
        \centering
        \includegraphics[width=\linewidth]{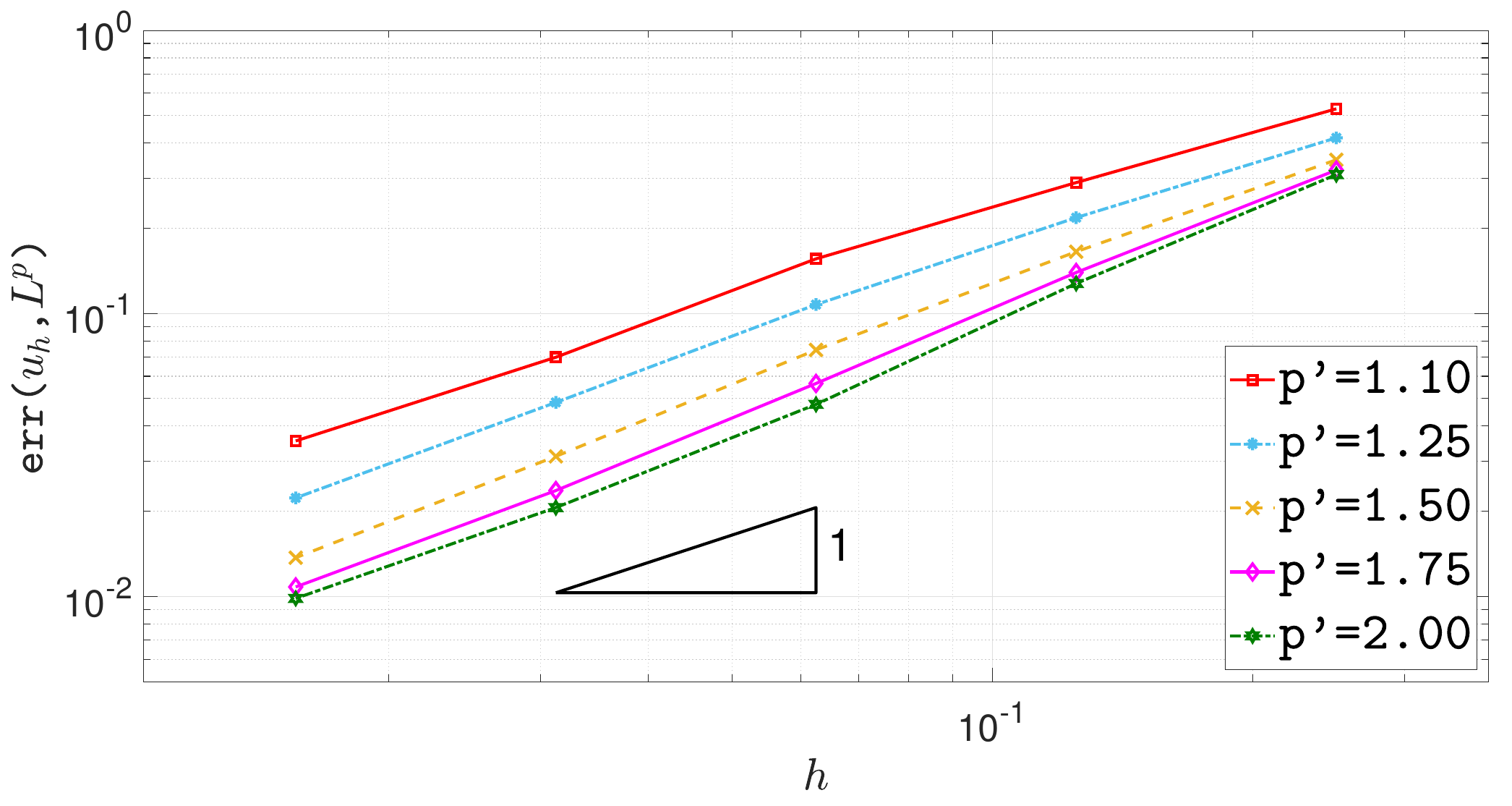}
        
    \end{subfigure}\qquad
    \begin{subfigure}{0.45\textwidth}
        \centering
        \includegraphics[width=\linewidth]{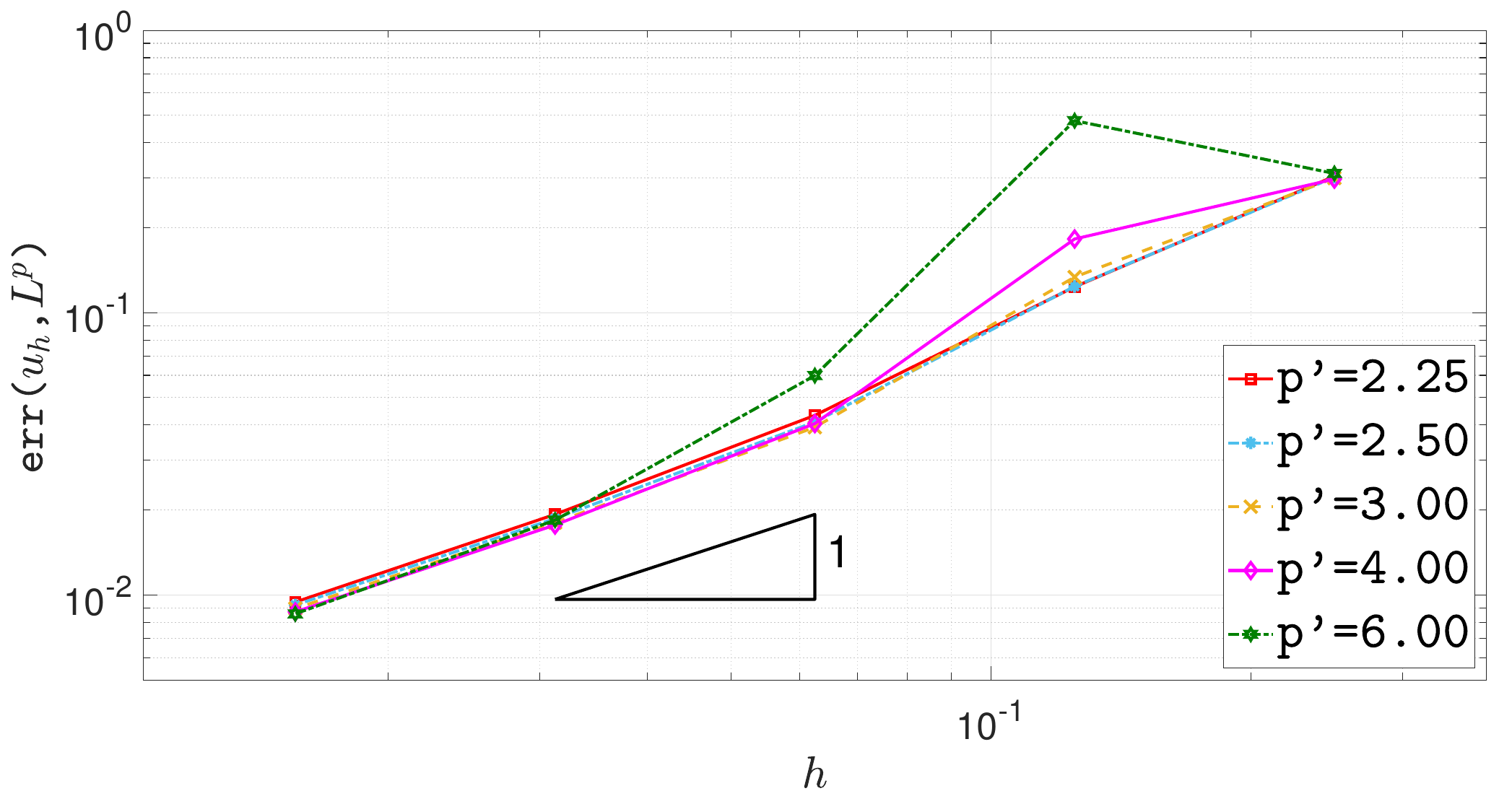}
    \end{subfigure}

    \vspace{0.25cm}

    \begin{subfigure}{0.45\textwidth}
        \centering
        \includegraphics[width=\linewidth]{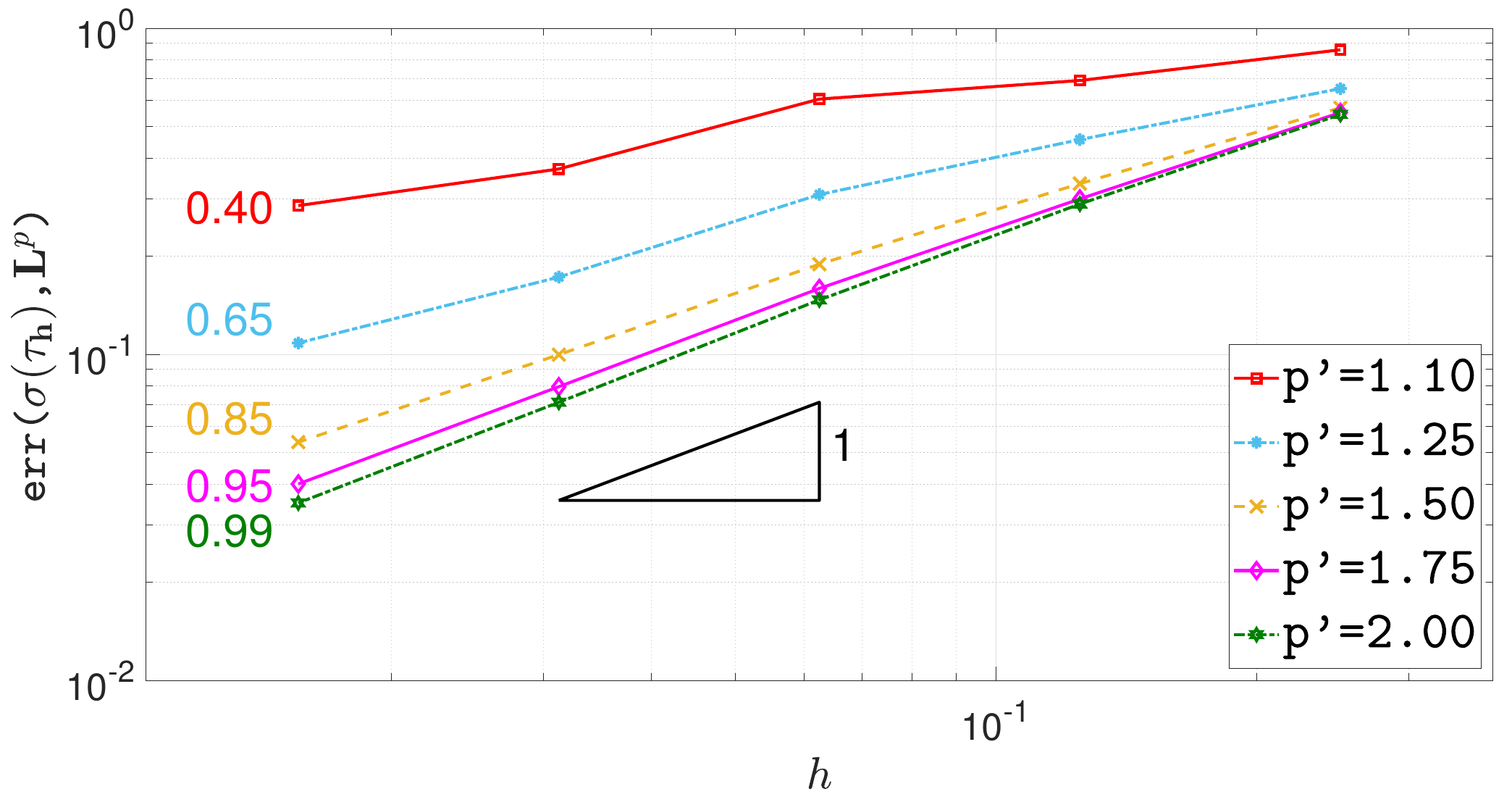}
    \end{subfigure}\qquad
    \begin{subfigure}{0.45\textwidth}
        \centering
        \includegraphics[width=\linewidth]{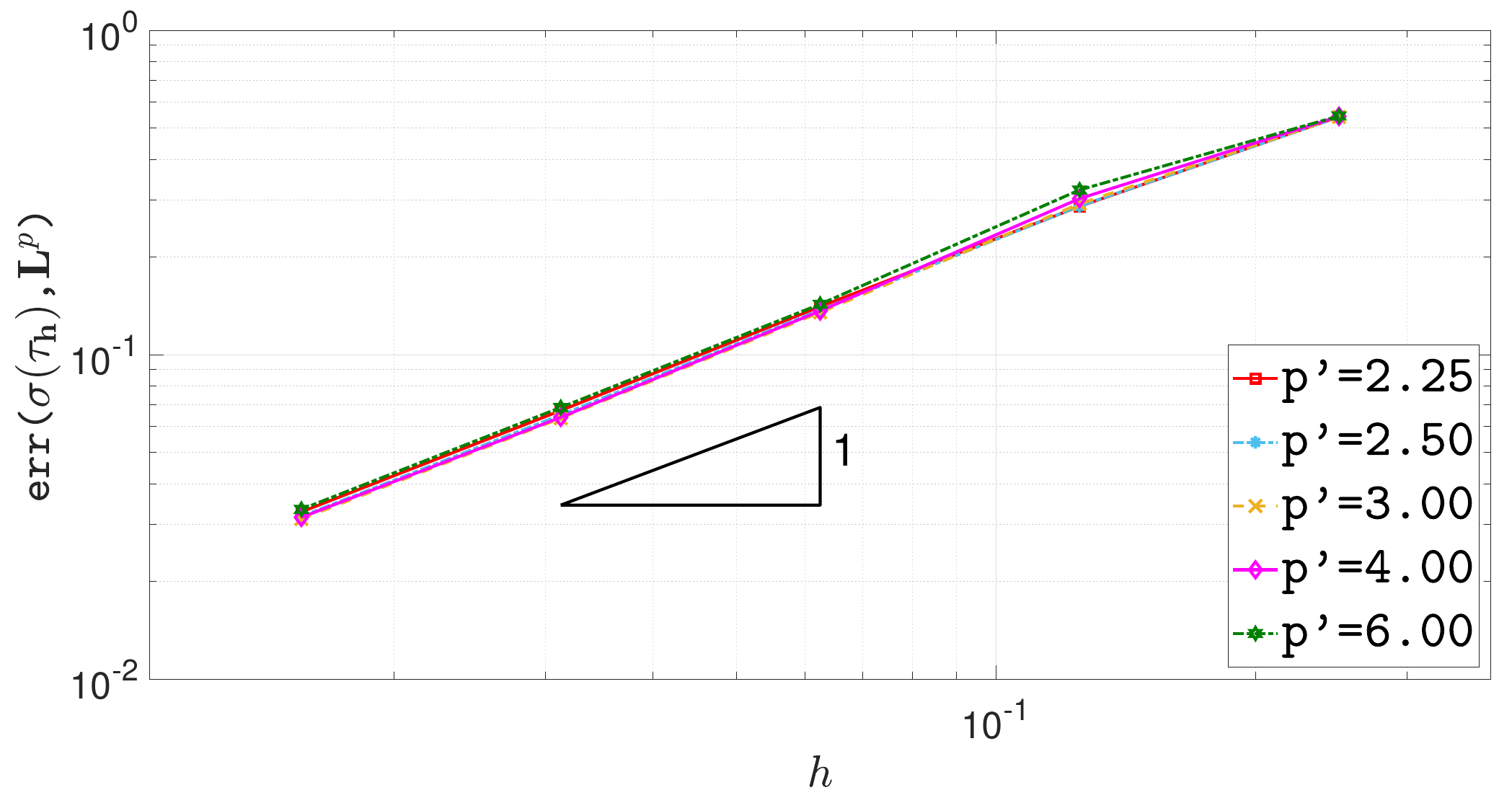}
    \end{subfigure}

   \caption{Test 1. Computed errors defined as in \eqref{eq:err_quant} as a function of the mesh size (loglog scale), for the mesh family \texttt{RANDOM}. Left panel: $1 <  \sobIndexConj \leq 2$, right panel: $\sobIndexConj > 2$.} 
    \label{fig:test1-V}
\end{figure}

\begin{figure}[htbp]
    \centering
    {\texttt{NON CONVEX MESHES}}
\\
    \begin{subfigure}{0.45\textwidth}
        \centering
        \includegraphics[width=\linewidth]{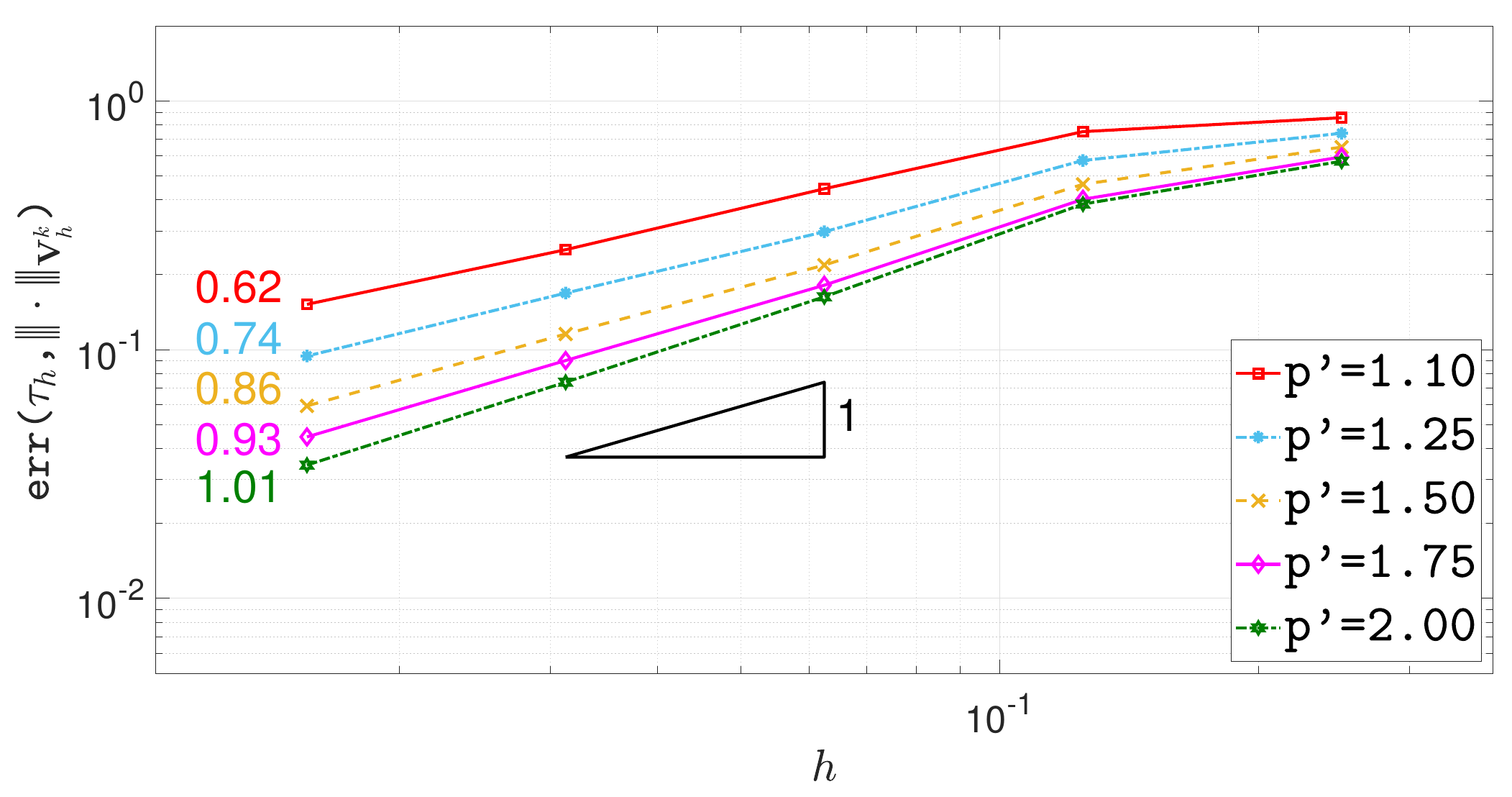}
    \end{subfigure}\qquad
    \begin{subfigure}{0.45\textwidth}
        \centering
        \includegraphics[width=\linewidth]{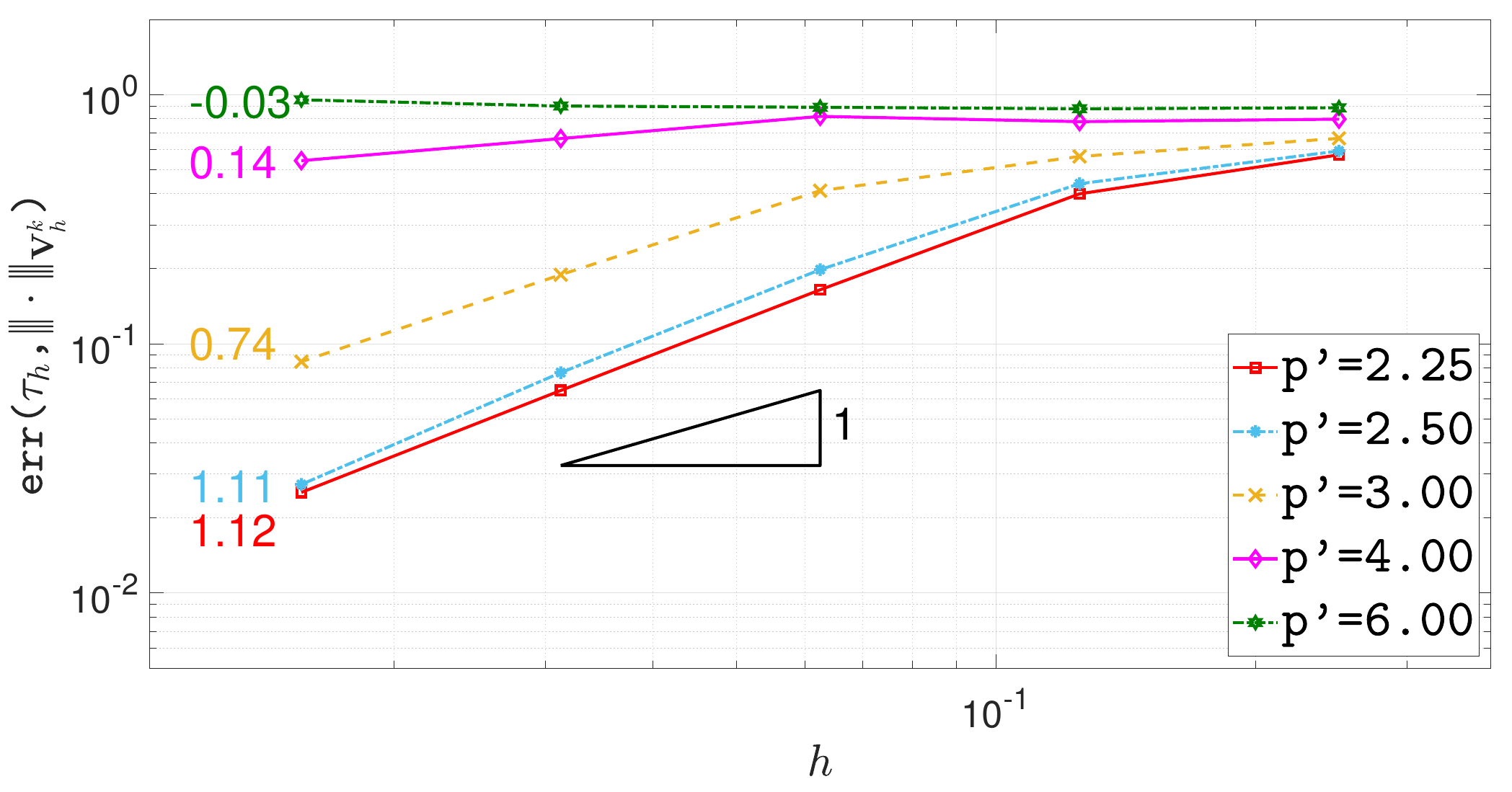}
    \end{subfigure}

    \vspace{0.25cm}
    
    \begin{subfigure}{0.45\textwidth}
        \centering
        \includegraphics[width=\linewidth]{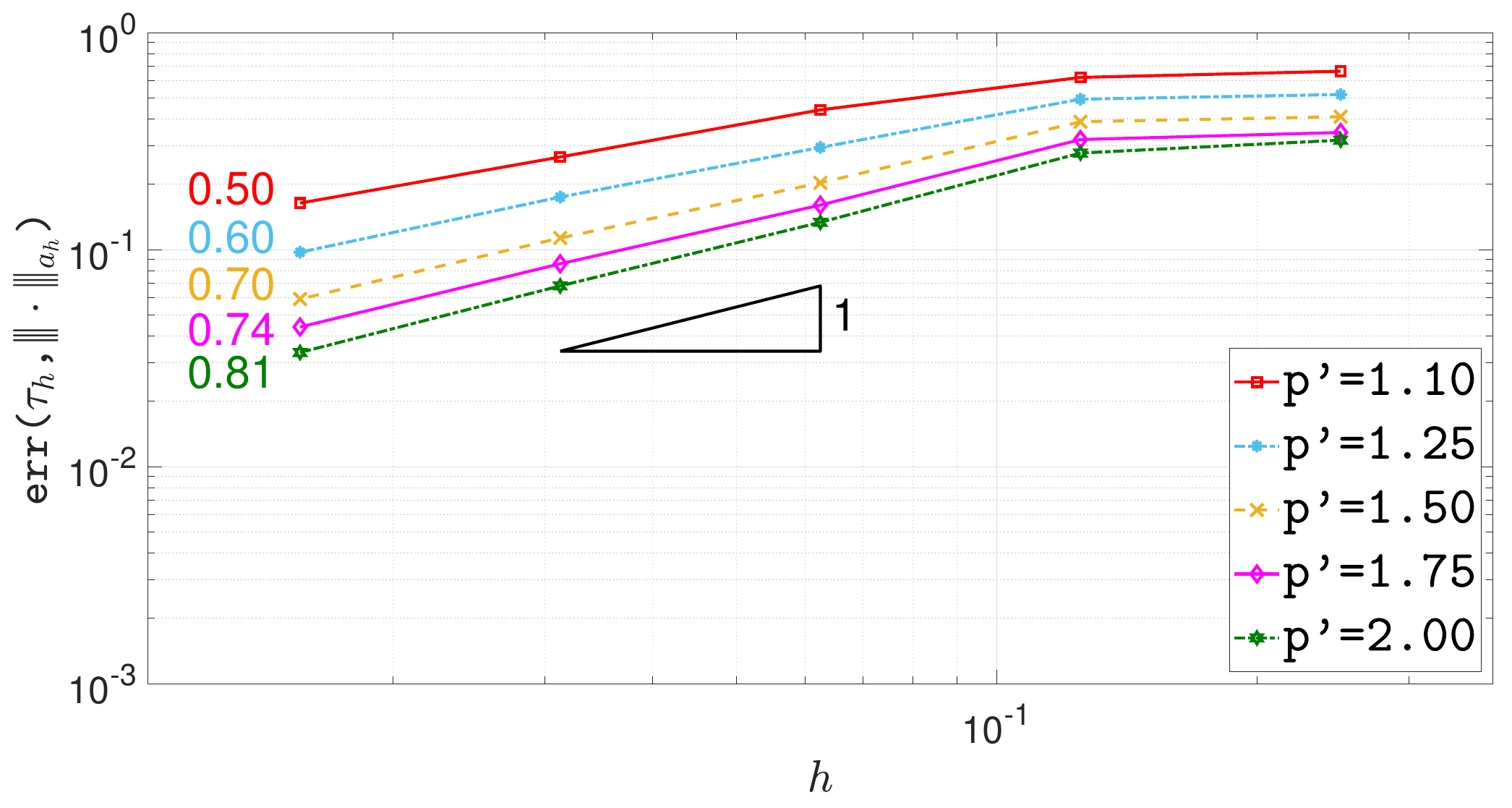}
    \end{subfigure}\qquad
    \begin{subfigure}{0.45\textwidth}
        \centering
        \includegraphics[width=\linewidth]{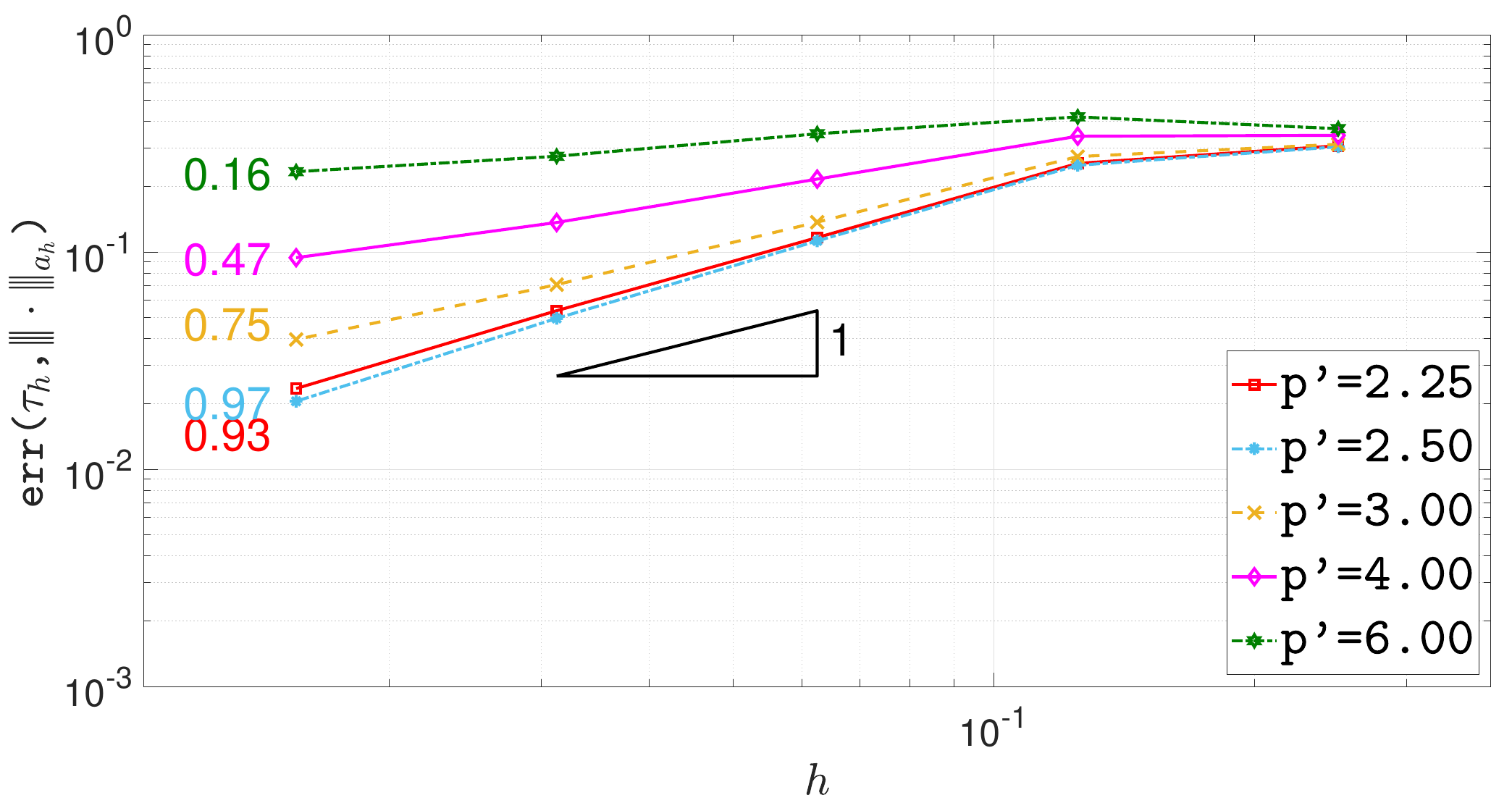}
    \end{subfigure}
    
    \vspace{0.25cm}
    
    \begin{subfigure}{0.45\textwidth}
        \centering
        \includegraphics[width=\linewidth]{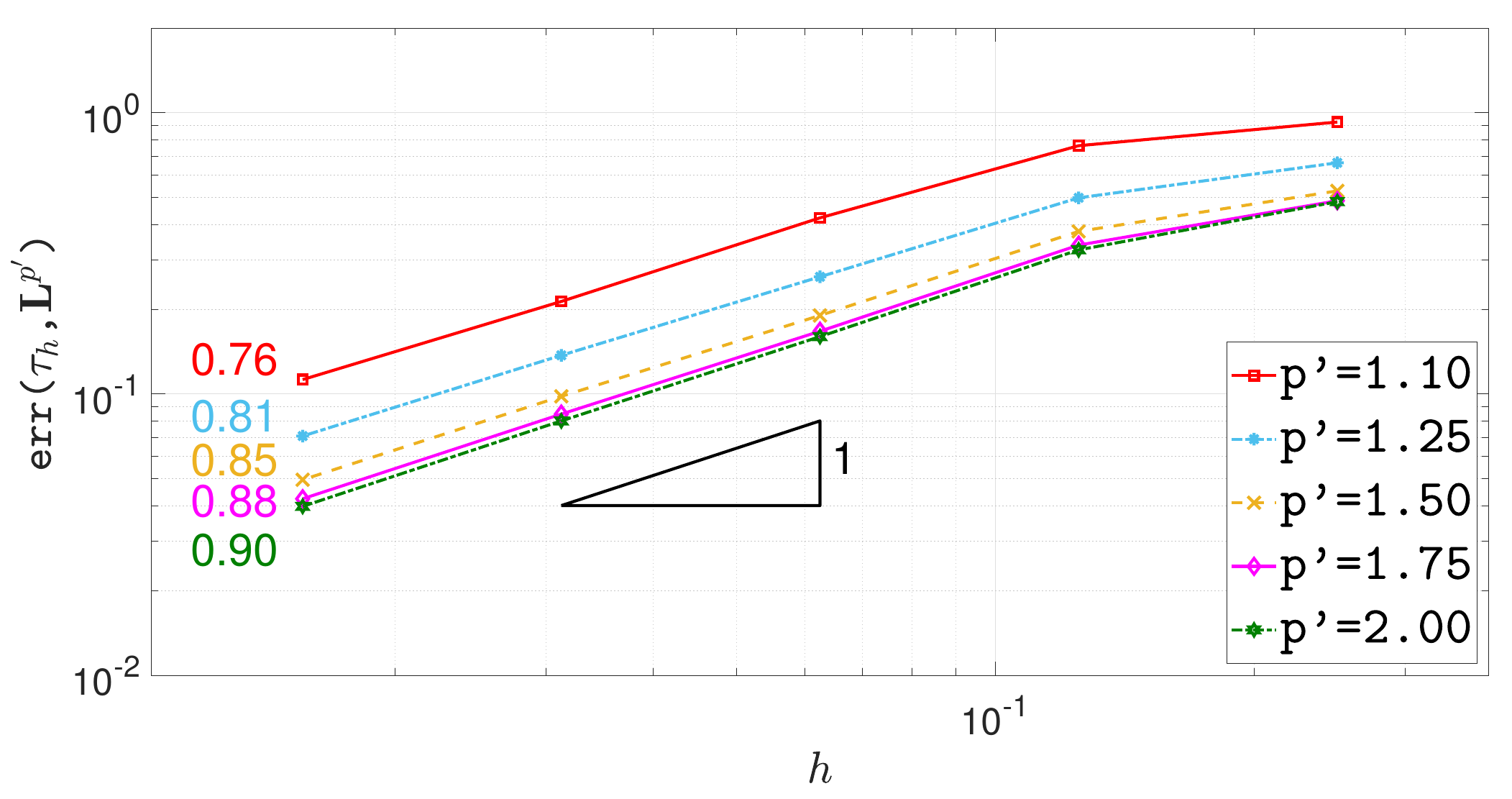}
    \end{subfigure}\qquad
    \begin{subfigure}{0.45\textwidth}
        \centering
        \includegraphics[width=\linewidth]{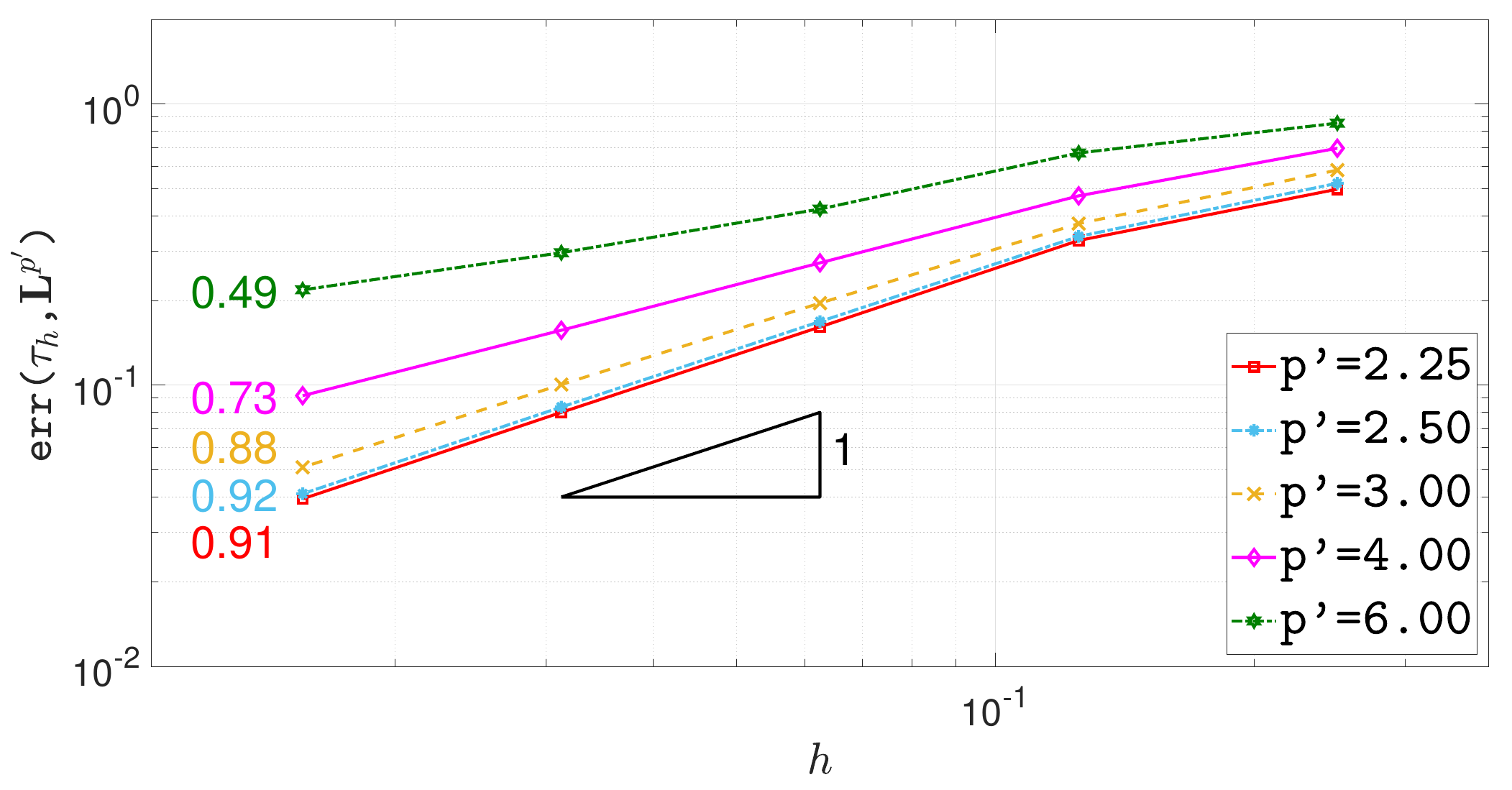}
    \end{subfigure}

    \vspace{0.25cm}    
    
    \begin{subfigure}{0.45\textwidth}
        \centering
        \includegraphics[width=\linewidth]{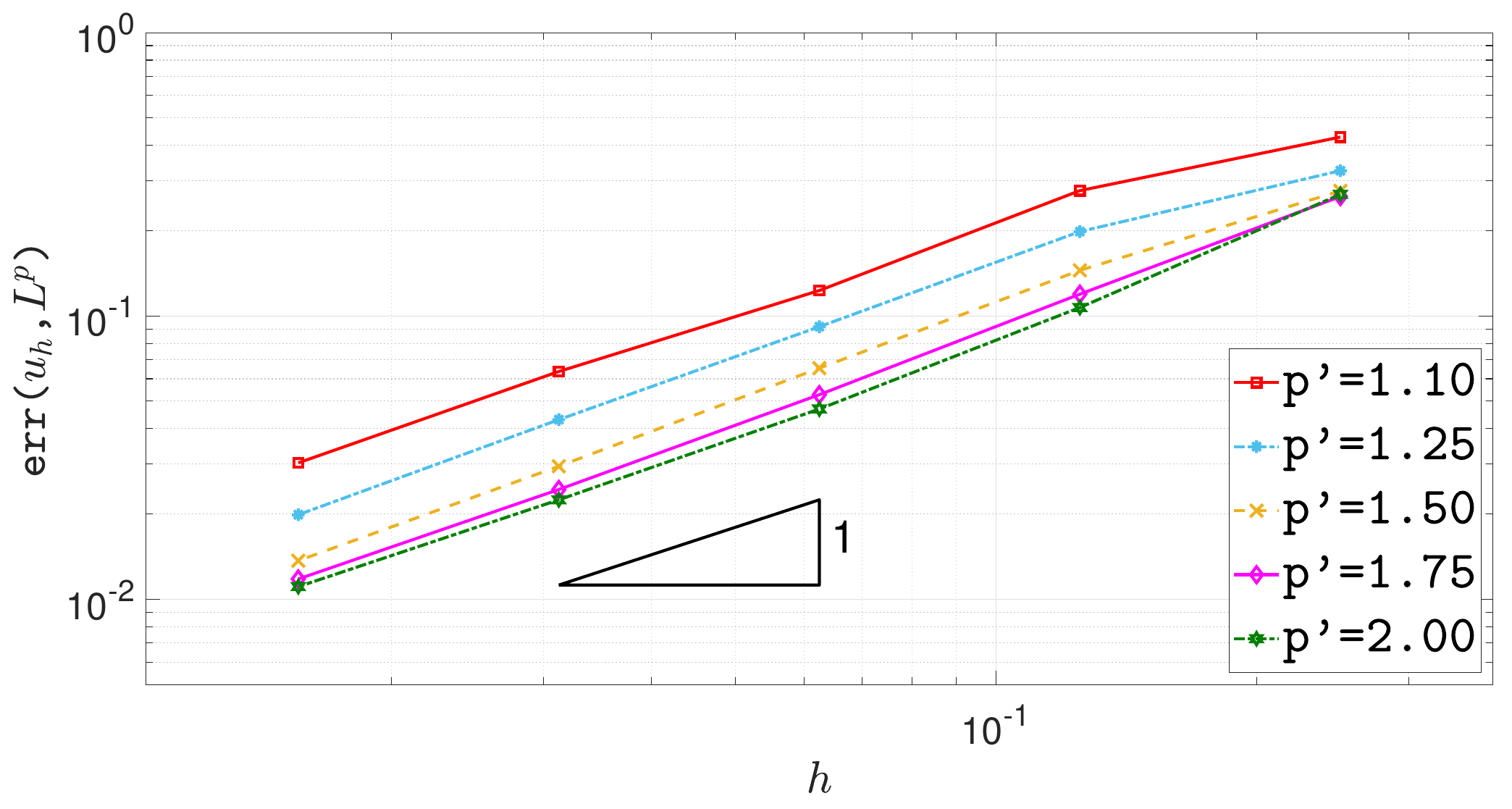}
        
    \end{subfigure}\qquad
    \begin{subfigure}{0.45\textwidth}
        \centering
        \includegraphics[width=\linewidth]{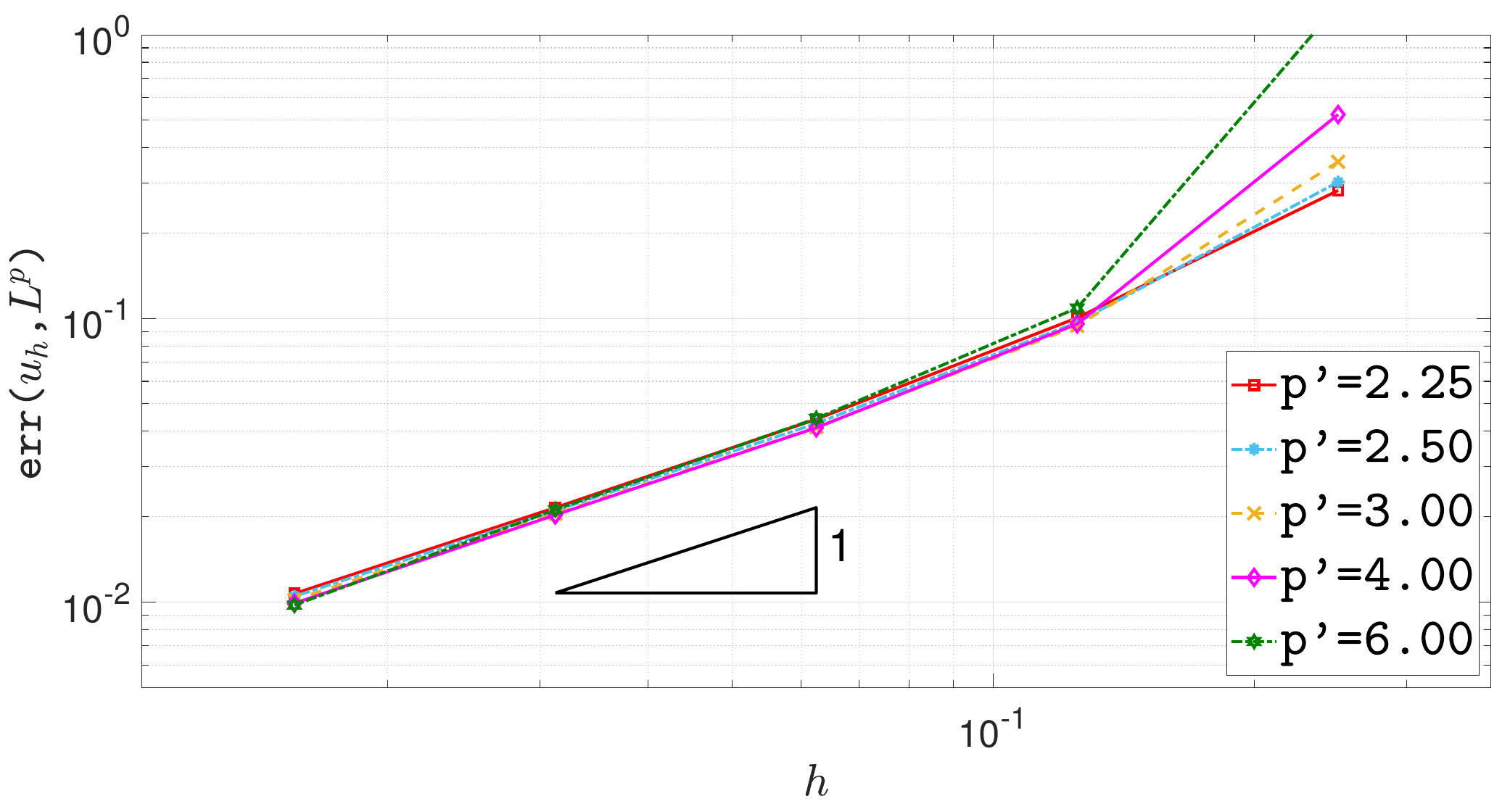}
    \end{subfigure}

    \vspace{0.25cm}

    \begin{subfigure}{0.45\textwidth}
        \centering
        \includegraphics[width=\linewidth]{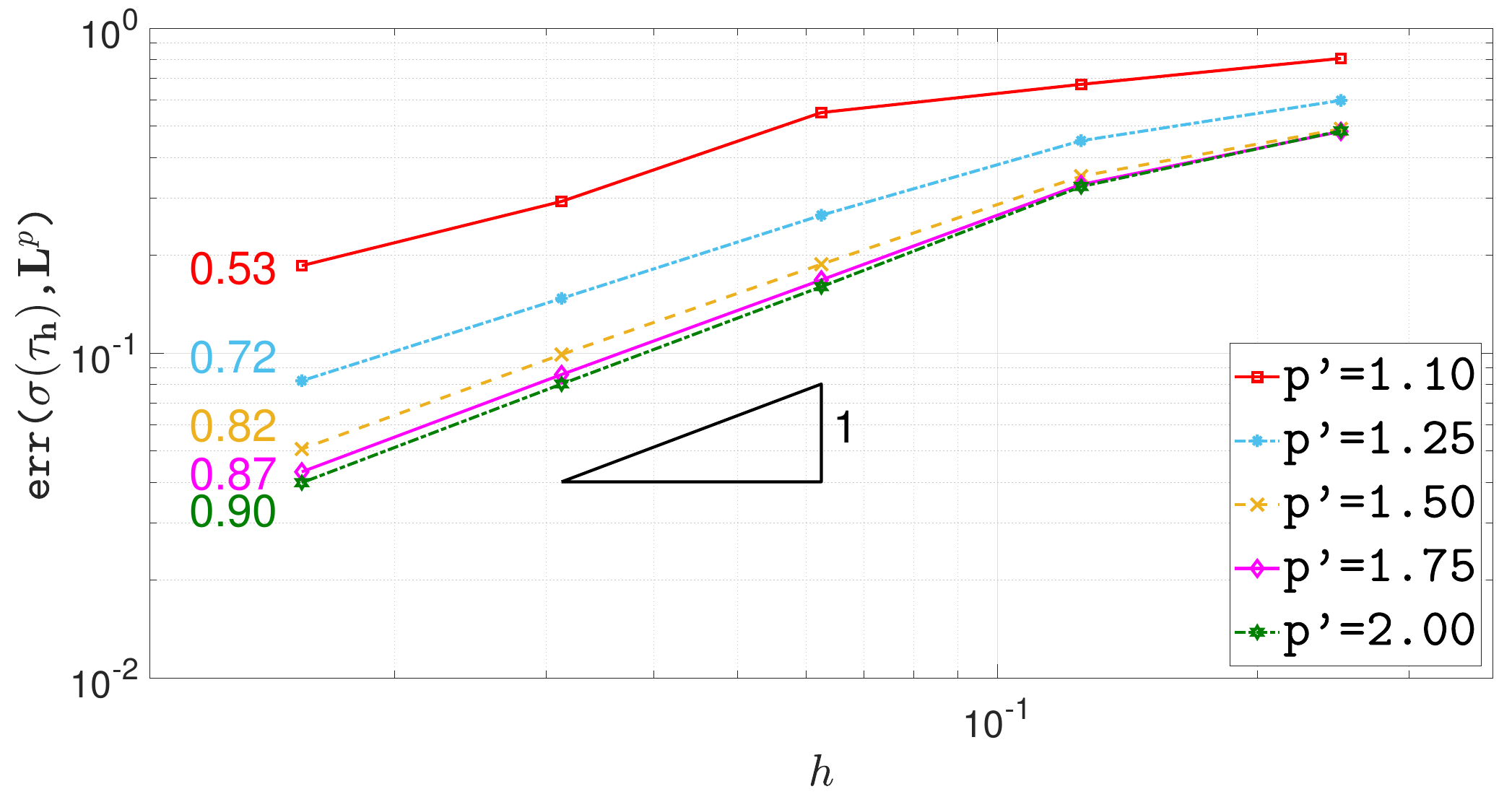}
    \end{subfigure}\qquad
    \begin{subfigure}{0.45\textwidth}
        \centering
        \includegraphics[width=\linewidth]{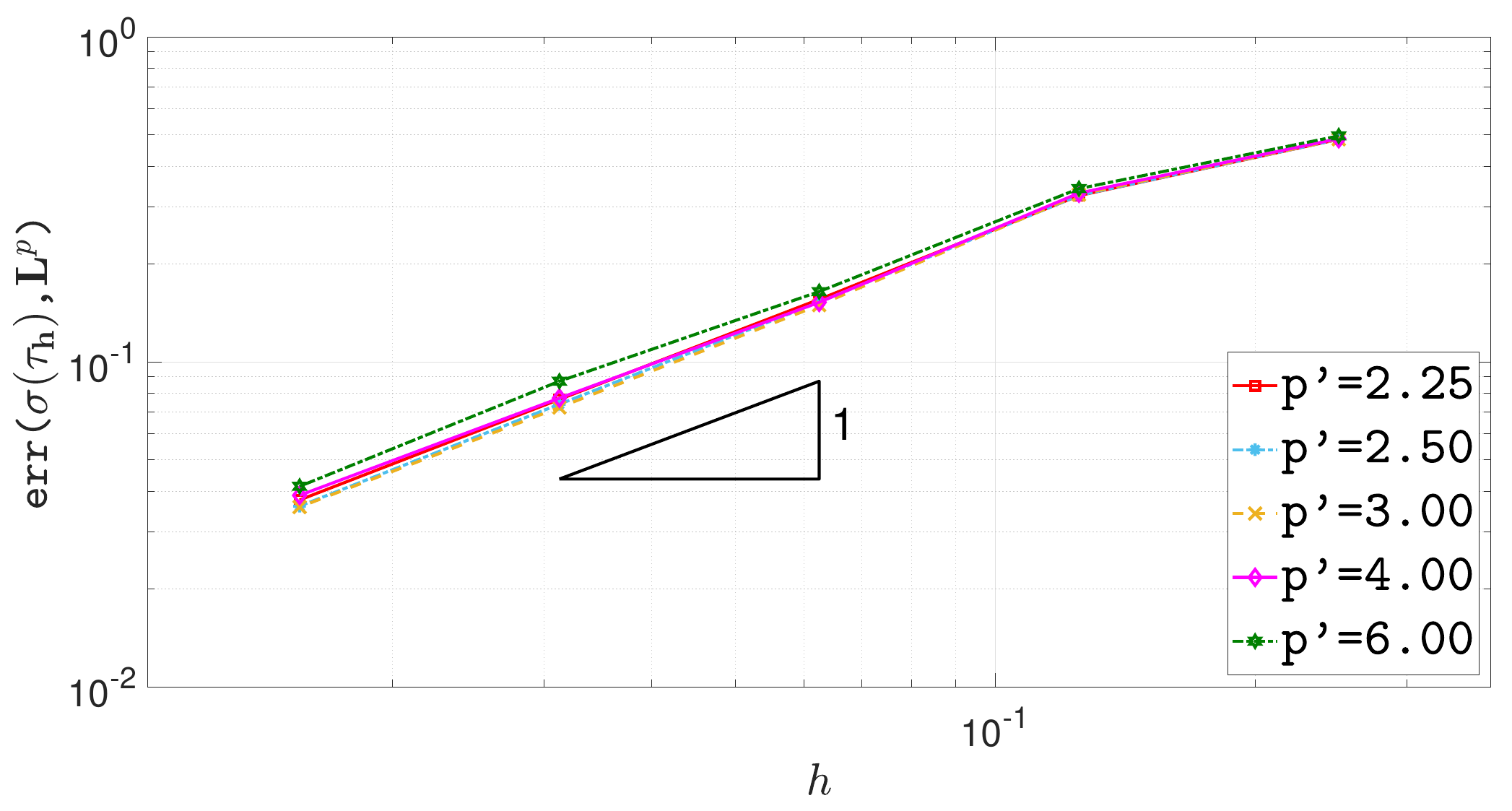}
    \end{subfigure}

   \caption{Test 1. Computed errors defined as in \eqref{eq:err_quant} as a function of the mesh size (loglog scale), for the mesh family \texttt{NON CONVEX}. Left panel: $1 <  \sobIndexConj \leq 2$, right panel: $\sobIndexConj > 2$.} 
    \label{fig:test1-W}
\end{figure}

In order to interpret the results illustrated in Fig.~\ref{fig:test1-Q} Fig.~\ref{fig:test1-V} and Fig.~\ref{fig:test1-W} with respect to the theoretical estimates established in Section~\ref{sec:a-priori}, Table~\ref{tab:test1-errl2} (resp. Table~\ref{tab:test1-errg2}) reports the expected convergence orders corresponding to the different sources of error derived in Corollaries~\ref{cor:total.error} and \ref{cor:total.error.scalar} (resp.  Corollaries \ref{cor:total.error.g2} and~\ref{cor:total.error.scalar.g2}).
We also report the rates for the interpolation errors  
$\texttt{err}(\vectProjF{0}\boldsymbol{\tau}, \boldsymbol{L}^{p'})$ and 
$\texttt{err}(\projL{0} u , L^{p})$.

\begin{table}[!htbp]
\centering
\begin{small}
\begin{tabular}{l|cc|cc|cc}
\toprule
\texttt{p'} & 
$\left(\mathcal{R}_h^{(1)}\right)^{\frac{1}{2}}$ &
$\left(\mathcal{R}_h^{(2)}\right)^{\frac{1}{2}}$ &
$\left(\mathcal{R}_h^{(1)}\right)^{\frac{1}{p}}$ &
$\left(\mathcal{R}_h^{(2)}\right)^{\frac{1}{p}}$ &
$\texttt{err}(\vectProjF{0}\boldsymbol{\tau}, \boldsymbol{L}^{p'})$ &
$\texttt{err}(\projL{0} u , L^{p})$ 
\\
\midrule
\texttt{1.10} &
\texttt{1.000} &
\texttt{0.550} &
\texttt{0.182} &
\texttt{0.100} &
\texttt{1.000} &
\texttt{1.000} 
\\
\texttt{1.25} &
\texttt{1.000} &
\texttt{0.625} &
\texttt{0.400} &
\texttt{0.250} &
\texttt{1.000} &
\texttt{1.000} 
\\
\texttt{1.50} &
\texttt{1.000} &
\texttt{0.750} &
\texttt{0.667} &
\texttt{0.500} &
\texttt{1.000} &
\texttt{1.000} 
\\
\texttt{1.75} &
\texttt{1.000} &
\texttt{0.875} &
\texttt{0.857} &
\texttt{0.750} &
\texttt{1.000} &
\texttt{1.000} 
\\
\texttt{2.00} &
\texttt{1.000} &
\texttt{1.000} &
\texttt{1.000} &
\texttt{1.000} &
\texttt{1.000} &
\texttt{1.000} 
\\
\bottomrule
\end{tabular}
\end{small}
\caption{Test 1. $1< p' \leq 2$. Expected orders of convergence for the terms appearing in the a priori error estimates in Corollaries~\ref{cor:total.error} and \ref{cor:total.error.scalar} and the interpolations errors. The first two residual columns refer to $\vectSol$, while the next two refer to $\scalarSol$.}
\label{tab:test1-errl2}
\end{table}
%
\begin{table}[!htbp]
\centering
\begin{small}
\begin{tabular}{l|cc|cc|cc}
\toprule
\texttt{p'} & 
$\left(\tilde{\mathcal{R}}_h^{(1)}\right)^{\frac{1}{\sobIndexConj}}$ &
$\left(\tilde{\mathcal{R}}_h^{(2)}\right)^{\frac{1}{\sobIndexConj}}$ &
$\left(\tilde{\mathcal{R}}_h^{(1)}\right)^{\frac{1}{\sobIndexConj}}$ &
$\left(\tilde{\mathcal{R}}_h^{(2)}\right)^{\frac{1}{\sobIndexConj}}$ &
$\texttt{err}(\vectProjF{0}\boldsymbol{\tau}, \boldsymbol{L}^{p'})$ &
$\texttt{err}(\projL{0} u , L^{p})$ 
\\
\midrule
\texttt{2.25} &
\texttt{0.800} &
\texttt{0.889} &
\texttt{0.800} &
\texttt{0.889} &
\texttt{1.000} &
\texttt{1.000} 
\\
\texttt{2.50} &
\texttt{0.667} &
\texttt{0.800} &
\texttt{0.667} &
\texttt{0.800} &
\texttt{1.000} &
\texttt{1.000} 
\\
\texttt{3.00} &
\texttt{0.500} &
\texttt{0.667} &
\texttt{0.500} &
\texttt{0.667} &
\texttt{1.000} &
\texttt{1.000} 
\\
\texttt{4.00} &
\texttt{0.333} &
\texttt{0.500} &
\texttt{0.333} &
\texttt{0.500} &
\texttt{1.000} &
\texttt{1.000} 
\\
\texttt{6.00} &
\texttt{0.200} &
\texttt{0.333} &
\texttt{0.200} &
\texttt{0.333} &
\texttt{1.000} &
\texttt{1.000} 
\\
\bottomrule
\end{tabular}
\end{small}
\caption{Test 1. $p' > 2$. Expected orders of convergence for the terms appearing in the a priori error estimates in Corollaries \ref{cor:total.error.g2} and~\ref{cor:total.error.scalar.g2} and the interpolations errors. The first two residual columns refer to $\vectSol$, while the next two refer to $\scalarSol$.}
\label{tab:test1-errg2}
\end{table}

We now discuss the observed convergence behavior of the error quantities introduced in \eqref{eq:err_quant}.
First of all, it can be observed from the plots that
$\texttt{err}(u_h,L^{p})$ is dominated by the interpolation error.

Let us now analyze the flux errors in the discrete norms.
The errors $\texttt{err}(\boldsymbol{\tau}_h,\discFullVectNorm{\cdot})$
exhibit convergence rates that are higher than those predicted by
Corollaries~\ref{cor:total.error} and \ref{cor:total.error.g2},
except for the case $p'=\texttt{6}$, for which the plots suggest that the
asymptotic regime has not yet been reached.
In contrast, the errors
$\texttt{err}(\boldsymbol{\tau}_h,\discDivFreeNorm{\cdot})$
are in closer agreement with the theoretical predictions.

The errors
$\texttt{err}(\boldsymbol{\tau}_h,\boldsymbol{L}^{p'})$
exhibit convergence rates lying between those of
$\texttt{err}(\boldsymbol{\tau}_h,\discDivFreeNorm{\cdot})$
and the interpolation error rate.

The errors
$\texttt{err}(\fluxFunction_{p'}(\boldsymbol{\tau}_h),\boldsymbol{L}^{p})$
exhibit convergence rates lying between those of
$\texttt{err}(\boldsymbol{\tau}_h,\discDivFreeNorm{\cdot})$
and the optimal first-order rate for $1<p'\leq2$,
whereas they converge linearly for $p'>2$.

We finally observe that no significant differences are observed between the results obtained on meshes with convex elements and those obtained on meshes with non convex elements.

Furthermore, we report the number of iterations required by the relaxed Ka\v{c}anov method described in Subsection~\ref{sub:kacanov}. 
As expected, in the regime $1<p'\leq 2$, the number of iterations increases as $p'$ decreases. In contrast, for $p'>2$, the iteration count remains essentially stable with respect to the value of $p'$. 
We also observe that the number of iterations is insensitive to both the mesh size and the mesh geometry.

\begin{table}[!ht]
\centering
\begin{small}
\begin{tabular}{l|r|ccccc|ccccc}
\toprule
&
& \multicolumn{10}{c}{\texttt{$p'$}} 
\\
   \texttt{MESHES}
&  \texttt{1/h}
& {\texttt{1.10}}
& {\texttt{1.25}}
& {\texttt{1.50}}
& {\texttt{1.75}}
& {\texttt{2.00}}
& {\texttt{2.25}}
& {\texttt{2.50}}
& {\texttt{3.00}}
& {\texttt{4.00}}
& {\texttt{6.00}}
\\
\midrule
\multirow{5}{*}{\texttt{QUAD.}} 
& {\texttt{4}}
& {\texttt{190}}
& {\texttt{112}}
& {\texttt{71}}
& {\texttt{50}}
& {\texttt{1}}
& {\texttt{32}}
& {\texttt{32}}
& {\texttt{33}}
& {\texttt{34}}
& {\texttt{33}}
\\
& {\texttt{8}}
& {\texttt{152}}
& {\texttt{120}}
& {\texttt{73}}
& {\texttt{48}}
& {\texttt{1}}
& {\texttt{31}}
& {\texttt{31}}
& {\texttt{31}}
& {\texttt{32}}
& {\texttt{32}}
\\
& {\texttt{16}}
& {\texttt{171}}
& {\texttt{123}}
& {\texttt{73}}
& {\texttt{46}}
& {\texttt{1}}
& {\texttt{30}}
& {\texttt{30}}
& {\texttt{29}}
& {\texttt{30}}
& {\texttt{31}}
\\
& {\texttt{32}}
& {\texttt{181}}
& {\texttt{123}}
& {\texttt{72}}
& {\texttt{44}}
& {\texttt{1}}
& {\texttt{29}}
& {\texttt{29}}
& {\texttt{29}}
& {\texttt{29}}
& {\texttt{30}}
\\
& {\texttt{64}}
& {\texttt{189}}
& {\texttt{125}}
& {\texttt{69}}
& {\texttt{41}}
& {\texttt{1}}
& {\texttt{28}}
& {\texttt{29}}
& {\texttt{29}}
& {\texttt{29}}
& {\texttt{30}}
\\
\midrule
\multirow{5}{*}{\texttt{RANDOM}} 
& {\texttt{4}}
& {\texttt{193}}
& {\texttt{111}}
& {\texttt{67}}
& {\texttt{49}}
& {\texttt{1}}
& {\texttt{33}}
& {\texttt{34}}
& {\texttt{34}}
& {\texttt{35}}
& {\texttt{37}}
\\
& {\texttt{16}}
& {\texttt{132}}
& {\texttt{109}}
& {\texttt{70}}
& {\texttt{47}}
& {\texttt{1}}
& {\texttt{31}}
& {\texttt{31}}
& {\texttt{31}}
& {\texttt{32}}
& {\texttt{34}}
\\
& {\texttt{32}}
& {\texttt{141}}
& {\texttt{106}}
& {\texttt{71}}
& {\texttt{46}}
& {\texttt{1}}
& {\texttt{30}}
& {\texttt{30}}
& {\texttt{30}}
& {\texttt{30}}
& {\texttt{35}}
\\
& {\texttt{64}}
& {\texttt{156}}
& {\texttt{112}}
& {\texttt{70}}
& {\texttt{44}}
& {\texttt{1}}
& {\texttt{29}}
& {\texttt{29}}
& {\texttt{29}}
& {\texttt{29}}
& {\texttt{34}}
\\
& {\texttt{128}}
& {\texttt{172}}
& {\texttt{120}}
& {\texttt{69}}
& {\texttt{41}}
& {\texttt{1}}
& {\texttt{28}}
& {\texttt{29}}
& {\texttt{29}}
& {\texttt{29}}
& {\texttt{33}}
\\
\midrule
\multirow{5}{*}{\texttt{NON CONV.}} 
& {\texttt{4}}
& {\texttt{168}}
& {\texttt{113}}
& {\texttt{72}}
& {\texttt{49}}
& {\texttt{1}}
& {\texttt{31}}
& {\texttt{30}}
& {\texttt{30}}
& {\texttt{32}}
& {\texttt{35}}
\\
& {\texttt{16}}
& {\texttt{195}}
& {\texttt{109}}
& {\texttt{73}}
& {\texttt{49}}
& {\texttt{1}}
& {\texttt{32}}
& {\texttt{31}}
& {\texttt{32}}
& {\texttt{33}}
& {\texttt{39}}
\\
& {\texttt{32}}
& {\texttt{147}}
& {\texttt{121}}
& {\texttt{74}}
& {\texttt{47}}
& {\texttt{1}}
& {\texttt{31}}
& {\texttt{31}}
& {\texttt{30}}
& {\texttt{32}}
& {\texttt{42}}
\\
& {\texttt{64}}
& {\texttt{172}}
& {\texttt{117}}
& {\texttt{72}}
& {\texttt{45}}
& {\texttt{1}}
& {\texttt{30}}
& {\texttt{30}}
& {\texttt{29}}
& {\texttt{30}}
& {\texttt{44}}
\\
& {\texttt{128}}
& {\texttt{197}}
& {\texttt{124}}
& {\texttt{71}}
& {\texttt{42}}
& {\texttt{1}}
& {\texttt{28}}
& {\texttt{29}}
& {\texttt{29}}
& {\texttt{30}}
& {\texttt{47}}
\\
\bottomrule
\end{tabular}
\end{small}
\caption{Test 1. 
Number of iterations  of the relaxed Ka\v{c}anov method for the considered mesh families and the values of $p'$ reported in \eqref{eq:test1-pp}.}
\label{tab:test1-iter}
\end{table}

\subsection{Test 2. Benchmark problem from torsional creep}
\label{sub:test2}

The aim of this test is to assess the performance of the proposed VEM scheme  \eqref{eq:discrete.problem} on the benchmark problem presented in \cite[Subsection 4.4.2]{COCKBURN.SHEN:2016} which models torsional creep. We consider the model problem \eqref{eq:continuous.problem.dual} posed on the domain $\Omega=(0,1)^2$ with piecewise constant loading term 
\[
f(x_1, x_2) = \left \{
\begin{aligned}
&8 &\qquad & \text{if $(x_1, x_2) \in [0.25, 0.75]^2$,}
\\
&0 &\qquad & \text{otherwise,}
\end{aligned}
\right.
\]
and different values of the Sobolev exponent $p' >1$.
It can be shown (see \cite{COCKBURN.SHEN:2016} and the references therein) that, as $p' \to 1$, the scalar solution $\scalarSol$ converges to the distance function from the boundary. Consequently, in the present setting, the solution is expected to develop a pyramid like profile
whose apex lies at $(0.5, 0.5)$.
On the other hand, as $p' \to \infty$, the solution  approaches a limit profile with increasingly steep gradients and, under suitable conditions, may develop discontinuities. 

We assess the performance of the proposed VEM scheme \eqref{eq:discrete.problem} in both regimes.
In order to discretize the problem, we consider a Cartesian tessellation of the domain with $64 \times 64$ elements. We remark that the adopted mesh is aligned with the discontinuity of $f$.

Fig.  \ref{fig:test2} displays the discrete scalar solutions $\discScalarSol$ obtained for different values of $p'$.
As expected, for smaller values of $p'$, the numerical solution exhibits the characteristic pyramid like profile. 
On the other hand, for larger values of $p'$ the discrete solution  develops steep gradients and exhibit discontinuous profile.
We finally note that the obtained plots are in good agreement with those reported in in \cite[Subsection 4.4.2]{COCKBURN.SHEN:2016}.

\begin{figure}[htbp]
    \begin{subfigure}{0.45\textwidth}
        \centering
        \begin{overpic}[width=\linewidth]{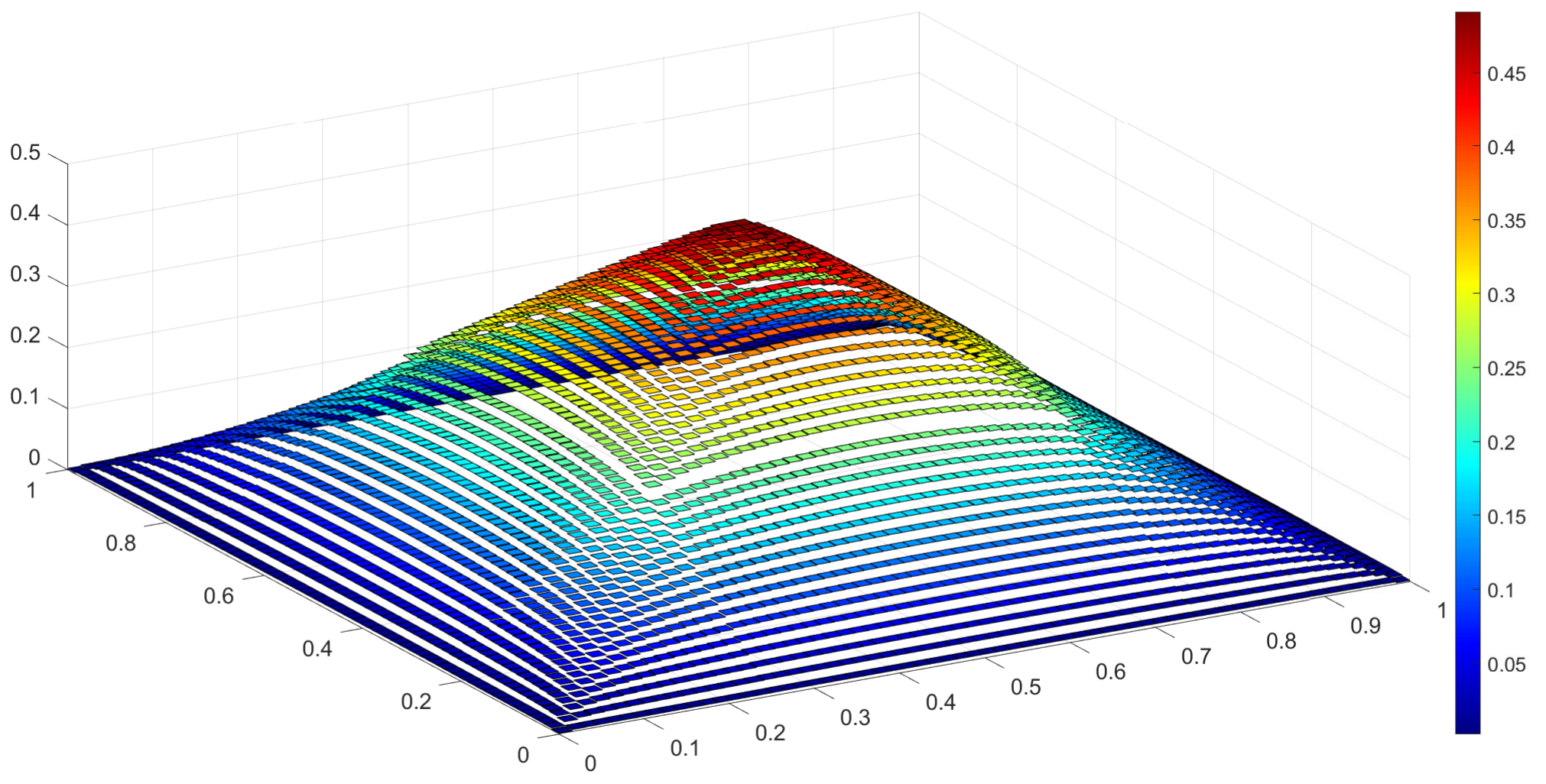}
        \put(45,-10){{{p'=\texttt{1.10}}}}
        \end{overpic}
    \end{subfigure}\qquad
    \begin{subfigure}{0.45\textwidth}
    \begin{overpic}[width=\linewidth]{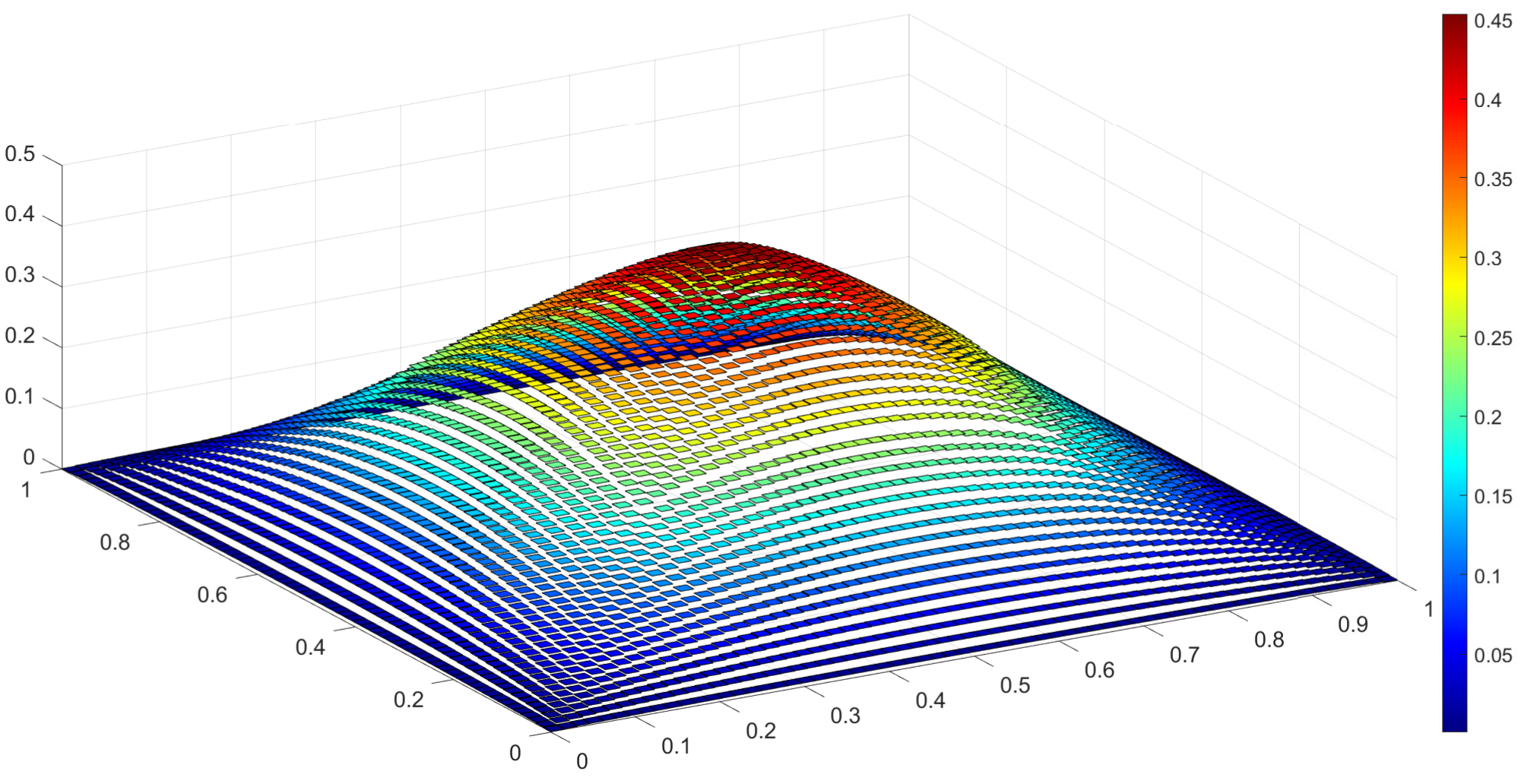}
        \put(45,-10){{{p'=\texttt{4/3}}}}
        \end{overpic}
    \end{subfigure}

    \vspace{1.5cm}
    
    \begin{subfigure}{0.45\textwidth}
        \centering
        \begin{overpic}[width=\linewidth]{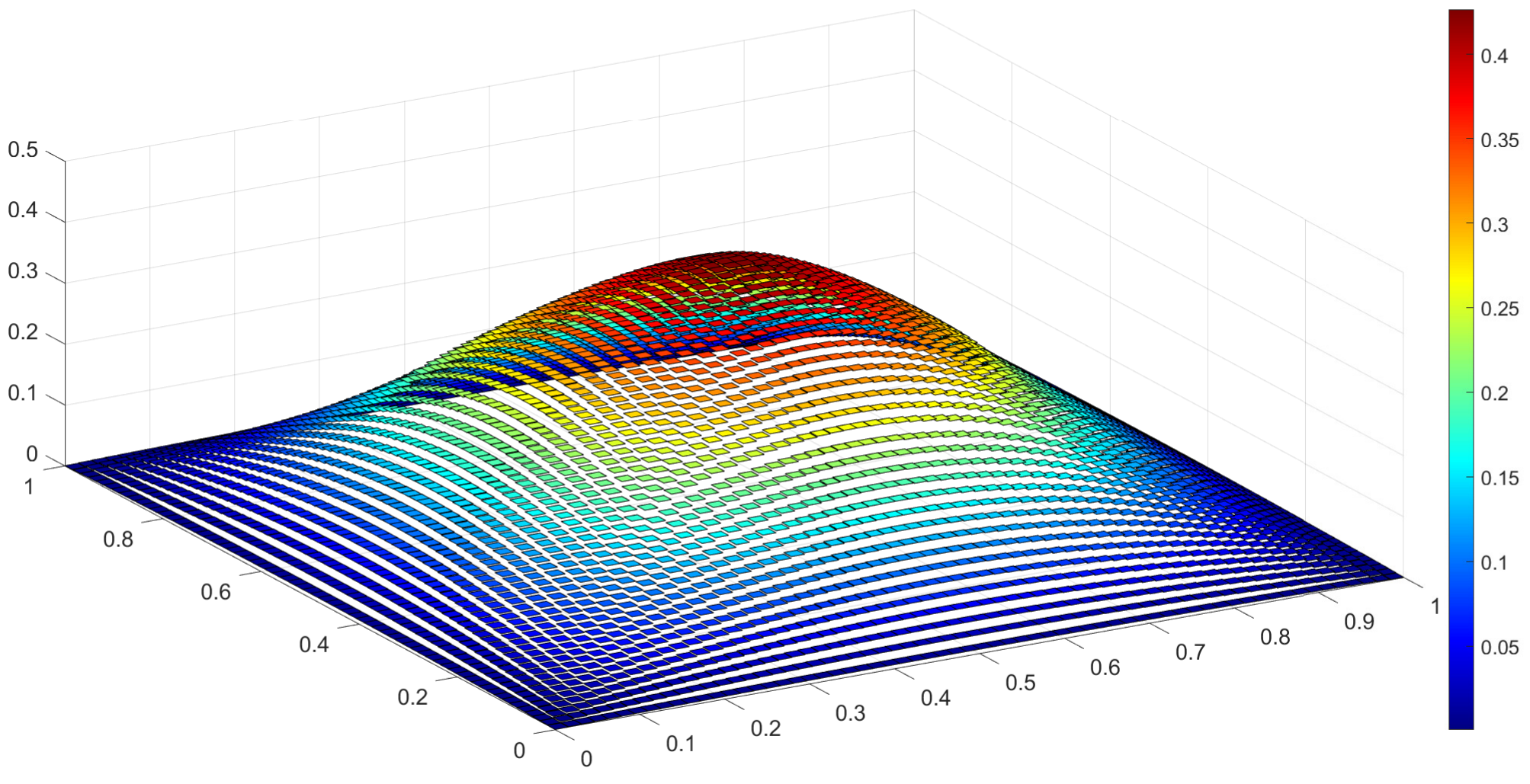}
        \put(45,-10){{{p'=\texttt{1.50}}}}
        \end{overpic}
    \end{subfigure}\qquad
    \begin{subfigure}{0.45\textwidth}
    \begin{overpic}[width=\linewidth]{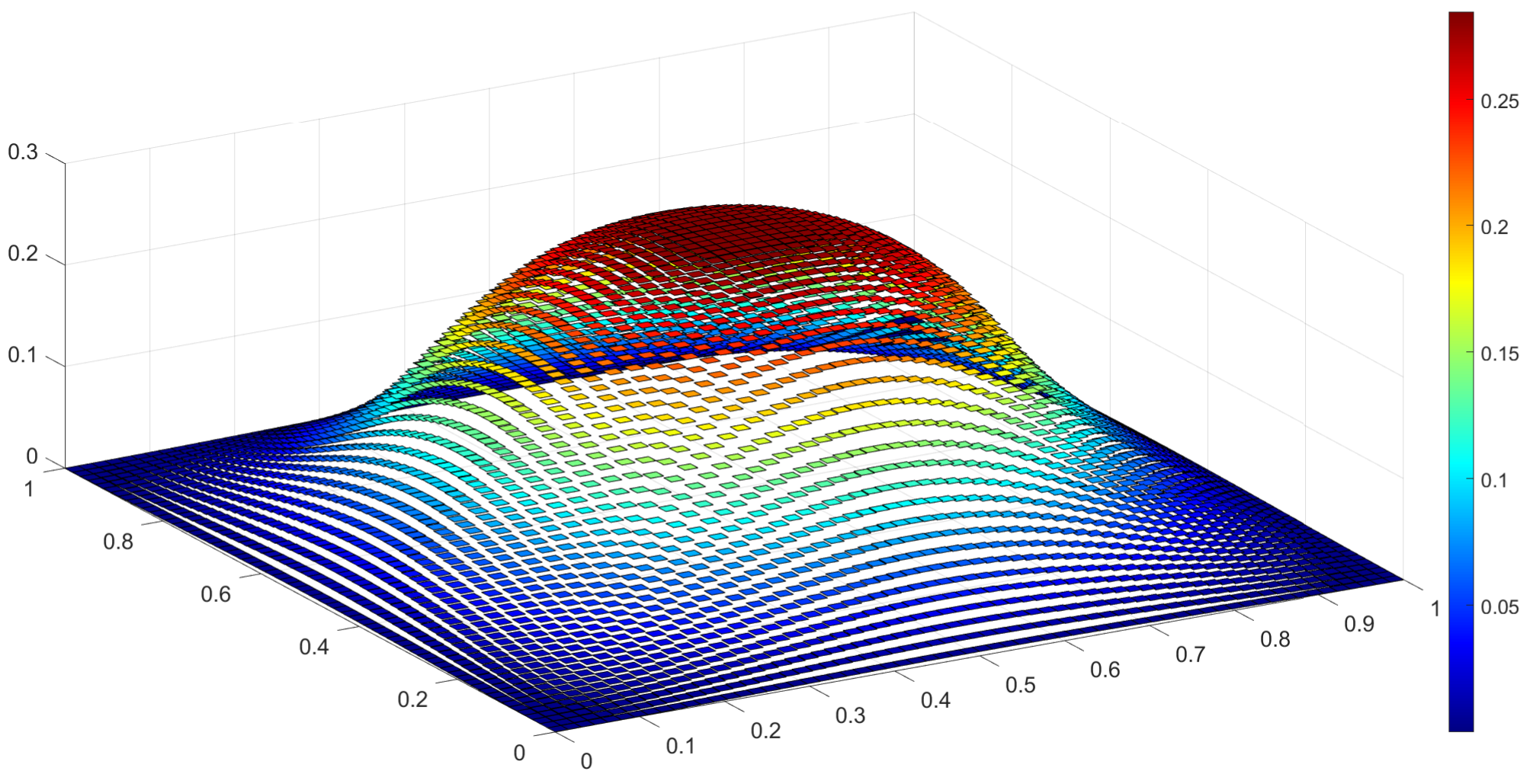}
        \put(45,-10){{{p'=\texttt{3.00}}}}
        \end{overpic}
    \end{subfigure}

    \vspace{1.5cm}
    
    \begin{subfigure}{0.45\textwidth}
        \centering
        \begin{overpic}[width=\linewidth]{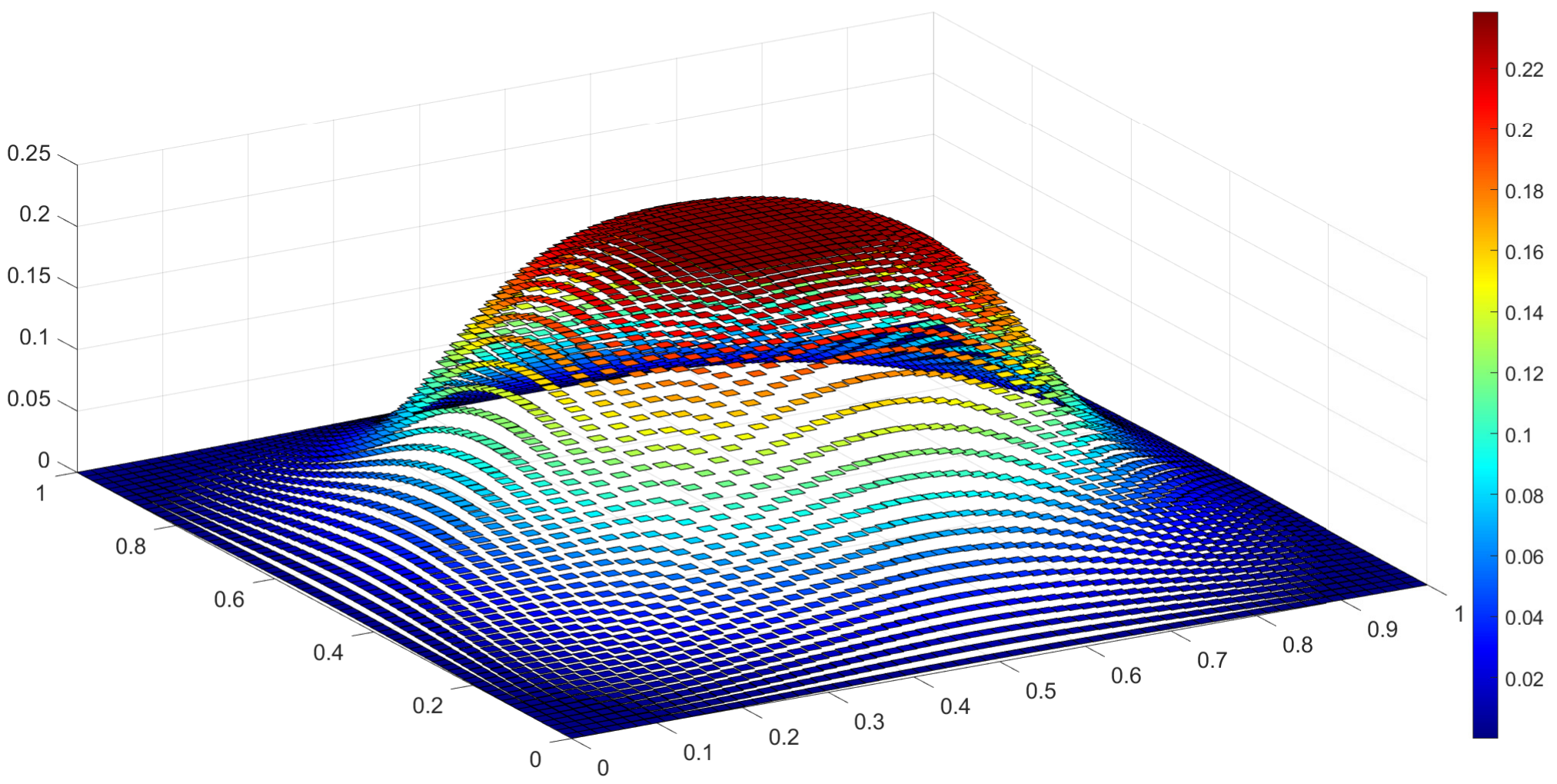}
        \put(45,-10){{{p'=\texttt{4.00}}}}
        \end{overpic}
    \end{subfigure}\qquad
    \begin{subfigure}{0.45\textwidth}
    \begin{overpic}[width=\linewidth]{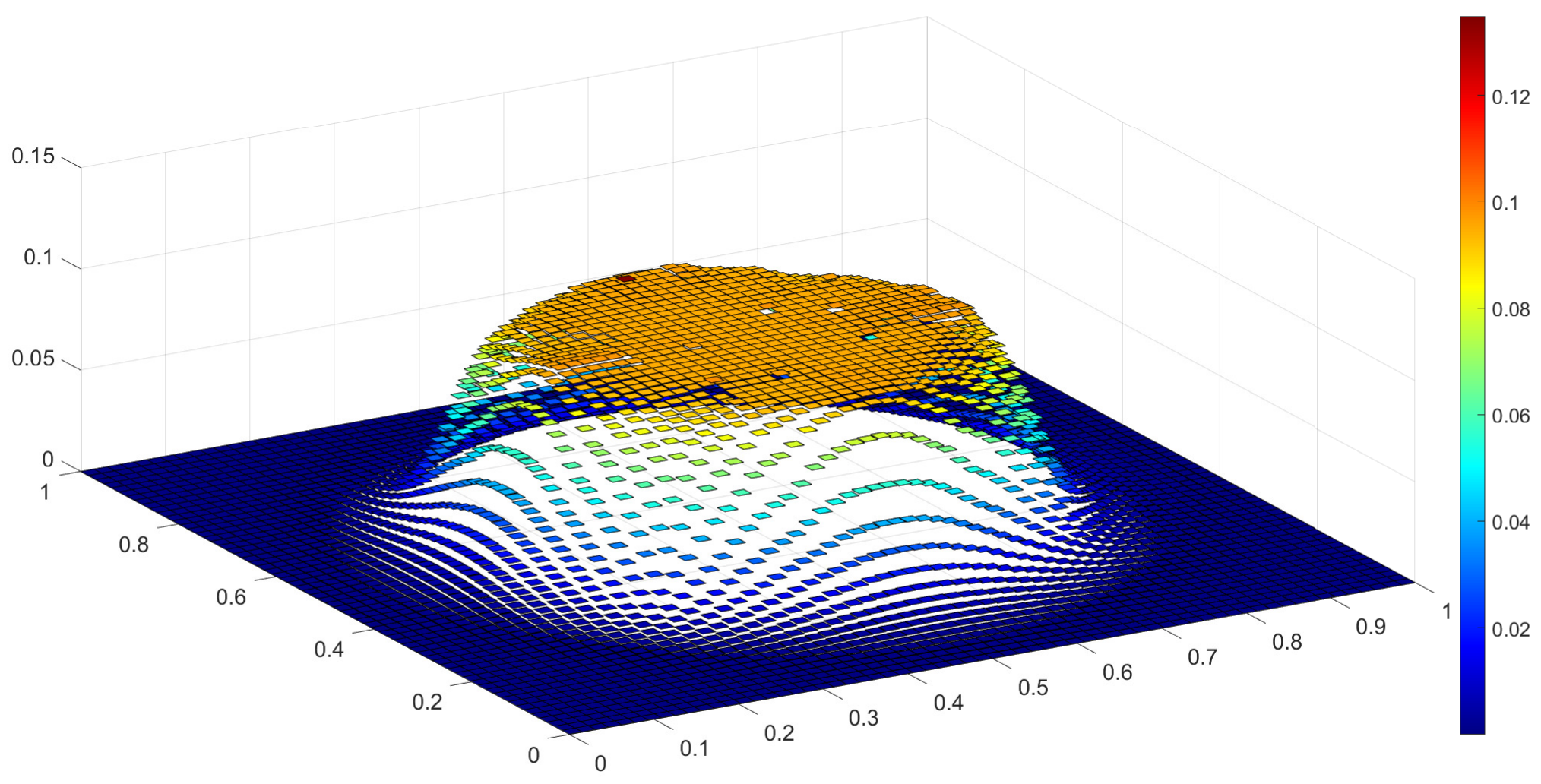}
        \put(45,-10){{{p'=\texttt{11.00}}}}
        \end{overpic}
    \end{subfigure}    

\vspace{1.5cm}

   \caption{Test 2. Discrete pressures $\discScalarSol$ obtained on a Cartesian tessellation of $(0, 1)^2$ with $64 \times 64$ elements. The values of $p'$ are reported in the plot.} 
    \label{fig:test2}
\end{figure}
\section*{Acknowledgements}
KBH has been partially funded by the European Union (ERC Synergy, NEMESIS, project number 101115663). Views and opinions expressed are however those of the authors only and do not necessarily reflect those of the European Union or the European Research Council Executive Agency.

\printbibliography

\appendix
\section{Proofs of technical lemmas}
\subsection{Auxiliary estimate for the normal component}\label{app:proof.r-norm.normal.component}
%
\begin{proof}[Proof of Lemma~\ref{lem:r-norm.normal.component}]
  Fix $\element \in \domain_h$, $\discVectFunOne \in \discVectSpace(\element)$ and $1<r \leq 2$.
  Let $\mathcal{T}_h(\element)$ be the shape-regular triangulation of $\element$ (see Remark \ref{lem.element.triangulation}). Notice that for each edge $\edge \in \edgeElement$ there
  exists a unique $T_{\edge} \in \mathcal{T}_h(\element)$ such that $\edge = \partial T_{\edge} \cap \partial \element$.
  Then, it is sufficient to show that for any $\edge \in \edgeElement$ it holds
  \begin{equation}\label{eq:r-norm.normal.component.local}
    \norm{\lebSpace{r}{\edge}}{\discVectFunOne \cdot \normalEdge}
    \lesssim
    h_{T_{\edge}}^{-\frac{1}{r}} \norm{\lebSpace{r}{T_{\edge}}}{ \discVectFunOne} + h_{T_{\edge}}^{1-\frac{1}{r}} \norm{\lebSpace{r}{T_{\edge}}}{\divergence \discVectFunOne},
  \end{equation}
  and then sum over $\edge \in \edgeElement$ to get the desired result.

  \medskip
  \noindent
  \textbf{Proof of \eqref{eq:r-norm.normal.component.local}}.
  Let $\edge \in \edgeElement$ and let $b_{\edge} \in \polySpace{2}(\edge)$ be the quadratic bubble function over $\edge$ such that $\norm{\lebSpace{\infty}{\edge}}{b_{\edge}} = 1$.
  Applying a $(\frac{2}{r},\frac{2}{2-r})$-H{\"o}lder inequality together with a standard bubble argument leads to
  \begin{equation}\label{eq:r-norm.normal.component.sup.new}
    \norm{\lebSpace{r}{\edge}}{\discVectFunOne \cdot \normalEdge} 
    \leq
    C(\polyDegree,\rho) h_{T_{\edge}}^{\frac{1}{r}-\frac{1}{2}} 
    \sup_{q^{\polyDegree+2} \in \polySpace{\polyDegree+2}(\edge)} \frac{\int_{\edge} (\discVectFunOne \cdot
    \normalEdge) b_{\edge}q^{\polyDegree+2} }{\norm{\lebSpace{2}{\edge}}{q^{\polyDegree+2}} }.
  \end{equation}
  Let $b_{\partial T_{\edge}} \in \polySpace{3}(T_{\edge})$ be the unique cubic function over $T_{\edge}$ that matches $b_{\edge}$ on $\edge$ and vanishes on the other two edges of $T_{\edge}$, and let $\hat{q}^{\polyDegree+2}$ be the extension (constant along lines parallel to the other edges) of $q^{\polyDegree+2}$. 
  Integration by parts followed by a $(r,r')$-H{\"o}lder inequality yields
  \begin{equation}\label{eq:r-norm.normal.component.sup.num.new}
    \begin{aligned}
      \int_{\edge} (\discVectFunOne \cdot \normalEdge) b_{\edge}q^{\polyDegree+2}
      &\leq
      \norm{\lebSpace{r}{T_{\edge}}}{\discVectFunOne}
      \norm{\lebSpace{r'}{T_{\edge}}}{\grad (b_{\partial T_{\edge}} \hat{q}^{\polyDegree+2})}
      +
      \norm{\lebSpace{r}{T_{\edge}}}{\divergence \discVectFunOne}
      \norm{\lebSpace{r'}{T_{\edge}}}{b_{\partial T_{\edge}} \hat{q}^{\polyDegree+2}} \\ 
      & \leq 
      C(r,\polyDegree,\rho) 
      \left( h_{T_{\edge}}^{\frac{2}{r'}-\frac{3}{2}} \norm{\lebSpace{r}{T_{\edge}}}{\discVectFunOne}
      +h_{T_{\edge}}^{\frac{2}{r'}-\frac{1}{2}} \norm{\lebSpace{r}{T_{\edge}}}{\divergence \discVectFunOne} \right) \norm{\lebSpace{2}{\edge}}{q^{\polyDegree+2}}
    \end{aligned}
  \end{equation}
  where the last inequality follows by applying \eqref{eq:grad.poly.inv.est} and \eqref{eq:poly.sob.embedding}.
  The desired result follows by substituting \eqref{eq:r-norm.normal.component.sup.num.new} into \eqref{eq:r-norm.normal.component.sup.new}, and noticing that $\frac{1}{r} - \frac{1}{2} + \frac{2}{r'} - \frac{3}{2} = -\frac{1}{r}$ and $\frac{1}{r} - \frac{1}{2} + \frac{2}{r'} - \frac{1}{2} = 1 - \frac{1}{r}$. 
\end{proof}
%
\subsection{Auxiliary estimate for the stabilization term}\label{app:proof.sE.approx.for.u}
%
\begin{proof}[Proof of Lemma~\ref{lem:sE.approx.for.u}]
  We start by splitting $s_{\element}(\discVectSol^{\perp}, \discVectSol^{\perp})$ as
  \begin{equation}\label{eq:sE.approx.for.u.initial}
    \begin{aligned}
      s_{\element}(\discVectSol^{\perp}, \discVectSol^{\perp})
      =
      \left[s_{\element}(\discVectSol^{\perp}, \discVectSol^{\perp}) 
      - s_{\element}((\vectProjF {\polyDegree} \vectSol)^{\perp}, \discVectSol^{\perp}) \right] 
      + s_{\element}((\vectProjF {\polyDegree} \vectSol)^{\perp}, \discVectSol^{\perp}) \eqqcolon \mathfrak{T}_{1,\element} + \mathfrak{T}_{2,\element}.
    \end{aligned}
  \end{equation}
  We bound each term separately.

  \medskip
  \noindent
  \textbf{First term.}
  Recalling the definition of $s_{\element}$ in \eqref{eq:def.sE}, we apply \eqref{eq:diffusive.flux.properties.original.hc} and \eqref{eq:simple.eu.norm.equivalence} to write
  \begin{align*}
    \mathfrak{T}_{1,\element}
    &\leq h_{\element}^{2}
    C(\sobIndex) \left( \euNorm{\locBoundaryDofVect(\discVectSol^{\perp}) } + \euNorm{\locBoundaryDofVect(\boldsymbol{\xi}_h^{\perp})} \right)^{\sobIndexConj-2} \euNorm{\locBoundaryDofVect(\boldsymbol{\xi}_h^{\perp}) } \euNorm{\locBoundaryDofVect(\discVectSol^{\perp})},
  \end{align*}
  where $\boldsymbol{\xi}_h \coloneqq \discVectSol - \vectSol_{I} $.
  Next, we apply \eqref{eq:young.type.inequality} together with \eqref{eq:simple.eu.norm.equivalence} and \eqref{eq:diffusive.flux.properties.original.mono}, obtaining, for any $\varepsilon > 0$,
  \begin{align*}
    &h_{\element}^2 C(\sobIndex) \left( \euNorm{\locBoundaryDofVect(\discVectSol^{\perp}) } + \euNorm{\locBoundaryDofVect(\boldsymbol{\xi}_h^{\perp})} \right)^{\sobIndexConj-2} \euNorm{\locBoundaryDofVect(\boldsymbol{\xi}_h^{\perp}) } \euNorm{\locBoundaryDofVect(\discVectSol^{\perp})} \\
    &\leq
    2^{\sobIndexConj-2}\varepsilon s_{\element}(\discVectSol^{\perp}, \discVectSol^{\perp}) + C(\sobIndex,\varepsilon) \left( \euNorm{\locBoundaryDofVect(\discVectSol^{\perp}) } + \euNorm{\locBoundaryDofVect(\vectSol_{I}^{\perp})} \right)^{\sobIndexConj-2} \euNorm{\locBoundaryDofVect(\boldsymbol{\xi}_h^{\perp}) }^2 \\
    &\leq
    2^{\sobIndexConj-2}\varepsilon s_{\element}(\discVectSol^{\perp}, \discVectSol^{\perp}) + C(\sobIndex,\varepsilon) 
    \left(\fluxFunction_{N_{\element}} ( \locBoundaryDofVect(\discVectSol^{\perp})  ) - \fluxFunction_{N_{\element}} (\locBoundaryDofVect(\vectSol_{I}^{\perp})) \right) \cdot  \locBoundaryDofVect(\boldsymbol{\xi}_h^{\perp}).
  \end{align*}
  Therefore, choosing $\varepsilon = \frac{2^{2-\sobIndexConj}}{4}$ leads to
  \begin{equation}\label{eq:sE.approx.for.u.term1}
    \begin{aligned}
      \mathfrak{T}_{1,\element}
      &\leq
      \frac{1}{4} s_{\element}(\discVectSol^{\perp}, \discVectSol^{\perp}) + C(\sobIndex) \left[
      s_{\element}(\discVectSol^{\perp} , \boldsymbol{\xi}_h^{\perp} ) - s_{\element}( \vectSol_{I}^{\perp} , \boldsymbol{\xi}_h^{\perp} ) \right].
    \end{aligned}
  \end{equation}

  \medskip
  \noindent
  \textbf{Second term.}
  Applying \eqref{eq:diffusive.flux.properties.original.hc} and a generalized $(\sobIndex,\sobIndexConj)$-Young inequality, we get
  \begin{align*}
    \mathfrak{T}_{2,\element}
    &\leq
    h_{\element}^{2}
    C(\sobIndex)\euNorm{\locBoundaryDofVect((\vectProjF {\polyDegree} \vectSol)^{\perp})}^{\sobIndexConj-1} \euNorm{\locBoundaryDofVect(\discVectSol^{\perp})} \\ 
    &\leq
    \varepsilon 
    s_{\element}(\discVectSol^{\perp}, \discVectSol^{\perp}) 
    +
    C(\sobIndex,\varepsilon) s_{\element}((\vectProjF {\polyDegree} \vectSol)^{\perp}, (\vectProjF {\polyDegree} \vectSol)^{\perp}).
  \end{align*}
  Therefore, choosing $\varepsilon = \frac{1}{4}$, we obtain
  \begin{equation}\label{eq:sE.approx.for.u.term2}
    \begin{aligned}
      \mathfrak{T}_{2,\element}
      &\leq
      \frac{1}{4} s_{\element}(\discVectSol^{\perp}, \discVectSol^{\perp}) + C(\sobIndex)
      s_{\element}((\vectProjF {\polyDegree} \vectSol)^{\perp}, (\vectProjF {\polyDegree} \vectSol)^{\perp}).
    \end{aligned}
  \end{equation}

  \medskip
  \noindent
  \textbf{Conclusion.}
  Combining \eqref{eq:sE.approx.for.u.term1} and \eqref{eq:sE.approx.for.u.term2} with \eqref{eq:sE.approx.for.u.initial}, after basic algebraic manipulations, we arrive at
  \begin{equation}
    \begin{aligned}
      s_{\element}(\discVectSol^{\perp}, \discVectSol^{\perp})
      &\leq
      C(\sobIndex) \Bigg[ 
      s_{\element}((\vectProjF {\polyDegree} \vectSol)^{\perp}, (\vectProjF {\polyDegree} \vectSol)^{\perp}) + 
      \left[ s_{\element}(\discVectSol^{\perp} , \boldsymbol{\xi}_h^{\perp} ) - s_{\element}( \vectSol_{I}^{\perp} , \boldsymbol{\xi}_h^{\perp} ) \right]
      \Bigg].
    \end{aligned}
  \end{equation}
  Observing that the term $s_{\element}((\vectProjF {\polyDegree} \vectSol)^{\perp}, (\vectProjF {\polyDegree} \vectSol)^{\perp}) = h_{\element}^{2}\euNorm{\locBoundaryDofVect((\vectProjF {\polyDegree} \vectSol)^{\perp}) }^{\sobIndexConj}$ is the same term that appears in \eqref{eq:aux.sE.ineq.for.g2.case} when $ 1 < \sobIndexConj \leq 2$, and in \eqref{eq:error.estimate.g2.rhs.T2E.initial} when $\sobIndex > 2$, applying the same argument that leads to \eqref{eq:error.estimate.s2.rhs.T2E} and \eqref{eq:error.estimate.g2.rhs.T2E}, yields the desired result.
\end{proof}
%
\subsection{Auxiliary estimate for the flux function}\label{app:proof.sigma.approx.for.u}
%
\begin{proof}[Proof of Lemma~\ref{lem:sigma.approx.for.u}]
  We start by observing that, since $(\sobIndex-1)(\sobIndexConj -2) <0$ and $(\sobIndex-1) (\sobIndexConj -2)+(\sobIndex-2) = 0$, after applying \eqref{eq:diffusive.flux.properties.original.hc}, we can write
  \begin{equation}\label{eq:sigma.approx.for.u.initial}
    \begin{aligned}
      \norm{\vectLebSpace{\sobIndex}{\element}}{\fluxFunction(\vectSol) -\fluxFunction(\vectProjL{\polyDegree}\discVectSol)}^{\sobIndex}
      &\leq
      C(\sobIndex)
      \int_{\element}
      \left( \euNorm{\vectSol} +\euNorm{\vectProjL{\polyDegree}\discVectSol} \right)^{\sobIndex (\sobIndexConj -2)} \euNorm{\vectSol - \vectProjL{\polyDegree}\discVectSol}^{\sobIndex} \\
      &\leq
      C(\sobIndex)
      \int_{\element}
      \euNorm{\vectSol - \vectProjL{\polyDegree}\discVectSol}^{(\sobIndex-1) (\sobIndexConj -2)+(\sobIndex-2)}
      \left( \euNorm{\vectSol} +\euNorm{\vectProjL{\polyDegree}\discVectSol} \right)^{\sobIndexConj -2} \euNorm{\vectSol - \vectProjL{\polyDegree}\discVectSol}^{2}\\
      &=
      C(\sobIndex)
      \int_{\element}
      \left( \euNorm{\vectSol} +\euNorm{\vectProjL{\polyDegree}\discVectSol} \right)^{\sobIndexConj -2} \euNorm{\vectSol - \vectProjL{\polyDegree}\discVectSol}^{2}.
    \end{aligned}
  \end{equation}
  We apply \eqref{eq:diffusive.flux.properties.original.mono} and split the integral above as
  \begin{equation}\label{eq:sigma.approx.for.u.initial.initial}
    \begin{aligned}
      \int_{\element}
      \left( \euNorm{\vectSol} +\euNorm{\vectProjL{\polyDegree}\discVectSol} \right)^{\sobIndexConj -2} \euNorm{\vectSol - \vectProjL{\polyDegree}\discVectSol}^{2}
      &\leq C(\sobIndex) \Big[ 
      \int_{\element}
      \left( \fluxFunction(\vectSol) - \fluxFunction(\vectProjL{\polyDegree}\vectSol_{I}) \right) \cdot (\vectSol - \vectProjL{\polyDegree}\discVectSol) \\ 
      &\quad + \int_{\element}
      \left( \fluxFunction(\vectProjL{\polyDegree}\vectSol_{I}) - \fluxFunction(\vectProjL{\polyDegree}\discVectSol) \right) \cdot (\vectSol - \vectProjL{\polyDegree}\discVectSol) \Big] \\
      &\eqqcolon C(\sobIndex) \left( \mathfrak{T}_{1,\element} + \mathfrak{T}_{2,\element} \right).
    \end{aligned}
  \end{equation}
  We bound the two terms above separately.

  \medskip
  \noindent
  \textbf{First term.}
  Using \eqref{eq:diffusive.flux.properties.original.hc} and \eqref{eq:young.type.inequality} together with \eqref{eq:simple.eu.norm.equivalence}, we obtain for any $\varepsilon > 0$,
  \begin{equation}\label{eq:sigma.approx.for.u.term.1}
    \begin{aligned}
      \mathfrak{T}_{1,\element}
      &\leq
      C(\sobIndex)
      \int_{\element}
      \left( \euNorm{\vectSol} + \euNorm{\vectSol - \vectProjL{\polyDegree}\vectSol_{I}} \right)^{\sobIndexConj - 2} \euNorm{\vectSol -
      \vectProjL{\polyDegree}\vectSol_{I}} \euNorm{\vectSol - \vectProjL{\polyDegree}\discVectSol} \\ 
      &\leq \varepsilon \int_{\element} \left( \euNorm{\vectSol} + \euNorm{\vectSol - \vectProjL{\polyDegree}\discVectSol} \right)^{\sobIndexConj - 2} \euNorm{\vectSol - \vectProjL{\polyDegree}\discVectSol}^2 \\ 
      &\quad 
      + C(\sobIndex,\varepsilon) \int_{\element} \left( \euNorm{\vectSol} + \euNorm{\vectSol - \vectProjL{\polyDegree}\vectSol_{I}} \right)^{\sobIndexConj - 2} \euNorm{\vectSol -
      \vectProjL{\polyDegree}\vectSol_{I}}^2 \\
      &\leq 2^{2-\sobIndexConj}\varepsilon \int_{\element} \left( \euNorm{\vectSol} + \euNorm{\vectProjL{\polyDegree}\discVectSol} \right)^{\sobIndexConj - 2} \euNorm{\vectSol -
      \vectProjL{\polyDegree}\discVectSol}^2 + C(\sobIndex,\varepsilon) \int_{\element} \euNorm{\vectSol - \vectProjL{\polyDegree}\vectSol_{I}}^{\sobIndexConj} \\ 
      &\leq 2^{2-\sobIndexConj} \varepsilon
      \int_{\element} \left( \euNorm{\vectSol} + \euNorm{\vectProjL{\polyDegree}\discVectSol} \right)^{\sobIndexConj - 2} \euNorm{\vectSol - \vectProjL{\polyDegree}\discVectSol}^2 +
      C(\polyDegree,\sobIndex,\polyDegree_{\vectSol},\rho,\varepsilon) h_{\element}^{\sobIndexConj \polyDegree_{\vectSol}}
      \seminorm{\vectSobSpace{\polyDegree_{\vectSol}}{\sobIndexConj}{\element}}{\vectSol}^{\sobIndexConj},
    \end{aligned}
  \end{equation}
  where the last inequality follows from \eqref{eq:projL.Fortin.approx}.

  \medskip
  \noindent
  \textbf{Second term.}
  Similarly, for the second term, we apply \eqref{eq:diffusive.flux.properties.original.hc}, \eqref{eq:young.type.inequality} and \eqref{eq:diffusive.flux.properties.original.mono} together with \eqref{eq:simple.eu.norm.equivalence} to write
  \begin{equation}\label{eq:sigma.approx.for.u.term.2}
    \begin{aligned}
      \mathfrak{T}_{2,\element}
      &\leq
      C(\sobIndex)
      \int_{\element}
      \left( \euNorm{\vectProjL{\polyDegree}\discVectSol} + \euNorm{\vectProjL{\polyDegree}\boldsymbol{\xi}_h } \right)^{\sobIndexConj - 2}
      \euNorm{\vectProjL{\polyDegree}\boldsymbol{\xi}_h } \euNorm{\vectSol - \vectProjL{\polyDegree}\discVectSol} \\ 
      &\leq 
       \varepsilon \int_{\element} \left( \euNorm{\vectProjL{\polyDegree}\discVectSol} + \euNorm{\vectSol - \vectProjL{\polyDegree}\discVectSol} \right)^{\sobIndexConj - 2}
      \euNorm{\vectSol - \vectProjL{\polyDegree}\discVectSol}^2 
      \\ 
      &\quad 
      + C(\sobIndex,\varepsilon) \int_{\element} \left( \euNorm{\vectProjL{\polyDegree}\discVectSol} +
      \euNorm{\vectProjL{\polyDegree}\boldsymbol{\xi}_h } \right)^{\sobIndexConj - 2} \euNorm{\vectProjL{\polyDegree}\boldsymbol{\xi}_h }^2 \\
      &\leq
        2^{2-\sobIndexConj} \varepsilon \int_{\element} \left( \euNorm{\vectSol} + \euNorm{\vectProjL{\polyDegree}\discVectSol} \right)^{\sobIndexConj - 2} \euNorm{\vectSol - \vectProjL{\polyDegree}\discVectSol}^2 \\ 
        &\quad+ C(\sobIndex,\varepsilon) \left[ a_{\element}(\vectProjL{\polyDegree}
        \discVectSol,\vectProjL{\polyDegree}\boldsymbol{\xi}_h) - a_{\element}(\vectProjL{\polyDegree}\vectSol_{I},\vectProjL{\polyDegree}\boldsymbol{\xi}_h) \right],
    \end{aligned}
  \end{equation}
  where $\boldsymbol{\xi}_h \coloneqq \discVectSol - \vectSol_{I}$. 

  \medskip
  \noindent
  \textbf{Conclusion.}
  Plugging \eqref{eq:sigma.approx.for.u.term.1} and \eqref{eq:sigma.approx.for.u.term.2} into \eqref{eq:sigma.approx.for.u.initial.initial}, choosing $\varepsilon = \frac{2^{\sobIndexConj-2}}{4 C_{1}(\sobIndex)}$ where $C_{1}(\sobIndex)$ is the positive constant in \eqref{eq:sigma.approx.for.u.initial.initial}, after basic algebraic manipulations, we get
  \begin{align*}
    \int_{\element}
    \left( \euNorm{\vectSol} +\euNorm{\vectProjL{\polyDegree}\discVectSol} \right)^{\sobIndexConj -2} \euNorm{\vectSol - \vectProjL{\polyDegree}\discVectSol}^{2}
    &\leq
    C(\polyDegree,\sobIndex,\polyDegree_{\vectSol},\rho) h_{\element}^{\sobIndexConj \polyDegree_{\vectSol}}
    \seminorm{\vectSobSpace{\polyDegree_{\vectSol}}{\sobIndexConj}{\element}}{\vectSol}^{\sobIndexConj} \\ 
    &\quad + C(\sobIndex) \left[ a_{\element}(\vectProjL{\polyDegree}
    \discVectSol,\vectProjL{\polyDegree}\boldsymbol{\xi}_h) - a_{\element}(\vectProjL{\polyDegree}\vectSol_{I},\vectProjL{\polyDegree}\boldsymbol{\xi}_h) \right].
  \end{align*}
  Combining this last estimate with \eqref{eq:sigma.approx.for.u.initial} leads to the desired result.
\end{proof}
\end{document}